\documentclass[a4paper, 12pt]{article}

\usepackage[utf8]{inputenc}
\usepackage[T1]{fontenc}
\usepackage[a4paper, margin=2.35cm]{geometry}
\usepackage{needspace}

\usepackage{libertine} 
\usepackage{inconsolata} 

\usepackage{amsmath, amsthm, amssymb}
\usepackage{graphicx}
\usepackage{enumerate}
\usepackage{enumitem}
\usepackage{authblk}
\usepackage[textsize=small, textwidth=2cm, color=yellow]{todonotes}
\usepackage{thmtools, mathtools}
\usepackage{thm-restate}
\colorlet{darkishGreen}{green!50!black}
\colorlet{darkishBlue}{blue!60!black}
\usepackage[colorlinks=true, citecolor=darkishGreen, linkcolor=darkishBlue, urlcolor=darkishBlue, bookmarks,hyperfootnotes=false, psdextra=true,hypertexnames=false]{hyperref}
\usepackage{float}
\usepackage{pifont}

\usepackage{tikz}
\usetikzlibrary{shadings, calc}

\usepackage[skip=6pt]{subcaption} 

\usepackage[nameinlink, capitalise, noabbrev]{cleveref}

\makeatletter
\newcommand\thankssymb[1]{\textsuperscript{\@fnsymbol{#1}}}
\makeatother

\renewenvironment{abstract}
{\small\vspace{-1em}
\begin{center}
\bfseries\abstractname\vspace{-.5em}\vspace{0pt}
\end{center}
\list{}{
\setlength{\leftmargin}{0.6in}
\setlength{\rightmargin}{\leftmargin}}
\item\relax}
{\endlist}

\declaretheorem[name=Theorem, numberwithin=section]{theorem}
\declaretheorem[name=Lemma, sibling=theorem]{lemma}
\declaretheorem[name=Proposition, sibling=theorem]{proposition}
\declaretheorem[name=Definition, sibling=theorem]{definition}
\declaretheorem[name=Corollary, sibling=theorem]{corollary}
\declaretheorem[name=Conjecture, sibling=theorem]{conjecture}

\declaretheorem[name=Claim, sibling=theorem]{claim}

\declaretheorem[name=Observation, style=remark, sibling=theorem]{observation}

\def\cqedsymbol{\ifmmode$\lrcorner$\else{\unskip\nobreak\hfil
\penalty50\hskip1em\null\nobreak\hfil$\lrcorner$
\parfillskip=0pt\finalhyphendemerits=0\endgraf}\fi}

\crefname{claim}{Claim}{Claims}

\newenvironment{claimproof}{\noindent\textit{Proof.}}{\hfill\ensuremath{\includegraphics[height=0.8em]{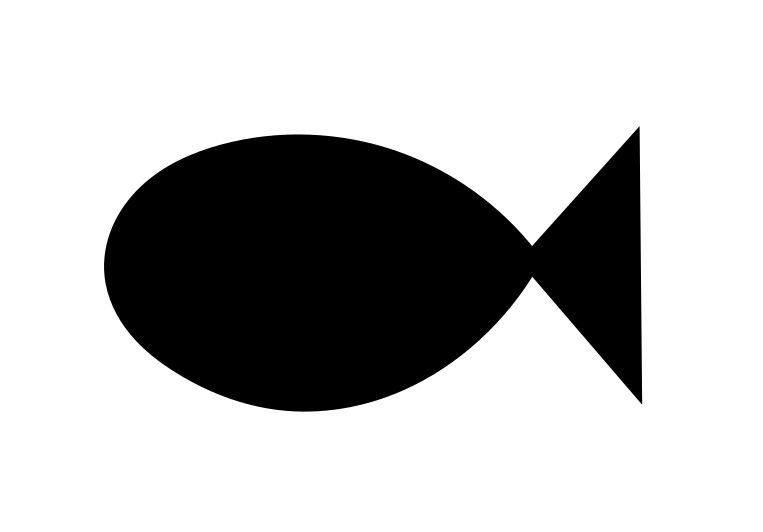}}\medskip}
\usepackage{etoolbox}

\newcommand{ \N } { \mathbb{N} }

\colorlet{darkishRed}{red!60!black}
\newcommand{\defn}[1]{{\color{darkishRed}{\emph{#1}}}}
\newcommand{\defnm}[1]{{\color{darkishRed}{#1}}}

\newcommand{\EP}{Erd\H{o}s-P\'{o}sa}

\let\le\leqslant
\let\ge\geqslant
\let\leq\leqslant
\let\geq\geqslant

\DeclareMathOperator{\diam}{diam}
\DeclareMathOperator{\rad}{rad}

\newcommand{\td}{tree-decom\-posi\-tion}
\newcommand{\tds}{tree-decom\-posi\-tions}
\newcommand{\gd}{graph-decom\-posi\-tion}
\newcommand{\fd}{forest-decom\-posi\-tion}

\usepackage{ifpdf}

\hypersetup{pdftitle={Happy hippos}}

\ifpdf

\title{The coarse Erd\H{o}s-P\'{o}sa theorem}

\author{Sandra Albrechtsen\thanks{Supported by the Alexander von Humboldt Foundation in the framework of the Alexander von Humboldt Professorship of Daniel Král' endowed by the Federal Ministry of Education and Research.}$^{\phantom{1}{\includegraphics[height=.6\baselineskip]{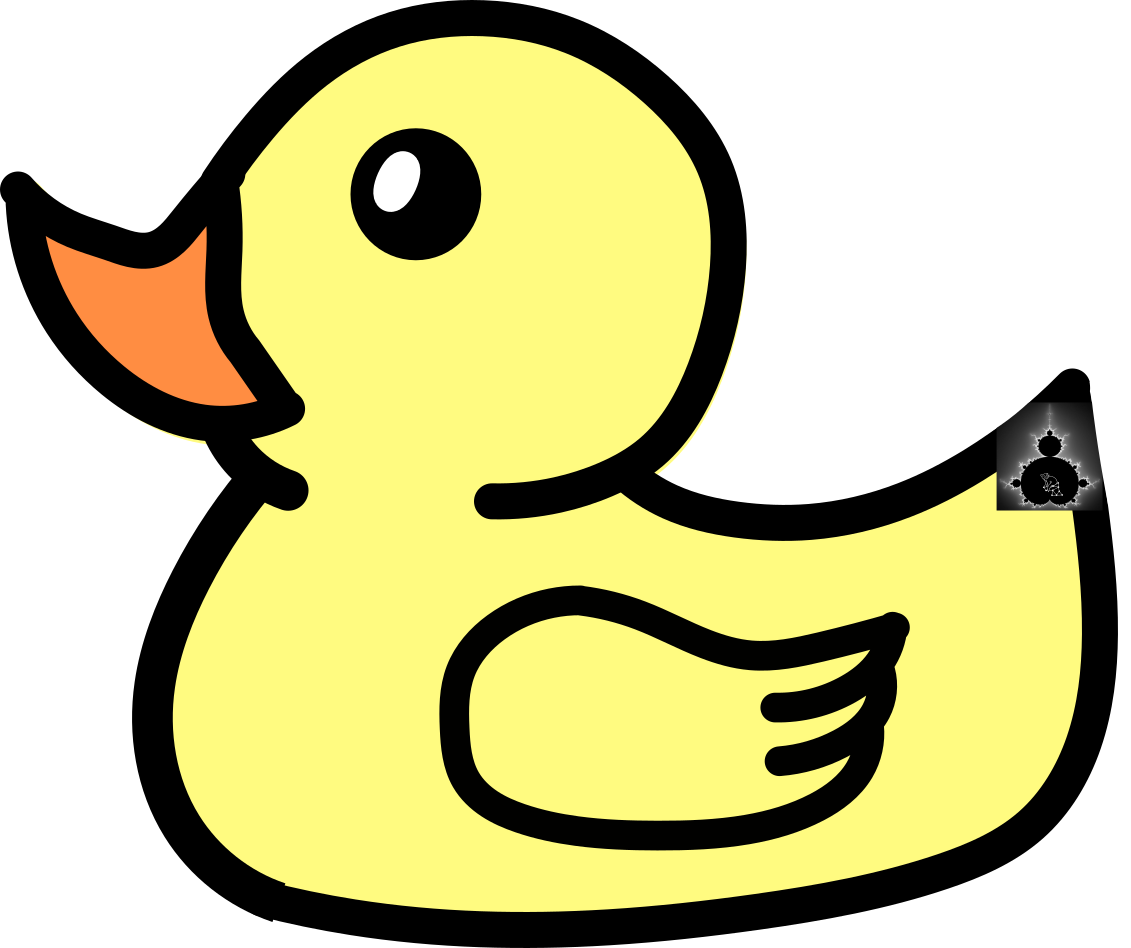}}}$}
\author{Marthe Bonamy\thanks{Supported by the ANR Project Mimétique (Metric Minors in Graphs, ANR-25-CE48-4089).}$^{\phantom{2}{\includegraphics[height=.45\baselineskip]{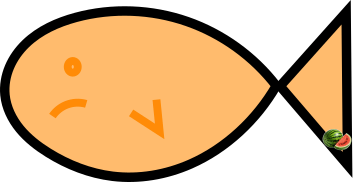}}}$}
\author{Romain Bourneuf\thankssymb{2}$^{{\includegraphics[height=.45\baselineskip]{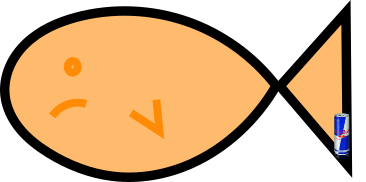}}}$}
\author{James Davies\thankssymb{1}$^{{\includegraphics[height=.6\baselineskip]{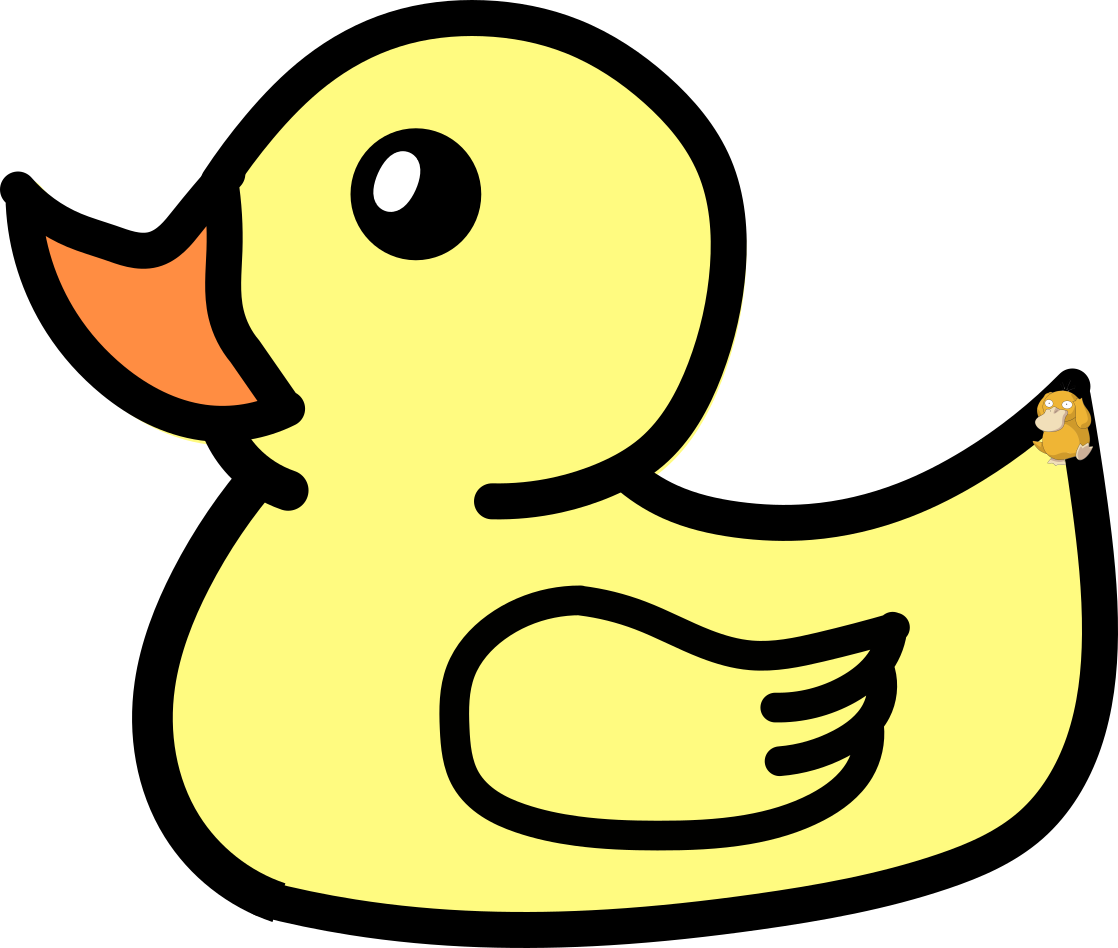}}}$}

\affil{$^{\includegraphics[height=.6\baselineskip]{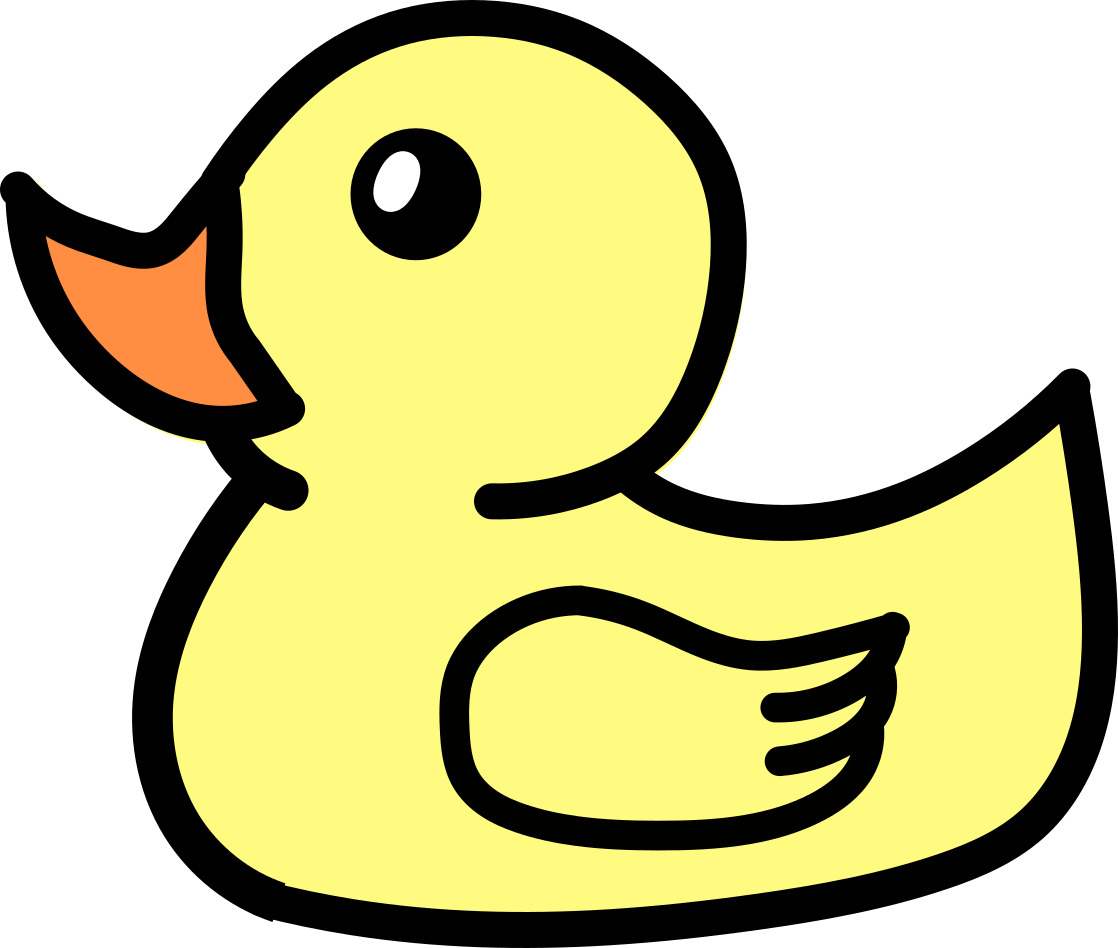}}$ Institute of Mathematics, Leipzig University, Augustusplatz 10, 04109 Leipzig, Germany}
\affil{$^{{\includegraphics[height=.45\baselineskip]{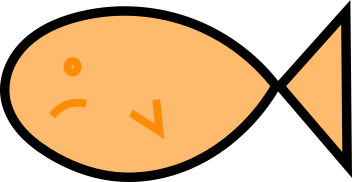}}}$ CNRS, LaBRI, Université de Bordeaux, Bordeaux, France.}

\date{}

\else

\title{The coarse Erd\H{o}s-P\'{o}sa theorem}

\author{Sandra Albrechtsen}
\affil{Institute of Mathematics, Leipzig University, Augustusplatz 10, 04109 Leipzig, Germany\\
*Supported by the Alexander von Humboldt Foundation in the framework of the Alexander von Humboldt Professorship of Daniel Král' endowed by the Federal Ministry of Education and Research.}

\author{Marthe Bonamy}
\affil{CNRS, LaBRI, Université de Bordeaux, Bordeaux, France.\\
$\dagger$ Supported by the ANR Project Mimétique (Metric Minors in Graphs, ANR-25-CE48-4089).}

\author{Romain Bourneuf\thankssymb{2}}
\affil{CNRS, LaBRI, Université de Bordeaux, Bordeaux, France.}

\author{James Davies\thankssymb{1}}
\affil{Institute of Mathematics, Leipzig University, Augustusplatz 10, 04109 Leipzig, Germany}

\date{}

\fi

\begin{document}

\maketitle

\begin{abstract}



We prove the coarse Erd\H{o}s-P\'{o}sa conjecture of Georgakopoulos and Papasoglu.
Informally, any graph either contains many fat cycles that are pairwise far apart, or there is a small number of bounded radius balls that together hit all of them.
To be more precise, if $G$ is a graph with no $q$-fat model of $k \cdot K_3$ for some $q, k \in \N$,
then there is a set $X\subseteq V(G)$ of $\mathcal{O}(k\log k)$ vertices such that every $q$-fat model of $K_3$ in~$G$ has distance $\mathcal{O}(q)$ from $X$.

In another form more closely resembling Manning's theorem that characterises quasi-trees: if $G$ is a graph with no $q$-fat model of $k \cdot K_3$ for some $q, k \in \N$,
then $G$ is $\mathcal{O}(q)$-quasi-isometric to a graph $H$ that contains a set $X\subseteq V(H)$ of $\mathcal{O}(k\log k)$ vertices such that $H-X$ is a forest.

Bootstrapping this result, we further prove that every graph with no $q$-fat model of $k \cdot K_3$ is quasi-isometric to a graph with no $k \cdot K_3$ minor, where the quasi-isometry can be chosen to only have additive distortion.
This result also holds for $k = \infty$.

We also obtain an Erd\H{o}s-P\'{o}sa theorem for long induced cycles that are far apart.
\end{abstract}

\section{A Long-expected Party}\label{sec:intro}

\sloppy

It is a simple observation that a graph is a forest if and only if it contains no cycles.
The \EP~theorem~\cite{EPTheorem} characterises when a graph $G$ is close to being a forest in the sense that either $G$ contains $k$ (vertex-)disjoint cycles (witnessing that $G$ is far from being a forest), or there is a set $X\subseteq V(G)$ of $\mathcal{O}(k\log k)$ vertices hitting all cycles (witnessing that $G$ is close to being a forest, since $G-X$ is a forest).
Here, the bound of $\mathcal{O}(k\log k)$ is best possible \cite{EPTheorem}.

Since the \EP~theorem from 1965~\cite{EPTheorem}, there has been extensive work on \EP-type results. The main direction has been to impose additional constraints on the cycles, including a version for cycles of length at least~$\ell$ \cite{AKAnticompleteLongS,EPLongCircuits,EPLongCyclesPrescribedSet,EPLongCyclesTighterFct,EPLongCyclesTightFct}, cycles of length~$\ell$ modulo~$z$ \cite{EPCycleGroup,EPCycleOddPrescribed,EPCycleModPrescribedSet,EPCyclesConnGroup,EPCycleOddConn,thomassen1988presence,EPCycleOddConn2,EPCycleMod}, cycles that are pairwise far apart \cite{ahn2025coarse,AKAnticompleteLongS,CDGKMMS,CMPRRAntiCompleteEPforLongHoles,DJMMDistanceErdosPosa}, and cycles intersecting a prescribed set~$S$ \cite{AKAnticompleteLongS,EPLongCyclesPrescribedSet,EPCycleOddPrescribed,EPCycleModPrescribedSet,EPCyclePrescribed,EPCyclePrescribed2}. Furthermore, \EP-type problems have been studied well beyond cycles, including e.g.\ an extension of the \EP~theorem to planar minors \cite{robertson1986graph,EPPlanar} and a half-integral version for all minors \cite{liu2022packing}.
For more \EP-type results, see \cite{RaymondDynamic}.
\medskip

\emph{Fat minors} are a coarse geometric analogue of the classical notion of graph minors (see \cref{subsec:FatMinors} for the definition). This notion belongs to a more general framework of studying large-scale properties of graphs from a geometric perspective, called \emph{coarse graph theory}, paralleling Gromov's approach to geometric group theory \cite{GroAsyInv,DruKapBook}. See \cite{AJKWFatK4,BLPPSep,HIMWAsDimPlanarFat,McMFat,NSSFatTree}
for a selection of results on fat minors, and \cite{GPCoarseGT} for further discussion and motivation.

Precursing the now blossoming field of coarse graph theory\footnote{Some of these ideas can in fact be traced back to the 90s~\cite{brandstadt1999distance}.}, Manning \cite{Manning05} (see \cite{GPCoarseGT}) characterised quasi-trees as being exactly the connected graphs (or length spaces) forbidding $K_3$ as a fat minor.
Manning's theorem~\cite{Manning05} (or its equivalent original bottleneck property version) has seen numerous applications in geometric group theory and coarse geometry, see for example \cite{bestvina2015constructing,benjamini2022triangulations,behrstock2017hierarchically,fujiwara2007note,hagen2014weak,Manning05,margolis2024coarse}.

The main result of this paper provides the \EP\ analogue of Manning's theorem~\cite{Manning05} and confirms the coarse \EP\ conjecture of Georgakopoulos and Papasoglu \cite[Conjecture~9.7]{GPCoarseGT}.
We state two versions, which will both follow from a more technical result (see  \cref{thm:CoarseEP:General}).
One is in terms of quasi-isometries as in Manning's theorem \cite{Manning05}.
The second is in terms of a collection of balls hitting all fat $K_3$ minor models, exactly as in the coarse \EP\ conjecture of Georgakopoulos and Papasoglu \cite[Conjecture~9.7]{GPCoarseGT}.
All of our results hold for infinite graphs.

\begin{restatable}{theorem}{main} \label{th:main}
    There exist a constant $\lambda_{\ref{th:main}}$ and a function $f_{\ref{th:main}}: \N \to \N$ such that the following holds.
    Let $q, k \in \N$ and let $G$ be a graph with no $q$-fat model of $k \cdot K_3$.
    Then $G$ is $(\lambda_{\ref{th:main}} \cdot q)$-quasi-isometric to a graph $H$ that contains a set $X\subseteq V(H)$ of size at most $f_{\ref{th:main}}(k)$ such that $H-X$ is a forest.

    Furthermore, we can take $f_{\ref{th:main}} \in \mathcal{O}(k\log k)$.
\end{restatable}

\begin{restatable}{theorem}{CoarseErdosPosa} \label{main:CoarseErdosPosa}
    There exist a constant $d_{\ref{main:CoarseErdosPosa}}$ and a function $f_{\ref{main:CoarseErdosPosa}}: \N \to \N$ such that the following holds.
    Let $q, k \in \N$ and let $G$ be a graph with no $q$-fat model of $k \cdot K_3$.
    Then there is a set $X\subseteq V(G)$ of size at most $f_{\ref{main:CoarseErdosPosa}}(k)$ such that every $q$-fat model of $K_3$ in $G$ is at distance at most $d_{\ref{main:CoarseErdosPosa}}\cdot q$ from $X$.

    Furthermore, we can take $f_{\ref{main:CoarseErdosPosa}} \in \mathcal{O}(k\log k)$.
\end{restatable}

We remark that in contrast to the \EP~theorem \cite{EPTheorem}, which can be extended to planar minors \cite{robertson1986graph}, this is not possible for \cref{main:CoarseErdosPosa} \cite{ADWeakCounterex}.
We will discuss this matter further in \cref{sec:discuss}.

\medskip

It is known that long cycles have the \EP\ property \cite{robertson1986graph,thomassen1988presence}. Ahn, Gollin, Huynh, and Kwon \cite{ahn2025coarse} proved an anticomplete\footnote{Here, anticomplete means pairwise disjoint and non-adjacent.} version of the \EP~theorem: if a graph $G$ does not contain $k$ anticomplete cycles then there is a set $X\subseteq V(G)$ of $\mathcal{O}(k\log k)$ vertices such that $G-B_G(X,1)$ is a forest.
Very recently, Czy{\.z}ewska, Masa{\v{r}}{\'\i}k, Ma.~Pilipczuk, Reinald,  and Rz{\k{a}}{\.z}ewski \cite{CMPRRAntiCompleteEPforLongHoles} transposed it to long induced cycles: a graph~$G$ either contains $k$ anticomplete induced cycles of length at least $\ell$, or there is a set $X\subseteq V(G)$ of $\mathcal{O}(\ell k\log k)$ vertices such that $G-B_G(X,1)$ has no induced cycle of length $\ell$ or more.
On a similar note, Chudnovsky, Dujmović, Joret, R.~Kaul, Micek, Morin, and Scott \cite{CDGKMMS} proved that for some constant $C$, a graph $G$ either contains $k$ cycles of length at least $\ell$ that are pairwise at distance at least $q$, or there exists a set $X\subseteq V(G)$ of $\mathcal{O}(\ell k\log k)$ vertices such that $G-B_G(X,Cq)$ has no cycle of length $\ell$ or more.

From our more technical result \cref{thm:CoarseEP:General}, we also deduce the following two \EP-type theorems for far apart long cycles and long induced cycles.

\begin{theorem} \label{maincor:FarApartEPForLongCycles}
    There exist a constant $d_{\ref{maincor:FarApartEPForLongCycles}}$ and a function $f_{\ref{maincor:FarApartEPForLongCycles}}: \N \to \N$ such that the following holds. Let $q, k,\ell \in \N$ and let $G$  be a graph that does not contain $k$ cycles of length at least~$\ell$ that are pairwise at distance at least~$q$. Then there is a set $X \subseteq V(G)$ of size at most~$f_{\ref{maincor:FarApartEPForLongCycles}}(k)$ such that every cycle of length at least~$\ell$ is at distance at most $d_{\ref{maincor:FarApartEPForLongCycles}} \cdot \max\{\ell,q\}$ from~$X$.

    Furthermore, we can take $f_{\ref{maincor:FarApartEPForLongCycles}} \in \mathcal{O}(k\log k)$.
\end{theorem}

\begin{theorem} \label{maincor:FarApartEPForLongInducedCycles}
    There exist a constant $d_{\ref{maincor:FarApartEPForLongInducedCycles}}$ and a function $f_{\ref{maincor:FarApartEPForLongInducedCycles}}: \N \to \N$ such that the following holds. Let $q, k,\ell \in \N$ and let $G$  be a graph that does not contain $k$ induced cycles of length at least~$\ell$ that are pairwise at distance at least~$q$. Then there is a set $X \subseteq V(G)$ of size at most~$f_{\ref{maincor:FarApartEPForLongInducedCycles}}(k)$ such that every induced cycle of length at least~$\ell$ is at distance at most $d_{\ref{maincor:FarApartEPForLongInducedCycles}} \cdot \max\{\ell,q\}$ from~$X$.

    Furthermore, we can take $f_{\ref{maincor:FarApartEPForLongInducedCycles}} \in \mathcal{O}(k\log k)$.
\end{theorem}

\cref{maincor:FarApartEPForLongCycles} differs from the main theorem of Chudnovsky et al.~\cite{CDGKMMS}, where the size of the hitting set $X$ also depends (linearly) on $\ell$.
However, neither result appears to directly imply the other, as our version (necessarily) displays some dependence on $\ell$ for the required radius of the balls hitting the cycles of length at least~$\ell$, when $q< \ell$, while \cite{CDGKMMS} does not. \cref{maincor:FarApartEPForLongInducedCycles} transposes \cref{maincor:FarApartEPForLongCycles} to the case of long induced cycles rather than long cycles.
We remark that we use the $\ell=3$ case of the main theorem of Chudnovsky et al.~\cite{CDGKMMS} (or the earlier \cite{DJMMDistanceErdosPosa}, though it yields worse bounds) in our proof of \cref{thm:CoarseEP:General}.
\medskip

Georgakopoulos and Papasoglu \cite{GPCoarseGT} conjectured that for every graph $H$, the graphs with no $q$-fat $H$ minor are quasi-isometric\footnote{In a slight abuse of notation, we mean here ``the class of [...] is quasi-isometric to the class of [...]'': this is the case throughout this paper.} to graphs with no $H$ minor.

\begin{conjecture}[\cite{GPCoarseGT}, disproved \cite{DHIMFatCounterexample}]\label{conj:qi}
    For every finite graph $H$, there exists a function $f_{\ref{conj:qi}}^H$ such that the following holds. Let $q \in \mathbb{N}$ and let $G$ be a graph that does not contain $H$ as a $q$-fat minor. Then $G$ is $f_{\ref{conj:qi}}^H(q)$-quasi-isometric to a graph $G'$ with no $H$ minor.
\end{conjecture}

It has been confirmed for several small graphs: notably when $H$ is a cycle \cite{Manning05}, a star \cite{GPCoarseGT}, $K_4$~\cite{AJKWFatK4}, and $K_{2,t}$~\cite{ADGFatK2t}. However, the conjecture is false \cite{DHIMFatCounterexample,ADGSmallCounterexFatMinorConj} even in the weaker setting of being quasi-isometric to some $H'$-minor-free graph \cite{ADWeakCounterex}.

We bootstrap \Cref{th:main} into confirming \Cref{conj:qi} when $H$ is a disjoint union of triangles, which in turn (as established in \cite[Lemma~5.3]{GPCoarseGT}) confirms it when $H$ is a disjoint union of cycles.
In fact, we can choose the quasi-isometry so that it has only additive error. This is consistent with Manning's theorem~\cite{Manning05}, which admits an analogous refinement~\cite{chepoi2012constant} (see also  \cite{berger2024bounded,GPCoarseGT,kerr2023tree}). Papasoglu and Swenson~\cite{papasoglu2026additive} also very recently proved the analogous refinement for graphs with no fat $K_{2,3}$ minor.
See~\cite{chepoi2012constant} for applications of such quasi-isometries having only an additive error.

\begin{theorem}\label{thm:mainFatkK3additive}
    There exists a function $f_{\ref{thm:mainFatkK3additive}} : \N^2 \to \N$ such that the following holds.
    Let $q,k\in \N$ and let $G$ be a graph with no $q$-fat model of $k\cdot K_3$.
    Then $G$ is $(1,f_{\ref{thm:mainFatkK3additive}}(k,q))$-quasi-isometric to a graph with no $k\cdot K_3$ minor.
\end{theorem}

Although \cref{th:main} gives a slightly less precise structure than \cref{thm:mainFatkK3additive}, let us highlight that the remarkable feature of \Cref{th:main} over \cref{thm:mainFatkK3additive} is that the constants for the quasi-isometry bounds in \cref{th:main} do not depend on the number $k$ of disjoint cycles forbidden as a fat minor.
Therefore, \cref{th:main} points to an example of an infinite graph ($H = \infty \cdot K_3$) for which \Cref{conj:qi} holds.
As with \cref{thm:mainFatkK3additive}, we further refine this so that the quasi-isometry has only additive error.

\begin{restatable}{theorem}{coarseEPinfadd} \label{main:CoarseEp:Infadd}
    There exists a function $f_{\ref{main:CoarseEp:Infadd}} : \N \to \N$ such that the following holds.
    Let $q \in \N$ and let $G$ be a graph with no $q$-fat model of $\infty \cdot K_3$.
    Then, $G$ is $(1, f_{\ref{main:CoarseEp:Infadd}}(q))$-quasi-isometric to a graph with no $\infty \cdot K_3$ minor.
\end{restatable}

One of the motivations for Georgakopoulos and Papasoglu to introduce \Cref{conj:qi} was that it would resolve a problem of Bonamy, Bousquet, Esperet, Groenland, Liu, Pirot, and Scott~\cite{BBEGLPSAsymptoticDimMinorClosed} asking if graphs with no fat $H$~minor have bounded asymptotic dimension. This is because quasi-isometries preserve asymptotic dimension \cite{bell2008asymptotic} and Bonamy et al.~\cite{BBEGLPSAsymptoticDimMinorClosed} proved that $H$-minor-free graphs have asymptotic dimension at most~$2$. For $H$ planar they even obtain an asymptotic dimension bound of~$1$.
Following this, there has been significant work on proving bounded asymptotic dimension for various classes of graphs; see e.g.\
\cite{CCTZ26,davies2025string,distel2026proper,dvovrak2025asymptotic,liu2025assouad,papasoglu2023polynomial}.
In particular, Hickingbotham, Illingworth, Micek, and Walczak (see \cite{HIMWAsDimPlanarFat}) recently made a breakthrough by proving that for any planar graph~$H$, the graphs with no fat $H$ minor have bounded asymptotic dimension. Bounded asymptotic dimension can be useful in such varied contexts as graph colouring~\cite{abrishami2026burling} or distributed graph algorithms~\cite{bonamy2025local, bonamy2026meta}.

If a class $\mathcal{C}$ has asymptotic dimension at most $d$, then\footnote{This is straightforward from the definition.} so does the class of graphs that are within bounded number of bounded radius balls deletion away from being in $\mathcal{C}$. Therefore, the class of graphs $G$ containing some vertex set $X\subseteq V(G)$ such that $|X|\le k$ and $G-X$ is a forest has asymptotic dimension at most~$1$.
Thus, by \cref{th:main} we obtain the following strengthened version of \cite{HIMWAsDimPlanarFat} in the special case when $H$ is a disjoint union of cycles.

\begin{corollary}\label{cor:asdim1}
    For every $q,k\in \N$,
    the class of graphs with no $q$-fat model of $k \cdot K_3$ has asymptotic dimension at most 1.
\end{corollary}

By \cref{main:CoarseEp:Infadd}, we also have the following infinite version (note that this is for fixed infinite graphs, not a class of graphs).

\begin{corollary}\label{cor:asdim2}
    Every graph with no $q$-fat model of $\infty \cdot K_3$ for some $q \in \mathbb{N}$ has asymptotic dimension at most 1.
\end{corollary}

We remark that our results on forbidding fat minors also hold more generally for length spaces. For proper geodesic metric spaces (which include the geometric realisations of connected locally finite graphs), we also obtain from this the following extension of Manning's theorem~\cite{Manning05}.

\begin{restatable}{theorem}{manningextension}\label{main:manningextension}
    Let $G$ be a length space.
    Then the following are equivalent.
    \begin{itemize}
        \item $G$ is quasi-isometric to a graph $H$ containing a vertex $v$ such that $H-v$ is a tree.
        \item For some $q>0$, $G$ does not contain $2 \cdot K_3$ as a $q$-fat minor.
        \item For some $q>0$, $G$ does not contain $\infty \cdot K_3$ as a $q$-fat minor.
    \end{itemize}
    If in addition $G$ is a proper geodesic metric space, then the above is also equivalent to
    \begin{itemize}
        \item $G$ is quasi-isometric to a tree.
        \item For some $q>0$, $G$ does not contain $K_3$ as a $q$-fat minor.
    \end{itemize}
\end{restatable}

For length spaces or even geodesic metric spaces, the five bullets are not equivalent, as can be seen by considering the (infinite) graph obtained by taking a cycle of every length and joining them all on a single vertex.

\paragraph{How this paper is organised.}
In \cref{sec:prelim}, we give necessary preliminaries.
In \cref{sec:ProofSketch} we give a high-level proof sketch. We then show the coarse \EP\ theorems, \cref{main:CoarseErdosPosa,th:main}, whose (mostly combined) proof spans \cref{sec:ApexForests,sec:HappyHippos,sec:Untangling,sec:Smooshing,sec:CoralReef,sec:FinalProof}. A more detailed overview of \cref{sec:ApexForests,sec:HappyHippos,sec:Untangling,sec:Smooshing,sec:CoralReef,sec:FinalProof} can be found at the end of \cref{sec:ProofSketch}.
The statement of our aforementioned technical main result (\cref{thm:CoarseEP:General}), from which \cref{main:CoarseErdosPosa,th:main} follow, can be found in \cref{sec:FinalProof}. There, we also give the proofs of our corollaries about long cycles, \cref{maincor:FarApartEPForLongCycles,maincor:FarApartEPForLongInducedCycles}.
\cref{sec:BattleOfTheBlackGate} contains the proofs of the fat minor conjecture for disjoint unions of cycles, \cref{thm:mainFatkK3additive,main:CoarseEp:Infadd}.  Finally, \cref{sec:lengthspaces} contains a discussion about length spaces together with the proof of \cref{main:manningextension}, and \cref{sec:discuss} contains a discussion about open problems.

\section{Preliminaries}\label{sec:prelim}

All graphs in this paper are possibly infinite unless specified to be finite.

We set $\defnm{\mathbb{N}} = \{1, 2, \ldots\}$.
For every $n \in \N$, we denote by \defn{$[n]$} the set $\{1, 2, \ldots, n\}$ and by $\defnm{\mathbb{Z}_n}$ the set $\mathbb{Z}/n\mathbb{Z}$.

\subsection{Distance, radius and balls}

Let $G$ be a graph.
We write~\defn{$d_G(u, v)$} for the distance between two vertices~$u$ and~$v$ in~$G$, namely the number of edges in a shortest path between $u$ and $v$ (and $\infty$ if there is no path between $u$ and $v$).
For two sets~$U$ and~$U'$ of vertices of~$G$, we write~\defn{$d_G(U, U')$} for the minimum distance between two elements of~$U$ and~$U'$, respectively.
For two subgraphs $G_1, G_2$ of $G$, we abbreviate $d_G(V(G_1), V(G_2))$ to~\defn{$d_G(G_1, G_2)$}, and similarly for the later notions.

Given a set~$U$ of vertices of~$G$, the \defn{ball (in~$G$) around~$U$ of radius $r \ge 0$}, denoted~\defn{$B_G(U, r)$}, is the set of all vertices of~$G$ at distance at most~$r$ from~$U$ in~$G$.

The \defn{radius}~\defn{$\rad(G)$} of~$G$ is the smallest number~$k \in \{0\} \cup \N \cup \{\infty\}$ such that there exists some vertex~$u \in V(G)$ with~$d_G(u, v) \le k$ for every vertex~$v$ of $G$. If $G$ is empty, we define its radius to be~$0$.
We remark that if $G$ is disconnected and nonempty, then its radius is $\infty$.
Note that~$G$ has radius at most~$k \in \mathbb{N}$ if and only if there is some vertex~$u$ of $G$ with~$V(G) = B_G(u, k)$.
Additionally, if $U \subseteq V(G)$, then the \defn{radius of $U$ in $G$}, denoted $\defnm{\rad_G(U)}$ is the smallest number $k \in \N$ such that there exists some vertex $v$ of $G$ (not necessarily in $U$) with $U \subseteq B_G(v, k)$ or $\infty$ if such a $k \in \N$ does not exist.

\subsection{Fat minors} \label{subsec:FatMinors}

Let $G, X$ be graphs.
A \defn{model} of $X$ in $G$ is a pair $(\mathcal V,\mathcal E)$ satisfying the following conditions: \begin{itemize}
    \item $\mathcal{V}=(V_x : x \in V(X))$ is a family of pairwise disjoint subsets of $V(G)$, such that $G[V_x]$ is connected for every $x\in V(X)$;
    \item $\mathcal{E}=(E_e : e \in E(X))$ is a family of pairwise internally disjoint paths in $G$ such that, for every edge $e=x_0x_1\in E(X)$, the path $E_e$ has one end in $V_{x_0}$ and the other in~$V_{x_1}$, and is disjoint from all other $V_x$.
\end{itemize}
The $V_x$ are its \defn{branch sets} and the $E_e$ are its \defn{branch paths}.
A model $(\mathcal{V}, \mathcal{E})$ of $X$ in $G$ is \defn{$q$-fat} for $q \in \N$ if $d_G(Y,Z) \geq q$ for every two distinct $Y,Z \in \mathcal{V} \cup \mathcal{E}$ unless $Y = E_e$ and $Z = V_x$ for some vertex $x \in V(X)$ incident to $e \in E(X)$, or vice versa.
If $G$ contains a ($q$-fat) model of~$X$, we say that $X$ is a \defn{($q$-fat) minor} of $G$.

The following observation is easy to check.

\begin{observation} \label{lem:fat-K3-cycle}
    Let $G$ be a graph, $X \subseteq V(G)$ and $q \in \N$. Then the following are equivalent:
    \begin{itemize}
        \item There is a $q$-fat model of $K_3$ in $G$ whose branch sets and branch paths are contained in~$G[X]$.
        \item There exists a cycle $C$ in $G[X]$ which can be divided into six paths $P_0, \ldots, P_5$ such that $d_G(P_i, P_j) \geq q$ whenever $i, j \in \mathbb{Z}_6$ are distinct and non-consecutive. \qed
    \end{itemize}
\end{observation}

We call such a cycle $C$ a \defn{$q$-fat cycle}, and say that such paths $P_0, \ldots, P_5$ \defn{witness} that $C$ is a $q$-fat cycle.
Because of \cref{lem:fat-K3-cycle}, we will frequently jump back and forth between $q$-fat models of $K_3$ and $q$-fat cycles.

\subsection{Quasi-isometries and graph-decompositions}

For~$M \in \mathbb{R}_{\geq 1}$ and~$A \in \mathbb{R}_{\geq 0}$, an \defn{$(M,A)$-quasi-isometry} from a graph~$G$ to a graph~$H$ is a map~$\phi\colon V(G) \to V(H)$ such that:

\begin{enumerate}[label=(Q\arabic*)]
    \item \label{quasiisom:1} $M^{-1} \cdot d_G(u, v) - A \leq d_H(\phi(u),\phi(v)) \leq M \cdot d_G(u,v) + A$ for every~$u,v \in V(G)$, and
    \item \label{quasiisom:2} for every vertex $h$ of $H$, there exists a vertex $v$ of $G$ such that $d_H(h,\phi(v)) \leq A$.
\end{enumerate}
We say that $G$ is \defn{$(M,A)$-quasi-isometric} to $H$ if there exists an $(M,A)$-quasi-isometry from $G$ to $H$.
If $\lambda \in \mathbb{R}_{\geq 1}$, we say that $G$ is \defn{$\lambda$-quasi-isometric} to $H$ if $G$ is $(\lambda, \lambda)$-quasi-isometric to~$H$.
\smallskip

The majority of the time we shall find it more convenient to work with graph-decompositions rather than quasi-isometries.
Graph-decompositions are a natural extension of tree-decompositions, where we allow the decomposition to be indexed by an arbitrary graph, instead of simply a tree.
Formally, a \defn{\gd} of a graph $G$ is a pair $(H, \mathcal{V})$ where $H$ is a graph and $\mathcal{V} = (V_h : h \in V(H))$ is a family of subsets of vertices of $G$, called \defn{bags}, which satisfies the following properties:
\begin{enumerate}[label=\rm{(H\arabic*)}]
    \item \label{itm:H1} $G = \bigcup_{h \in V(H)}G[V_h]$,
    \item \label{itm:H2} For every $v \in V(G)$, the subgraph $\defnm{H_v} \coloneqq H[\{h \in V(H) : v \in V_h\}]$ of $H$ is connected.
\end{enumerate}
We also say that $(H, \mathcal{V})$ is an \defn{$H$-decomposition} of $G$.
We refer to $H_v$ as the \defn{copart} of $v$ in $(H, \mathcal{V})$.
By \cite[Lemma~3.2]{ADEFJKW23}, every \gd\ not only satisfies \ref{itm:H2} but also
\begin{enumerate}[label=\rm{(H2$^\prime$)}]
    \item \label{itm:H2'} For every connected subgraph $Y$ of $G$, the subgraph\\ $\defnm{H_Y} := H[\{h \in V (H) : V (Y) \cap V_h \neq \emptyset\}]$ of $H$ is connected.
\end{enumerate}

For a \gd\ $(H, \mathcal{V})$ of a graph $G$, its \defn{(outer-)radial width} is $\max_{h \in V(H)}\rad_{G}(V_h)$ and its \defn{(inner-)radial spread} is $\max_{v \in V(G)}\rad(H_v)$.

The graph-decomposition is \defn{honest} if for every $h \in V(H)$ we have $V_h \neq \emptyset$, and for every edge $hh' \in E(H)$ we have $V_h \cap V_{h'} \neq \emptyset$.
Observe that if $(H, \mathcal{V})$ is a graph-decomposition of~$G$, then there exists a subgraph $H'$ of $H$ such that $(H', \mathcal{V})$ is an honest graph-decomposition of~$G$, with radial width and radial spread not larger than those of $(H, \mathcal{V})$.

Given a set $X \subseteq V(G)$, a \defn{partial \gd} of $G$ with \defn{support} $X$ is a graph-decomposition of $G[X]$.
The radial width and radial spread are defined similarly. Importantly, for the radial width, the radius is still measured in $G$.
For brevity, given $r_1, r_2 \in \N$, we call a (partial) graph-decomposition \defn{$(r_1,r_2)$-radial} if its (outer-)radial width is at most~$r_1$ and its (inner-)radial spread is at most~$r_2$.

The existence of graph-decompositions with bounded radial width and radial spread is equivalent to the existence of quasi-isometries with bounded parameters, as follows.

\begin{proposition}[{\cite[Proposition~1.2]{ADEFJKW23}}] \label{prop:q.i.-graph-dec}
There exist functions $\lambda_{\ref{prop:q.i.-graph-dec}} : \mathbb{N}^2 \to \mathbb{N}$ and $w_{\ref{prop:q.i.-graph-dec}}, s_{\ref{prop:q.i.-graph-dec}} : \mathbb{N} \to \mathbb{N}$ such that the following holds for all graphs $G, H$:
\begin{enumerate}[label=\rm{(\roman*)}]
    \item \label{itm:q.i.-graph-dec:DecToQI} If $G$ admits an honest $(r_1, r_2)$-radial $H$-decomposition then $G$ is $\lambda_{\ref{prop:q.i.-graph-dec}}(r_1, r_2)$-quasi-isometric to $H$.
    \item \label{itm:q.i.-graph-dec:QItoDec} If $G$ is $\lambda$-quasi-isometric to $H$, then $G$ admits an honest $(w_{\ref{prop:q.i.-graph-dec}}(\lambda), s_{\ref{prop:q.i.-graph-dec}}(\lambda))$-radial $H$-decom\-position.
\end{enumerate}
\end{proposition}

Given the above connection between quasi-isometries and honest graph-decompositions of bounded radial width and radial spread, it is natural to expect that the distances between vertices in the decomposed graph~$G$ are closely related to the distances in the indexing graph~$H$. This is made precise by the following two lemmas.

\begin{lemma} \label{lem:dist-H-decomp}
    Let $(H, \mathcal{V})$ be an $(r_1, r_2)$-radial graph-decomposition of a graph $G$. Then, for all disjoint sets $X, Y \subseteq V(G)$, we have \[d_H(H_X, H_Y) \leq 2r_2 \cdot (d_G(X, Y)-1).\]
    If $(H, \mathcal{V})$ is honest, we also have \[d_G(X, Y) \leq 2r_1 \cdot (d_H(H_X, H_Y) + 1).\]
\end{lemma}

\begin{proof}
    Let $P$ be a shortest $X$--$Y$ path in $G$, write $P \eqqcolon  v_0v_1 \ldots v_\ell$.
    For every $i \in [\ell]$, there exists a node $h_i \in V(H)$ such that $v_{i-1}, v_i \in V_{h_i}$.
    For every $i \in [\ell-1]$, we have $v_i \in V_{h_i} \cap V_{h_{i+1}}$.
    Since $(H, \mathcal{V})$ has radial spread at most $r_2$, this implies $d_H(h_i, h_{i+1}) \leq 2r_2$.
    Then, $d_H(H_X, H_Y) \leq d_H(h_1, h_{\ell}) \leq 2r_2 \cdot (\ell-1) = 2r_2 \cdot (d_G(X, Y) - 1)$.

    For the second inequality, let $Q$ be a shortest $H_X$--$H_Y$ path in $H$, write $Q \eqqcolon h_0 h_1 \ldots h_\ell$.
    There exists $v_0 \in X \cap V_{h_0}$ and $v_{\ell+1} \in Y \cap V_{h_\ell}$.
    Since $(H, \mathcal{V})$ is honest, for every $i \in [\ell]$, there exists $v_i \in V_{h_{i-1}} \cap V_{h_i}$.
    For every $i \in \{0, \ldots, \ell\}$, we have $v_{i}, v_{i+1} \in V_{h_i}$.
    Since $(H, \mathcal{V})$ has radial width at most $r_1$, this implies $d_G(v_i, v_{i+1}) \leq 2r_1$.
    Thus, $d_G(X, Y) \leq d_G(v_0, v_{\ell+1}) \leq 2r_1 \cdot (\ell+1) = 2r_1 \cdot (d_H(H_X, H_Y)+1)$.
\end{proof}

\begin{lemma}\label{lem:far-apart-lifts-far-apart}
    Let $(H, \mathcal{V})$ be an honest $(\cdot, r_2)$-radial graph-decomposition of a graph $G$.
    Let $d \in \mathbb{N}$ and let $A, B \subseteq V(H)$ such that $d_H(A, B) \geq f_{\ref{lem:far-apart-lifts-far-apart}}(d, r_2) \coloneqq 2r_2(d+1)$.
    Then, \[d_G\left(\bigcup_{h \in A}V_h, \bigcup_{h \in B}V_h\right) \geq d.\]
    Moreover, if the first inequality is strict, so is the second.
\end{lemma}

\begin{proof}
    Let $P$ be a path in $G$ between $\bigcup_{h \in A}V_h$ and $\bigcup_{h \in B}V_h$.
    Write $P \eqqcolon v_0v_1\ldots v_\ell$ with $v_0 \in \bigcup_{h \in A}V_h$ and $v_{\ell} \in \bigcup_{h \in B}V_h$.
    By \cref{lem:dist-H-decomp}, $d_H(H_{v_0}, H_{v_{\ell}}) \leq 2r_2(\ell-1)$.
    Since $v_0 \in \bigcup_{h \in A}V_h$, we have $H_{v_0} \cap A \neq \emptyset$, and similarly $H_{v_{\ell}} \cap B \neq \emptyset$, so $d_H(A, B) \leq d_H(H_{v_0}, H_{v_{\ell}}) + 4r_2$.
    Thus, $2r_2(d+1) \leq d_H(A, B) \leq d_H(H_{v_0}, H_{v_{\ell}}) + 4r_2 \leq 2r_2(\ell+1)$.
    Therefore, $\ell \geq d$.
    Since $P$ was arbitrary, this concludes the proof of the first part of the statement.

    The `Moreover' part follows immediately from the proof.
\end{proof}

Just as the composition of two quasi-isometries is still a quasi-isometry, we can also compose graph-decompositions. We actually prove a slightly stronger result.

\begin{lemma} \label{herd-of-hyenas}
    Let $(G', \mathcal{V}')$ be an honest $(r'_1, r'_2)$-radial graph-decomposition of a graph $G$.
    Let $X' \subseteq V(G')$, and let $(H, \mathcal{V})$ be an honest $(r_1, r_2)$-radial partial graph-decomposition of $G'$ with support $V(G'-X')$.
    Let $Y \subseteq V(G)$ such that $G'_y \cap X' = \emptyset$ for every $y \in Y$.
    Then, there exists a subgraph $H'$ of $H$ and an honest partial $H'$-decomposition of $G$ with support $Y$, radial width $w_{\ref{herd-of-hyenas}}(r'_1, r_1) \coloneqq 2r'_1(2r_1+1)$ and radial spread $s_{\ref{herd-of-hyenas}}(r'_2, r_2) \coloneqq 2r_2(2r'_2+1)$.
\end{lemma}

\begin{proof}
    We compose the two graph-decompositions in the natural way: for every node $h \in V(H)$, let $\defnm{U_h} := \big(\bigcup_{w' \in V_h}V'_{w'}\big) \cap Y$.
    We show that $(H, \mathcal{U})$ is a $(w_{\ref{herd-of-hyenas}}(r'_1,r_1), s_{\ref{herd-of-hyenas}}(r'_2,r_2))$-radial partial graph-decomposition of $G$ with support $Y$, which concludes the proof.
    \medskip

    We first show that $(H, \mathcal{U})$ is a partial graph-decomposition of $G$ with support $Y$; that is, that $(H, \mathcal{U})$ is a graph-decomposition of $G[Y]$.

    \ref{itm:H1}: Let $y \in Y$ be arbitrary.
    There exists a node $v' \in V(G')$ such that $y \in V'_{v'}$.
    Since $y \in Y$, we have $G'_y \cap X' = \emptyset$, and so $v' \notin X'$.
    Thus, there exists a node $h \in V(H)$ such that $v' \in V_h$.
    Then, $y \in \bigcup_{w' \in V_h}V'_{w'}$ and $y \in Y$, so $y \in \big(\bigcup_{w' \in V_h}V'_{w'}\big) \cap Y = U_h$.
    Similarly, if $y_1y_2 \in E(G)$ with $y_1, y_2 \in Y$, there exists a node $h \in V(H)$ such that $\{y_1, y_2\} \subseteq U_h$.

    \ref{itm:H2}: Let $y \in Y$ be arbitrary.
    Then, $G'_{y}$ is a connected subgraph of $G'$ which is disjoint from~$X'$, so $G'_y$ is a connected subgraph of $G'-X'$.
    By \ref{itm:H2'}, $H_{G'_y}$ induces a connected subgraph of $H$.
    However, by definition we have $H_y = H_{G'_y}$.

    This concludes the proof that $(H, \mathcal{U})$ is a graph-decomposition of $G[Y]$.
    \medskip

    Next, we show that the radial width of $(H, \mathcal{U})$ is at most~$w_{\ref{herd-of-hyenas}}(r'_1,r_1)$. For this, let $h \in V(H)$ be arbitrary and let $y_1, y_2 \in U_h$.
    By definition of $U_h$, there exist $v'_1, v'_2 \in V_h$ such that $y_1 \in V'_{v'_1}$ and $y_2 \in V'_{v'_2}$.
    Since $(H, \mathcal{V})$ has radial width at most~$r_1$, we have $d_{G'}(v'_1, v'_2) \leq 2r_1$.
    Thus, $d_{G'}(G'_{y_1}, G'_{y_2}) \leq 2r_1$.
    By \cref{lem:dist-H-decomp}, this implies $d_G(y_1, y_2) \leq 2r'_1 \cdot (d_{G'}(G'_{y_1}, G'_{y_2}) + 1) \leq 2r'_1(2r_1+1) = w_{\ref{herd-of-hyenas}}(r'_1,r_1)$.
    \medskip

    Finally, we show that the radial spread of $(H, \mathcal{U})$ is at most $s_{\ref{herd-of-hyenas}}(r'_2,r_2)$. For this, let $y \in Y$ be arbitrary and let $h_1, h_2 \in H_y$.
    By definition of $\mathcal{U}$, there exist $v'_1, v'_2 \in G'_y$ such that $h_1 \in H_{v'_1}$ and $h_2 \in H_{v'_2}$.
    Since $(G', \mathcal{V}')$ has radial spread at most $r'_2$, we have $d_{G'_y}(v'_1, v'_2) \leq 2r'_2$.
    Observe that the restriction of $(H_y, \mathcal{V})$ to $G'_y$ is a $(\cdot, r_2)$-radial graph-decomposition of $G'_y$.
    By \cref{lem:dist-H-decomp}, we then have $d_{H_y}(H_{v'_1}, H_{v'_2}) \leq 2r_2 \cdot (d_{G'_y}(v'_1, v'_2) - 1) \leq 2r_2(2r'_2-1)$.
    Finally, using that $H_{v'_1}$ and $H_{v'_2}$ have radius at most $r_2$ and are contained in $H_y$, we get $d_{H_y}(h_1, h_2) \leq 2r_2 + d_{H_y}(H_{v'_1}, H_{v'_2}) + 2r_2 \leq 2r_2(2r'_2+1)$.
\end{proof}

\section{Proof sketch and overview} \label{sec:ProofSketch}

The proof of \cref{main:CoarseErdosPosa,th:main} is centered on the following structural statement:
\begin{equation} \label{eq:Sketch}
\begin{aligned}
&\emph{\text{If }} G \emph{\text{ does not contain }} k \cdot K_3 \emph{\text{ as a }} q\emph{\text{-fat minor, then }} G \emph{\text{ admits a graph-decomposition}}\\
&\emph{\text{modelled on a graph of girth }} g \gg q \emph{\text{ whose bags have radius bounded by a function of }} q.
\end{aligned}
\tag{\ding{100}}
\end{equation}

This formulation is intentionally slightly simplified for the sake of this proof sketch; the precise statement can be found in \cref{thm:ResultAfterSmooshing}.
\medskip

Let us first explain how \eqref{eq:Sketch} implies the result.
We may assume that $G$ does not contain $k\cdot K_3$ as a $q$-fat minor, since otherwise we are done.
We thus obtain a graph-decomposition~$(H,\mathcal V)$ of~$G$ whose bags have small radius and whose decomposition graph $H$ has girth $g\gg q$.
\smallskip

The large girth of $H$ allows us to pass from ordinary cycles to fat cycles.
Indeed, we show in \cref{lem:dumb-rhino} that every cycle in a graph of sufficiently large girth can be transformed into a fat cycle contained in a bounded-radius ball around the original cycle.
Consequently, after adjusting the parameters, any collection of pairwise sufficiently far-apart cycles in $H$ gives rise to a collection of pairwise still far-apart $q'$-fat cycles in~$H$, for a suitable $q'\gg q$.
\smallskip

These fat cycles can then be lifted from $H$ to $G$.
Indeed, we show in \cref{lem:lift-fat-graph-dec} that every $q'$-fat cycle in $H$ gives rise to a $q$-fat cycle in $G$, and cycles that are pairwise sufficiently far apart in $H$ give rise to cycles that are pairwise at distance at least $q$ in $G$.
Thus, we have reduced the problem to finding pairwise far-apart (ordinary) cycles in $H$ (instead of fat cycles). With this, we arrive at a problem already solved by Dujmovi\'c, Joret, Micek, and Morin~\cite{DJMMDistanceErdosPosa}.
\smallskip

Their result yields either many pairwise far-apart cycles in $H$, in which case the preceding argument gives the desired $q$-fat model of $k\cdot K_3$ in $G$, or a small set $U\subseteq V(H)$ such that $H-B_H(U,r)$ is a forest, for a suitable $r\in\N$.
\smallskip

In the latter case, for a suitable $R\in\N$, let $G^- \coloneqq G-\bigcup_{u\in U}B_G(V_u,R)$.
The part of the decomposition graph corresponding to $G^-$ is a forest, and hence $G^-$ admits a partial forest-decomposition whose bags have small radius in $G$.
We may then apply the Erd\H{o}s--P\'osa property for subtrees of a forest (\cref{lem:helly}) to obtain either many cycles in $G^-$ that are $q$-fat and pairwise far apart in $G$, or a small collection of bags of $(H, \mathcal{V})$, and hence of vertex sets of $G^-$ of small radius in $G$, meeting every $q$-fat cycle that is contained in $G^-$.
\smallskip

In the former case, we obtain a $q$-fat model of $k \cdot K_3$; in the latter case, these vertex sets together with the bounded-radius bags of $(H, \mathcal{V})$ associated with~$U$ form a bounded collection of vertex sets of bounded radius that meet every $q$-fat model of~$K_3$ in~$G$.
\medskip

We now explain how to obtain the decomposition from \eqref{eq:Sketch}.
This is the main part of the proof.
A natural first attempt is to greedily choose $q$-fat cycles $D_1,\ldots,D_n$ that are pairwise at distance at least~$d$, for a suitable $d\gg q$.
If we find at least~$k$ such cycles, they form a $q$-fat model of $k\cdot K_3$, and we are done.
Otherwise, the maximality of the collection suggests that the part of the graph far from the selected cycles is tree-like.
Indeed, we prove in \cref{lem:ApexForestWithoutApex} that $G'\coloneqq G-\bigcup_{i\leq n}B_G(D_i,d)$ admits a \td\ $(T,\mathcal V)$ whose bags have small radius in~$G$, since $G'$ contains no model of $K_3$ that is $q$-fat in $G$.
\smallskip

The next step would be to extend this \td\ of $G'$ to a \gd\ $(H,\mathcal V)$ of $G$ by incorporating the cycles $D_i$, while ensuring that the bags of the resulting decomposition still have small radius.
This naive construction, however, gives no reason for the decomposition graph $H$ to have large girth.
\medskip

To ensure that $H$ has large girth, we must choose the cycles $D_i$ more carefully.
We begin by selecting those $q$-fat cycles that are contained in a ball of small radius, where ``small'' means larger than $q$ but bounded by a function of $q$.
We call such cycles \emph{goldfish cycles}.
Formally, a cycle $C$ is \defn{$b$-goldfish}\footnote{Just like a goldfish that is trapped inside a small glass bowl its whole life.} if there exists a vertex $x\in V(G)$ such that $V(C) \subseteq B_G(x, b)$ (see \cref{fig:GoldfishCycle}).
\smallskip

\begin{figure}[ht]
    \centering
    \includegraphics[width=0.3\linewidth]{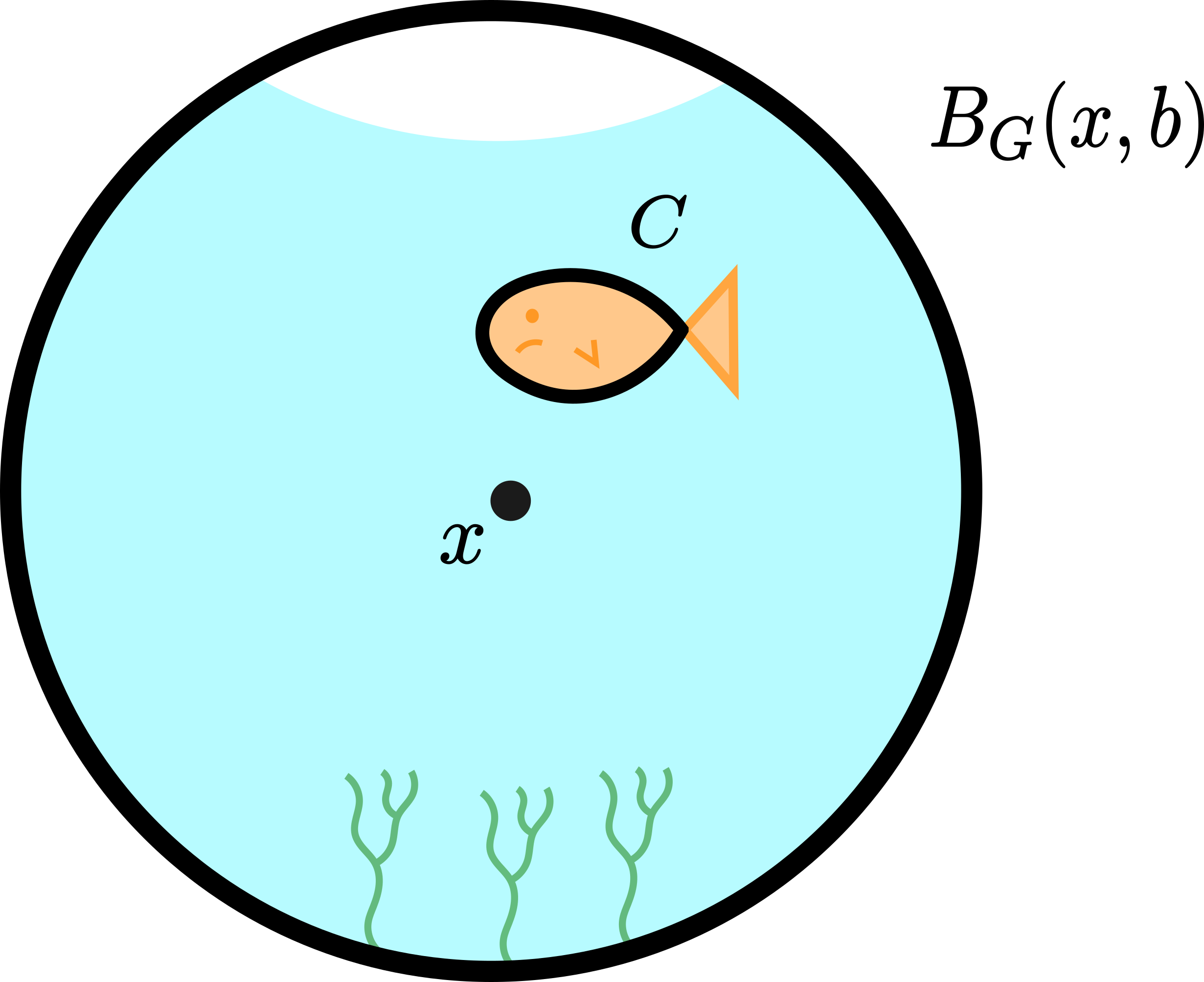}
    \caption{A $b$-goldfish cycle $C$.}
    \label{fig:GoldfishCycle}
\end{figure}

We choose a maximal collection of $q$-fat goldfish cycles that are pairwise at distance at least~$d$.
Any $q$-fat cycle $C$ lying far from all the selected cycles must then be locally tree-like. More precisely, every bounded-radius ball centered at a vertex of $C$ admits a \td\ whose bags have small radius in $G$.
Indeed, otherwise such a ball would contain an additional $q$-fat goldfish cycle, contradicting the maximality of the collection.
\smallskip

A key ingredient of the proof is showing that these local \td s around $C$ can be combined into a single decomposition around most of the cycle, possibly after modifying $C$ slightly; see \cref{lem:AngryHippo}.
This leads to the second type of cycle used in the construction, which we call a \emph{corsola cycle}.
Roughly speaking, a cycle $C$ is \defn{$(r_1,r_2,t, \cdot)$-corsola} 
if it can be written as the union of two internally disjoint paths $P_C$ and $W_C$, where the length of $W_C$ is bounded in terms of $t$, and $G$ admits an $(r_1,r_2)$-radial partial \td\ with support $B_G(P_C,t)$. Thus, most of the cycle is surrounded by a tree-like region, while the remaining path $W_C$ is short. See \cref{fig:CorsulaCycle} for an illustration and \cref{def:CorsolaCycle} for the formal definition.
\smallskip

\begin{figure}[ht]
    \centering
    \includegraphics[width=0.5\linewidth]{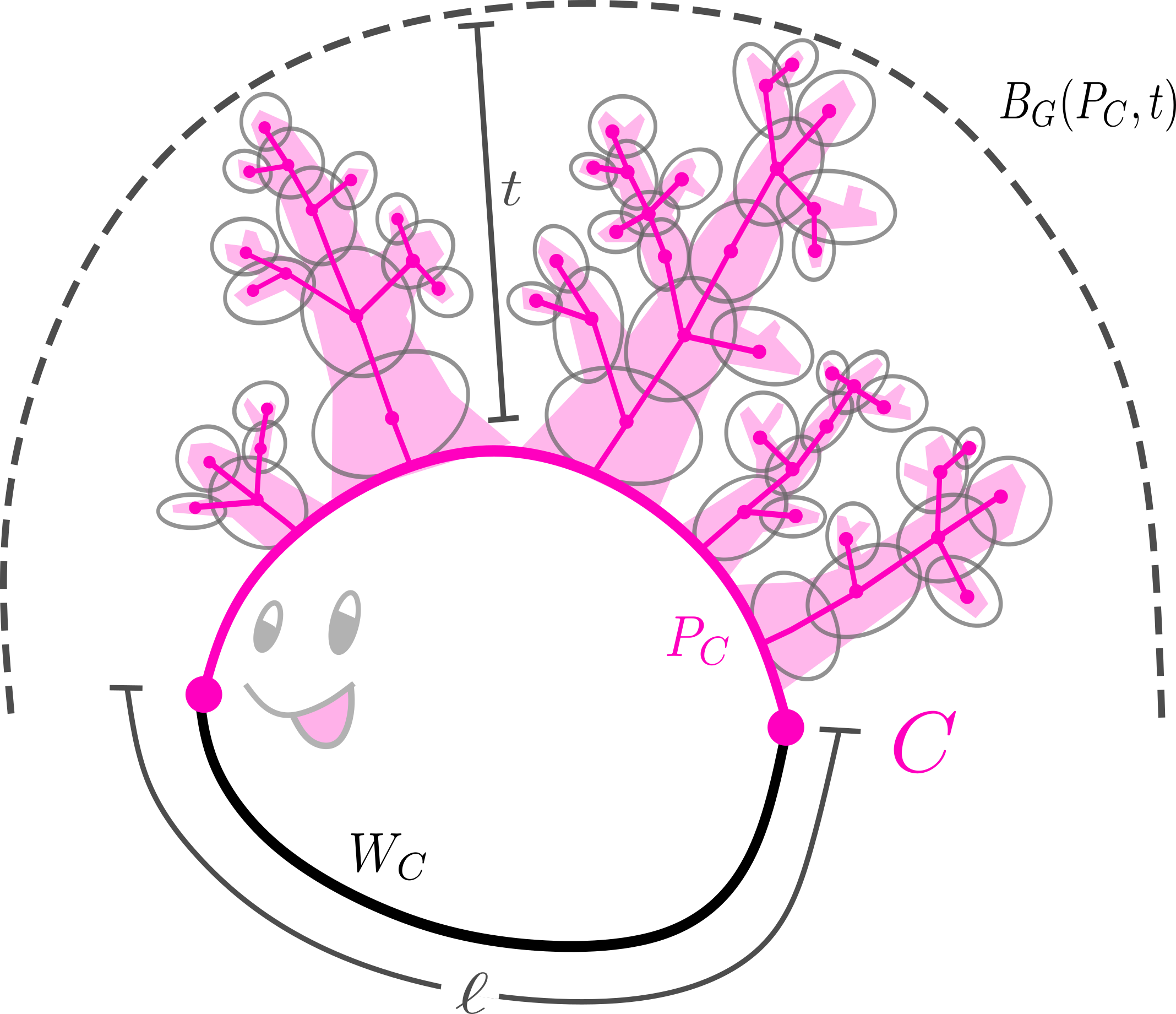}
    \caption{An $(r_1,r_2,t, \cdot)$-corsola cycle~$C = P_C \cup W_C$, where $W_C$ has length at most $\ell = 77 \cdot t$ and $B_G(P_C, t)$ is tree-like.}
    \label{fig:CorsulaCycle}
\end{figure}

We therefore refine the naive packing procedure as follows.
We first choose a maximal collection of pairwise far-apart $q$-fat goldfish cycles, and then extend it by a maximal collection of $q$-fat corsola cycles so that all selected cycles remain pairwise far apart.
By the preceding discussion, once no further $q$-fat goldfish or corsola cycle can be added, \emph{every} $q$-fat cycle in $G$ (not necessarily goldfish or corsola) lies close to one of the selected cycles $D_i$.
Consequently, the remaining subgraph $\defnm{G'} \coloneqq G-\bigcup_i B_G(D_i,d)$ contains no $q$-fat cycle in $G$ and hence admits a \td\ whose bags have small radius in $G$.
\smallskip

Thus, $G$ can essentially be divided into three types of regions:
\begin{enumerate}
\item a bounded number of bounded-radius regions containing the selected goldfish cycles;
\item a bounded number of regions around the selected corsola cycles, each admitting a \gd\ $(H_i,\mathcal V_i)$ whose bags have small radius in $G$, where $H_i$ contains a unique cycle and this cycle is long; and
\item the remaining graph $G'$, which admits a \td\ whose bags have small radius in $G$.
\end{enumerate}

To combine these decompositions, we need slightly more than a tree-decomposition of $G'$ itself.
With some additional care, we ensure that, for a suitable $R\in\N$, the enlarged subgraph $G[B_G(G',R)]$ admits a \td\ $(T,\mathcal V)$ whose bags have small radius in $G$; see the overview of \cref{sec:Untangling} below.
This tree-decomposition overlaps with the \gd s $(H_i,\mathcal V_i)$ around the corsola cycles, while the latter decompositions remain pairwise far apart.

These controlled overlaps allow us to combine the individual decompositions into a graph-decomposition $(H,\mathcal V)$ of $G$ without creating short cycles in the decomposition graph $H$; see \cref{lem:Smooshing}.
Consequently, $H$ has large girth, while all bags of $(H,\mathcal V)$ have small radius.
This gives precisely the decomposition required in \eqref{eq:Sketch} above and completes the proof sketch.
\medskip

[\cref{sec:ApexForests}]: We prove \cref{lem:ApexForestWithoutApex}: If, at some stage of the construction, we can no longer pack another $q$-fat model of $K_3$, then the remaining graph admits a \td\ whose bags have small radius in $G$.
\smallskip

[\cref{sec:HappyHippos}]: We prove \cref{lem:AngryHippo}: If $G$ contains a $q''$-fat cycle~$C$ that is far from every $q$-fat goldfish cycle in $G$, for a suitable $q'' \gg q$, then there exists a $q$-fat corsola cycle~$C'$ contained in a bounded-radius ball around $C$.
\smallskip

[\cref{sec:Untangling}]: We prove \cref{lem:untangling}: When choosing a new corsola cycle $D_i$ using \cref{lem:ApexForestWithoutApex,lem:AngryHippo}, we cannot immediately guarantee that $D_i$ is as far from the previously chosen cycles~$D_j$ as required. Indeed, we need $G[B_G(G',R)]$ to admit a \td\ whose bags have small radius in $G$, and must therefore ensure that $G-(\bigcup B_G(D_i,d-R))$ contains no $q$-fat cycle, rather than merely requiring this of $G'=G-(\bigcup B_G(D_i,d))$.
To overcome this difficulty, \cref{lem:untangling} establishes the following auxiliary statement: given two corsola cycles at distance at least~$d-R \gg q$, either all paths of length less than~$d$ between them can be hit by a small number of bounded-radius balls, or there already exist $k$ pairwise far-apart $q$-fat cycles lying `between' the two corsola cycles.
\smallskip

[\cref{sec:Smooshing}]: We prove \cref{lem:Smooshing}: This lemma describes how to combine the \gd s around the corsola cycles with the \td\ of $G[B_G(G', R)]$ into a \gd\ of $G$ whose bags have small radius and whose decomposition graph has high girth.\smallskip

[\cref{sec:CoralReef}]: This section contains the heart of the proof: We prove \eqref{eq:Sketch} (i.e.\ \cref{thm:ResultAfterSmooshing}) by combining \cref{lem:ApexForestWithoutApex,lem:AngryHippo,lem:untangling,lem:Smooshing} to show that either $G$ contains a $q$-fat model of $k \cdot K_3$, or $G$ admits a \gd\ modelled on a graph of large girth.
\smallskip

[\cref{sec:FinalProof}]: We complete the proof of \cref{th:main} and \cref{main:CoarseErdosPosa} by considering the case in which $G$ admits a \gd\ modelled on a graph of large girth. We resolve this case using the Erd\H{o}s--P\'osa property for pairwise far-apart cycles established in \cite{DJMMDistanceErdosPosa,CDGKMMS}.
In addition, this section also contains the proofs of \cref{maincor:FarApartEPForLongCycles}, and \cref{maincor:FarApartEPForLongInducedCycles}.
\medskip

[\cref{sec:BattleOfTheBlackGate,sec:lengthspaces,sec:discuss}]: In the final three sections we provide a proof of the fat minor conjecture, \cref{conj:qi}, for disjoint unions of cycles (see \cref{thm:mainFatkK3additive,main:CoarseEp:Infadd} in \cref{sec:BattleOfTheBlackGate}), then we give some remarks on length spaces and geodesic metric spaces including a proof of \cref{main:manningextension} (see \cref{sec:lengthspaces}), and we finish the paper off with some discussion and open problems (see \cref{sec:discuss}).

\section{The forest-decomposition lemma} \label{sec:ApexForests}

In this section, we prove the following lemma, which provides a partial forest-decom\-position of a graph $G$ under the assumption that every fat model of $K_3$ lies close to a prescribed set~$U \subseteq V(G)$.

\begin{restatable}[Forest-decomposition lemma]{lemma}{apexforestwithoutapex} \label{lem:ApexForestWithoutApex}
    There exist functions $w_{\ref{lem:ApexForestWithoutApex}}, s_{\ref{lem:ApexForestWithoutApex}} : \N \to \N$ and $r_{\ref{lem:ApexForestWithoutApex}}:\N^2\to\N$ such that the following holds.
    Let $U$ be a set of vertices in a graph~$G$, and let $q,d \in \N$. If every $q$-fat model of $K_3$ in $G$ is at distance at most $d$ from~$U$ in $G$, then $G$ admits an honest $(w_{\ref{lem:ApexForestWithoutApex}}(q), s_{\ref{lem:ApexForestWithoutApex}}(q))$-radial partial \fd\ with support $V(G-B_G(U, r_{\ref{lem:ApexForestWithoutApex}}(q,d)))$.
\end{restatable}

The proof of \cref{lem:ApexForestWithoutApex} relies on the following two theorems.
Informally, they assert that if a graph $G$ contains no $q$-fat model of $K_3$, respectively $K_4$, then $G$ admits a quasi-isometry, or equivalently a \gd, to a graph containing no model of $K_3$, respectively~$K_4$.

\begin{theorem}[\cite{Manning05,GPCoarseGT}] \label{thm:K3-q.i}
    For every $q \in \N$, every graph with no $q$-fat model of $K_3$ is $10q$-quasi-isometric to a forest.
\end{theorem}

\begin{theorem}[{\cite[Theorem~$2'$]{AJKWFatK4}}]
\label{thm:K4-q.i}
    For every $q \in \N$, every graph with no $q$-fat model of~$K_4$ admits an honest $(25235q + 71, 22)$-radial \gd\ $(H, \mathcal{V})$ modelled on a $K_4$-minor-free graph~$H$.
\end{theorem}

We also need the following lemma, which shows that every $q'$-fat cycle in the decomposition graph gives rise to a $q$-fat cycle in the decomposed graph.

\begin{lemma} \label{lem:lift-fat-graph-dec}
    Let $(H, \mathcal{V})$ be an honest $(r_1, r_2)$-radial graph-decomposition of a graph $G$.
    Let $q \in \mathbb{N}$ and let $q' \geq f_{\ref{lem:lift-fat-graph-dec}}(q, r_1, r_2) \coloneqq 2r_2(q+1) + 4r_1r_2$.
    Let $C'$ be a $q'$-fat cycle of $H$.
    Then, there exists a $q$-fat cycle $C$ of $G$ such that $V(C) \subseteq \bigcup_{h \in B_H(C', 2r_1r_2)}V_{h}$.

    Moreover, if $u \in V(G)$ satisfies $d_{H}(H_u, C') > q'$ then $d_G(u, C) > q$.
\end{lemma}

\begin{proof}
    Since $C'$ is $q'$-fat, it can be divided into six paths $P'_0, \ldots, P'_5$ such that $d_H(P'_i, P'_j) \geq q'$ whenever $i, j \in \mathbb{Z}_6$ are distinct and non-consecutive.
    Choose arbitrarily one of the two cyclic orientations of $C'$. This induces an orientation of all $P'_i$. For each $P'_i$, denote by $s'_i$ the start of $P'_i$ and by $t'_i$ the end of $P'_i$.
    Since the graph-decomposition $(H, \mathcal{V})$ is honest, for every $i \in \mathbb{Z}_6$, there exists a vertex $v_{i,i+1} \in V_{t'_i} \cap V_{s'_{i+1}}$.
    Since $(H, \mathcal{V})$ has radial width at most $r_1$, the set $B_G(V_h, r_1)$ induces a connected subgraph of $G$ for every $h \in V(H)$.
    For every $i \in \mathbb{Z}_6$, $P'_i$ is connected, so since $(H, \mathcal{V})$ is honest, the set $\bigcup_{h \in V(P'_i)}B_G(V_h, r_1)$ induces a connected subgraph of $G$.
    For every $i \in \{0, 2, 4\}$, fix a $v_{i-1, i}$ -- $v_{i,i+1}$-path $P_i$ in $G$ contained in $\bigcup_{h \in V(P'_i)}B_G(V_h, r_1)$.
    Then, for every $i \in \{1, 3, 5\}$, fix a $P_{i-1}$--$P_{i+1}$-path $P_i$ in $G$ that is internally disjoint from $P_{i-1}$ and $P_{i+1}$, and that is contained in $\bigcup_{h \in V(P'_i)}B_G(V_h, r_1)$.

    Since $(H, \mathcal{V})$ has radial spread at most $r_2$,
    \cref{lem:dist-H-decomp} implies that for every $d \in \mathbb{N}$, every vertex at distance at most $d$ from a vertex of $\bigcup_{h \in V(P'_i)}V_h$ belongs to $\bigcup_{h \in B_H(P'_i, 2dr_2)}V_h$.
    Thus, $V(P_i) \subseteq \bigcup_{h \in B_H(P'_i, 2r_1r_2)}V_h$ for every $i \in \mathbb{Z}_6$.

    Let $i, j \in \mathbb{Z}_6$ be distinct and non-consecutive.
    Then, $d_H(P'_i, P'_j) \geq q' \geq 2r_2(q+1) + 4r_1r_2$ so
    \[
    d_H(B_H(P'_i, 2r_1r_2), B_H(P'_j, 2r_1r_2)) \geq d_H(P'_i, P'_j) - 4r_1r_2 \geq 2r_2(q+1) = f_{\ref{lem:far-apart-lifts-far-apart}}(q, r_2).
    \]
    Thus, by \cref{lem:far-apart-lifts-far-apart} we have \[d_G\left(\bigcup_{h \in B_H(P'_i, 2r_1r_2)}V_h, \bigcup_{h \in B_H(P'_j, 2r_1r_2)}V_h\right) \geq q.\]
    Therefore, $d_G(P_i, P_j) \geq q$.
    Thus, these paths contain a $q$-fat cycle $C$ in $G$ that is formed by $P_0\cap C, P_1\cap C, P_2\cap C, P_3\cap C, P_4\cap C, P_5\cap C$ and with $ V(C) \subseteq \bigcup_{h \in B_H(C', 2r_1r_2)}V_{h}$.

    For the `Moreover' part, let $u \in V(G)$ such that $d_H(H_u, C') > q'$.
    Then, $d_H(H_u, H_C) > q'-2r_1r_2 -2r_2 \geq 2r_2(q+1) = f_{\ref{lem:far-apart-lifts-far-apart}}(q, r_2)$.
    Applying \cref{lem:far-apart-lifts-far-apart} then gives $d_G(u, C) > q$.
\end{proof}

We now turn to the proof of \cref{lem:ApexForestWithoutApex}.
We remark that the lemma could likely be proved using the BFS-based approach of Georgakopoulos and Papasoglu in their proof of \cref{thm:K3-q.i}~\cite{GPCoarseGT}.
Adapting their argument to our setting, however, would require overcoming several technical difficulties, since the relevant subgraph would have to be partitioned while preserving distances measured in the ambient graph $G$.
We therefore take a different approach, which we sketch now.

Let $G/W$ be obtained from $G$ by contracting $W=B_G(U,d)$ to a single vertex~$w$. Then every $q$-fat model of~$K_3$ in~$G/W$ contains $w$ by the assumption on~$G$ and $U$, and hence $G/W$ contains no $q$-fat model of $K_4$.

By \cref{thm:K4-q.i}, the graph $G/W$ admits an honest $(\mathcal{O}(q),\mathcal{O}(1))$-radial \gd\ $(H,\mathcal V')$, where $H$ is $K_4$-minor-free.
Fix a node $w'\in V(H)$ with $w\in V'_{w'}$.
Using both the fact that $H$ is $K_4$-minor-free and the fact that every $q$-fat model of $K_3$ in $G/W$ contains $w$, we show that, after deleting a sufficiently large ball around $w'$, the remaining graph $H'$ contains no $q'$-fat cycle, for some $q'$ depending only on $q$.

We may therefore apply \cref{thm:K3-q.i} to obtain a forest-decomposition of $H'$ with bounded radial width and spread.
Finally, we compose these two decompositions using \cref{herd-of-hyenas} to obtain the required honest partial forest-decomposition of $G/W$, which lifts directly to $G$.

We now prove \cref{lem:ApexForestWithoutApex}, which we restate for convenience.

\apexforestwithoutapex*

\begin{proof}[Proof of \cref{lem:ApexForestWithoutApex}]
    We first define the functions $w_{\ref{lem:ApexForestWithoutApex}}, s_{\ref{lem:ApexForestWithoutApex}}$ and $r_{\ref{lem:ApexForestWithoutApex}}$. For this, we introduce several parameters that will be reintroduced in the course of the proof. We do it to clarify the dependency between the different parameters. The following parameters all only depend on~$q$.
    Let \begin{align*}
        r'_1 &\coloneqq 25235q+71, \\
        r'_2 &\coloneqq 22, \\
        q' &\coloneqq 7 \cdot f_{\ref{lem:lift-fat-graph-dec}}(q, r'_1, r'_2), \\
        r' &\coloneqq 2q'+2r'_2, \\
        r_1 &\coloneqq w_{\ref{prop:q.i.-graph-dec}}(30q'), \\
        r_2 &\coloneqq s_{\ref{prop:q.i.-graph-dec}}(30q'), \text{ and } \\
        \rho &\coloneqq \max\{2r'_1 \cdot (r' + 1), w_{\ref{herd-of-hyenas}}(r'_1, r_1)\}.\\
        \intertext{Finally, set}
        w_{\ref{lem:ApexForestWithoutApex}}(q) &\coloneqq w_{\ref{herd-of-hyenas}}(r'_1, r_1), \\
        s_{\ref{lem:ApexForestWithoutApex}}(q) &\coloneqq s_{\ref{herd-of-hyenas}}(r'_2, r_2), \text{ and } \\
        r_{\ref{lem:ApexForestWithoutApex}}(q,d) &\coloneqq \rho + d.
    \end{align*}

    Let $G, U, q, d$ be given and suppose that every $q$-fat model of~$K_3$ in $G$ is at distance at most~$d$ from~$U$ in~$G$.
    Let $\defnm{W} := B_G(U, d)$, and let $\defnm{G/W}$ be the graph obtained from $G$ by identifying all vertices of $W$ into a single vertex $\defnm{w}$.
    Observe that any $q$-fat model of~$K_3$ in $G/W$ that does not contain $w$ would also be a $q$-fat model of~$K_3$ in~$G$, and it would avoid $B_G(U, d)$, which is impossible by assumption.
    Therefore, every $q$-fat model of $K_3$ in $G/W$ contains $w$.
    \smallskip

    We first show that $G/W$ contains no $q$-fat model of $K_4$.
    Indeed, suppose for a contradiction that $G/W$ contains a $q$-fat model of $K_4$. Let $((V_x : x \in V(K_4)), (E_e : e \in E(K_4)))$ be such a model.
    Then at most one branch set $V_x$ and at most one branch path $E_e$ contain the vertex~$w$ (and if both exist, $x$ and $e$ must be incident).
    Thus, there exist three branch sets~$V_x$ that do not contain $w$ and that are connected by branch paths $E_e$ that also do not contain $w$.
    They form a $q$-fat model of $K_3$ in $G/W$ that does not contain $w$, a contradiction.
    This proves that $G/W$ contains no $q$-fat model of $K_4$.
    \smallskip

    By \cref{thm:K4-q.i}, $G/W$ admits an honest $(r'_1, r'_2)$-radial \gd\ \defn{$(H, \mathcal{V}')$} modelled on a $K_4$-minor-free graph $H$ with $r'_1 = 25235q+71$ and $r'_2 = 22$. The main part of the proof is showing the following claim about fat cycles in~$H-B_H(w',r')$.

    \begin{claim} \label{claim:ApexForestWithoutApex}
        Let $q' = 7 \cdot f_{\ref{lem:lift-fat-graph-dec}}(q, r'_1, r'_2)$ and $r' = 2q'+2r'_2$. Let $\defnm{w'} \in V(H)$ such that  $w \in V'_{w'}$. Then $H - B_{H}(w', r')$ has no $(3q')$-fat cycle.
    \end{claim}

    \begin{claimproof}
        Suppose for a contradiction that there is a $(3q')$-fat cycle \defn{$C'$} in $H - B_{H}(w', r')$.

        Assume first that $C'$ is $q'$-fat in $H$.
        We have $V(C') \subseteq V(H) \setminus B_{H}(w', r')$, and therefore $d_{H}(w', C') > r' \geq q'+2r'_2$.
        Since $(H, \mathcal{V}')$ has radial spread at most $r'_2$, this in particular implies $d_{H}(H_w, C') > q' \geq f_{\ref{lem:lift-fat-graph-dec}}(q, r'_1, r'_2)$.
        Thus, by \cref{lem:lift-fat-graph-dec}, there exists a $q$-fat cycle $C$ of $G/W$ such that $d_{G/W}(w, C) > q$.
        By \cref{lem:fat-K3-cycle}, the cycle $C$ gives rise to a $q$-fat model of~$K_3$ in~$G/W$ that does not contain $w$, which is a contradiction.
        \smallskip

        Thus, the cycle $C'$ is not $q'$-fat in $H$.
        Let $P'_0, \ldots, P'_5$ be paths that witness that $C'$ is $(3q')$-fat in $H - B_{H}(w', r')$.
        Since $C'$ is not $q'$-fat in $H$, there exist distinct non-consecutive $i, j \in \mathbb{Z}_6$ such that $P'_i$ and $P'_j$ are at distance less than~$q'$ in~$H$.
        Let \defn{$Q'$} be a shortest $P'_i-P'_j$ path in~$H$; in particular, $Q'$ has length less than~$q'$.
        Furthermore, since $C'$ is $(3q')$-fat in $H - B_{H}(w', r')$, the path~$Q'$ must intersect $B_{H}(w', r')$.
        Let \defn{$p'_1$} be the last vertex of $C'$ that is visited by $Q'$ before $B_{H}(w', r')$ and let \defn{$p'_2$} be the first vertex of $C'$ that is visited by $Q'$ after $B_{H}(w', r')$, and observe that $p'_1 \neq p'_2$ (see \cref{subfig:ForestDecomp1}).

        In the remainder of the proof of this claim, we will find two paths $R',S'$ of $H$ between $p'_1$ and $p'_2$ which avoid $B_H(H_w, q')$ and which form a fat model of $K_3$ in $H$ (see \cref{subfig:ForestDecomp2}). By \cref{lem:lift-fat-graph-dec}, it will then follow that $G/W$ contains a $q$-fat model of $K_3$ avoiding $w$, which contradicts our observation right before \cref{claim:ApexForestWithoutApex}.
        \smallskip

        \begin{figure}[ht]
            \centering
            \begin{subfigure}[b]{0.42\linewidth}
                \centering
                \includegraphics[width=0.8\linewidth]{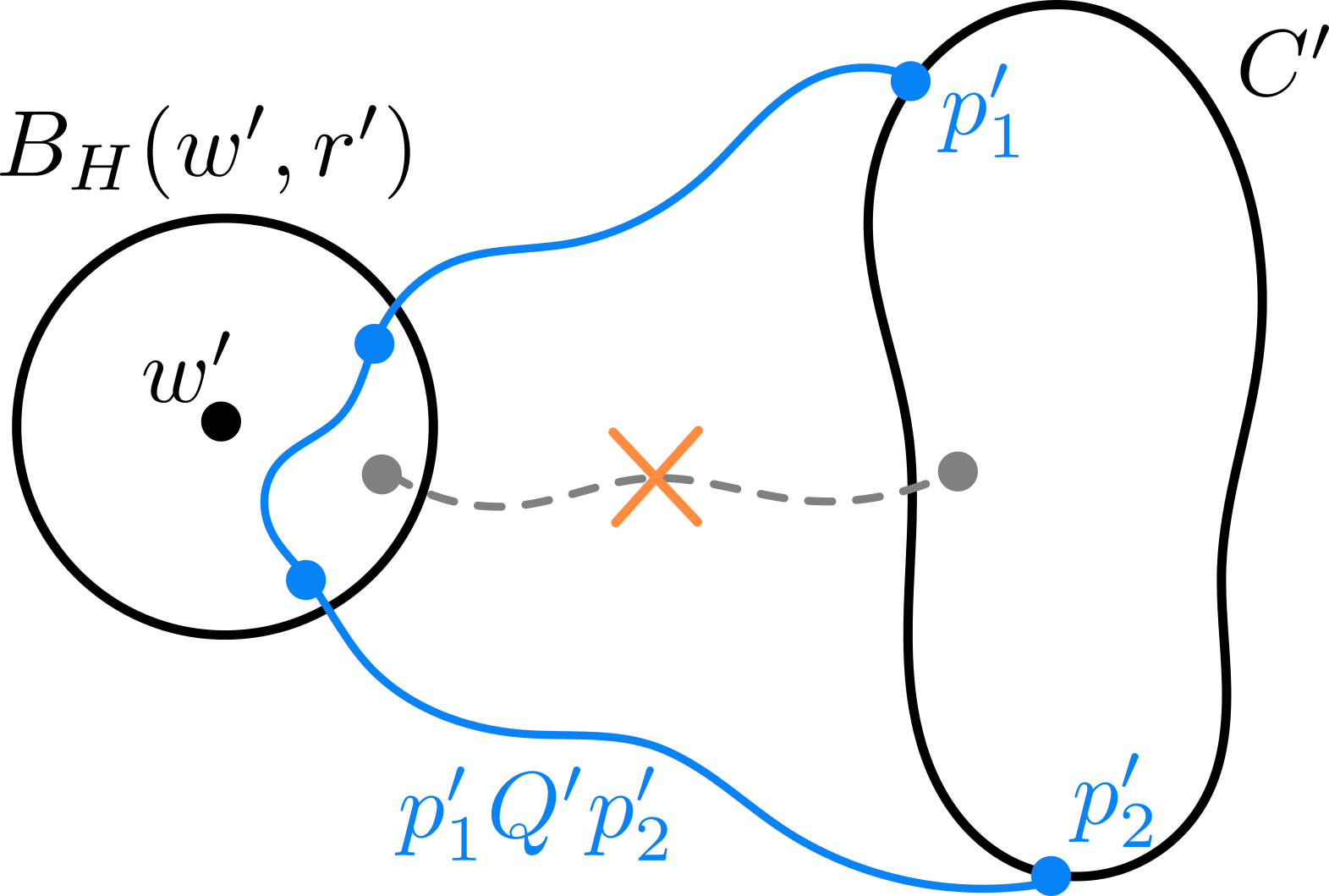}
                \caption{}
                \label{subfig:ForestDecomp1}
            \end{subfigure}
            \begin{subfigure}[b]{0.56\linewidth}
                \centering
                \includegraphics[width=0.8\linewidth]{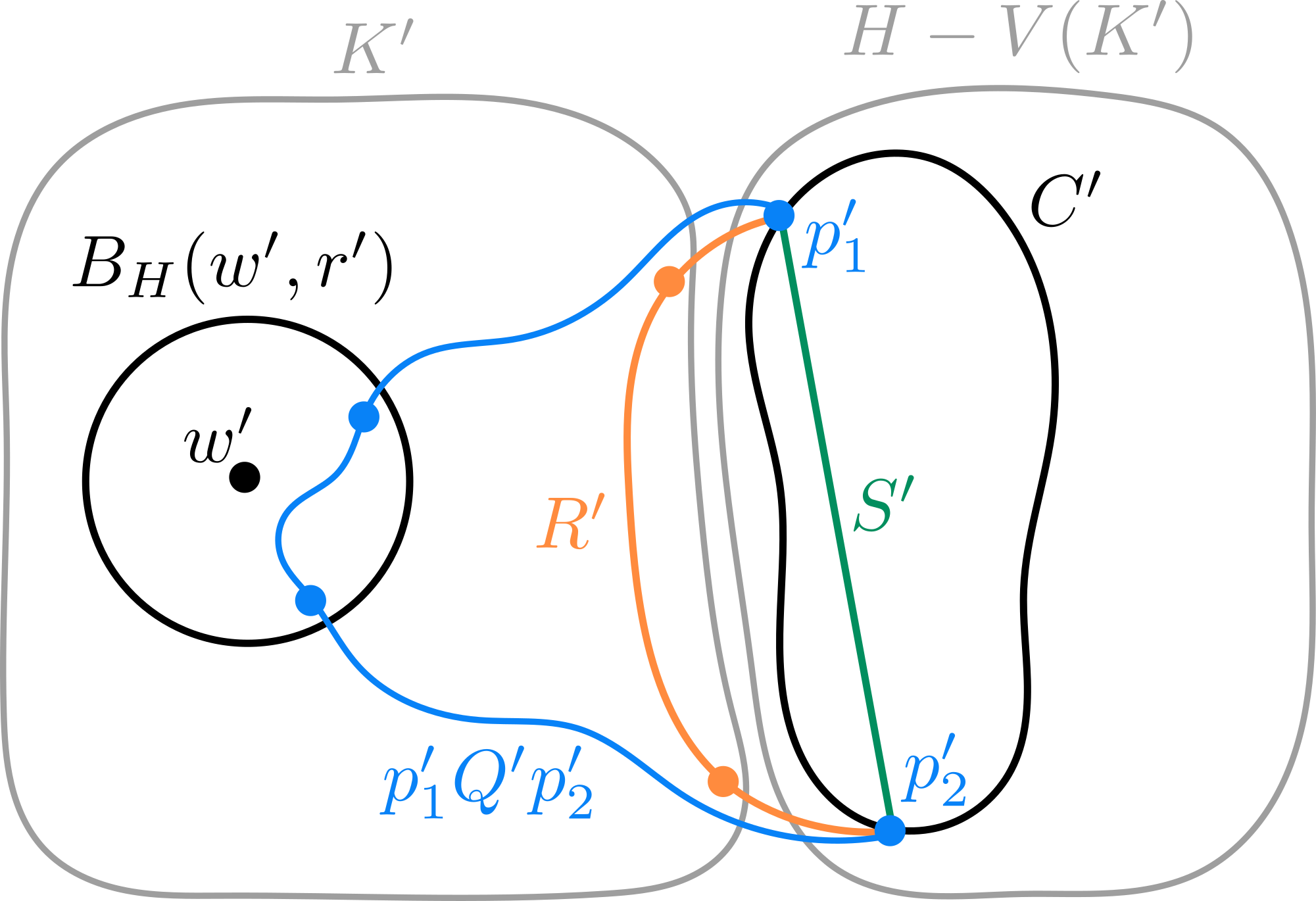}
                \caption{}
                \label{subfig:ForestDecomp2}
            \end{subfigure}
            \caption{An illustration of the situation in the proof of \cref{claim:ApexForestWithoutApex}.}
            \label{fig:ForestDecompLem}
        \end{figure}

        Note that $Q'$ contains a $B_{H}(w', r')-p'_1$ path and a $B_{H}(w', r')-p'_2$ path that are internally disjoint from~$V(C')$.
        Since $H$ is $K_4$-minor-free, this implies that every $B_{H}(w', r')-C'$ path intersects $\{p'_1, p'_2\}$, as otherwise $B_H(w',r'), \{p'_1\}, \{p'_2\}$, and the new intersection point on $C'$ would form the branch sets of a model of~$K_4$ in~$H$ (see \cref{subfig:ForestDecomp1}).
        In other words, $\{p'_1, p'_2\}$ is a $2$-separator of~$H$, which separates $B_H(w',r')$ from~$C'$.

        Let \defn{$K'$} be the component of $H - \{p'_1, p'_2\}$ which contains~$w'$.
        Then, $B_{H}(w', r') \subseteq V(K')$ and $V(K') \cap V(C') = \emptyset$. Moreover, $N_H(K') = \{p'_1,p'_2\}$. (See \cref{subfig:ForestDecomp2}.)
        We now choose the aforementioned $p'_1$--$p'_2$ paths $R', S'$ inside $K'$ and $H-V(K')$, respectively (see \cref{subfig:ForestDecomp2}).
        \smallskip

        Let \defn{$S'$} be a shortest $p'_1-p'_2$ path in $H - V(K') \subseteq H - B_H(w',r')$. We first show that $S'$ has length at least~$q'$.
        The path $Q'p'_1$ has length at most~$q'$ (since $Q'$ has length at most~$q'$) and does not intersect $B_{H}(w', r')$ (by the choice of $p'_1$), so $p'_1$ is at distance at most $q'$ from $P'_i$ in $H - B_{H}(w', r')$. Similarly, $p'_2$ is at distance at most $q'$ from $P'_j$ in $H - B_{H}(w', r')$.
        Since $C'$ is $(3q')$-fat in $H-B_{H}(w', r')$, it follows that $p'_1$ and $p'_2$ are at distance at least~$q'$ in $H-B_{H}(w', r')$, so $S'$ has length at least~$q'$.

        Thus, we can divide $S'$ into 7 pairwise edge-disjoint subpaths $S'_0, S'_1, \ldots, S'_6$, each of length at least $q'/7$.
        \smallskip

        Let \defn{$R'$} be a shortest path between $p'_1$ and $p'_2$ in $H$ all whose internal vertices are contained in~$V(K')$.
        Since $Q'$ contains such a path (except for `shortest') and has length at most~$q'$, also $R'$ has length at most~$q'$.
        \smallskip

        Let $T' \coloneqq V(S'_0) \cup V(S'_6) \cup V(R')$, and observe that $T'$ induces a connected subgraph of $H$, and that there is a model of $K_3$ in $H$ with branch sets $T', S'_2$ and $S'_4$, and branch paths $S'_1, S'_3$ and $S'_5$.

        Since $S'$ is a shortest $p'_1$--$p'_2$ path in $H-V(K')$, the vertices $p'_1$ and $p'_2$ are at distance at least $\min\{||S'_0||, ||S'_6||\} \geq q'/7 = f_{\ref{lem:lift-fat-graph-dec}}(q, r'_1, r'_2)$ from $V(S'_1) \cup \ldots \cup V(S'_5)$ in $H - V(K')$.
        Furthermore, $K'$ is a component of $H - \{p'_1, p'_2\}$ which is disjoint from $V(S')$, so $p'_1$ and $p'_2$ are in fact at distance at least $f_{\ref{lem:lift-fat-graph-dec}}(q, r'_1, r'_2)$ from $V(S'_1) \cup \ldots \cup V(S'_5)$ in $H$.

        Since $S'$ is a shortest path in $H-V(K')$, and because $R'$ is internally contained in $K'$, it follows that this model of $K_3$ in $H$ is $f_{\ref{lem:lift-fat-graph-dec}}(q, r'_1, r'_2)$-fat.

        Moreover, $p'_1, p'_2 \in V(C') \subseteq V(H) \setminus B_H(w',r')$, so the endvertices $p'_1, p'_2$ of $R'$ are at distance greater than $r'$ from $w'$. Since $R'$ has length at most~$q'$, it follows that $R'$ stays at distance greater than~$r'-q'$ from~$w'$ in~$H$.
        Thus, this fat model of $K_3$ in $H$ avoids $B_{H}(w', r'-q') \supseteq B_{H}(H_w, r'-q'-2r'_2) = B_{H}(H_w, q')$.
        By \cref{lem:lift-fat-graph-dec} (and \cref{lem:fat-K3-cycle}), there is a $q$-fat model of $K_3$ in $G/W$ that does not contain $w$, a contradiction.
    \end{claimproof}

    Let $\rho = \max\{2r'_1 \cdot (r' + 1), w_{\ref{herd-of-hyenas}}(r'_1, r_1)\}$, and set
    \[
    \defnm{Y} := V(G) \setminus B_{G}(U, \rho+d) = V(G) \setminus B_G(U, r_{\ref{lem:ApexForestWithoutApex}}(q,d)).
    \]
    We now show that there is an honest $(w_{\ref{lem:ApexForestWithoutApex}}(q), s_{\ref{lem:ApexForestWithoutApex}}(q))$-radial partial forest-decomposition of~$G$ with support~$Y$, which concludes the proof.
    \smallskip

    By \cref{claim:ApexForestWithoutApex} and \cref{thm:K3-q.i}, there is a $(30q')$-quasi-isometry from $H - B_{H}(w', r')$ to a forest $F$.
    By \cref{prop:q.i.-graph-dec}, it follows that there is an honest $(r_1, r_2)$-radial forest-decomposition $(F, \mathcal{V})$ of $H - B_{H}(w', r')$ with $r_1 = w_{\ref{prop:q.i.-graph-dec}}(30q')$ and $r_2 = s_{\ref{prop:q.i.-graph-dec}}(30q')$.

    We now want to apply \cref{herd-of-hyenas} to $(F, \mathcal{V})$ with $Y$ and $X' = B_H(w',r')$.
    For this, we need to show that $H_y \cap B_{H}(w', r') = \emptyset$ for every $y \in Y$.
    Let $y \in Y$ be arbitrary.
    Then, $d_G(y, U) > \rho + d$, and so $d_{G/W}(y, w) > \rho$.
    Using \cref{lem:dist-H-decomp}, we have $d_H(w', H_y) \geq d_H(H_w, H_y) \geq \frac{d_{G/W}(w, y)}{2r'_1}-1 > \frac{\rho}{2r'_1} - 1 \geq r'$.
    Therefore, $H_y \cap B_H(w', r') = \emptyset$.
    By \cref{herd-of-hyenas}, there exists an honest $(w_{\ref{herd-of-hyenas}}(r'_1, r_1), s_{\ref{herd-of-hyenas}}(r'_2, r_2))$-radial partial forest-decomposition of $G/W$ with support~$Y$.

    Every $y \in Y$ satisfies $d_{G/W}(y, w) > \rho \geq w_{\ref{herd-of-hyenas}}(r'_1, r_1)$, so this forest-decomposition is actually an honest $( w_{\ref{lem:ApexForestWithoutApex}}(q), s_{\ref{lem:ApexForestWithoutApex}}(q))$-radial partial forest-decomposition of $G$ with support $Y$.
\end{proof}

\section{The happy hippo lemma (aka Finding a corsola cycle)} \label{sec:HappyHippos}

In this section, we prove \cref{lem:AngryHippo} below.
Informally, this lemma states that if a graph $G$ contains a $q'$-fat cycle~$C$ far from a set $U \subseteq V(G)$, and no $q$-fat goldfish cycle lies close to $C$, then $G$ contains a $q$-fat corsola cycle that is also far from~$U$.

We begin by formally defining corsola cycles.

\begin{definition}[Corsola cycle] \label{def:CorsolaCycle}
    Let $r_1, r_2, t, \kappa \in \mathbb{N}$.
    A cycle $C$ in a graph $G$ is \defn{$(r_1, r_2, t, \kappa)$-corsola} if it consists of two internally disjoint paths $P_C, W_C$ such that $C = P_C \cup W_C$ and such that the following properties hold (see \cref{fig:CorsulaCycle}):
    \begin{enumerate}[label=\rm{(C\arabic*)}]
        \item \label{itm:Corsola:WC} $W_C$ has length at most $77\cdot t$,
        \item \label{itm:Corsola:TreeDecomp} $G$ admits an $(r_1, r_2)$-radial partial \td\ with support $B_G(P_C,t)$,
        \item \label{itm:Corsola:QuasiGeodesic} For every $\tau \leq t$, if $u, v \in V(P_C)$ satisfy $d_{G}(u, v) \leq \tau$ then $d_{P_C}(u, v) \leq 4\tau + \kappa$.
    \end{enumerate}
\end{definition}

With this terminology in place, we can state the main result of the section.

\begin{restatable}[Happy hippo lemma]{lemma}{happyhippo} \label{lem:AngryHippo}
    There exist functions $q'_{\ref{lem:AngryHippo}}, w_{\ref{lem:AngryHippo}}, s_{\ref{lem:AngryHippo}}, \kappa_{\ref{lem:AngryHippo}}, a_{\ref{lem:AngryHippo}} : \mathbb{N} \to \mathbb{N}$, and $b_{\ref{lem:AngryHippo}} : \mathbb{N}^2 \to \mathbb{N}$ such that the following holds.
    Let $q, t \in \N$ with $t \geq 5q+3$, let $G$ be a graph and let $U \subseteq V(G)$.
    Let $C$ be a $q'_{\ref{lem:AngryHippo}}(q)$-fat cycle in $G$ with $d_G(C, U) \geq a_{\ref{lem:AngryHippo}}(q)$, such that for every $c \in V(C)$, the ball $B_G(c, b_{\ref{lem:AngryHippo}}(q, t))$ contains no model of $K_3$ that is $q$-fat in~$G$.
    Then, there is a $q$-fat $(w_{\ref{lem:AngryHippo}}(q), s_{\ref{lem:AngryHippo}}(q), t, \kappa_{\ref{lem:AngryHippo}}(q))$-corsola cycle $C' = (P_{C'}, W_{C'})$ such that $d_G(P_{C'}, U) \geq q$.
\end{restatable}

We prove \cref{lem:AngryHippo} at the end of this section.
The proof relies on two auxiliary lemmas, corresponding to the two main difficulties in constructing a corsola cycle.
The first provides the partial tree-decomposition required by \ref{itm:Corsola:TreeDecomp} given some suitable path~$P_C$, while the second produces a controlled sequence of geodesic paths from which a cycle~$C'$ with a partition $C' = P_{C'} \cup W_{C'}$ can be extracted such that $W_{C'}$ is short, and such that $P_{C'}$ satisfies \ref{itm:Corsola:QuasiGeodesic} and the assumptions of the first auxiliary lemma.

We begin with the decomposition lemma.
It gives sufficient local conditions on a path $P$ under which $G$ admits a partial tree-decomposition with support $B_G(P,t)$ and bounded radial width and radial spread.
In the proof of \cref{lem:AngryHippo}, we will apply it to the path $P_{C'}$ of the corsola cycle $C'$ constructed using the second auxiliary lemma.

\begin{lemma}\label{hairy-yak}
    There exist functions $w_{\ref{hairy-yak}} : \mathbb{N}^2 \to \mathbb{N}$ and $s_{\ref{hairy-yak}} : \mathbb{N} \to \mathbb{N}$ such that the following holds.
    Let $q, t, b, \varepsilon \in \mathbb{N}$ such that $t \geq 5q+3$ and $b \geq 2t+30q$.
    Let $P$ be a path in a graph~$G$ with the following properties:
    \begin{description}
        \item[(Locally tree-like)] For every $u \in V(P)$, no $q$-fat model of $K_3$ in $G$ is contained in $B_G(u, b)$.
        \item[(Locally almost geodesic)] If $u, v \in V(P)$ satisfy $d_G(u, v) \leq 20q+4$ then $d_{P}(u, v) \leq \varepsilon$.
        \item[(No long shortcut)] If $u, v \in V(P)$ satisfy $d_G(u, v) \leq 2t+15q$ then $d_{P}(u, v) \leq b$.
    \end{description}
    Then, $G$ admits a $(w_{\ref{hairy-yak}}(q, \varepsilon), s_{\ref{hairy-yak}}(\varepsilon))$-radial partial tree-decomposition with support $B_G(P, t)$.
\end{lemma}

The proof of \cref{hairy-yak} uses a coarse notion of connectivity. We introduce this terminology first.
Let $G$ be a graph and let $M \in \mathbb{R}^+$.
A set $X \subseteq V(G)$ is \defn{$M$-near-connected} if, for every $x, x' \in X$, there exists a sequence $x = x_0, x_1, \ldots, x_{\ell} = x'$ of vertices of $X$ such that $d_G(x_{i-1}, x_i) \leq M$ for every $i \in [\ell]$.
The maximal $M$-near-connected subsets of $X$ are called its \defn{$M$-near-components}.
Observe that the $M$-near-components of $X$ form a partition of $X$.
\smallskip

We briefly describe the construction underlying the proof.
It follows the same general strategy as the proof of \cite[Theorem~3.1]{GPCoarseGT}.
We partition $B_G(P,t)$ into layers $L_0,L_1,\ldots,L_t$ according to the distance from $P$, and, for each $n\in\{0,\ldots,t\}$, consider the set $\mathcal{C}_n$ of $5q$-near-components of $L_n$.
\smallskip

The `Locally tree-like' and `No long shortcut' assumptions imply that, for every $n>5q$, each near-component $C\in\mathcal{C}_n$ has bounded diameter in $G$ and is adjacent to a unique near-component in $\mathcal{C}_{n-1}$.
These properties allow us to organise the near-components into a tree containing $P$ as a central path: each member of $\mathcal{C}_{5q+1}$ is joined to a suitable vertex of $P$, while each near-component in a subsequent layer is joined to its unique predecessor. (See \cref{fig:HairyYak}.)
\smallskip

The bag corresponding to a near-component $C$ is simply $B_G(C,1)$.
For a vertex $x\in V(P)$, the corresponding bag is the union of the balls $B_G(y,5q+2)$ over all vertices $y\in V(P)$ satisfying $d_P(x,y)\leq\varepsilon$.
The bounded diameter of the near-components and the `Locally almost geodesic' assumption are then used both to verify that these bags define a partial tree-decomposition and to bound its radial width and radial spread.
This verification is routine, although somewhat technical.
We now give the details.

\begin{proof}[Proof of~\cref{hairy-yak}.]
    We partition the vertices in $B_G(P, t)$ into layers $L_0, L_1, \ldots, L_t$ according to their distance to $P$, i.e.\ \defn{$L_i$} comprises precisely those vertices in $B_G(P,t)$ that are at distance~$i$ from~$P$.
    For every $n \in \{0, \ldots, t\}$, let \defn{$\mathcal{C}_n$} be the set of $5q$-near-components of $L_n$.

    \medskip

    \noindent\textbf{Bounding the diameter of the near-components:}
    \begin{claim} \label{cl:small-diameter}
        If $5q < n \leq t$, then every near-component $C \in \mathcal{C}_n$ has diameter at most $10q$ in~$G$.
    \end{claim}

    \begin{claimproof}
        Suppose for a contradiction that some $5q$-near-component $\defnm{C} \in \mathcal{C}_n$ has diameter greater than~$10q$.
        In what follows, we construct a $q$-fat model of~$K_3$ in~$G$ that is contained in the ball of radius~$b$ around some vertex of~$P$, contradicting the `Locally tree-like' assumption on~$P$.

        Let \defn{$Q$} be a path between two vertices $\defnm{x_1}, \defnm{x_2} \in V(C)$ with the following properties: \begin{itemize}
            \item $10q < d_G(x_1, x_2) \leq 15q$,
            \item $V(Q) \subseteq B_G(L_n, 5q/2)$, and
            \item $d_G(x_1,v) \leq 15q$ for every $v \in V(Q)$.
        \end{itemize}
        To see that such a path exists, choose vertices $x_1,x'_1\in C$ with $d_G(x_1,x'_1)>10q$, and let $R$ be an $x_1$--$x'_1$ path such that every subpath of $R$ of length $5q$ meets $C$; such a path exists by the definition of $C$.
        Let $x_2$ be the first vertex of $C$ encountered along $R$ whose distance from $x_1$ is greater than $10q$.
        Then $d_G(x_1,x_2)\leq 15q$, and the subpath of $R$ between $x_1$ and $x_2$ has the required properties for $Q$.
        \smallskip

        For $i \in [2]$, let \defn{$P'_i$} be a shortest path from $x_i$ to $P$ in $G$.
        Let \defn{$y_i$} be the endvertex of $P'_i$ in~$P$, and let $\defnm{B_i} := B_G(x_i, 7q/2)$.
        Let \defn{$P_i$} be the subpath of $P'_i$ of length $q$ that starts at the last vertex of $P'_i$ that is contained in~$B_i$.
        Let \defn{$Q'$} be a subpath of $Q$ joining $B_1$ and $B_2$.
        Let $\defnm{B_P} := B_G(y_1Py_2, n-9q/2)$.

        Then, $d_G(B_1 \cup B_2 \cup Q', P) \geq n-7q/2$ and $B_P \subseteq B_G(P, n-9q/2)$, and therefore $d_G(B_1 \cup B_2 \cup Q', B_P) \geq q$.
        If $d_G(P_1 \cup B_1, P_2 \cup B_2) \leq q$ then, using that $P_1 \cup B_1 \subseteq B_G(x_1, 9q/2)$ and $P_2 \cup B_2 \subseteq B_G(x_2, 9q/2)$, we would have $d_G(x_1, x_2) \leq 10q$, which is impossible by the choice of $x_1,x_2$, so $d_G(P_1 \cup B_1, P_2 \cup B_2) > q$.
        Finally, we have $d_G(Q', P) \geq n-5q/2$, and $P_1 \cup P_2 \subseteq B_G(P, n-7q/2)$, so $d_G(Q', P_1 \cup P_2) \geq q$.
        Thus, the sets $B_P, B_1, B_2, P_1, P_2, Q'$ form a $q$-fat model of $K_3$ in $G$.

        Let $S$ be a shortest path between $x_1$ and $x_2$ in $G$. Then, $P'_1x_1Sx_2P'_2$ is a walk from $y_1$ to $y_2$ in~$G$ of length $2n + d_G(x_1, x_2) \leq 2t + 15q$.
        Thus, $y_1, y_2 \in V(P)$ satisfy $d_G(y_1, y_2) \leq 2t+15q$, so by the `No long shortcut' assumption of \cref{hairy-yak}, we have $d_P(y_1, y_2) \leq b$.
        It is then straightforward to check using $b\ge 2t+30q$ that the above constructed $q$-fat model of $K_3$ in $G$ is contained in $B_G(y, b)$, where $y$ is a midpoint of the subpath of $P$ between $y_1$ and $y_2$.
        This contradicts the `Locally tree-like' assumption of \cref{hairy-yak}.
    \end{claimproof}

    \noindent\textbf{The near-components are arranged in a tree-like way:}

    \begin{claim}\label{cl:single-parent}
        If $5q < n < t$, then every near-component $C \in \mathcal{C}_{n+1}$ is adjacent to exactly one near-component $C' \in \mathcal{C}_n$.
    \end{claim}

    \begin{claimproof}
        The `at least one' part is obvious by the definition of layer and because $\mathcal{C}_n$ partitions $L_n$.
        For the `at most one' part, suppose for a contradiction that there exists a near-component $C \in \mathcal{C}_{n+1}$ that is adjacent to two distinct near-components in $\mathcal{C}_n$. Once more, we construct a $q$-fat model of~$K_3$ in~$G$ that is contained in the ball of radius~$b$ around some vertex of~$P$, contradicting the `Locally tree-like' assumption on~$P$.

        Let \defn{$A$} be a path of $G$ with the following properties: \begin{itemize}
            \item The endvertices $a, z$ of $A$ lie in distinct elements of $\mathcal{C}_n$,
            \item The interior of $A$ is at distance at least $n+1$ from $P$, and
            \item Every vertex of $A$ is at distance at most $5q/2$ from $C$.
        \end{itemize}
        To see that such a path exists, let $C' \neq C'' \in \mathcal{C}_n$ be both adjacent to $C$. Consider an edge $e \in E(G)$ between $C$ and $C'$, an edge $f \in E(G)$ between $C$ and $C''$, and a path $A'$ between the endvertices of $e$ and $f$ in $C$ such that every segment of $A'$ of length $5q$ intersects $C$; such a path exists by definition of $C$.
        Then, consider a minimal subpath~$A$ of $A' \cup \{e, f\}$ between distinct elements of $\mathcal{C}_n$. Then the interior of~$A$ must be contained in $\bigcup_{m \geq n+1} L_m$. Moreover, it must have length greater than~$5q$ because the elements of $\mathcal{C}_n$ are $5q$-near-components, so it must contain a vertex of $C$.
        It is then easy to check that this subpath~$A$ has the desired properties.
        \medskip

        Let \defn{$Q_a$} be a shortest path from $a$ to $P$ and define \defn{$Q_z$} similarly. Let $\defnm{y_a}, \defnm{y_z}$ be their respective endvertices in~$P$.
        Let $Z$ be an $a-z$ path contained in $Q_a \cup P \cup Q_z$.
        Let \defn{$B_a$} be the initial subpath of $Q_a$ of length $2q$ and define \defn{$B_z$} similarly.
        Let $\defnm{Z'} := Z - B_a - B_z$.
        Let \defn{$p_a$} be the last point along $A$ (when going from $a$ to $z$) that is at distance at most $q$ from $B_a$, and let \defn{$p_z$} be the first point after $p_a$ in $A$ (when going from $a$ to $z$) that is at distance at most $q$ from $B_z$.
        Let \defn{$B_w$} be the subpath of $A$ between them.
        Let \defn{$P_a$} be a shortest path between $p_a$ and $B_a$, and let \defn{$P_z$} be a shortest path between $p_z$ and $B_z$.
        Since $a$ and $z$ lie in distinct elements of $\mathcal{C}_n$, we have $d_G(a, z) > 5q$, so $P_a$ and $P_z$ both have length exactly $q$.

        By definition of $p_a$, we have $d_G(B_a, B_w) \geq q$.
        By definition of $p_z$, we have $d_G(B_w, B_z) \geq q$.
        Since $P_a$ has length~$q$ and one of its endvertices, namely $p_a$, lies at distance greater than~$n$ from~$P$, its other endvertex, which lies in $B_a$, is at distance greater than $n-q$ from $P$.
        Consequently, this endvertex is at distance less than $q$ from $a$.
        The same argument applies to the endvertex of $P_z$ contained in $B_z$.
        Thus, $P_a \cup B_a \subseteq B_G(a, 2q)$ and $P_z \cup B_z \subseteq B_G(z, 2q)$ so $d_G(P_a \cup B_a, P_z \cup B_z) > q$, otherwise we would get $d_G(a, z) \leq 5q$.
        Finally, $d_G(P_a \cup B_w \cup P_z, P) > n-q$ and $Z' \subseteq B_G(P, n-2q)$ so $d_G(Z', P_a \cup B_w \cup P_z) > q$.
        Thus, the sets $B_a, P_a, B_w, P_z, B_z, Z'$ form a $q$-fat model of $K_3$ in $G$.

        By \cref{cl:small-diameter}, $C$ has diameter at most $10q$ so there is an $a-z$ path $S$ in $G$ of length at most $10q+2$.
        Then, $Q_aaSzQ_z$ is a walk of length at most $2n+10q+2 \leq 2t+10q$ between $y_a$ and $y_z$.
        Thus, $y_a, y_z \in V(P)$ satisfy $d_G(y_a, y_z) \leq 2t+10q \leq 2t+15q$ so by the `No long shortcut' assumption of \cref{hairy-yak}, we have $d_P(y_a, y_z) \leq b$.
        It is then straightforward to check using $b
        \ge 2t+30q$ that this $q$-fat model of $K_3$ in $G$ is contained in $B_G(y, b)$, where $y$ is a midpoint of the subpath of $P$ between $y_a$ and $y_z$.
        This contradicts the `Locally tree-like' assumption of \cref{hairy-yak}.
    \end{claimproof}

    \begin{figure}[ht]
        \centering
        \includegraphics[width=0.95\linewidth]{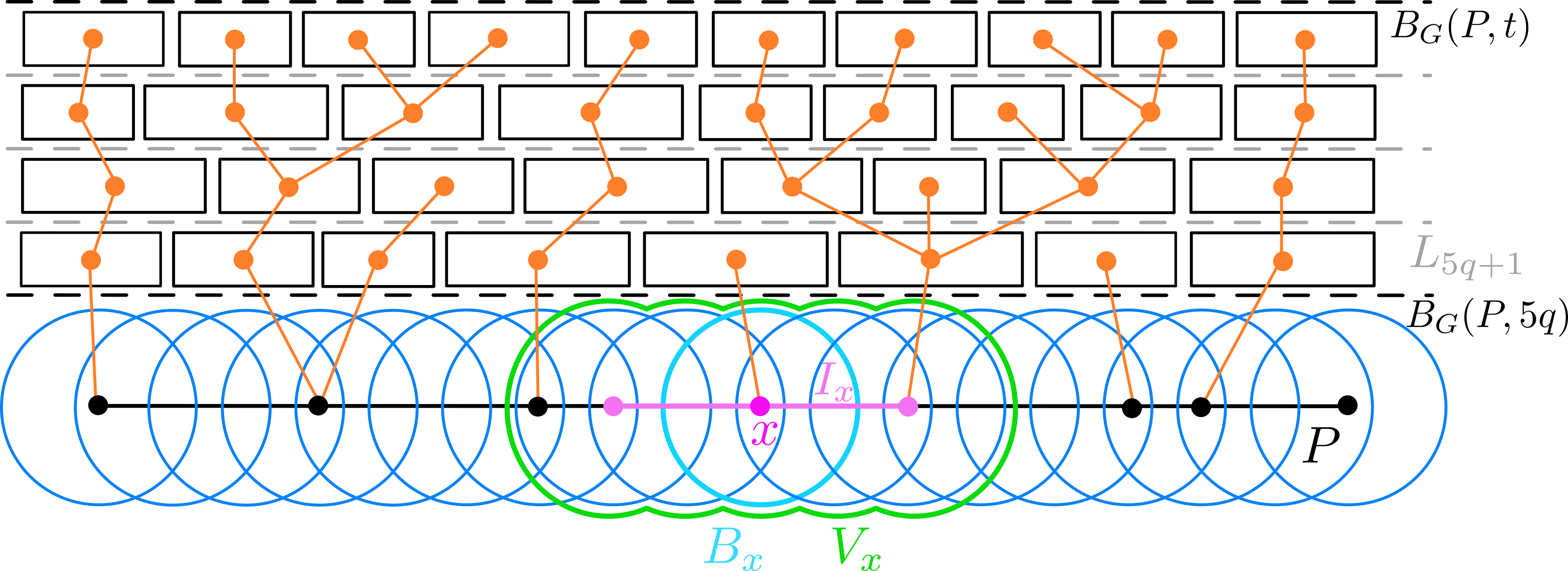}
        \caption{An illustration of the construction in \cref{hairy-yak}, except that the balls~$B_x$ should have radius~$5q+2$ instead of~$5q$. We did this to simplify the figure because otherwise the balls~$B_x$ would overlap with the near-components in layers $L_{5q+1}$ and $L_{5q+2}$.}
        \label{fig:HairyYak}
    \end{figure}

    \noindent\textbf{Definition of the partial tree-decomposition:}
    We now define a partial tree-decomposition \defn{$(T, \mathcal{V})$} of $G$ with support $B_G(P, t)$.
    We first construct the decomposition tree \defn{$T$}:
    Set $V(T) := V(P) \cup \bigcup_{n = 5q+1}^t \mathcal{C}_n$.
    Next, describe the edges of~$T$.
    For all $uv \in E(P)$, we let $uv \in E(T)$.
    For every $C \in \mathcal{C}_{5q+1}$, we add to $T$ an edge joining $C$ to an arbitrary vertex of $P$ at distance $5q+1$ from $C$ in $G$.
    If $C, C' \in \bigcup_{n = 5q+1}^t \mathcal{C}_n$ are adjacent in $G$, we add an edge $CC' \in E(T)$.
    This finishes the construction of~$T$.

    By construction, if $C, C' \in \mathcal{C}_{n}$ with $5q+1 \leq n \leq t$ then $C$ and $C'$ are not adjacent in $G$, so $CC' \notin E(T)$.
    Similarly, if $C \in \mathcal{C}_n$ and $C' \in \mathcal{C}_m$ with $|n-m| > 1$ then $C$ and $C'$ are not adjacent in $G$, so $CC' \notin E(T)$.
    Furthermore, if $C \in \mathcal{C}_{n}$ with $5q+1 < n \leq t$, then by \cref{cl:single-parent} there is a unique $C' \in \mathcal{C}_{n-1}$ which is adjacent to $C$ in $G$.
    Since $P$ is a path, this implies that $T$ is a tree.

    We now describe the bags of the partial tree-decomposition.
    For every vertex $x \in V(P)$, let $\defnm{B_x} \coloneqq B_G(x, 5q+2)$ and let $\defnm{I_x} \coloneqq \{y \in V(P) : d_P(x, y) \leq \varepsilon\}$. Then, set $\defnm{V_x} \coloneqq \bigcup_{y \in I_x}B_y$.
    For every vertex $C \in \bigcup_{n = 5q+1}^t \mathcal{C}_n$, set $\defnm{V_C} \coloneqq B_G(C, 1) \cap B_G(P, t)$.
    This completes the construction of $(T, \mathcal{V})$.
    \medskip

    \noindent\textbf{Verification that it is a  tree-decomposition:}
    We now show that $(T, \mathcal{V})$ is indeed a $(w_{\ref{hairy-yak}}(q, \varepsilon)$, $ s_{\ref{hairy-yak}}(\varepsilon))$-radial partial tree-decomposition of $G$ with support $B_G(P, t)$.
    \smallskip

    \ref{itm:H1}: Since the $5q$-near-components of $L_n$ form a partition of $L_n$ for every $5q < n \leq t$, and the vertices at distance at most $5q$ from $P$ are covered by the bags indexed by the vertices of $P$, it follows that every vertex in $B_G(P, t)$ is contained in some bag in $\mathcal{V}$.
    Let $u, v \in B_G(P, t)$ be adjacent vertices.
    Suppose first that $u, v \in B_G(P, 5q)$.
    Let $x \in V(P)$ be a vertex at distance at most $5q$ from $u$.
    Then, $v$ is at distance at most $5q+1$ from $x$.
    Thus, $u, v \in B_x \subseteq V_x$.
    Otherwise, one of $u,v$, say $u$, is at distance at least $5q+1$ from $P$, so belongs to some near-component~$C$. Then, $u, v \in B_G(C, 1) \cap B_G(P, t) = V_C$.
    \medskip

    \ref{itm:H2}: We view $T$ as a tree rooted at the path~$P$.
    Accordingly, we may speak of \emph{parents} and \emph{ancestors} in $T$, with the convention that distinct vertices of $P$ are incomparable.
    By construction, every node of $T$ has a unique ancestor in $P$.

    \begin{claim}\label{cl:ancestors-in-Tv}
        Let $C \in \mathcal{C}_{5q+1}$, and let $y \in V(P)$ be the parent of $C$ in $T$.
        Then, $V_C \subseteq V_y$.
    \end{claim}

    \begin{claimproof}
        Let $v \in V_C$ be arbitrary.
        Then $v \in B_G(C, 1)$, so there exists $u \in C$ such that $v \in B_G(u, 1)$.
        Let $z \in V(P)$ such that $d_G(u, z) = d_G(u, P) = 5q+1$.
        Then, $v \in B_G(z, 5q + 2) = B_z$.
        By the construction of~$T$, there exists $u' \in C$ such that $d_G(u', y) = 5q+1$.
        By \cref{cl:small-diameter}, we then have $d_G(y, z) \leq d_G(y, u') + d_G(u', u) + d_G(u, z) \leq 5q+1 + 10q + 5q+1 \leq 20q+4$.
        By the `Locally almost geodesic' assumption of \cref{hairy-yak}, this implies $d_P(y, z) \leq \varepsilon$, so $z \in I_y$. Hence, $B_z \subseteq V_y$, and thus $v \in V_y$.
    \end{claimproof}

    \begin{claim}\label{cl:Pv-interval}
        Let $v \in B_G(P, t)$ be arbitrary, and set $\defnm{P_v} := T_v \cap V(P)$.
        Then $P_v$ is an interval of $P$ of length at most $3\varepsilon$.
    \end{claim}

    \begin{claimproof}
        Choose arbitrarily one of the two orientations of $P$.
        First, we prove that $P_v$ forms an interval of $P$.
        If not, there exist $x_1, x_2, x_3 \in V(P)$ which appear in this order along $P$, so that $v \in V_{x_1} \cap V_{x_3}$ and $v \notin V_{x_2}$.
        Then, there exist $y_1 \in I_{x_1}$ and $y_3 \in I_{x_3}$ such that $v \in B_{y_1} \cap B_{y_3}$.
        Since $v \notin V_{x_2}$, we have $y_1, y_3 \notin I_{x_2}$.
        Since $x_1$ appears before $x_2$ in $P$ and $y_1 \in I_{x_1} \setminus I_{x_2}$, $y_1$ must appear before $x_2$ in $P$.
        Similarly, $y_3$ must appear after $x_2$ in $P$.
        Thus, $y_1$ appears before $I_{x_2}$ in $P$ and $y_3$ appears after $I_{x_2}$ in $P$, so $d_P(y_1, y_3) > 2\varepsilon$.
        However, $d_G(y_1, y_3) \leq d_G(y_1, v) + d_G(v, y_3) \leq 2 \cdot (5q+2) \leq 20q+4$.
        By the `Locally almost geodesic' assumption of \cref{hairy-yak}, this implies $d_P(y_1, y_3) \leq \varepsilon$, a contradiction.
        \smallskip

        Next, we prove that the interval $P_v$ has length at most $3\varepsilon$.
        Let $x_1, x_2 \in P_v$.
        There exist $y_1 \in I_{x_1}$ and $y_2 \in I_{x_2}$ such that $v \in B_{y_1} \cap B_{y_2}$.
        Then, $d_G(y_1, y_2) \leq d_G(y_1, v) + d_G(v, y_2) \leq 2 \cdot (5q+2) \leq 20q+4$.
        By the `Locally almost geodesic' assumption of \cref{hairy-yak}, this implies $d_P(y_1, y_2) \leq \varepsilon$.
        Thus, $d_P(x_1, x_2) \leq d_P(x_1, y_1) + d_P(y_1, y_2) + d_P(y_2, x_2) \leq 3\varepsilon$.
    \end{claimproof}

    We can now verify that $(T, \mathcal{V})$ satisfies~\ref{itm:H2} and has radial spread at most $s_{\ref{hairy-yak}}(\varepsilon) \coloneqq 3\varepsilon + 2$.

    \begin{claim} \label{claim:HairyYak:RadialSpread}
        For every $v \in B_G(P, t)$, $T_v$ is a subtree of $T$ of radius at most $3\varepsilon + 2$.
    \end{claim}

    \begin{claimproof}
        Let $v \in B_G(P, t)$ be arbitrary.

        Suppose first that $d_G(v, P) > 5q$.
        Let $C \in \bigcup_{n = 5q+1}^t \mathcal{C}_n$ be the unique near-component that contains $v$.
        Let $x\in V(T_v)$ be arbitrary. We show that the unique $x$--$C$ path in $T$ has length at most $3\varepsilon+2$ and lies entirely in $T_v$.

        Assume first that $x \in V(P)$, so $x \in P_v$.
        Since $P_v \neq \emptyset$, we must have $d_G(v, P) \leq 5q + 2$, so $C$ is at distance at most $2$ in $T$ from its unique ancestor in $P$.
        Since $P_v$ forms an interval of $P$ of length at most $3\varepsilon$ by \cref{cl:Pv-interval}, it suffices to prove that all the ancestors of $C$ in $T$ are in $T_v$.
        If $C \in \mathcal{C}_{5q+1}$, we are done by \cref{cl:ancestors-in-Tv}.
        Suppose now that $C \in \mathcal{C}_{5q+2}$, and let $u$ be any vertex of $G$ such that $u \in L_{5q+1} \cap B_G(v, 1)$.
        Let $C' \in \mathcal{C}_{5q+1}$ be the near-component that contains $u$.
        Since $uv\in E(G)$, we have $CC'\in E(T)$. Hence, $C'$ is the unique ancestor of $C$ belonging to~$\mathcal{C}_{5q+1}$.
        Now, $v \in B_G(u, 1)$, so $v \in B_G(C', 1) \cap B_G(P, t) = V_{C'}$ and so $C' \in T_v$.
        By \cref{cl:ancestors-in-Tv}, the parent of $C'$ in $T$ is also in $T_v$, and we are done.

        Assume now that $x \notin V(P)$, so $x = C'$ for some $C' \in \bigcup_{n = 5q+1}^t \mathcal{C}_n$.
        Since $x \in T_v$, we have $v \in V_{x} = V_{C'} \subseteq B_G(C', 1)$, so there exists $u \in C'$ such that $d_G(u, v) \leq 1$.
        Then, either $C' = C$, or $C'$ is adjacent to $C$ in $G$.
        In both cases, the unique path between $C$ and $C'$ in $T$ has length at most $1$ and is contained in $T_v$.
        \smallskip

        Consider now the case where $d_G(v, P) \leq 5q$.
        Let $x \in T_v \setminus P_v$ be arbitrary.
        Then $x=C$ for some $5q$-near-component $C$ of a layer $L_n$.
        Since $v \in V_x = V_C \subseteq B_G(C, 1)$, we must have $n = 5q+1$, so $x$ is at distance $1$ from its ancestor in $P$.
        Furthermore, by \cref{cl:ancestors-in-Tv}, the parent of $x$ in $T$ is also in~$T_v$ (hence is in~$P_v$).
        This concludes the proof since $P_v$ forms an interval of~$P$ of length at most~$3\varepsilon$ by \cref{cl:Pv-interval}.
    \end{claimproof}

    \medskip

    \noindent\textbf{Radial width:}

    It remains to show that $(T, \mathcal{V})$ has radial width at most $w_{\ref{hairy-yak}}(q, \varepsilon) \coloneqq \max\{10q+1, \varepsilon + 5q + 2\}$.

    \begin{claim}
        For every node $x \in V(T)$, the bag $V_x$ has radius at most $w_{\ref{hairy-yak}}(q, \varepsilon)$ in $G$.
    \end{claim}

    \begin{claimproof}
        Suppose first that $x = C$ where $C$ is a near-component $C \in \bigcup_{n = 5q+1}^t \mathcal{C}_n$.
        By \cref{cl:small-diameter}, $C$ has diameter at most $10q$ in $G$, so $V_x = V_C \subseteq B_G(C, 1)$ has radius at most $10q+1$ in $G$.

        Suppose now that $x \in V(P)$ and let $u \in V_x$ be arbitrary.
        By definition, there exists $y \in I_x$ such that $u \in B_y$.
        Then, $d_{G}(x, u) \leq d_P(x, y) + d_G(y, u) \leq \varepsilon + 5q+2$, so $V_x$ has radius at most $\varepsilon + 5q+2$ in $G$.
    \end{claimproof}

    Thus, $(T, \mathcal{V})$ is a $(w_{\ref{hairy-yak}}(q, \varepsilon)$, $ s_{\ref{hairy-yak}}(\varepsilon))$-radial partial tree-decomposition of $G$ with support $B_G(P, t)$. \end{proof}

We now turn to the second auxiliary lemma, which supplies the main construction used in the proof of \cref{lem:AngryHippo}.
We are given a $q'$-fat cycle $C$ that is locally tree-like, in the sense that no ball of radius $\ell$ centered on a vertex of $C$ contains a $q$-fat model of $K_3$.
Our goal is to construct a $q$-fat corsola cycle near $C$.
\smallskip

The original cycle $C$ cannot necessarily be used directly: it need not satisfy either the tree-decomposition condition \ref{itm:Corsola:TreeDecomp} or the quasi-geodesic condition \ref{itm:Corsola:QuasiGeodesic}.
To enforce \ref{itm:Corsola:QuasiGeodesic}, it is natural to replace suitable portions of $C$ by geodesic shortcuts.
Such replacements must be chosen carefully, however, since introducing one shortcut may create new shortcuts in the resulting path.
\smallskip

The next lemma resolves this difficulty by constructing a controlled sequence of geodesic paths, which we call \emph{jumps}.
The construction also detects the first long shortcut that would prevent an application of \cref{hairy-yak}.
Rather than eliminating this final shortcut, we use it as the short path $W_{C'}$ of the corsola cycle.
The remaining part of the jump sequence becomes $P_{C'}$.
A carefully chosen buffer around the final jump ensures that the resulting cycle remains $q$-fat.
\smallskip

More precisely, starting from the $q'$-fat cycle $C$, we construct vertices $v_0,\ldots,v_k$ of $C$ and geodesic paths $J_1,\ldots,J_k$, where each $J_i$ joins $v_{i-1}$ to $v_i$ and has length $\ell$.
The situation and the process are illustrated in \cref{fig:HappyHippo}.
Before choosing $J_{i+1}$, we place a buffer near the end of $J_i$ and ask whether there is a long shortcut from $v_i$ to an earlier point of the jump sequence that avoids this buffer.
If such a shortcut exists, the construction stops.
Otherwise, we carefully choose the next endvertex $v_{i+1}$ and continue.
The resulting sequence satisfies precisely the properties needed to construct the desired corsola cycle.
\smallskip

We now formalize this construction.

\begin{lemma} \label{lem:NotSoAngryHippo}
    Let $q, L, r, \ell, q' \in \N$ satisfy $r \geq 10q, \ell > 8r$ and $q' > \ell+1$.
    Let $C$ be a $q'$-fat cycle in a graph $G$ such that, for every $c\in V(C)$, the ball $B_G(c,\ell)$ contains no $q$-fat model of $K_3$ in $G$.
    There exist vertices $v_0, \ldots, v_k$ of $C$ and paths $J_1, \ldots, J_k$, called \defn{jumps}, satisfying the following properties.
    For every $i \in [k]$, the path $J_i$ has length $\ell$ and joins $v_{i-1}$ to $v_i$.
    Moreover, for each $i \in [k]$, let $\defnm{o_i}$ be the vertex at distance $2r$ from $v_i$ along $J_i$, and let $\defnm{O_i} := B_G(o_i, r)$; see \cref{subfig:HappyHippo:AllJumps}. Then:
    \begin{enumerate}[label=\rm{(\arabic*)}]
        \item\label{item:geodesic} For every $i \in [k]$, the path $J_i$ is geodesic in $G$.
        \item\label{item:avoids-Oi} For every $i \in \{2, \ldots, k\}$, we have $J_i \subseteq G - O_{i-1}$.
        \item\label{item:no-shortcut} For every $i \in [k-1]$, every path of length at most $L$ in~$G$ from $v_i$ to a vertex appearing before the last occurrence of $o_i$ in the jump sequence $J_1 \cup \cdots \cup J_{i}$ intersects $B_G(o_i, r/2)$.
        \item\label{item:final-shortcut} There exists a path of length at most $L$ from $v_k$ to a vertex appearing before the last occurrence of $o_k$ in the jump sequence $J_1 \cup \cdots \cup J_k$ that avoids $B_G(o_k, r/2)$.
    \end{enumerate}
\end{lemma}

The role of the individual conclusions of \cref{lem:NotSoAngryHippo} is worth emphasizing.
The path provided by \ref{item:final-shortcut} will become $W_{C'}$.
Since this path avoids the buffer around the final jump, properties \ref{item:geodesic}, \ref{item:avoids-Oi} and \ref{item:no-shortcut} allow us to extract a subpath~$P_{C'}$ of the jump sequence containing a protected geodesic segment inside the buffer, which will witness the fatness of $C'$.
These properties also control how short paths can traverse the successive buffers, yielding \ref{itm:Corsola:QuasiGeodesic} and the metric hypotheses needed to apply \cref{hairy-yak}.

We remark that we later apply \cref{lem:NotSoAngryHippo} with $L > \ell$. It then follows that each $J_i$ is disjoint from all $J_j$ with $j \leq i-2$, and that if $J_i$ meets $J_{i-1}$ in some vertex~$u$, then $u$ appears on $J_{i-1}$ after~$o_{i-1}$. Indeed, since $J_i$ has length $\ell < L$ and because $J_i \subseteq G-O_{i-1} \subseteq G-B_G(o_{i-1},r/2)$ by~\ref{item:avoids-Oi}, property \ref{item:no-shortcut} ensures that $J_i$ avoids all $J_j$ with $j \leq i-2$, and that it cannot meet $J_{i-1}$ before~$o_{i-1}$.

\begin{proof}[Proof of \cref{lem:NotSoAngryHippo}.]
    \noindent\textbf{Setup:}
    Let $\defnm{P_0,\ldots,P_5}$ witness that $C$ is a $q'$-fat cycle in $G$; that is, $V(C) = V(P_0)\; \dot\cup \ldots \dot\cup\; V(P_5)$, and $d_G(P_i,P_j) \geq q'$ whenever $i,j\in\mathbb{Z}_6$ are distinct and non-consecutive.
    Choose arbitrarily one of the two cyclic orientations of $C$, so that we can speak of the \emph{successor} and the \emph{predecessor} of a vertex of $C$.
    This induces an orientation of each $P_i$ so that the last vertex of $P_i$ is adjacent to the first vertex of $P_{i+1}$ for every $i \in \mathbb{Z}_6$.
    Then, for every $i \in \mathbb{Z}_6$ we can view $P_i \cup P_{i+1}$ as a path from the first vertex of $P_i$ to the last vertex of $P_{i+1}$, and we can talk about the \emph{last} vertex along this path with a certain property.
    Throughout the construction, we additionally maintain the following invariant, on which we will rely to justify that the recursive construction is well defined:
    \begin{enumerate}[label=\rm{(\arabic*)}, start=5]
        \item\label{item:last-in-ball} For every $i \in [k]$, if $v_i \in V(P_j)$ for some $j \in \mathbb{Z}_6$, then $v_i$ is the last vertex of $P_j \cup P_{j+1}$ contained in $B_G(o_i, 2r)$.
    \end{enumerate}
    \smallskip

    \begin{figure}[ht]
        \centering
        \begin{subfigure}[b]{0.45\linewidth}
            \centering
            \includegraphics[width=1\linewidth]{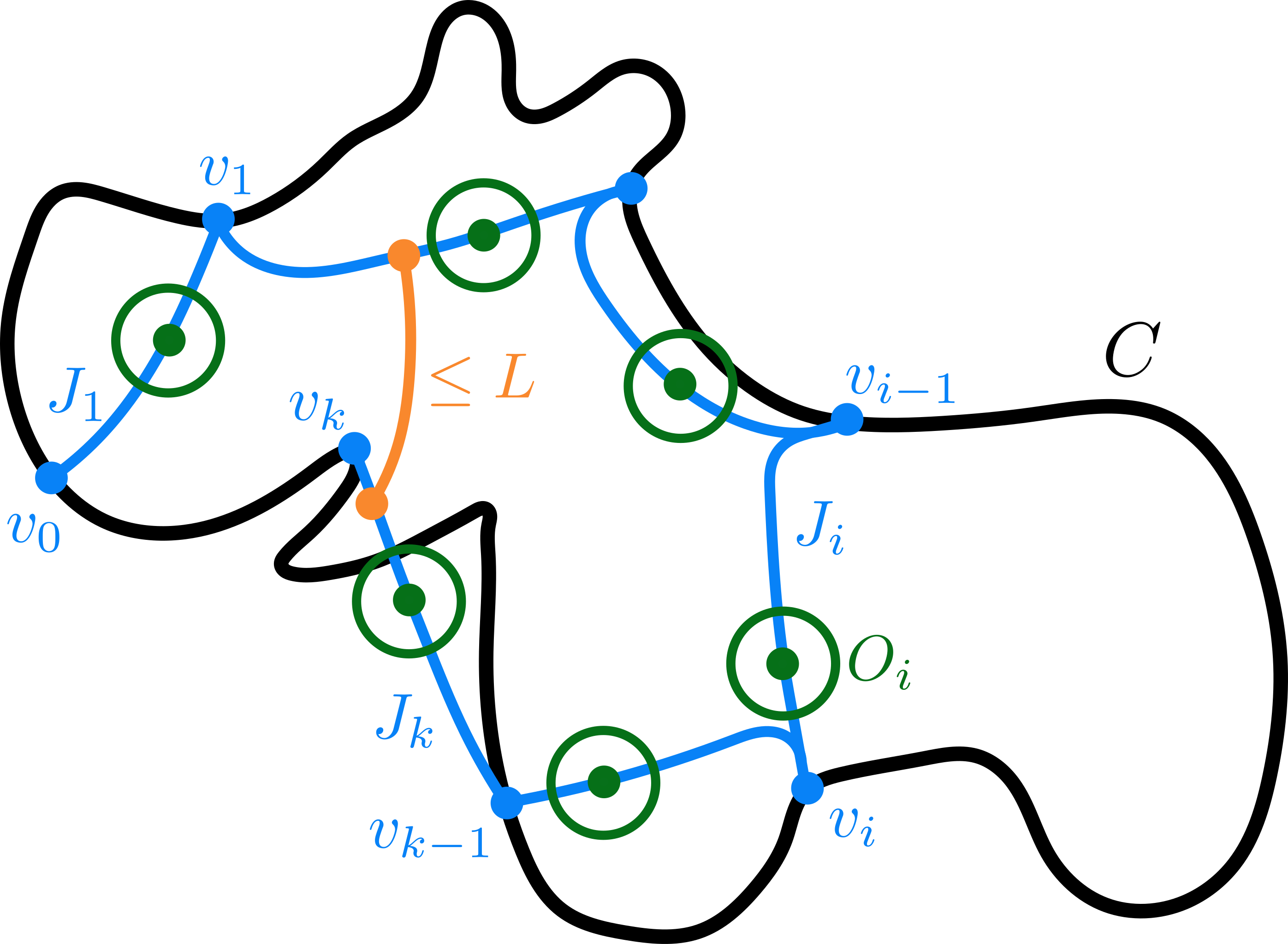}
            \caption{The jumps $J_i$.}
            \label{subfig:HappyHippo:AllJumps}
        \end{subfigure}
        \begin{subfigure}[b]{0.45\linewidth}
            \centering
            \includegraphics[width=1\linewidth]{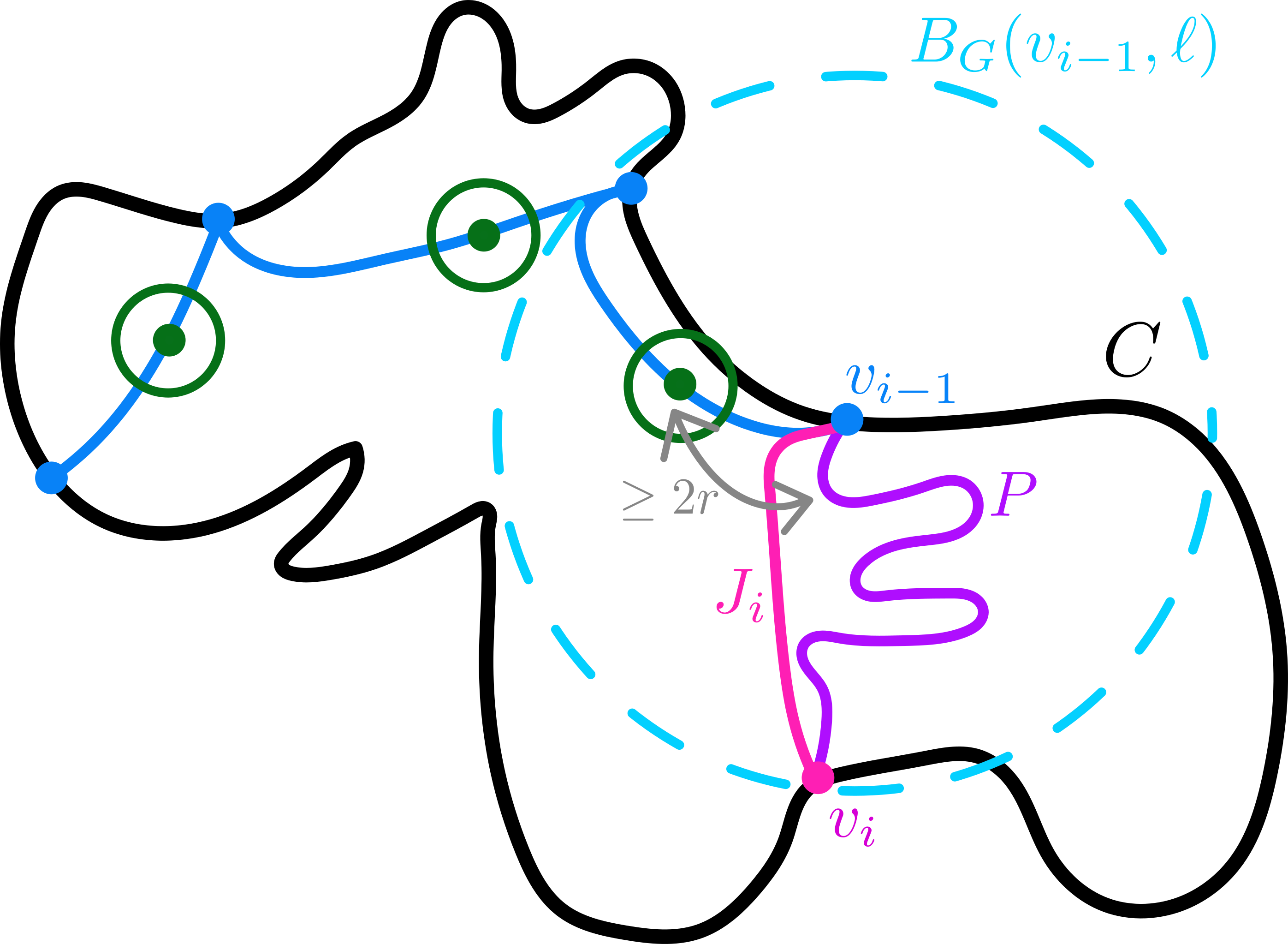}
            \caption{Choice of a new jump $J_i$.}
            \label{subfig:HappyHippo:NewJump}
        \end{subfigure}
        \caption{An illustration of the situation in the proof of~\cref{lem:NotSoAngryHippo}. The $q'$-fat cycle $C$ is depicted in black, and the jumps~$J_i$ are indicated in blue.
        Each jump $J_i$ avoids the previous ball $O_{i-1} = B_G(o_{i-1}, r)$ (depicted in green).
        The process of choosing the jumps $J_i$ terminates as soon as we find a shortcut of length $\leq L$ as in \ref{item:final-shortcut} (indicated in orange).}
        \label{fig:HappyHippo}
    \end{figure}

    \noindent\textbf{Definition of the first jump:} Let \defn{$v_0$} be the first vertex of $P_0$ and let \defn{$v_1$} be the last vertex along $P_0 \cup P_1$ that is at distance at most $\ell$ from $v_0$ in $G$.
    Let \defn{$J_1$} be a shortest path in $G$ from $v_0$ to $v_1$.
    Then $v_0, v_1 \in V(C)$, and $J_1$ is a path from $v_0$ to $v_1$ in $G$.
    The last vertex of $P_5$ is adjacent to $v_0$.
    Since $P_5$ and $P_1$ are non-consecutive, we have $d_G(v_0, P_1) \geq q'-1 > \ell$.
    Thus, $v_1 \in V(P_0)$ and $v_1$ is at distance exactly $\ell$ from $v_0$, so $J_1$ has length $\ell$.

    \ref{item:geodesic} holds by construction. There is nothing to verify for \ref{item:avoids-Oi}, and we do not have to care about \ref{item:no-shortcut} and \ref{item:final-shortcut} yet.
    By definition of~$o_1$ and since $J_1$ is a geodesic, we have $d_G(v_0, o_1) = \ell - 2r$.
    Thus, every vertex in $B_G(o_1, 2r)$ is at distance at most $\ell$ from $v_0$ in $G$.
    Since $v_1$ is the last vertex along $P_0 \cup P_1$ that is at distance at most $\ell$ from $v_0$ in $G$, it follows that $v_1$ is the last vertex of $P_0 \cup P_{1}$ in $B_G(o_1, 2r)$, so \ref{item:last-in-ball} holds.
    \medskip

    \noindent\textbf{Definition of the $(i+1)$st jump:} Suppose that, for some $i\geq 1$, we have already defined vertices $v_0,\ldots,v_i \in V(C)$ and jumps $J_1,\ldots,J_i$ of length~$\ell$ satisfying \ref{item:geodesic}, \ref{item:avoids-Oi} and \ref{item:last-in-ball}.
    If there exists a path of length at most $L$ from $v_i$ to a vertex appearing before the last occurrence of~$o_i$ in the jump sequence $J_1 \cup \cdots \cup J_i$ that avoids $B_G(o_i, r/2)$, then we stop the construction and set $\defnm{k} \coloneqq i$.
    This ensures that \ref{item:no-shortcut} holds throughout the procedure, and that \ref{item:final-shortcut} will hold in the end.
    Note that \ref{item:no-shortcut} implies that all $v_i$ are pairwise distinct for $i \in [k-1]$, so the construction eventually stops (as all $v_i$ lie on $C$ and $C$ is finite), and the jump sequence then not only satisfies \ref{item:geodesic} to \ref{item:no-shortcut} but also~\ref{item:final-shortcut}.

    Suppose now that there is no such shortcut. We define $v_{i+1}$ and $J_{i+1}$ as follows (see \cref{subfig:HappyHippo:NewJump}).
    Let $\defnm{j} \in \mathbb{Z}_6$ such that $v_i \in V(P_j)$.
    Let \defn{$v_{i+1}$} be the last vertex along $P_j\cup P_{j+1}$ that lies at distance at most $\ell$ from $v_i$ in $G$ and for which there exists a $v_i$--$v_{i+1}$ path contained in $B_G(v_i,\ell)$ whose vertex closest to $o_i$ in $G$ is $v_i$.
    (Note that $v_{i}$ is a candidate for $v_{i+1}$.)
    Let \defn{$J_{i+1}$} be a shortest path between $v_i$ and $v_{i+1}$ in $G$.
    \medskip

    \noindent\textbf{Verification of the properties of the jumps:}
    We now prove that $v_{i+1}$ and $J_{i+1}$ satisfy the desired properties, i.e.\ $J_{i+1}$ has length~$\ell$ and satisfies \ref{item:geodesic}, \ref{item:avoids-Oi} and \ref{item:last-in-ball}.
    Since our criterion for terminating the construction process (as mentioned in the last paragraph) ensures that \ref{item:no-shortcut} is satisfied (and \ref{item:final-shortcut} will also be satisfied in the end for the same reason), this then concludes the proof.

    First, observe that \ref{item:geodesic} holds by construction (as $J_{i+1}$ is a shortest path).

    \begin{claim}
        $J_{i+1}$ has length $\ell$.
    \end{claim}

    \begin{claimproof}
        By the choice of $v_{i+1}$ and $J_{i+1}$, we have $||J_{i+1}|| = d_G(v_i,v_{i+1}) \leq \ell$.
        Now suppose for a contradiction that $d_G(v_i, v_{i+1}) < \ell$, and consider the successor $v'$ of $v_{i+1}$ in $C$.
        Then, $d_G(v_i, v') \leq \ell \leq q'-1$ so $v' \in V(P_j) \cup V(P_{j+1})$ (as $v_i \in V(P_j)$ is at distance at least~$q'$ from $P_{j+2}$ since $C$ is $q'$-fat), and appending $v'$ to any $v_i$--$v_{i+1}$ path contained in $B_G(v_i, \ell)$ whose vertex closest to $o_i$ in $G$ is $v_i$ yields a $v_i$--$v'$ path contained in $B_G(v_i, \ell)$ whose vertex closest to $o_i$ in $G$ is $v_i$ by \ref{item:last-in-ball} for $i$ (and because $d_G(v_i,o_i) = 2r$ by the definition of $o_i$).
        This would contradict the definition of $v_{i+1}$, so we must have $d_G(v_i, v_{i+1}) = \ell$.
    \end{claimproof}

    \ref{item:avoids-Oi}: Suppose for a contradiction that $J_{i+1}$ does not satisfy~\ref{item:avoids-Oi}, i.e.\ $J_{i+1}$ intersects $O_{i}$.
    Let $P$ be a $v_i$--$v_{i+1}$ path contained in $B_G(v_i, \ell)$ whose vertex closest to $o_i$ in $G$ is $v_i$; such a $P$ exists by the choice of $v_{i+1}$.
    Since $J_{i+1}$ intersects $O_i$ and $d_G(v_i, O_i) = r$, the jump $J_{i+1}$ contains a subpath $J'$ of length at least $r/2$ that is contained in $B_G(O_i, r/2)$, and hence in $B_G(o_i, 3r/2)$.
    However, $d_G(o_i, P) = 2 r$, so $d_G(P, J') \geq r/2 > q$.
    Since $r/2 \geq 5q$, we can divide $J'$ into five pairwise disjoint subpaths $J'_1, \ldots, J'_5$, each of length at least $q-1$.
    Then, since $J'$ is geodesic, we have $d_G(J'_j, J'_{j'}) > q$ whenever $j, j' \in [5]$ are distinct and non-consecutive, and $d_G(J_{i+1} \setminus J', J'_2 \cup J'_3 \cup J'_4) > q$.
    Therefore, $J'_1, \ldots, J'_5$ and $(J_{i+1} \setminus J') \cup P$ form a $q$-fat model of $K_3$ in $G$ that is contained in $B_G(v_i, \ell)$, a contradiction.
    Thus, $J_{i+1} \subseteq G - O_i$.
    \medskip

    \ref{item:last-in-ball}: Let $j' \in \mathbb{Z}_6$ such that $v_{i+1} \in V(P_{j'})$. We need to show that $v_{i+1}$ is the last vertex of $P_{j'} \cup P_{j'+1}$ contained in $B_G(o_{i+1}, 2r)$.

    For this, let $u$ be an arbitrary vertex of $P_{j'} \cup P_{j'+1}$ contained in $B_G(o_{i+1}, 2r)$, and let $j \in \mathbb{Z}_6$ such that $v_i \in V(P_j)$.
    Then, $u \in V(P_{j'}) \cup V(P_{j'+1}) \subseteq V(P_j) \cup V(P_{j+1}) \cup V(P_{j+2})$ and
    $d_G(v_i, u) \leq d_G(v_i, o_{i+1}) + d_G(o_{i+1}, u) \leq (\ell - 2r) + 2r = \ell < q'$, so $u \in V(P_j) \cup V(P_{j+1})$.
    Let $P$ be a $v_i$--$v_{i+1}$ path contained in $B_G(v_i, \ell)$ whose vertex closest to $o_i$ in $G$ is $v_i$ (which exists by the choice of $v_{i+1})$.
    Let $P'$ be the path obtained from $P$ by appending a shortest path from $v_{i+1}$ to $o_{i+1}$ and a shortest path from $o_{i+1}$ to $u$.
    Then, $P'$ is a $v_i$--$u$ walk contained in $B_G(v_i, \ell)$ since all the vertices we added to $P$ are at distance at most $2r$ from $o_{i+1}$, hence at distance at most $\ell$ from $v_i$.
    For the same reason, all the vertices we added to $P$ are at distance at least $\ell - 4r$ from $v_i$, hence at least $\ell - 6r > 2r$ from $o_i$, so the vertex of $P'$ that is closest to $o_i$ in $G$ is $v_i$.
    Thus, $u$ is a vertex along $P_j \cup P_{j+1}$ that lies at distance at most $\ell$ from $v_i$ in $G$ and for which there exists a $v_i$--$u$ path contained in $B_G(v_i,\ell)$ whose vertex closest to $o_i$ in $G$ is $v_i$.
    By definition, $v_{i+1}$ appears no earlier than $u$ on $P_j \cup P_{j+1}$, and hence in $P_{j'} \cup P_{j'+1}$, which concludes the proof.
\end{proof}

We are now ready to prove \cref{lem:AngryHippo}, which we restate here for convenience.

\happyhippo*

We start by applying \cref{lem:NotSoAngryHippo} to obtain a jump sequence with the properties described above.
The final shortcut given by \ref{item:final-shortcut} becomes the path $W_{C'}$.
From the portion of the jump sequence between its endvertices, we extract the complementary path $P_{C'}$.
\smallskip

The buffer associated with the final jump separates a long geodesic subpath of $P_{C'}$ from $W_{C'}$, and hence ensures that the resulting cycle $C'=P_{C'}\cup W_{C'}$ is $q$-fat.
The construction of the jump sequence gives the quasi-geodesic property \ref{itm:Corsola:QuasiGeodesic}, as well as the `Locally almost-geodesic' and `No long shortcut' conditions required by \cref{hairy-yak}.
Applying \cref{hairy-yak} to $P_{C'}$ therefore yields the partial tree-decomposition required by \ref{itm:Corsola:TreeDecomp}.
The remaining bounds, including the distance from $P_{C'}$ to $U$, follow from the choice of the parameters and from the fact that the entire construction remains close to the original cycle $C$.

\renewcommand{\qedsymbol}{
\includegraphics[height=0.8em]{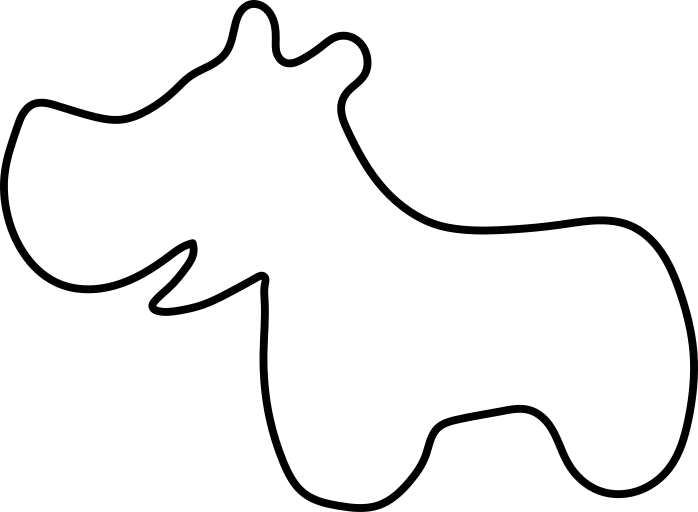}
}

\begin{proof}[Proof of \cref{lem:AngryHippo}]
    \noindent\textbf{Definition of parameters and functions:}
    We first introduce several parameters that will appear in the course of the proof. Let \begin{align*}
        r &\coloneqq 40q+8, \\
        \ell &\coloneqq 8r+1,\\
        q' &\coloneqq \defnm{q'_{\ref{lem:AngryHippo}}}(q) \coloneqq \ell + 2,\\
        \defnm{w_{\ref{lem:AngryHippo}}}(q) &\coloneqq w_{\ref{hairy-yak}}(q, 2\ell),\\
        \defnm{s_{\ref{lem:AngryHippo}}}(q) &\coloneqq s_{\ref{hairy-yak}}(2\ell),\\
        \kappa &\coloneqq \defnm{\kappa_{\ref{lem:AngryHippo}}}(q) \coloneqq 3\ell, \\
        a &\coloneqq a_{\ref{lem:AngryHippo}}(q) \coloneqq q + \ell,\\
        b' &\coloneqq 8t+60q+\kappa,\\
        b &\coloneqq \defnm{b_{\ref{lem:AngryHippo}}}(q, t) \coloneqq b' + \ell, \text{ and } \\
        L &\coloneqq 2t+15q+\ell+r \leq 77t.
    \end{align*}

    \noindent\textbf{Application of \cref{lem:NotSoAngryHippo}:}
    By \cref{lem:NotSoAngryHippo}, there exist vertices $\defnm{v_0, \ldots, v_k}$ of $C$ and jumps $\defnm{J_1, \ldots, J_k}$, satisfying the following properties.
    For every $i \in [k]$, the path $J_i$ has length $\ell$ and joins $v_{i-1}$ to $v_i$.
    Moreover, for each $i \in [k]$, let $\defnm{o_i}$ be the vertex at distance $2r$ from $v_i$ along $J_i$, and let $\defnm{O_i} := B_G(o_i, r)$. Then:
    \begin{enumerate}[label=\rm{(\arabic*)}]
        \item\label{item':geodesic} For every $i \in [k]$, the path $J_i$ is geodesic in $G$.
        \item\label{item':avoids-Oi} For every $i \in \{2, \ldots, k\}$, we have $J_i \subseteq G - O_{i-1}$.
        \item\label{item':no-shortcut} For every $i \in [k-1]$, every path of length at most $L$ from $v_i$ to a vertex appearing before the last occurrence of $o_i$ in the jump sequence $J_1 \cup \cdots \cup J_{i}$ intersects $B_G(o_i, r/2)$.\\
        In particular, $J_i$ is disjoint from all $J_j$ with $j \leq i-2$, and if $J_i$ meets $J_{i-1}$ in some vertex~$u$, then $u$ appears on $J_{i-1}$ after~$o_{i-1}$.
        \item\label{item':final-shortcut} There exists a path of length at most $L$ from $v_k$ to a vertex appearing before the last occurrence of $o_k$ in the jump sequence $J_1 \cup \cdots \cup J_k$ that avoids $B_G(o_k, r/2)$.
    \end{enumerate}

    Note that the `in particular' part of \ref{item':no-shortcut} was explained in the paragraph following \cref{lem:NotSoAngryHippo}.
    \smallskip

    \noindent\textbf{Some further properties of the jump sequence:}
    We analyse in greater detail where vertices may occur in the jump sequence. Our first observation shows that vertices close to $o_i$ can occur only in $J_i$.

    \begin{claim}\label{claim:Oi-unique}
        If $v \in B_G(o_i, r/2)$ for some $i \in [k]$, then $v$ does not appear in any $J_{i'}$ with $i' \neq i$.
    \end{claim}

    \begin{claimproof}
        Let $v \in B_G(o_i, r/2)$ for some $i \in [k]$, and suppose for a contradiction that $v$ lies in some jump~$J_{i'}$ with $i' \neq i$.
        Let $P$ be a shortest $o_i$--$v$ path in $G$, and note that $P$ has length at most $r/2$.

        Suppose first that $i' < i$. We have $d_G(v_{i-1}, v) \geq d_G(v_{i-1},o_i) - d_G(o_i,v) \geq (\ell-2r)-r/2 > 2r$. However, $d_G(v_{i-1}, o_{i-1}) = 2r$ and $J_{i-1}$ is a geodesic that visits $o_{i-1}$ and $v_{i-1}$, so $v$ appears before the last occurrence of $o_{i-1}$ in the jump sequence $J_1 \cup \cdots \cup J_{i-1}$.
        Then, $v_{i-1}J_io_iPv$ is a walk of length at most $(\ell - 2r) + r/2 \leq \ell \leq L$ from $v_{i-1}$ to a vertex appearing before the last occurrence of $o_{i-1}$ in the jump sequence $J_1 \cup \cdots \cup J_{i-1}$, that avoids $B_G(o_{i-1}, r/2)$ by \ref{item':avoids-Oi} and since $P$ has length at most $r/2$, contradicting \ref{item':no-shortcut}.

        Suppose now that $i' > i$. Then $i' > i+1$ because $J_{i+1} \subseteq G - O_i$ by \ref{item':avoids-Oi}.
        But then again $v_{i'-1}J_{i'}vPo_i$ is a walk of length at most $\ell + r/2 \leq L$ from $v_{i'-1}$ to a vertex appearing before the last occurrence of $o_{i'-1}$ in the jump sequence $J_1 \cup \cdots \cup J_{i'-1}$, that avoids $B_G(o_{i'-1}, r/2)$ by \ref{item':avoids-Oi} and since $P$ has length at most $r/2$, contradicting \ref{item':no-shortcut}.
    \end{claimproof}

    By \cref{claim:Oi-unique}, and since $J_k$ is a geodesic in $G$ by \ref{item':geodesic}, every vertex in $B_G(o_k,r/2)$ appears at most once in the jump sequence.
    We next show that no vertex can occur on both sides of $B_G(o_k,r/2)$.
    More precisely, let $p_k$ and $q_k$ denote, respectively, the first and last vertices of $B_G(o_k,r/2)$ appearing in the jump sequence.
    By the preceding observation, each of $p_k$ and $q_k$ occurs exactly once, namely in $J_k$.

    We say that a vertex appears \defn{before} $B_G(o_k,r/2)$ if it occurs strictly before $p_k$ in the jump sequence, and \defn{after} $B_G(o_k,r/2)$ if it occurs strictly after $q_k$.

    \begin{claim}\label{claim:Jk-unique}
        No vertex appears both before and after $B_G(o_k, r/2)$ in the jump sequence.
    \end{claim}

    \begin{claimproof}
        Suppose for a contradiction that some vertex $v$ appears both before and after $B_G(o_k, r/2)$ in the jump sequence.
        Then, $v$ appears after $o_k$ in $J_k$, so $d_G(o_k, v) \leq 2r$.

        We have $d_G(v_{k-1}, v) \geq d_G(v_{k-1},o_k) - d_G(o_k,v) \geq (\ell-2r)-2r > 2r$. However, $J_{k-1}$ is a geodesic that visits $o_{k-1}$ and $v_{k-1}$, and $d_G(v_{k-1}, o_{k-1}) = 2r$, so $v$ appears before the last occurrence of $o_{k-1}$ in the jump sequence $J_1 \cup \cdots \cup J_{k-1}$.
        Then, $v_{k-1}J_kv$ is a path of length at most $\ell \leq L$ from $v_{k-1}$ to a vertex appearing before $o_{k-1}$ in the jump sequence $J_1 \cup \cdots \cup J_{k-1}$, that avoids $B_G(o_{k-1}, r/2)$ by \ref{item':avoids-Oi}, contradicting \ref{item':no-shortcut}.
    \end{claimproof}

    \noindent\textbf{Definition of the corsola cycle:}
    By \ref{item':final-shortcut}, there is a path \defn{$W_{C'}$} of length at most $L$ between a vertex \defn{$y$} appearing after $o_k$ in the jump sequence and a vertex \defn{$x$} appearing before $o_k$ in the jump sequence, that avoids $B_G(o_k, r/2)$.
    Without loss of generality, we may assume that $W_{C'}$ is internally disjoint from the jump sequence.

    The concatenation of the jumps $J_i$ forms a walk $W_0$ from $v_0$ to $v_k$, which contains a subwalk between $x$ and $y$.
    Let \defn{$P_{C'}$} be a shortest path between $x$ and $y$ that is contained (edge-wise) in this subwalk. See\ \cref{fig:THEcycle} for an illustration.
    Note that by \ref{item':avoids-Oi} and \ref{item':no-shortcut}, if there are vertices that appear on two distinct jumps $J_i$ and $J_j$ with $i<j$, then $i = j-1$ and by \ref{item':geodesic}, they appear on $J_j$ in reverse order.
    In particular, this implies that the successor of a vertex in $P_{C'}$ is the successor of its last appearance along the jump sequence, and thus $P_{C'}$ proceeds only forward along the jump sequence.
    \medskip

    \begin{figure}[ht]
        \centering
        \includegraphics[width=0.5\linewidth]{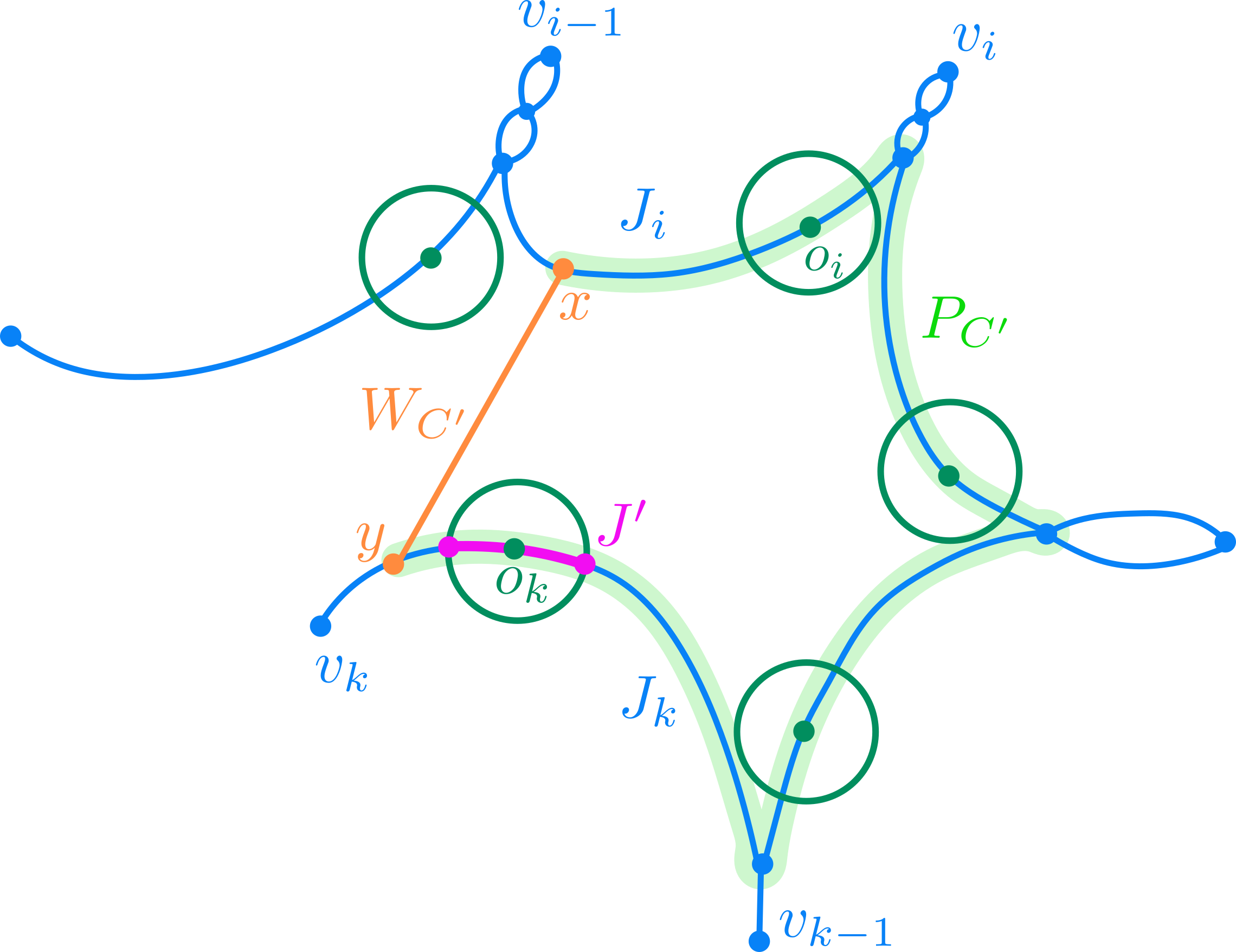}
        \caption{An illustration of the jumps $J_i$ (in blue), the path $W_{C'}$ (in orange) and the path $P_{C'}$ (in light green). The path $J'$, which we use to show that $P_{C'} \cup W_{C'}$ is $q$-fat, is indicated in pink.}
        \label{fig:THEcycle}
    \end{figure}

    Since $W_{C'}$ is internally disjoint from the jump sequence and since $P_{C'}$ is contained in the jump sequence, it follows that $\defnm{C'} \coloneqq P_{C'} \cup W_{C'}$ is a cycle.
    We claim that $C'$ is a $q$-fat $(w_{\ref{lem:AngryHippo}}(q), s_{\ref{lem:AngryHippo}}(q), t, \kappa_{\ref{lem:AngryHippo}}(q))$-corsola cycle. Since every $v \in V(P_{C'})$ is at distance at most $\ell$ from a vertex of $C$, and thus at distance at least $d_G(C, U) - \ell \geq a - \ell = q$ from $U$, the cycle~$C'$ will be as desired.
    \medskip

    \noindent\textbf{Fatness:}
    We first show that $C'$ is $q$-fat.
    By \cref{claim:Jk-unique}, no vertex appears both before and after $B_G(o_k,r/2)$ in the jump sequence.
    Moreover, by \cref{claim:Oi-unique}, every vertex of $B_G(o_k,r/2)$ appears at most once in the jump sequence, necessarily in $J_k$.
    It follows that $P_{C'}$ contains $J' \coloneqq J_k \cap B_G(o_k,r/2)$.
    Observe that $J'$ is a path of length $r \geq 5q$ since $J_k$ is a geodesic that visits $o_k$ by \ref{item':geodesic}.
    Thus, we can divide $J'$ into five pairwise disjoint subpaths $J'_1, \ldots, J'_5$, each of length at least $q-1$.
    Then, since $J'$ is geodesic, we have $d_G(J'_i, J'_{j}) > q$ whenever $i,j \in [5]$ are distinct and non-consecutive.
    Furthermore, by \cref{claim:Oi-unique} and the definition of $W_{C'}$, we have $d_G((W_{C'} \cup P_{C'}) \setminus J', J'_2 \cup J'_3 \cup J'_4) > q$.
    Hence, $J'_1, \ldots, J'_5$ and $(W_{C'} \cup P_{C'}) \setminus J'$ witness that $C'$ is $q$-fat in $G$.
    \medskip

    \noindent\textbf{Corsola properties:} We now show that $C'$ is a $(w_{\ref{lem:AngryHippo}}(q), s_{\ref{lem:AngryHippo}}(q), t, \kappa_{\ref{lem:AngryHippo}}(q))$-corsola cycle.
    \smallskip

    \ref{itm:Corsola:WC}: By construction, $W_{C'}$ has length at most $L \leq 77t$.
    \smallskip

    \ref{itm:Corsola:QuasiGeodesic}: We show that for every $\tau \leq L - \ell-r$, if $u, v \in V(P_{C'})$ satisfy $d_G(u, v) \leq \tau$ then $d_{P_{C'}}(u, v) \leq 4\tau + \kappa$. (To obtain \ref{itm:Corsola:QuasiGeodesic} note that $t \leq L - \ell-r$.)

    So let $\tau \leq L - \ell - r$, and let $u, v \in V(P_{C'})$ such that $d_G(u, v) \leq \tau$.
    Without loss of generality, suppose that $u$ occurs before $v$ along $P_{C'}$ when going from $x$ to $y$.
    Let $i \in [k]$ be maximal such that $u \in V(J_i)$, and let $j \in [k]$ be maximal such that $v \in V(J_j)$.
    Since $P_{C'}$ proceeds only forward along the jump sequence as noted earlier,
    we have $i \leq j$, and the subpath of $P_{C'}$ between $u$ and $v$ has length at most $(j-i+1) \cdot \ell$.
    Let $Q$ be a shortest path between $u$ and $v$ in~$G$; in particular, $Q$ has length at most $\tau$.

    We first show that $Q$ intersects all sets $O_{s}$ with $i < s < j$.
    Suppose otherwise, and let $s$ be maximal subject to $i<s<j$ and $Q\cap O_s=\emptyset$.
    Then, $Q$ intersects $B_G(J_{s+1},r)$ (either in $O_{s+1}$ if $s < j-1$ or in~$v$ if $s = j-1$).
    Thus, using $J_{s+1}$ and $Q$, we get a path of length at most $\ell + r + \tau \leq L$ from $v_s$ to a vertex appearing before the last occurrence of $o_s$ in the jump sequence $J_1 \cup \cdots \cup J_s$, that avoids $B_G(o_s,r/2)$ (as $J_{s+1}$ avoids $O_s$ by~\ref{item':avoids-Oi}, and $Q$ avoids $O_s$ by assumption on $s$). This contradicts \ref{item':no-shortcut}.
    The same argument also shows that $Q$ must intersect all $O_s$ with $i < s < j$ in decreasing order.

    However, any two consecutive $O_i$ are at distance at least $\ell - 6r$ in $G$.
    Since $Q$ visits all $O_s$ with $i < s < j$ in decreasing order, it follows that $Q$ has length at least $(j-i-2) \cdot (\ell - 6r)$.
    Thus, $\tau \geq ||Q|| \geq (j-i-2) \cdot (\ell - 6r)$ so $j - i \leq \frac{\tau}{\ell - 6r} + 2$, so
    \[
    d_{P_{C'}}(u, v) \leq (j-i+1) \cdot \ell \leq \ell \cdot \left(\frac{\tau}{\ell - 6r} + 3\right) = \frac{8r+1}{2r+1} \cdot \tau + 3\ell \leq 4\tau + \kappa,
    \]
    which concludes the proof of \ref{itm:Corsola:QuasiGeodesic}.
    \smallskip

    \ref{itm:Corsola:TreeDecomp}: We need to show that $G$ admits a $(w_{\ref{lem:AngryHippo}}(q), s_{\ref{lem:AngryHippo}}(q))$-radial partial tree-decomposition with support $B_G(P_{C'}, t)$.
    To do so, we will apply \cref{hairy-yak} with $q, t, b', \varepsilon = 2\ell$.
    To this end, we first verify that the hypotheses of \cref{hairy-yak} are satisfied:
    \smallskip

    \textbf{(Locally tree-like)}: Every $u \in V(P_{C'})$ is a vertex of the jump sequence, so it is at distance at most $\ell$ from a vertex of $C$. Thus, for every $u \in V(P_{C'})$, no $q$-fat model of $K_3$ in $G$ is contained in $B_G(u, b-\ell) = B_G(u, b')$.
    \smallskip

    \textbf{(Locally almost geodesic)}: We need to show that if $u, v \in V(P_{C'})$ satisfy $d_G(u, v) \leq 20q+4$ then $d_{P_{C'}}(u, v) \leq 2\ell$. So suppose for a contradiction that there exist $u, v \in V(P_{C'})$ such that $d_G(u, v) \leq 20q+4$ and $d_{P_{C'}}(u, v) > 2\ell$.
    Let $Q$ be a shortest path between $u$ and $v$ in $G$.
    Without loss of generality, suppose that $u$ occurs before $v$ along $P_{C'}$ when going from $x$ to $y$.
    Let $i \in [k]$ be maximal such that $u \in V(J_i)$, and let $j \in [k]$ be maximal such that $v \in V(J_j)$.
    Since $P_{C'}$ proceeds only forward along the jump sequence as noted earlier,
    we have $i \leq j$, and the subpath of $P_{C'}$ between $u$ and $v$ has length at most $(j-i+1) \cdot \ell$.
    Since $d_{P_{C'}}(u, v) > 2\ell$, this implies $i < j-1$.
    Consider the path $v_{j-1} J_j v Q u$.
    This path has length at most $\ell + 20q + 4 \leq \ell + r \leq L$.
    Furthermore, $J_j \subseteq G - O_{j-1}$ by \ref{item':avoids-Oi}, and $Q$ has length at most $20q+4 \leq r/2$, so this path avoids $B_G(o_{j-1}, r/2)$, and goes from $v_{j-1}$ to before $o_{j-1}$ in the jump sequence $J_1 \cup \cdots \cup J_{j-1}$, which contradicts~\ref{item':no-shortcut}.\smallskip

    \textbf{(No long shortcut)}: This follows from the statement that we proved for \ref{itm:Corsola:QuasiGeodesic} by setting $\tau = 2t+15q$: If $u, v \in V(P_{C'})$ satisfy $d_G(u, v) \leq 2t+15q$, then $d_{P_{C'}}(u, v) \leq 8t + 60q + \kappa = b'$.\medskip

    Thus, $P_{C'}$ satisfies the premises of \cref{hairy-yak}.
    Therefore, by \cref{hairy-yak}, $G$ admits a partial tree-decomposition with support $B_G(P_{C'}, t)$, radial width $w_{\ref{hairy-yak}}(q, 2\ell) = w_{\ref{lem:AngryHippo}}(q)$ and radial spread $s_{\ref{hairy-yak}}(2\ell) = s_{\ref{lem:AngryHippo}}(q)$.
    This proves that the cycle $C' = P_{C'} \cup W_{C'}$ satisfies \ref{itm:Corsola:TreeDecomp}, and hence concludes the proof.
\end{proof}

\renewcommand{\qedsymbol}{\openbox}

\section{Untangling corsola cycles} \label{sec:Untangling}

In this section, we prove \cref{lem:untangling}, which roughly states that in a graph $G$ with no $q$-fat model of $k \cdot K_3$, any small collection of corsola paths of $G$ that are somewhat distant can be made arbitrarily distant by removing few balls of small radius. For this, we will use the following theorem of Simonovits \cite{simonovits1967new} used in his proof of the \EP~theorem \cite{EPTheorem}.

\begin{theorem}[\cite{simonovits1967new}] \label{lem:simon}
    Let $k$ be a positive integer and let $G$ be a finite graph with all vertices of degree~$2$ or~$3$ that contains no $k$ disjoint cycles. Then $G$ has at most $24k\log k$ vertices of degree~$3$.
\end{theorem}

We call a path $P$ satisfying \ref{itm:Corsola:TreeDecomp} and \ref{itm:Corsola:QuasiGeodesic} from the definition of corsola cycle (\cref{def:CorsolaCycle}) an \defn{$(r_1, r_2,t, \kappa)$-corsola} path.

\begin{restatable}[Untangling lemma]{lemma}{untangling}\label{lem:untangling}
    There exists a function $R_{\ref{lem:untangling}} : \N^3 \to \N$ such that the following holds.
    Let $q, d, k, t, \kappa \in \N$ with $t \geq 2q$, let $G$ be a graph and let $P_1, \ldots , P_n$ be $(\cdot, \cdot, t, \kappa)$-corsola paths of $G$ which are pairwise at distance more than~$5q$.
    Then, either $G$ contains a $q$-fat model of $k \cdot K_3$, or there exists a set $X$ of size at most~$2n+24k \log k $ such that for every $1\le i < j \le n$, every $P_i$--$P_j$ path of length less than~$d$ in $G$ meets $B_G(X,R_{\ref{lem:untangling}}(q, d, \kappa))$.
\end{restatable}

\begin{proof}
    For all $q, d, \kappa \in \N$, let
    \begin{align*}
        R_{\ref{lem:untangling}}(q, d, \kappa) &\coloneqq 10q +2d + \kappa+1.
    \end{align*}
    Let $q, d, k, t, \kappa\in \mathbb{N}$ with $t \geq 2q$.
    Set $\defnm{R} \coloneqq R_{\ref{lem:untangling}}(q, d, \kappa)$.
    Let $G$ be a graph and let $P_1, \ldots , P_n$ be $(\cdot, \cdot, t, \kappa)$-corsola paths of $G$ which are pairwise at distance more than~$5q$.
    \smallskip

    Choose a collection \defn{$\mathcal{S}$} of geodesic paths in~$G$, each of length less than~$d$ and with endvertices in $P_i$ and $P_j$ for some $1 \leq i < j \leq n$, such that $\mathcal{S}$ is maximal subject to the paths in $\mathcal{S}$ being pairwise at distance at least $R-2d=10q+\kappa+1$.
    \medskip

    Suppose first that $|\mathcal{S}| \le n + 12k\log k$.
    Let \defn{$X$} be the endvertices of the paths in $\mathcal{S}$, so $|X|\le 2n + 24k\log k$.
    Now suppose for the sake of contradiction that for some $1\le i < j \le n$ there is a $P_i$--$P_j$ path $Q$ of length less than $d$ in $G - B_G(X, R)$.
    Let $p_i$ be the endvertex of $Q$ in $P_i$ and $p_j$ be the endvertex of $Q$ in $P_j$.
    Let $S$ be a shortest path in $G$ between $p_i$ and $p_j$.
    As $Q$ has length less than~$d$, so does $S$. Furthermore, as $Q$ and $S$ have the same endvertices, it follows that every vertex of~$S$ is at distance at most~$d$ from~$Q$.
    Therefore, $S$ is at distance at least $R-d$ from~$X$. Since all the geodesics in $\mathcal{S}$ have length less than~$d$, it follows that $S$ is at distance more than $R-2d$ from~$\mathcal{S}$.
    This contradicts the maximality of $\mathcal{S}$.
    Thus, in this case $X$ is a set of size at most $2n + 24k\log k$ such that for every $1\le i < j \le n$, every $P_i$--$P_j$ path of length less than~$d$ in $G$ meets $B_G(X,R_{\ref{lem:untangling}}(q, d, \kappa))$, as desired.
    \medskip

    Suppose now that $|\mathcal{S}| > n + 12k\log k$. In the remainder of the proof, we show that $G$ contains a $q$-fat model of $k \cdot K_3$.
    For this, we first modify $\mathcal{S}$ to obtain a nicer collection of paths.
    Since $\mathcal{S}$ is nonempty, we have that $d>5q$ since the paths $P_1, \ldots , P_n$ are pairwise at distance at least~$5q$ but the paths in $\mathcal{S}$ have length less than~$d$.

    Fix some $S\in \mathcal{S}$ for now.
    There exist some $1 \leq i < j \leq n$ and vertices $w_i, w_j \in V(S)$ such that $w_i \in B_G(P_i, q)$ as well as $w_j \in B_G(P_j, q)$ and the subpath \defn{$J_2$} of $S$ between $w_i$ and $w_j$ internally avoids $\bigcup_{m \leq n} B_G(P_m,q)$.
    Let \defn{$J_1$} be a shortest $P_i$--$w_i$ path, and let \defn{$J_3$} be a shortest $w_j$--$P_j$ path.
    Each of $J_1,J_2,J_3$ is geodesic, $J_1,J_3$ both have length~$q$, and $J_2$ has length more than $5q-2q=3q$ and is at distance at least~$q$ from every path $P_1, \ldots , P_n$.

    Then $\defnm{J_S} := J_1\cup J_2\cup J_3$ is a $P_i$--$P_j$ path that is entirely within distance $q$ of~$S$.
    Let $\defnm{\mathcal{J}} = \{J_S : S \in \mathcal{S}\}$.
    Then the paths in $\mathcal{J}$ are pairwise at distance at least $R - 2d - 2q = 8q+\kappa + 1\ge q$.
    \smallskip

    Let \defn{$X$} be the endvertices of the paths in $\mathcal{J}$.
    So, $|X| > 2n+24k \log k$.
    Note that by the argument above, for each $1\le i \le n$, the vertices of $X\cap V(P_i)$ are pairwise at distance at least $8q+\kappa+1$ in $G$.
    Let \defn{$\mathcal{Q}$} be the collection of all the (non-trivial) subpaths of $P_1, \ldots, P_n$ which have both their endvertices in $X$ but are otherwise disjoint from~$X$. Clearly the paths of $\mathcal{J} \cup \mathcal{Q}$ are pairwise disjoint, except possibly at their endvertices.
    \smallskip

    We will construct the desired $q$-fat model of $k \cdot K_3$ in $\mathcal{J} \cup \mathcal{Q}$. For this, we first establish the following claim.

    \begin{claim}\label{cl:distpaths}
        Every two disjoint paths $P,P'\in \mathcal{J} \cup \mathcal{Q}$ are at distance at least $q$ in $G$.
    \end{claim}

    \begin{claimproof}
        We have already established that distinct paths in $\mathcal{J}$ are at distance at least $8q+\kappa+1 \geq q$ in~$G$.

        Now consider disjoint paths $Q,Q'\in \mathcal{Q}$.
        Clearly, if $Q,Q'$ are contained in distinct paths of $P_1, \ldots , P_n$ then they are at distance more than $5q\ge q$.
        So we may assume that they are both contained in the same path~$P_i$.
        Since the four endvertices of $Q,Q'$ are contained in $X\cap V(P_i)$, and hence contained in pairwise distinct paths in $\mathcal{J}$, they are pairwise at distance at least $8q+\kappa+1$ in $G$ and therefore in $P_i$.
        In particular, $Q$ and $Q'$ are at distance at least $8q+\kappa+1$ in $P_i$.
        By \ref{itm:Corsola:QuasiGeodesic} and because $P_i$ is $(\cdot, \cdot, t, \kappa)$-corsola, any two vertices of $P_i$ at distance $\tau \le 2q \leq t$ in $G$ are at distance at most $4\tau + \kappa$ in $P_i$.
        Taking $\tau=q$, this implies that $Q,Q'$ must be at distance at least $q$ in $G$.
        \smallskip

        Finally, consider disjoint paths $J\in \mathcal{J}$, $Q\in \mathcal{Q}$. Then $J$ is a $P_i$--$P_j$ path for some $1\le i < j \le n$.
    If $Q$ is not contained in $P_i$ or $P_j$, then it follows by the definition of the paths in~$\mathcal{J}$ that the distance between $J$ and $Q$ is at least~$q$.
    So, we may assume without loss of generality that $Q$ is contained in $P_i$.
    Let $z$ be the endvertex of $J$ in $P_i$.

    The vertices of $J$ within distance $q$ of $P_i$ are exactly the vertices of the (geodesic) subpath~$J_1$ of $J$ (which ends in $z$) as in the definition of $J$.
    So, showing that $z$ is at distance at least $2q$ from $Q$ in~$G$ will imply that $J$ and $Q$ are at distance at least~$q$ in $G$.

    Let $x,y$ be the endvertices of $Q$.
    As before, since $x,y,z\in X\cap V(P_i)$, they are pairwise at distance at least $8q+\kappa+1$ in $G$ and therefore in $P_i$.
    In particular, $z$ is at distance at least $8q+\kappa+1$ from $Q$ in $P_i$.
    By \ref{itm:Corsola:QuasiGeodesic}, any two vertices of $P_i$ at distance $\tau \le 2q \leq t$ in $G$ are at distance at most $4\tau + \kappa$ in $P_i$.
    Taking $\tau=2q$, this implies that $z$ is at distance at least $2q$ from $Q$ in $G$.
    Therefore, $J$ and $Q$ are at distance at least $q$ in $G$.
    \end{claimproof}

    Let $\defnm{H}=\bigcup_{P\in \mathcal{J} \cup \mathcal{Q}} P$.
    By construction, $H$ has maximum degree~$3$, and all its vertices of degree~$3$ are contained in $X$. In particular, all internal vertices of the paths in $\mathcal{J} \cup \mathcal{Q}$ have degree~$2$ in~$H$.
    Now we show the key claim that disjoint cycles in $H$ are $q$-fat in $G$.

    \begin{claim}\label{cl:fatcycles}
        If $H$ contains $k$ disjoint cycles, then $G$ contains $k\cdot K_3$ as a $q$-fat minor.
    \end{claim}

    \begin{claimproof}
        By the structure of~$H$ described just before the claim, every cycle of~$H$ is the union of some paths of $\mathcal{J} \cup \mathcal{Q}$. So by \cref{cl:distpaths}, disjoint cycles of $H$ are pairwise at distance at least~$q$ in $G$. Therefore, it is enough to show that each cycle $C$ of $H$ is a $q$-fat cycle of $G$.

        Clearly, $C$ must contain some $J\in \mathcal{J}$ since $\bigcup_{Q\in \mathcal{Q}}Q$ is contained in $\bigcup_{i=1}^n P_i$.
        Let $F=C-J$.
        Let $J_1\cup J_2\cup J_3=J$ as in the definition of the paths of $\mathcal{J}$, and note that $J_2$ is at distance at least~$q$ from $F\subset H-J$.
        Since $J_2$ is a geodesic path in $G$ of length at least~$3q$ and $J_1,J_3$ are at distance at least $3q$ in $G$, we can partition $J_2$ into three geodesic subpaths $W_1,W_2,W_3$, each of length at least $q$, with $J_1$ sharing an endvertex with $W_1$ and $W_3$ sharing an endvertex with $J_3$, and such that $J_1$ is at distance at least $q$ from $W_2,W_3$ in $G$ and $J_3$ is at distance at least $q$ from $W_1,W_2$ in $G$.
        Note that $W_1$ and $W_3$ are at distance at least $q$ from each other in $G$ and also that $W_1,W_2,W_3$ are at distance at least $q$ from $F\subset H-J$ in $G$. So, any two non-consecutive subpaths $J_1,W_1,W_2,W_3,J_3,F$ of the cycle $C$ are at distance at least $q$ in $G$.
        As $J_1,W_1,W_2,W_3,J_3,F$ partition the cycle $C$, they therefore witness that $C$ is a $q$-fat model of $K_3$, as desired.
    \end{claimproof}

    It now remains to find $k$ disjoint cycles of~$H$.
    Let $H'$ be the graph obtained from $H$ by iteratively deleting all the vertices of~$H$ of degree at most~$1$. Since $H$ has maximum degree~$3$, all vertices of~$H'$ have degree~$2$ or~$3$.
    It remains to show that $H'$ has more than $24k\log k$ vertices of degree~$3$. Then applying \cref{lem:simon} to $H'$ yields that $H'$ (and thus~$H$) contains $k$ disjoint cycles, which concludes the proof by \cref{cl:fatcycles}.

    An easy counting argument yields that $H'$ has at most $n_1$ fewer degree~$3$ vertices than~$H$, where $n_1$ is the number of degree~$1$ vertices of $H$.
    Hence, we have to show that $H$ has more than $n_1+24k\log k$ vertices of degree~$3$.

    By definition of~$H$, all its degree~$1$ vertices lie in~$X$, and a vertex in~$X$ has degree~$1$ in~$H$ if and only if it is the only vertex of~$X$ that lies on~$P_i$, for some $i \leq n$. Moreover, the vertices in~$X$ of degree~$2$ in $H$ are those that lie on some $P_i$ such that $|V(P_i) \cap X| \geq 2$ and that are closest in~$P_i$ to one of the two endvertices of $P_i$. It follows that $2n_1+n_2 \leq 2n$ where $n_2$ is the number of vertices in~$X$ of degree~$2$ in $H$.
    Therefore, $X$, and hence~$H$, has more than $2n+24k \log k -n_1-n_2 \ge n_1+24k \log k$ vertices of degree~$3$ in~$H$, as desired.
\end{proof}

\section{Smooshing the tree-decompositions together} \label{sec:Smooshing}

Recall that, in the overall proof (cf.\ \cref{sec:ProofSketch}), after packing as many corsola cycles $D_i$ as possible, we aim to combine the \td s around these cycles with the \td\ of the remaining subgraph $G-\bigcup B_G(D_i,d)$ obtained by applying \cref{lem:ApexForestWithoutApex}, into a graph-decomposition of $G$ without introducing short cycles in the decomposition graph.
The following lemma formalises this step.

Recall that a path $P$ in a graph $G$ is \defn{$(r_1, r_2,t, \kappa)$-corsola} if it satisfies \ref{itm:Corsola:TreeDecomp} and \ref{itm:Corsola:QuasiGeodesic} from the definition of corsola cycle (\cref{def:CorsolaCycle}).
Additionally, we say that $P$ is \defn{$(r_1, r_2,t)$-corsola} if it satisfies \ref{itm:Corsola:TreeDecomp} from the definition of corsola cycle (\cref{def:CorsolaCycle}).

\begin{restatable}[Smooshing the tree-decompositions together]{lemma}{smooshing} \label{lem:Smooshing}
    There exist functions $y_{\ref{lem:Smooshing}} : \N^2 \to \N$, $z'_{\ref{lem:Smooshing}} : \N^2 \to \N$, $w_{\ref{lem:Smooshing}} : \N \to \N$, $s_{\ref{lem:Smooshing}} : \N \to \N$ and $d_{\ref{lem:Smooshing}} : \N \to \N$ such that the following holds.

    Let $r_1, r_2, g \in \N$, let $z \geq z'_{\ref{lem:Smooshing}}(r_1, g)$ and let~$G$ be a graph.
    Let $P_1, \dots, P_n$ be $(r_1, r_2, z)$-corsola paths in $G$ that are pairwise at distance at least~$d_{\ref{lem:Smooshing}}(z)$, and let $P := \bigcup_{i \in [n]} P_i$.
    Let $C \subseteq V(G) - B_G(P, z)$, and assume that there is an $(r_1, r_2)$-radial partial \fd\ of~$G$ with support $B_G(C,y_{\ref{lem:Smooshing}}(r_1, g))$.

    Then, $G$ admits a $(w_{\ref{lem:Smooshing}}(r_1), s_{\ref{lem:Smooshing}}(r_2))$-radial partial graph-decomposition modelled on a graph of girth at least~$g$, with support $B_G(P, z) \cup C$.
\end{restatable}

We remark that for now, one may think of $C$ as the remaining subgraph $G-B_G(P,z)$. We need the freedom of choosing $C$ smaller because when we apply \cref{lem:Smooshing} later, we will have some small number of exceptional bounded-radius balls (containing e.g.\ the goldfish cycles) which are neither covered by $B_G(P,z)$ nor by the forest-decomposition of the remaining subgraph.

Note that it does not affect the proof whether $C=G-B_G(P,z)$ or $C\subsetneq G-B_G(P,z)$.

\begin{proof}[Proof of \cref{lem:Smooshing}]
    Let $r_1, r_2, g\in \N$ be given.
    \bigskip

    \noindent\textbf{\large Preparation}
    \smallskip

    \noindent\textbf{Definition of parameters and functions:} For the proof, we need to define functions $y_{\ref{lem:Smooshing}}$, $z'_{\ref{lem:Smooshing}}$, $w_{\ref{lem:Smooshing}}$, $s_{\ref{lem:Smooshing}}$, and $d_{\ref{lem:Smooshing}}$. Additionally, we need two more parameters $z_1, z_2 \in \N$. We will choose them so that
    \begin{align*}
    r_1 &\ll z_1 \ll w_{\ref{lem:Smooshing}}(r_1),\\
    r_1,z_1,g &\ll z_2 \ll y_{\ref{lem:Smooshing}}(r_1,g),\\
    z'_{\ref{lem:Smooshing}}(r_1, g) &=\phantom{:} z_1+z_2 \ll d_{\ref{lem:Smooshing}}(z), \text{ and}\\
    r_2 &\ll s_{\ref{lem:Smooshing}}(r_2).
    \end{align*}
    See \cref{fig:Smooshing:Parameter} for an illustration of the roles that $y_{\ref{lem:Smooshing}}, z_1, z_2, z'_{\ref{lem:Smooshing}}$ and $d_{\ref{lem:Smooshing}}$ play in the construction.
    We now provide concrete values for the constants listed above that satisfy our requirements, but the reader can choose to ignore these values; what matters is that we choose them large in comparison to $r_1,r_2,g$ as indicated above. The values that we obtain are
    \begin{align*}
        z_1 &:= 13r_1+1, \\
        z_2 &:= \max\{z_1-2r_1, r_1 \cdot (g+6)\} = \max\{11r_1+1, r_1 \cdot (g+6)\}, \\
        z'_{\ref{lem:Smooshing}}(r_1, g) &:= z_1 + z_2  = \max\{24r_1+2, r_1 \cdot (g+19)+1\},\\
        y_{\ref{lem:Smooshing}}(r_1, g) &:= z_1+2z_2+9r_1+2 = \max\{44r_1+5, r_1 \cdot (2g+34)+3\}, \\
        w_{\ref{lem:Smooshing}}(r_1) &:= 6z_1+22r_1 = 100r_1+6,\\
        s_{\ref{lem:Smooshing}}(r_2) &:= 3r_2, \text{ and } \\
        d_{\ref{lem:Smooshing}}(z) &:= 2z+2.
    \end{align*}
    \medskip

    \noindent\textbf{Setup:} (See \cref{fig:Smooshing:Parameter} for an illustration.) Let $z \geq z'_{\ref{lem:Smooshing}}(r_1,g)$ and $r_1,r_2,g$ be given, and let $G, P_1, \dots, P_n, C$ satisfy the assumptions of the lemma.
    For convenience, set $\defnm{z'} \coloneqq z'_{\ref{lem:Smooshing}}(r_1, g)$, $\defnm{y} \coloneqq y_{\ref{lem:Smooshing}}(r_1,g)$ and $\defnm{d} \coloneqq d_{\ref{lem:Smooshing}}(z)$ and let $\defnm{X} \coloneqq B_G(P, z) \cup C$.
    Let $\defnm{G^{P, z}}$ be the graph induced on the set $B_G(P, z)$.
    For every $i \in [n]$, let $\defnm{(T^{i,z}, \mathcal{V}^{i,z})}$ be an $(r_1, r_2)$-radial partial tree-decomposition of $G$ with support $B_G(P_i, z)$; such a tree-decomposition exists since $P_i$ is $(r_1, r_2, z)$-corsola.
    By assumption, any two $P_i$ are at distance at least $d > 2z+1$, so the supports of the $(T^{i,z}, \mathcal{V}^{i,z})$'s are pairwise disjoint and anticomplete.
    Let \defn{$(T^{P,z}, \mathcal{V}^{P,z})$} be the combined partial \fd\ of the $(T^{i,z}, \mathcal{V}^{i,z})$'s, i.e.\ \defn{$T^{P,z}$} is obtained from the disjoint union of the trees $T^{i,z}$, and $\defnm{\mathcal{V}^{P,z}} = \bigcup_{i \in [n]} \mathcal{V}^{i,z}$. In particular, $(T^{P,z}, \mathcal{V}^{P,z})$ is an $(r_1, r_2)$-radial partial graph-decomposition of $G$ modelled on the forest $T^{P,z}$ and with support~$B_G(P, z)$.

    Further, let \defn{$(\widehat{T}^C, \widehat{\mathcal{V}}^C)$} be an $(r_1, r_2)$-radial partial \fd\ of $G$ with support $B_G(C, y)$; such a forest-decomposition exists by the assumptions of the lemma.
    \medskip

    \begin{figure}[ht]
        \centering
        \includegraphics[width=0.6\textwidth]{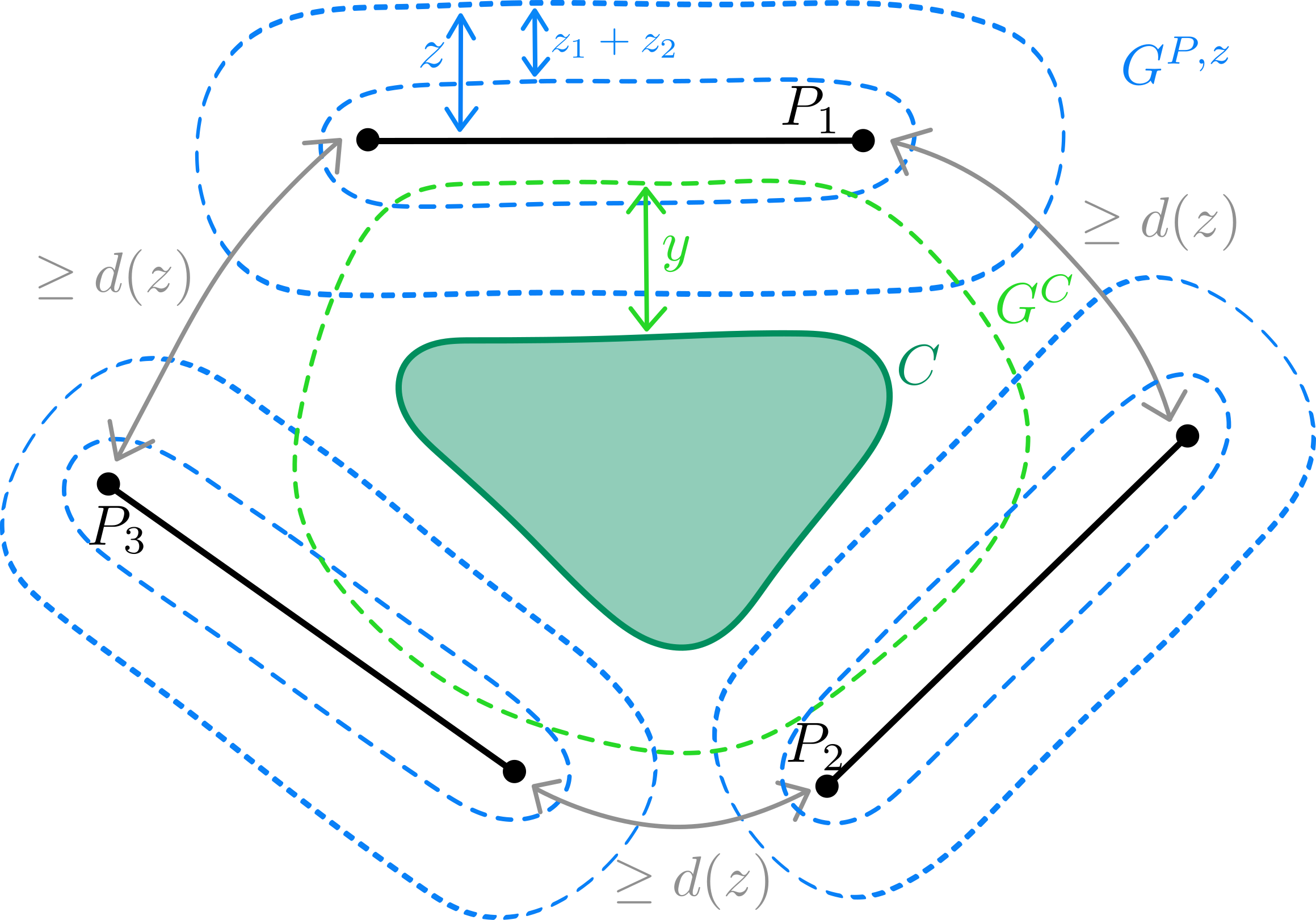}
        \caption{A sketch of the situation at the beginning of the proof of \cref{lem:Smooshing}. By the assumption of the lemma, there exist $(r_1, r_2)$-radial partial tree-decompositions with support $V(G^C)$ and with support $V(G^{P,z})$.}
        \label{fig:Smooshing:Parameter}
    \end{figure}

    \noindent\textbf{Proof outline:} (Cf.\ \cref{fig:Smooshing:Sketch}.)
    We want to glue the \fd s $(T^{P,z}, \mathcal{V}^{P,z})$ and $(\widehat{T}^C, \widehat{\mathcal{V}}^C)$ together.
    The most naive approach for this would be to define a \gd\ modelled on the graph $H$ obtained by taking the disjoint union of $T^{P,z}$ and $\widehat{T}^C$ and adding an edge between every pair of nodes $s \in V(T^{P,z})$ and $t \in V(\widehat{T}^C)$ whose bags intersect.
    However, the resulting graph~$H$ may contain many short cycles.
    Since we need $H$ to have large girth, we modify the two \fd s before gluing them together.

    For this, we first define induced subgraphs $G^P$ and $G^C$ of $G$ such that $G[B_G(P, z-z_2)] \subseteq G^P \subseteq G[B_G(P,z)]$ and $C \subseteq V(G^C) \subseteq B_G(C,y)$ (see \cref{subfig:Smooshing:Sketch:1}).
    We choose them so that $G^P \cup G^C = G[X]$ and so that they only share vertices at distance exactly $z-z_2$ from $P$.
    We then obtain $(T^P, \mathcal{V}^P)$ and $(T^C, \mathcal{V}^C)$ from $(T^{P,z}, \mathcal{V}^{P,z})$ and $(\widehat{T}^C, \widehat{\mathcal{V}}^C)$ by restricting them to $G^P$ and $G^C$, respectively (see \cref{subfig:Smooshing:Sketch:2}).

    We next modify these two \fd s by a sequence of contractions.
    Consider the subgraph of $T^P$ induced by the nodes whose bags intersect the annulus $B_G(P,z-z_2)\setminus B_G(P, z-z_1-z_2)$ around~$P$.
    We contract in $T^P$ each component $O$ of this subgraph to a single node $t_O$, thereby obtaining a \fd\ $(\widetilde{T}^P,\widetilde{\mathcal{V}}^P)$; see \cref{fig:Smooshing:TP}.
    For each contraction node $t_O\in V(\widetilde{T}^P)$, we then contract in $T^C$ the set of nodes of $T^C$ whose bags intersect $\widetilde{V}^P_{t_O}$ to a single node $s_O$; see \cref{fig:Smooshing:TC}.
    This yields a \fd\ $(\widetilde{T}^C,\widetilde{\mathcal{V}}^C)$.

    Finally, we glue $(\widetilde{T}^P,\widetilde{\mathcal{V}}^P)$ and $(\widetilde{T}^C,\widetilde{\mathcal{V}}^C)$ together by identifying $t_O$ with $s_O$ for every component $O$; see \cref{fig:Smooshing:ItAllTogether}.
    We claim that the resulting decomposition graph $H$ has large girth.

    Indeed, since both $\widetilde{T}^P$ and $\widetilde{T}^C$ are forests, every cycle in $H$ must pass through the identified nodes corresponding to two distinct components $O$ and $O'$.
    Equivalently, before the identifications, such a cycle must contain $t_O,s_O,s_{O'}$, and $t_{O'}$.

    The contractions in~$T^{P}$ ensure that, for any two distinct contraction nodes $t_{O},t_{O'} \in V(\widetilde{T}^{P})$ and any vertices $u \in \widetilde{V}^P_{t_{O}}$ and $v \in \widetilde{V}^P_{t_{O'}}$ at distance exactly $z-z_2$ from $P$, the vertices $u$ and $v$ are far apart in~$G^C$; see \cref{fig:Smooshing:Sketch}.
    Recall that the vertices at this distance from $P$ are the only vertices that may be shared by $G^P$ and $G^C$.
    More precisely, the distance between $u$ and $v$ in $G^C$ is roughly at least $2z_2$.
    Indeed, since $u$ and $v$ occur in bags belonging to distinct contracted components, the path between these components in $T^P\subseteq T^{P,z}$ contains a bag that avoids the annulus.
    By the separator property of the \fd, every $u$--$v$ path in $G^{P,z}$ meets such a bag, which is disjoint from $G^C$.
    Thus, every $u$--$v$ path in $G^C$ must leave $B_G(P,z)$.

    Consequently, although $t_O$ and $t_{O'}$ may have distance $2$ in $\widetilde{T}^P$ as a result of the contractions, the corresponding nodes $s_O$ and $s_{O'}$ are far apart in $\widetilde{T}^C$.
    More precisely, their bags contain vertices at distance roughly at least $2z_2$ in $G^C$, while $(T^C,\mathcal{V}^C)$ has radial width at most $r_1$.
    Thus, the distance between $s_O$ and $s_{O'}$ in $\widetilde{T}^C$ is roughly at least $2z_2/(2r_1)$.
    Since $z_2\gg r_1,g$, every cycle created in $H$ has length at least $g$.
    \begin{figure}
        \centering
        \begin{subfigure}[b]{0.49\linewidth}
            \centering
            \includegraphics[width=0.98\linewidth]{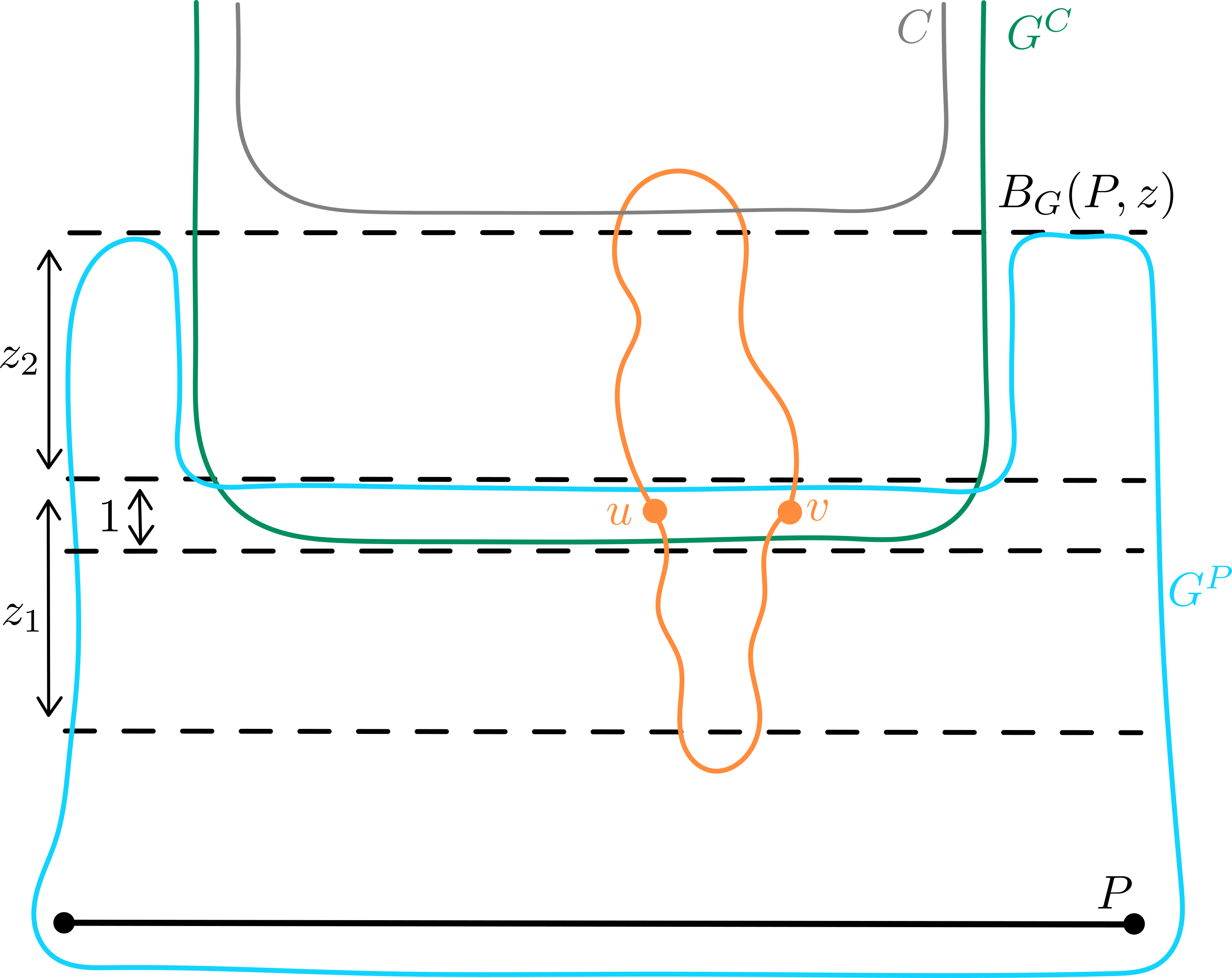}
            \caption{The subgraphs $G^P$ and $G^C$.}
            \label{subfig:Smooshing:Sketch:1}
        \end{subfigure}
        \begin{subfigure}[b]{0.49\linewidth}
            \centering
            \includegraphics[width=0.98\linewidth]{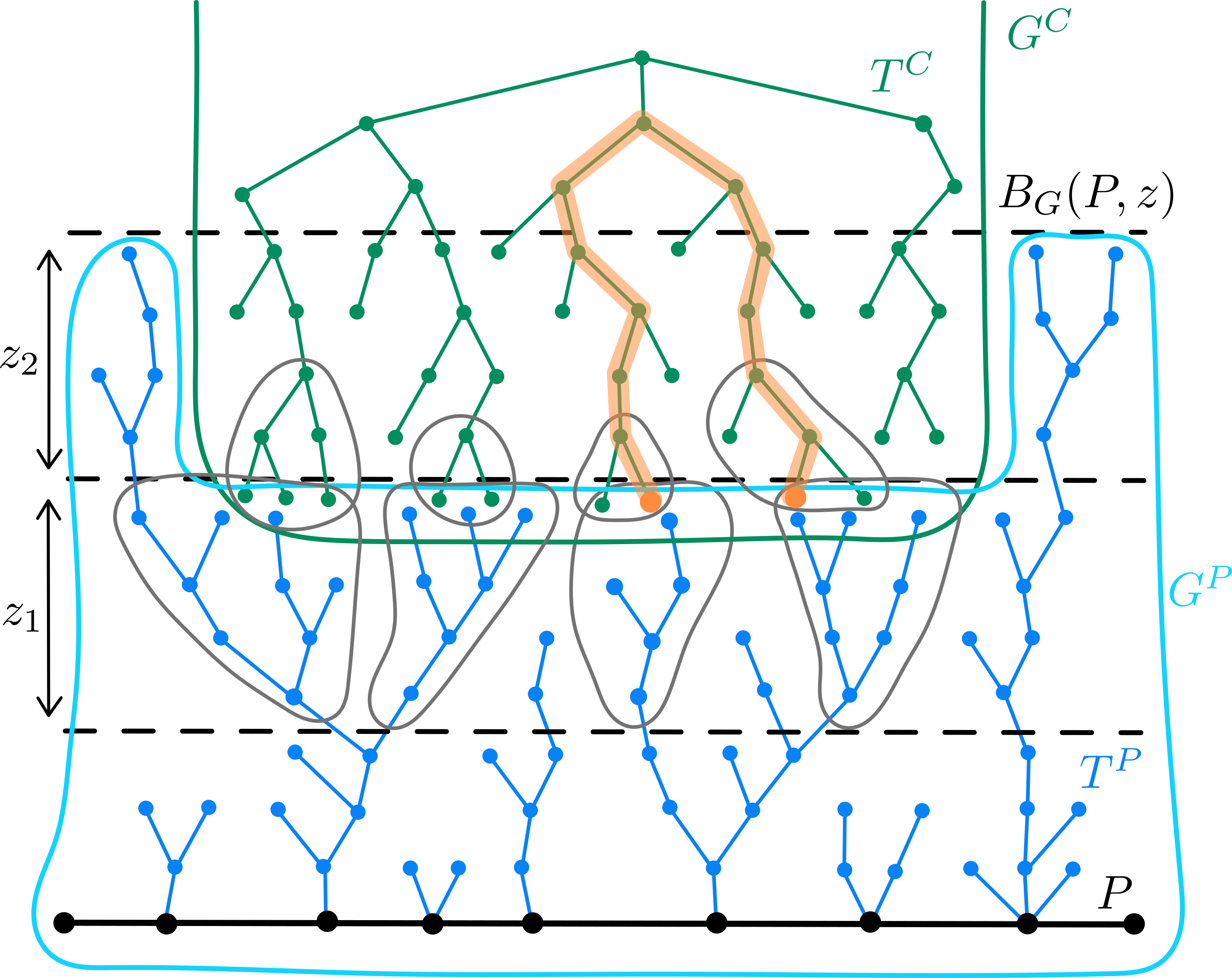}
            \caption{The \fd s of $G^P$ and $G^C$.}
            \label{subfig:Smooshing:Sketch:2}
        \end{subfigure}
        \caption{An illustration of the construction in \cref{lem:Smooshing}.
        Given two vertices $u,v \in V(G^P) \cap V(G^C)$ that lie in distinct components of $G[B_G(P,z-z_2)\setminus B_G(P,z-z_2-z_1)]$, every $u$--$v$ path must meet either $C$ or $B_G(P,z-z_2-z_1)$.
        Since $G^C$ is disjoint from $B_G(P,z-z_2-z_1)$, every $u$--$v$ path in $G^C$ must meet $C$ and therefore has length at least $2z_2$.
        On the other hand, any path $Q'$ in $\widetilde{T}^C$ between distinct nodes $s_O$ and $s_{O'}$ can be converted into a path in $G^C$ of length at most $2r_1 ||Q'||$.
        Such a path $Q'$ must then have length at least $2z_2/(2r_1)$.
        Thus, identifying the corresponding nodes of the two \fd s does not create short cycles in $H$.
        }
        \label{fig:Smooshing:Sketch}
    \end{figure}
    \bigskip

    \noindent\textbf{\large Definition of $\mathbf{G^{P}}$ and $\mathbf{G^C}$:}
    \smallskip

    Let \defn{$G^P$} be the induced subgraph of $G$ obtained from $G[B_G(P, z-z_2)]$ by adding to it all the components of $G^{P,z} - B_G(P, z-z_2)$ that do \emph{not} send an edge to $C$.
    Let further \defn{$(T^{P}, \mathcal{V}^{P})$} be the honest \fd\ obtained from $(T^{P,z}, \mathcal{V}^{P,z})$ by restricting it to $G^P$ (i.e.\ intersecting every bag $V^{P,z}_t$ with~$V(G^P)$, and then deleting all nodes/edges of~$T^{P,z}$ whose bag/adhesion set is now empty).
    In particular, $(T^P, \mathcal{V}^P)$ is an honest $(r_1, r_2)$-radial partial forest-decomposition of~$G$ with support~$V(G^{P}) \supseteq B_G(P, z-z_2)$.
    Observe that by the definition of~$G^P$, every component of~$G^P$ contains a unique path~$P_i$. Hence, by \ref{itm:H2'}, every component of~$T^P$ contains a unique path~$P_i$ in its support.
    \smallskip

    Let $\mathcal{K}$ be the union of those components of $G^{P,z}-B_G(P,z-z_2)$ that send an edge to $C$.
    Let $\defnm{G^C}$ be the subgraph of $G$ induced by $C\cup V(\mathcal{K})\cup(N_G(V(\mathcal{K}))\cap V(G^P))$.
    From the definition of $G^P$ and $G^C$ it is immediate that $G^P \cup G^C = G[X]$ and that $G^P$ and $G^C$ only share vertices that are at distance exactly $z-z_2$ from $P$.

    \begin{claim}\label{cl:G^C-close-to-Y}
        $V(G^C) \subseteq B_G(C, 2z_2+6r_1+2) \subseteq B_G(C,y)$.
    \end{claim}

    \begin{claimproof}
    Let $K$ be a component of $G^{P,z}-B_G(P,z-z_2)$ that sends an edge to $C$, and let $v \in V(K)$.
    Pick a vertex $u \in V(K)$ that sends an edge to~$C$.
    We prove that $d_G(u, v) \leq 2z_2+6r_1+1 \leq y-1$, which will conclude the proof as $d_G(u,C) = 1$ and thus $d_G(v,C) \leq 2z_2+6r_1+2$.

    Let $i\in[n]$ be the unique index such that $P_i$ is contained in the same component of $G^{P,z}$ as~$K$.
    Let $Q_u, Q_v$ be shortest paths between $u$ and $P_i$, and $v$ and $P_i$, respectively.
    In particular, $Q_u,Q_v \subseteq G^{P,z}$ have length at most~$z$.
    Since $K$ and $P_i$ are both connected, the subgraphs $T^{P,z}_K$ and $T^{P,z}_{P_i}$ of $T^{P,z}$ are also connected by \ref{itm:H2'}.
    For every $t \in V(T^{P, z})$, the bag $V^{P, z}_t$ has diameter at most $2r_1$ in $G$.
    However, $d_G(P_i, K) > z-z_2 \geq z' - z_2 = z_1 \geq 2r_1$ so no bag $V^{P, z}_t$ intersects both $P_i$ and $K$.
    Thus, the subtrees $T^{P,z}_K$ and $T^{P,z}_{P_i}$ are disjoint, so there is a unique node $t \in V(T^{P,z}_K)$ that is closest to $T^{P,z}_{P_i}$ in $T^{P,z}$.
    Then, $t$ separates $T^{P,z}_K$ from $T^{P,z}_{P_i}$ in $T^{P,z}$, which implies that $V^{P,z}_t$ meets both $Q_u$ and $Q_v$.
    Since $t \in V(T^{P,z}_K)$, it follows that $V^{P,z}_t$ avoids $B_G(P, z-z_2-2r_1)$.
    Thus, the subpaths of $Q_u,Q_v$ between $u,v$, respectively, and $V^{P,z}_t$ have length at most $z_2+2r_1$.
    Since $\diam_G(V^{P,z}_t) \leq 2r_1$, it follows that $d_G(u,v) \leq 2(z_2+2r_1)+2r_1 = 2z_2 + 6r_1$,  as desired. 
    \end{claimproof}

    Let \defn{$(T^C, \mathcal{V}^C)$} be the honest \fd\ obtained from $(\widehat{T}^C, \widehat{\mathcal{V}}^C)$ by restricting it to $G^C$ (i.e.\ intersecting every bag $\widehat{V}^{C}_t$ with~$V(G^C)$, and then deleting all nodes/edges of~$\widehat{T}^{C}$ whose bag/adhesion set is now empty).
    \bigskip

    \noindent\textbf{\large Definition of the forest-decomposition ($\mathbf{\widetilde{T}^P, \widetilde{\mathcal{V}}^P}$)}
    \smallskip

    Let \defn{$\mathcal{O}$} be the set of components of~$T^P$ after deleting all the nodes~$t$ of~$T^P$ whose bag~$V^P_t$ either does not meet $B_G(P,z-z_2)$ or contains a vertex of $B_G(P,z-z_1-z_2)$; these components are indicated in orange in \cref{fig:Smooshing:TP}.
    For each $O \in \mathcal{O}$, let $\defnm{U_O} := \bigcup_{o \in O} V^P_o$, and let \defn{$Y_O$} be the set of vertices in $U_O$ that have a neighbour outside of $G^P$ in $G[X]$.
    Finally, set $\defnm{Y} := \bigcup_{O \in \mathcal{O}} Y_O$; these are the vertices indicated in pink in \cref{fig:Smooshing:TP}.

    \begin{figure}[ht]
        \centering
        \includegraphics[width=0.8\textwidth]{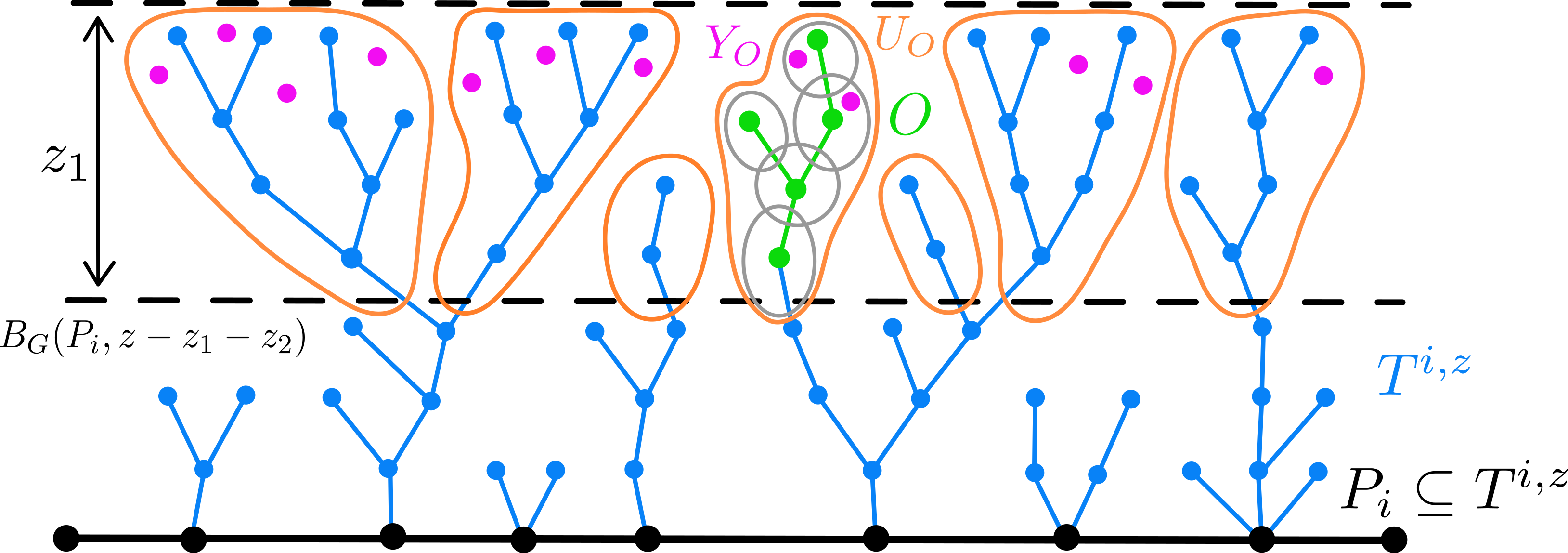}
        \caption{An illustration of a component $T^i$ of the forest $T^P$. A subtree $O \in \mathcal{O}$ of $T^P$ is indicated in green, and the corresponding vertex sets $U_O$ and $Y_O$ are indicated in orange and pink, respectively. Note that $O$ is a subgraph of $T^P$, while $U_O$ and $Y_O$ are subsets of $V(G)$.}
        \label{fig:Smooshing:TP}
    \end{figure}

    Note that each $O\in\mathcal{O}$ is a subtree of $T^P$, whereas each $Y_O$ is a subset of $V(G^P)$.
    The following result implies that $Y$ separates $G^P$ from $G^C$ in $G[X] = G^P \cup G^C$, and that every vertex in $Y$ is at distance exactly $z-z_2$ from $P$.

    \begin{claim}\label{Y=G^P_inter_G^C}
        We have $Y = V(G^P) \cap V(G^C)$.
    \end{claim}

    \begin{claimproof}
        Every vertex in $Y$ belongs to $G^P$ and has a neighbour outside $G^P$ in $G[X]$.
        Since $G[X]=G^P\cup G^C$, it also belongs to $G^C$, and thus $Y \subseteq V(G^P) \cap V(G^C)$.

        Conversely, let $v\in V(G^P)\cap V(G^C)$; recall that $v$ lies at distance exactly $z-z_2$ from $P$.
        Since $v\in V(G^P)\cap V(G^C)$, the definition of $G^C$ implies that $v$ was added to $G^C$ because it has a neighbour in $X$ outside of~$G^P$.
        Choose $t\in V(T^P)$ such that $v\in V_t^P$. Since $(T^P,\mathcal{V}^P)$ has radial width at most $r_1$ and $z_1>2r_1$, the bag $V_t^P$ does not meet $B_G(P,z-z_1-z_2)$. On the other hand, it meets $B_G(P,z-z_2)$ because it contains $v$. Hence, $t$ belongs to some $O\in\mathcal{O}$, and therefore $v\in Y_O\subseteq Y$.
    \end{claimproof}

    We now establish two further properties of the sets $Y_O$.
    \begin{claim} \label{claim:Smooshing:Y}
        For all distinct $O, O' \in \mathcal{O}$ the following two statements hold:
        \begin{enumerate}[label=\rm{(\arabic*)}]
            \item \label{itm:Smooshing:Y:Diameter} For any two vertices $u,v \in U_O$, there exists a $u$--$v$ path through $B_G(P,z-z_2+2r_1)$ of length at most $2z_1+6r_1$.
            \item \label{itm:Smooshing:Y:Distance2} $d_G(Y_O, Y_{O'}) \geq 2z_1-4r_1$ and every $Y_O$--$Y_{O'}$ path that avoids $B_G(P, z-z_1-z_2+2r_1)$ has length greater than $2z_2$.
        \end{enumerate}
    \end{claim}

    \begin{claimproof}
        \ref{itm:Smooshing:Y:Diameter}: Let $O \in \mathcal{O}$ be given, and let $u,v \in U_O$.
        Let $i\in[n]$ be the unique index such that $P_i$ is contained in the support of the component of $T^P$ that contains $O$; this is well defined by a previous observation.
        By definition of $U_O$ and since $(T^P, \mathcal{V}^P)$ has radial width at most $r_1$, we have $d_G(u, P_i), d_G(v, P_i) \leq z-z_2+2r_1$.
        Let $Q_u, Q_v$ be shortest paths between $u$ and $P_i$, and $v$ and $P_i$, respectively.
        In particular, $Q_u, Q_v \subseteq G^P$ have length at most $z-z_2+2r_1$.
        We first show that there is a node $o \in O$ such that $V^P_o$ intersects both $Q_u$ and $Q_v$.
        By \ref{itm:H2'}, the subgraph $T^P_{P_i}$ of $T^P$ is a subtree of $T^P$ that lives in the same component of $T^P$ as $O$.
        Since $O$ is connected and disjoint from $T^P_{P_i}$ by definition, there is a unique node $o \in O$ that is closest to $T^P_{P_i}$ in $T^P$.
        Then $o$ separates $O$ and $T^P_{P_i}$ in~$T^P$ so $V^P_o$ separates $U_O$ from $P_i$ in~$G^P$.
        Thus, $V^P_o$ intersects both~$Q_u$ and~$Q_v$.

        Let $Q'_u, Q'_v$ be subpaths of $Q_u, Q_v$ that start in $u,v$, respectively, and end in $V^P_o$.
        Since $u,v \in B_G(P_i,z-z_2+2r_1)$ and $V^P_o$ does not meet $B_G(P_i,z-z_1-z_2)$ (by the definition of~$O \in \mathcal{O}$), and because $Q_u,Q_v$ are shortest paths to $P_i$, it follows that $Q'_u, Q'_v$ have length at most~$z_1+2r_1$. Therefore, they can be combined to a $u$--$v$ path through $B_G(P_i,z-z_2+2r_1)$ of length at most
        \[
        2(z_1+2r_1)+\diam_G(V^P_o) \leq 2z_1 + 6r_1,
        \]
        where we used that $(T^P, \mathcal{V}^P)$ has radial width at most~$r_1$. This completes the proof of \ref{itm:Smooshing:OHatC:Diameter}.
        \medskip

        \ref{itm:Smooshing:Y:Distance2}: (See \cref{subfig:Smooshing:Sketch:1} and its caption for an informal explanation.)
        Let $O \neq O' \in \mathcal{O}$, and let $y \in Y_O$ and $y' \in Y_{O'}$; recall that $d_G(P, y) = z-z_2 = d_G(P, y')$.

        If $y$ and $y'$ do not belong to the same component of $G[B_G(P, z-z_2)]$, then there are distinct indices $i, j \in [n]$ such that $y \in B_G(P_i,z-z_2)$ and $y' \in B_G(P_j, z-z_2)$.
        Thus, $d_G(y, y') \geq d_G(P_i, P_j) - d_G(P_i, y) - d_G(P_j, y') \geq d - 2(z-z_2) > 2z_2 \geq 2z_1-4r_1$.

        Otherwise, $y$ and $y'$ belong to the same component of $G[B_G(P, z-z_2)]$.
        Thus, there exists a $y$--$y'$ path $Q'$ in $G[B_G(P, z-z_2)] \subseteq G^P$.
        By \ref{itm:H2'}, $T^P_{Q'}$ is a connected subtree of $T^P$ that intersects both $O$ and $O'$, hence contains the unique $O$--$O'$ path in $T^P$.
        Moreover, the bag of every node of $T^P_{Q'}$ meets $Q'$ hence $B_G(P, z-z_2)$.
        Since $O \neq O'$, there exists a node $t$ along the unique $O$--$O'$ path in $T^P$ whose bag contains a vertex of $B_G(P, z-z_1-z_2)$.
        Since $T^P$ is an induced subgraph of $T^{P, z}$, the unique $O$--$O'$ path in $T^{P, z}$ also contains this node $t$.

        Let $Q$ be an arbitrary $y$--$y'$ path in $G$.
        If $V(Q) \subseteq V(G^{P, z})$ then $Q$ must intersect $V^{P, z}_t$.
        Since $(T^{P, z}, \mathcal{V}^{P, z})$ has radial width at most $r_1$, this implies that $Q$ contains a vertex at distance at most $z-z_2-z_1+2r_1$ from $P$.
        Since $y$ and $y'$ both lie at distance exactly $z-z_2$ from $P$, it follows that $Q$ has length at least $2z_1-4r_1$, as desired.

        If $V(Q) \not \subseteq V(G^{P, z})$, then $Q$ contains a vertex of distance more than~$z$ from $P$, and hence at distance more than~$z_2$ from $y,y'$. Thus, $Q$ has length greater than $2z_2 \geq 2z_1-4r_1$.

        In particular, note that we showed before that $Q \subseteq G^{P,z}$ implies that $Q$ meets $B_G(P,z-z_1-z_2+2r_1)$. Thus, every $y$--$y'$ path in $G$ that avoids $B_G(P, z-z_1-z_2+2r_1)$ has length greater than $2z_2$.
    \end{claimproof}

    Let \defn{$(\widetilde{T}^P, \widetilde{\mathcal{V}}^P)$} be the \fd\ obtained from $(T^P, \mathcal{V}^P)$ by contracting each $O \in \mathcal{O}$ to a single node~$\defnm{t^P_O}$ and assigning to it the bag $\defnm{\widetilde{V}^P_{t^P_O}} :=\bigcup_{s \in O} V_s^P= U_O$; for every other node $t$ of $\widetilde{T}^P$ we set $\widetilde{V}^P_t := V^P_t$. For notational simplicity, we denote $\widetilde{V}^P_{t^P_O}$ by \defn{$\widetilde{V}^P_O$}.

    Note that by definition, the components $O \in \mathcal{O}$ are pairwise disjoint, and hence the nodes~$t^P_O$ are pairwise distinct.

    By definition, every $Y_O$ is contained in the bag $\widetilde{V}^P_{O}$. In fact, $\widetilde{V}^P_{O}$ is the only bag of $\widetilde{\mathcal{V}}^P$ meeting $Y_O$ since $Y_O$ avoids $\widetilde{V}^P_{{O'}} = U_{O'}$ for every $O' \neq O \in \mathcal{O}$ by \cref{claim:Smooshing:Y}~\ref{itm:Smooshing:Y:Distance2}, and we already saw that if $t \in V(T^P)$ meets $Y$ then $t$ is contained in some $O' \in \mathcal{O}$.
    \bigskip

    \noindent\textbf{\large Definition of the forest-decomposition ($\mathbf{\widetilde{T}^C, \widetilde{\mathcal{V}}^C}$)}
    \smallskip

    Recall that $(T^C, \mathcal{V}^C)$ is an honest partial \fd\ of $G$ with support $V(G^C)$.
    Moreover, we have $V(G^P) \cap V(G^C) =Y$ by \cref{Y=G^P_inter_G^C}.
    Hence, every bag in $\mathcal{V}^C$ intersects only bags in $\widetilde{\mathcal{V}}^P$ of the form $\widetilde{V}^P_{t_O}$ for some component $O \in \mathcal{O}$.
    \medskip

    For every $O \in \mathcal{O}$, let \defn{$\widetilde{Y}_{O}$} be the set obtained from $Y_O$ by adding between any two vertices a path in $B_G(P,z-z_2+2r_1)$ of length at most $2z_1+6r_1$; such a path exists by \cref{claim:Smooshing:Y}~\ref{itm:Smooshing:Y:Diameter}.
    Observe that $\widetilde{Y}_{O}$ is nonempty and connected, and $\widetilde{Y}_{O} \subseteq B_G(Y_O, z_1+3r_1)$.

    For every $O \in \mathcal{O}$, let \defn{$\widehat{O}_C$} be the set of all nodes $t$ of $\widehat{T}^C$ whose bag $\widehat{V}^C_t$ contains a vertex of~$\widetilde{Y}_{O}$, and let $\defnm{\widetilde{O}_C} \coloneqq \widehat{O}_C \cap V(T^C)$.
    Let us now establish two properties of the sets $\widetilde{O}_C$ that we need to conclude the proof.

    \begin{figure}[ht]
        \centering
        \includegraphics[width=0.7\linewidth]{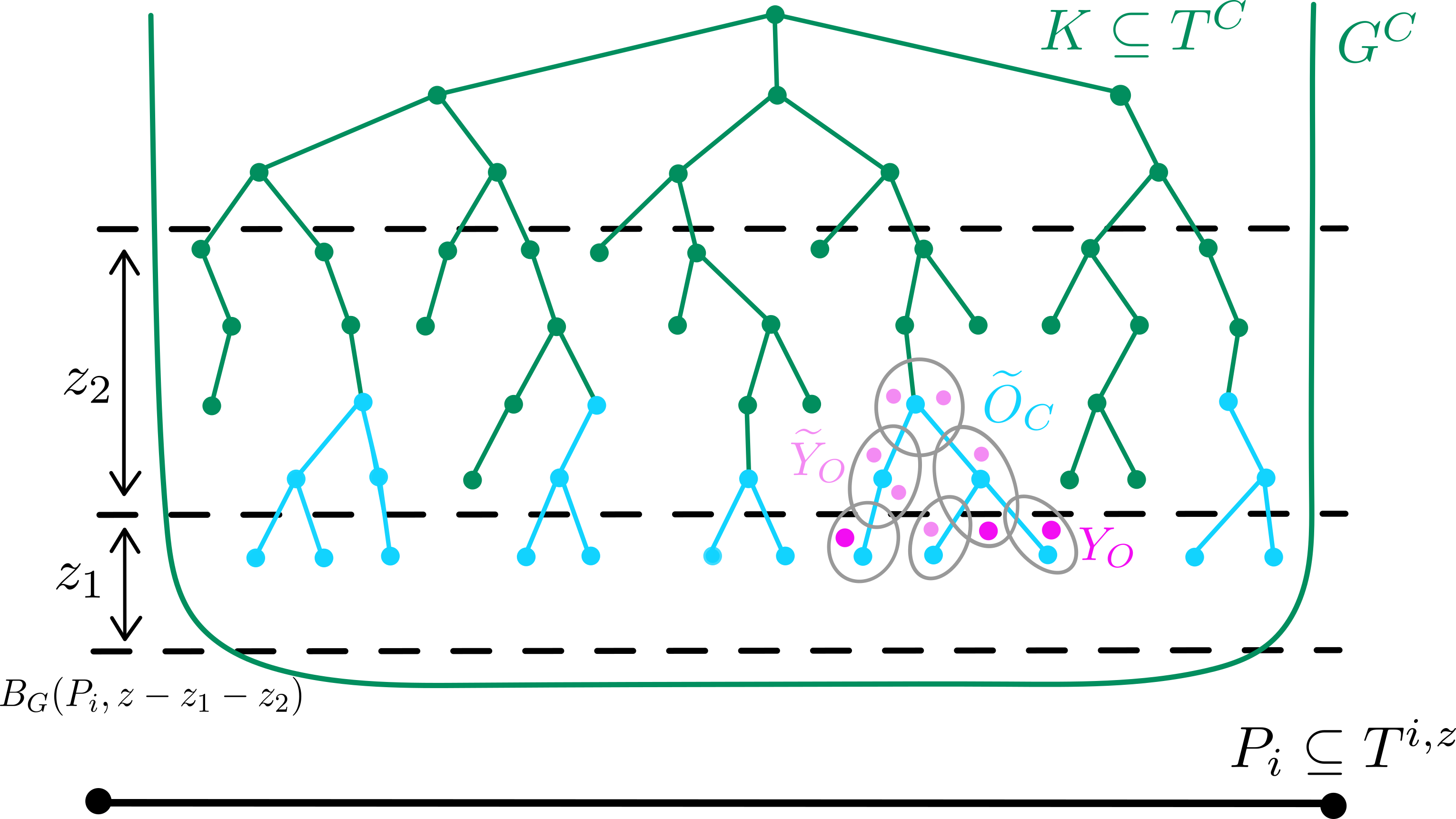}
        \caption{An illustration of a component~$K$ of $T^C$.
        The subtrees $\widetilde{O}_C \cap K$ of $T^C$ are indicated in light blue; they are connected and far apart by \cref{claim:Smooshing:OHatC}. The sets $Y_O \subseteq V(G)$ and $\widetilde{Y}_{O} \subseteq V(G)$, for some $O \in \mathcal{O}$, are indicated in dark and light pink, respectively.
        }
        \label{fig:Smooshing:TC}
    \end{figure}

    \begin{claim} \label{claim:Smooshing:OHatC}
        For all distinct $O, O' \in \mathcal{O}$ the following two statements hold:
        \begin{enumerate}[label=\rm{(\arabic*)}]
            \item \label{itm:Smooshing:OHatC:Diameter} For every component $K$ of $T^C$, $\widetilde{O}_C \cap K$ is connected (or empty).
            \item \label{itm:Smooshing:OHatC:Distance} $d_{T^C}(\widetilde{O}_C, \widetilde{O}'_C) \geq g$.
        \end{enumerate}
    \end{claim}

    \begin{claimproof}
        \ref{itm:Smooshing:OHatC:Diameter}: Let $O \in \mathcal{O}$. By definition, $\widetilde{Y}_{O}$ is connected and contained in
        \[
        B_G(Y_O, z_1+3r_1) \subseteq B_G(C, z_1 + 2z_2 + 9r_1+2) \subseteq B_G(C, y),
        \]
        where we used for the first inequality that $Y \subseteq V(G^C) \subseteq B_G(C, 2z_2+6r_1+2)$ by \cref{Y=G^P_inter_G^C} and \cref{cl:G^C-close-to-Y}.
        Therefore, by \ref{itm:H2'} applied to the partial \fd\ $(\widehat{T}^C, \widehat{\mathcal{V}}^C)$ of $G$ with support~$B_G(C, y)$, the set $\widehat{O}_C = \{t \in V(\widehat{T}^C) : \widehat{V}^C_t \cap \widetilde{Y}_{O} \neq \emptyset\}$ induces a subtree of the forest $\widehat{T}^C$.

        Let $K$ be a component of $T^C$.
        Then, $K$ is a connected subgraph of the forest $\widehat{T}^C$, so its intersection with the subtree $\widehat{T}^C[\widehat{O}_C]$ is connected (or empty).
        Thus, $\widetilde{O}_C \cap K = \widehat{O}_C \cap V(T^C) \cap K = \widehat{O}_C \cap K$ is connected (or empty).
        \medskip

        \ref{itm:Smooshing:OHatC:Distance}: (See \cref{subfig:Smooshing:Sketch:2} and its caption for an informal explanation.) Suppose for a contradiction that there are $O \neq O' \in \mathcal{O}$ such that $d_{T^C}(\widetilde{O}_C, \widetilde{O}'_C) < g$.
        Let $s \in \widetilde{O}_C$ and $t \in \widetilde{O}'_C$ witness this, i.e.\ $d_{T^C}(s,t) < g$, and let $Q$ be the (unique) path between $s$ and $t$ in $T^C$.
        Since~$\widetilde{O}_C \subseteq \widehat{O}_C$ and $\widetilde{O}'_C \subseteq \widehat{O}'_C$ and by definition of $\widehat{O}_C, \widehat{O}'_C$, there exist vertices $u_s \in \widehat{V}^C_s \cap \widetilde{Y}_{O}$ and $u_t \in \widehat{V}^C_t \cap \widetilde{Y}_{O'}$.
        Since $(T^C, \mathcal{V}^C)$ is an honest partial forest-decomposition of $G$ with support $V(G^C)$, there exist $v_s \in V^C_s$ and $v_t \in V^C_t$.
        Since $V^C_s \subseteq \widehat{V}^C_s$ and $(\widehat{T}^C, \widehat{\mathcal{V}}^C)$ has radial width at most $r_1$, it follows that $d_G(u_s, v_s) \leq 2r_1$, and similarly $d_G(u_t, v_t) \leq 2r_1$.

        Since the paths we added to $Y_O$ and $Y_{O'}$ to obtain $\widetilde{Y}_{O}$ and $\widetilde{Y}_{O'}$ are contained in $B_G(P, z-z_2+2r_1)$, the vertices $u_s, u_t$ are contained in $B_G(P, z-z_2+2r_1)$.
        Therefore, the vertices $v_s, v_t$ are contained in $B_G(P, z-z_2+4r_1) \subseteq B_G(P, z)$.
        Let $y_s, y_t \in Y$ lie on a shortest path between $v_s$ and $P$, and $v_t$ and $P$, respectively; such vertices exist since these shortest paths are contained in $B_G(P, z) \subseteq X$ and $Y$ separates $G^P$ and $G^C$ in $G[X]$ by \cref{Y=G^P_inter_G^C}.
        In particular, we have $d_G(v_s, y_s), d_G(v_t, y_t) \leq 4r_1$.
        Furthermore, since $u_s \in \widetilde{Y}_{O}$, there exists $y'_s \in Y_O$ such that $d_G(u_s, y'_s) \leq z_1+3r_1$.
        Then, $d_G(y_s, y'_s) \leq d_G(y_s, v_s) + d_G(v_s, u_s) + d_G(u_s, y'_s) \leq 4r_1 + 2r_1 + z_1 + 3r_1 = z_1 + 9r_1 < 2z_1-4r_1$.
        Thus, by \cref{claim:Smooshing:Y}~\ref{itm:Smooshing:Y:Distance2}, we have $y_s \in Y_O$.
        Similarly, we have $y_t \in Y_{O'}$.

        As we obtained $(T^C, \mathcal{V}^C)$ from $(\widehat{T}^C, \widehat{\mathcal{V}}^C)$ by restricting it to some subgraph of $G$, its radial width is at most that of $(\widehat{T}^C, \widehat{\mathcal{V}}^C)$, which is at most~$r_1$.
        Hence, as $(T^C, \mathcal{V}^C)$ is honest, the proof of \cref{lem:dist-H-decomp} shows that there is a path $Q'$ between $v_s$ and $v_t$ through $\bigcup_{q \in V(Q)} B_G(V^C_q, r_1) \subseteq B_G(G^C, r_1)$ of length at most $2r_1\cdot(d_{T^C}(s,t)+1) \leq 2r_1\cdot g$.
        Let $Q''$ be the $y_s$--$y_t$ path obtained by appending to $Q'$ a shortest $y_s$--$v_s$ path and a shortest $v_t$--$y_t$ path.
        Note that $V(Q'') \subseteq B_G(\{y_s,y_t\}, 6r_1) \cup V(Q') \subseteq B_G(G^C, 6r_1)$ and that $Q''$ has length at most $2r_1\cdot (g+6)$.
        Consequently, $Q''$ avoids $B_G(P, z-z_1-z_2+2r_1)$ since $G^C$ avoids $B_G(P, z-z_2-1)$ and $6r_1 < z_1-2r_1$, and $Q''$ has length at most $2z_2$, contradicting \cref{claim:Smooshing:Y}~\ref{itm:Smooshing:Y:Distance2}.
    \end{claimproof}

    By \cref{claim:Smooshing:OHatC}~\ref{itm:Smooshing:OHatC:Diameter}, for every component $K$ of $T^C$ and every $O \in \mathcal{O}$ the subgraph $\widetilde{O}_C \cap K$ is connected (or empty).
    We may therefore define \defn{$(\widetilde{T}^C, \widetilde{\mathcal{V}}^C)$} by contracting each nonempty set $\widetilde{O}_C \cap V(K)$ to a node \defn{$t^C_{O,K}$}, with bag $\defnm{\widetilde{V}_{t^C_{O, K}}^C} := \bigcup_{s \in \widetilde{O}_{C} \cap V(K)} V^C_s$. For every other node $t$ of $\widetilde{T}^C$, set $\widetilde{V}^C_t := V^C_t$. For notational simplicity, we denote $\widetilde{V}^C_{t^C_{O,K}}$ by \defn{$\widetilde{V}^C_{O,K}$}.

    By \cref{claim:Smooshing:OHatC}~\ref{itm:Smooshing:OHatC:Distance}, the sets~$\widetilde{O}_{C}$ are pairwise disjoint, and hence so are the sets being contracted.
    In particular, the nodes~$t^C_{O, K}$ are pairwise distinct.
    \bigskip

    \noindent\textbf{\large Definition of the graph-decomposition ($\mathbf{H, \mathcal{V}}$)}
    \smallskip

    We now define a graph-decomposition \defn{$(H,\mathcal{V})$} of $G$, and show that it is as desired.
    Let \defn{$H$} be obtained from the disjoint union of $\widetilde{T}^P$ and $\widetilde{T}^C$ by simultaneously identifying every pair of nodes $s\in V(\widetilde{T}^P)$ and $t\in V(\widetilde{T}^C)$ whose bags $\widetilde{V}_s^P$ and $\widetilde{V}_t^C$ intersect.

    By \cref{Y=G^P_inter_G^C}, every pair of identified bags contains a vertex of $Y$.
    For each $O\in\mathcal{O}$, the only bag of $(\widetilde{T}^P,\widetilde{\mathcal{V}}^P)$ that intersects $Y_O$ is the bag of $t_O$, while the only bags of $(\widetilde{T}^C,\widetilde{\mathcal{V}}^C)$ that can intersect $Y_O$ are those of the nodes $t^C_{O,K}$.
    Hence, $H$ is obtained by merging, for each $O\in\mathcal{O}$, the node $t^P_O$ with all nodes $t^C_{O,K}$ whose bags intersect $\widetilde{V}^P_{O}$.
    Denote the resulting node by \defn{$h_O$}.

    We define the bags of $(H,\mathcal{V})$ as follows.
    If a node $h$ comes from an unmerged node of $\widetilde{T}^P$, set $\defnm{V_h}\coloneqq\widetilde{V}_h^P$; if it comes from an unmerged node of $\widetilde{T}^C$, set $\defnm{V_h}\coloneqq\widetilde{V}_h^C$.
    Finally, if $h=h_O$ is obtained by merging nodes of $\widetilde{T}^P$ and $\widetilde{T}^C$, let $\defnm{V_h}$ be the union of their bags.

    \begin{figure}[ht]
        \centering
        \includegraphics[width=0.65\linewidth]{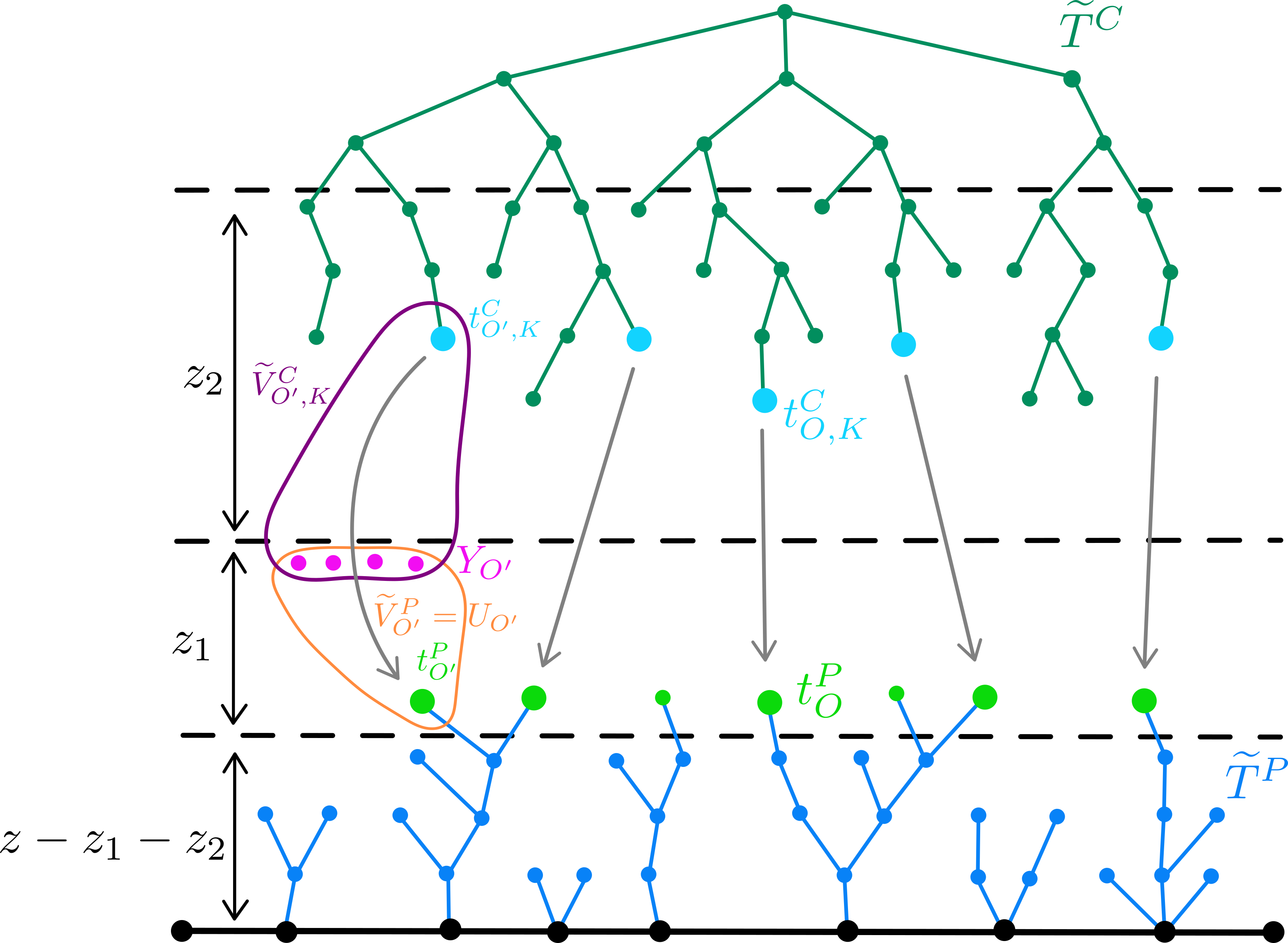}
        \caption{Smooshing (i.e.\ gluing) the tree-decompositions together.}
        \label{fig:Smooshing:ItAllTogether}
    \end{figure}

    Since $(\widetilde{T}^P,\widetilde{\mathcal{V}}^P)$ and $(\widetilde{T}^C,\widetilde{\mathcal{V}}^C)$ are forest-decompositions of $G^P$ and $G^C$, respectively, and since $G^P\cup G^C=G[X]$, the bags of $\mathcal{V}$ cover $G[X]$.
    Thus, $(H,\mathcal{V})$ satisfies \ref{itm:H1}.
    Moreover, identifying all nodes whose bags intersect ensures that, for every vertex of $G[X]$, the nodes whose bags contain it induce a connected subgraph of $H$.
    Hence, \ref{itm:H2} also holds, and $(H,\mathcal{V})$ is a partial \gd\ of~$G$ with support~$X$.

    It remains to show that $(H,\mathcal{V})$ is $(w_{\ref{lem:Smooshing}}(r_1),s_{\ref{lem:Smooshing}}(r_2))$-radial and that $H$ has girth at least $g$.
    \medskip

    \noindent\textbf{Radial spread:}
    Since $(\widetilde{T}^P, \widetilde{\mathcal{V}}^P)$ and $(\widetilde{T}^C, \widetilde{\mathcal{V}}^C)$ are obtained from $(T^P, \mathcal{V}^P)$ and $(T^C, \mathcal{V}^C)$ by restricting to subgraphs of $G$ and contracting some nodes of the decomposition graphs, their radial spreads do not increase.
    Hence, both have radial spread at most $r_2$.

    The decomposition $(H,\mathcal{V})$ is obtained by gluing $(\widetilde{T}^C,\widetilde{\mathcal{V}}^C)$ onto $(\widetilde{T}^P,\widetilde{\mathcal{V}}^P)$.
    Its radial spread is therefore at most the radial spread of $(\widetilde{T}^P,\widetilde{\mathcal{V}}^P)$ plus twice the radial spread (`diameter spread') of $(\widetilde{T}^C,\widetilde{\mathcal{V}}^C)$, and hence at most $3r_2=s_{\ref{lem:Smooshing}}(r_2)$.

    \medskip

    \noindent\textbf{Radial width:}
    Recall that $(T^P,\mathcal{V}^P)$ has radial width at most $r_1$.
    Moreover, for every $O\in\mathcal{O}$, the set $U_O$ has diameter at most $2z_1+6r_1$ by \cref{claim:Smooshing:Y}~\ref{itm:Smooshing:Y:Diameter}.
    Since $(\widetilde{T}^P,\widetilde{\mathcal{V}}^P)$ is obtained by contracting each $O$ to a node $t^P_O$ with bag $\widetilde{V}_{O}^P=U_O$, every bag in $\widetilde{\mathcal{V}}^P$ has radius at most $2z_1+6r_1$.

    Now let $O\in\mathcal{O}$, let $K$ be a component of $T^C$, and let $s\in\widetilde{O}_C\cap V(K)$.
    Then $V_s^C \subseteq \widehat{V}_s^C$, and $\widehat{V}_s^C$ intersects $\widetilde{Y}_O$.
    Since $(\widehat{T}^C,\widehat{\mathcal{V}}^C)$ has radial width at most $r_1$, it follows that $V_s^C \subseteq B_G(\widetilde{Y}_O,2r_1)$.
    Consequently, $\widetilde{V}_{{O,K}}^C=\bigcup_{s\in\widetilde{O}_C\cap V(K)}V_s^C
    \subseteq B_G(\widetilde{Y}_O,2r_1)$.
    As $\widetilde{Y}_O$ has diameter at most $2(2z_1+6r_1) = 4z_1+12r_1$ by \cref{claim:Smooshing:Y}~\ref{itm:Smooshing:Y:Diameter}, the bag $\widetilde{V}_{{O,K}}^C$ has diameter at most $4z_1+16r_1$.
    Every other bag in $\widetilde{\mathcal{V}}^C$ is contained in a bag of $\widehat{\mathcal{V}}^C$ and hence has diameter at most $2r_1$.

    Finally, each merged bag of $(H,\mathcal{V})$ is the union of one bag of $\widetilde{\mathcal{V}}^P$ and some bags of $\widetilde{\mathcal{V}}^C$, all of which intersect that bag of $\widetilde{\mathcal{V}}^P$.
    Hence, its radius is at most $(2z_1+6r_1) + (4z_1+16r_1) = 6z_1 + 22r_1 = w_{\ref{lem:Smooshing}}(r_1)$.
    Thus, $(H,\mathcal{V})$ has radial width at most $w_{\ref{lem:Smooshing}}(r_1)$.
    \medskip

    \noindent\textbf{Girth:} Since $\widetilde{T}^P$ and $\widetilde{T}^C$ are forests, every cycle $D$ in $H$ must contain a path in $\widetilde{T}^C$ whose endvertices are two distinct identified nodes $h_O$ and $h_{O'}$ and whose internal vertices do not belong to $\widetilde{T}^P$.
    These endvertices correspond to nodes $t^C_{O,K}$ and $t^C_{O',K'}$ of $\widetilde{T}^C$.
    By \cref{claim:Smooshing:OHatC}~\ref{itm:Smooshing:OHatC:Distance}, their distance in $\widetilde{T}^C$ is at least $g$.
    Hence, $D$ has length at least $g$.
\end{proof}

\section{Coral reef (aka Packing many fat cycles)} \label{sec:CoralReef}

In this section we combine \cref{lem:Smooshing,lem:AngryHippo,lem:ApexForestWithoutApex,lem:untangling} to prove the following key structural result of this paper (this is the precise statement behind \eqref{eq:Sketch} from the proof sketch in \cref{sec:ProofSketch}):

\begin{theorem} \label{thm:ResultAfterSmooshing}
    There exist functions $f_{\ref{thm:ResultAfterSmooshing}},w'_{\ref{thm:ResultAfterSmooshing}},s_{\ref{thm:ResultAfterSmooshing}} : \N \to \N$ and $w_{\ref{thm:ResultAfterSmooshing}} : \N^2 \to \N$ such that for every $k,q,g \in \N$ and for every graph $G$ at least one of the following statements holds:
    \begin{enumerate}[label=\rm{(\roman*)}]
        \item \label{itm:Smooshing:main:FatMinor} $k\cdot K_3$ is a $q$-fat minor of $G$.
        \item \label{itm:Smooshing:main:GraphDecomp} There exist a graph $H$ and a set $U$ of at most $f_{\ref{thm:ResultAfterSmooshing}}(k)$ vertices of $H$ such that
        \begin{itemize}
            \item $U$ hits all cycles in $H$ of length at most $g$;
            \item $G$ admits an honest $(w_{\ref{thm:ResultAfterSmooshing}}(q,g), s_{\ref{thm:ResultAfterSmooshing}}(q))$-radial $H$-decomposition such that every bag associated with a node in $V(H) \setminus U$ has radius at most $w'_{\ref{thm:ResultAfterSmooshing}}(q)$ in $G$.
        \end{itemize}
    \end{enumerate}
    Furthermore, this holds with $f_{\ref{thm:ResultAfterSmooshing}} \in \mathcal{O}(k \log k)$.
\end{theorem}

Let us recall \cref{lem:AngryHippo,lem:ApexForestWithoutApex,lem:untangling,lem:Smooshing} in one place for convenience.

\apexforestwithoutapex*

\happyhippo*

\untangling*

\smooshing*

We now proceed with the proof of \cref{thm:ResultAfterSmooshing}.

\begin{proof}[Proof of \cref{thm:ResultAfterSmooshing}.]
    Let $k,q,g \in \N$ and let $G$ be a graph.
    \bigskip

    \noindent\textbf{\large Preparation}
    \smallskip

    \noindent \textbf{Proof outline:}
    We proceed, roughly speaking, as follows:
    We recursively construct a sequence $D_1, D_2, \dots$ of $q$-fat cycles that are pairwise at distance at least $q$, together with a small vertex set $Y \subseteq V(G)$ that has small distance to every cycle constructed so far.
    At each step, we either find a new cycle $D_i$ to add to the sequence, and a new such set $Y$, or show that \ref{itm:Smooshing:main:GraphDecomp} holds.
    In the latter case, the construction terminates, and the set $Y$ constructed so far is used to obtain the hitting set $U$ for $H$.
    If the construction does not terminate within the first $k$ steps, then it produces $k$ cycles that are $q$-fat and pairwise at distance at least $q$, and hence \ref{itm:Smooshing:main:FatMinor} holds.
    \smallskip

    The preceding description is slightly oversimplified.
    In fact, it is more convenient to allow the cycles $D_i$ to be close to one another, or even to intersect, provided that these interactions occur in a controlled manner.
    \medskip

    \noindent\textbf{Definition of coral reefs:}
    An \defn{$(r_1, r_2, d, \kappa, q)$-coral reef}\footnote{Save the coral reefs.} is a collection $\mathcal{D} = \{D_1, \dots, D_n\}$ of $(r_1, r_2, d, \kappa)$-corsola cycles $D_i = (P_{D_i}, W_{D_i})$ satisfying the following properties.
    For each $i \in [n]$, let $\defnm{x_i}$ denote the `middle' vertex\footnote{That is, if $W_{D_i} = w_0 \dots w_s$ (with $s\le \ell$), then we may take $x_i=w_{\lfloor s/2\rfloor}$.} of the path~$W_{D_i}$, and set $\defnm{X_\mathcal{D}} \coloneqq \{x_i : i \in [n]\}$ and $\ell := 77d$.
    Then, as illustrated in \cref{fig:CoralReef}, we have:
    \begin{enumerate}[leftmargin=3em,label=\rm{(CR\arabic*)}]
        \item \label{itm:CoralReef:Distance} for all $i < j \in [n]$, all vertices $u \in V(D_i)$ and $v \in V(D_j)$, we have $d_G(u, v) \geq q$, unless $v \in B_G(x_j, \ell/2)$;
        \item \label{itm:CoralReef:degenerate} for each $j \in [n]$, at most $78$ cycles $D_i$ with $i<j$ are at distance less than $q$ from $D_j$.
    \end{enumerate}
    The \defn{size} of a coral reef $\mathcal{D}$ is the number $n$ of cycles in $\mathcal{D}$.
    Note that if a coral reef $\mathcal{D}$ has size $n$, then $|X_\mathcal{D}| \leq n$.
    \medskip

    \begin{figure}[ht]
        \centering
        \includegraphics[width=0.6\linewidth]{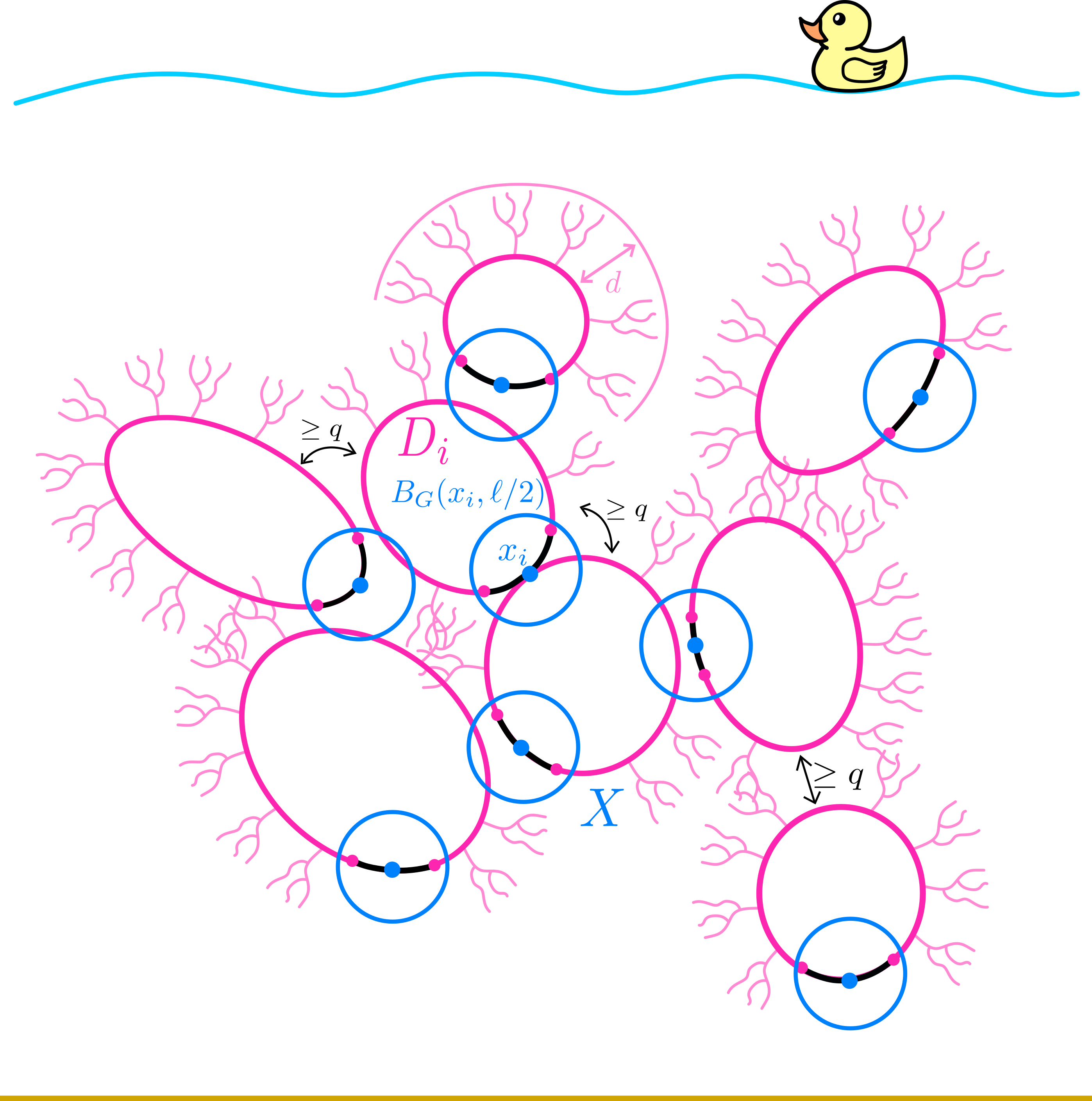}
        \caption{The corsola cycles~$D_i$ of an $(r_1,r_2,d,\kappa,q)$-coral reef $\mathcal{D}$ are shown in black/pink.
        The vertices of $X_\mathcal{D}$ are depicted in blue.}
        \label{fig:CoralReef}
    \end{figure}

    \noindent\textbf{Proof outline (continued):}
    For the proof, we follow the procedure described above, with one modification.
    We first choose $q$-fat goldfish cycles $D'_1, \dots, D'_m$, and then construct $q$-fat corsola cycles $D_1, \dots, D_n$ that form a coral reef.
    Throughout the construction, we ensure that each goldfish cycle $D'_i$ is at distance at least $q$ from every other goldfish or corsola cycle; see \cref{fig:CoralReef}.

    As shown in \cref{claim:GreatBarrierReef} below, every coral reef of size $n$ contains a subcollection of at least $n/79$ cycles that are pairwise at distance at least~$q$.
    Consequently, the modified construction terminates after at most $79(k-m)$ steps.
    \medskip

    \noindent\textbf{Definition of parameters and functions:}
    Recall that $k,q,g\in\mathbb{N}$ are fixed, where $k$ is the desired number of fat cycles, $q$ specifies both their desired fatness and pairwise distance, and $g$ is the desired girth of the decomposition graph $H$ in the outcome \ref{itm:Smooshing:main:GraphDecomp}.

    We need to define functions $f_{\ref{thm:ResultAfterSmooshing}}$, $w_{\ref{thm:ResultAfterSmooshing}}$, $w'_{\ref{thm:ResultAfterSmooshing}}$ and $s_{\ref{thm:ResultAfterSmooshing}}$.
    We also introduce auxiliary parameters $r_1,r_2,R_1,R_2,d,\ell,R, b$ and~$N$, which we will use throughout the proof, as well as some more auxiliary parameters, which we will only need `locally' at some point in the proof.
    We choose the former set of parameters so that
    \begin{align*}
        q &\overset{\ref{lem:ApexForestWithoutApex},\ref{lem:AngryHippo}}{\ll} r_1, r_2 \overset{\ref{lem:Smooshing}}{\ll} R_1,R_2\\
        q,g &\overset{\ref{lem:Smooshing}}{\ll} d \overset{\ref{lem:untangling}}{\ll} \ell,R \overset{\ref{lem:AngryHippo}}{\ll} b, \text{ and}\\
        \defnm{N(n)} &\overset{\ref{lem:untangling}}{:=}\phantom{ } 2n + \lfloor 24 k \log k \rfloor \text{ for every } n \in \N.
    \end{align*}
    We then set
    \begin{align*}
        \defnm{f_{\ref{thm:ResultAfterSmooshing}}(k)} &:= N(79k) + 79k = 237k + \lfloor 24 k \log k \rfloor,\\
        \defnm{w'_{\ref{thm:ResultAfterSmooshing}}(q)} &:= R_1, \text{ and}\\
        \defnm{s_{\ref{thm:ResultAfterSmooshing}}(q)} &:= R_2+1.
    \end{align*}
    The definition of $w_{\ref{thm:ResultAfterSmooshing}}(q,g)$ is more involved, so we postpone it until the end of the proof, where it is first needed; see \eqref{eq:DefinitionOfWidth}.
    In particular, $w_{\ref{thm:ResultAfterSmooshing}}(q,g)$ will be chosen much larger than all the parameters introduced above.
    \medskip

    We now specify the constants explicitly, although the reader may prefer to skip these details.
    We collect them here to keep all values in one place, to make their dependencies transparent, and to indicate the order in which the relevant lemmas are applied.
    We will also indicate later in the proof where each constant is used or determined.
    The resulting values are as follows; the parameters listed above are highlighted in red:
    \smallskip

    \noindent First, let $Q \coloneqq 5q+1$. We apply \cref{lem:AngryHippo} (Happy hippo lemma) for $Q$, which yields
    \begin{equation} \label{eq:HappyHippo}
        q' \coloneqq q'_{\ref{lem:AngryHippo}}(Q),\; r'_1 \coloneqq w_{\ref{lem:AngryHippo}}(Q),\; r'_2 \coloneqq s_{\ref{lem:AngryHippo}}(Q),\; \kappa \coloneqq \kappa_{\ref{lem:AngryHippo}}(Q)\; \text{ and }\; a \coloneqq a_{\ref{lem:AngryHippo}}(Q) \geq q.
    \end{equation}
    We then apply \cref{lem:ApexForestWithoutApex} (Forest-decomposition lemma) with $q'$ and $a$, which determines
    \begin{equation} \label{eq:ForestDecomp}
        s_1 \coloneqq w_{\ref{lem:ApexForestWithoutApex}}(q'),\; s_2 \coloneqq s_{\ref{lem:ApexForestWithoutApex}}(q')\; \text{ and }\; u \coloneqq r_{\ref{lem:ApexForestWithoutApex}}(q', a).
    \end{equation}
    We let
    \begin{equation} \label{eq:Maxr1Andr2}
    \defnm{r_1} := \max\{r'_1,s_1\}\; \text{ and }\; \defnm{r_2} := \max\{r'_2,s_2\}.
    \end{equation}
    Next, we apply \cref{lem:Smooshing} (Smooshing lemma), which determines
    \begin{equation} \label{eq:Smooshing}
    \begin{aligned}
        y &\coloneqq y_{\ref{lem:Smooshing}}(r_1, g+1),\; z' \coloneqq z'_{\ref{lem:Smooshing}}(r_1, g+1),\; \defnm{R_1} \coloneqq w_{\ref{lem:Smooshing}}(r_1) \geq r_1,\; \defnm{R_2} \coloneqq s_{\ref{lem:Smooshing}}(r_2) \geq r_2\; \\z &\coloneqq \max\{z', u+y\},\; \text{ and}\;
        \defnm{d} := \max\{d_{\ref{lem:Smooshing}}(z), 233q+8\} \geq z.
    \end{aligned}
    \end{equation}
    Note that $R_1, R_2$ only depend on $r_1, r_2$, and hence only on~$q$.
    Since $d\ge 2q$, we may apply \cref{lem:untangling} (Untangling lemma), which determines
    \begin{equation} \label{eq:Untangling}
        \defnm{R} \coloneqq R_{\ref{lem:untangling}}(q, d, \kappa).
    \end{equation}

    Set
    \[
    \defnm{\ell} \coloneqq 77d.
    \]
    Having determined $d$, we now invoke \cref{lem:AngryHippo} again for $Q$ and $d$, which determines
    \begin{equation} \label{eq:HappyHippo2}
    \defnm{b} \coloneqq \max\{b_{\ref{lem:AngryHippo}}(Q, d),\; \lceil \ell/2\rceil\}.
    \end{equation}
    The choice of $d$ implies $d \geq 25q+8 = 5Q+3$, as required in \cref{lem:AngryHippo}.
    It also implies \begin{align}
        \frac{77d}{78}+2q-2  < d, \label{eq:d}
    \end{align}
    which we will use later.

    This completes the definition of the parameters. We now start with the proof.
    \bigskip

    \noindent\textbf{\large Construction of the fat cycles (outcome \ref{itm:Smooshing:main:FatMinor})}
    \smallskip

    \noindent\textbf{Choice of goldfish cycles:}
    Recall that a cycle $C$ in a graph $G$ is \defn{$b$-goldfish} if there exists a vertex $z \in V(G)$ such that $V(C) \subseteq B_G(z, b)$; see \cref{fig:GoldfishCycle}.

    We start our construction by picking a maximal collection $\defnm{\mathcal{D}'} = \{\defnm{D'_1}, \dots, \defnm{D'_m}\}$ of $q$-fat $b$-goldfish cycles that are pairwise at distance at least $q$. If $m \geq k$, then we are done, as $D'_1 \cup \dots \cup D'_m$ forms a $q$-fat model of $m \cdot K_3$ as required by \ref{itm:Smooshing:main:FatMinor}.

    We may therefore assume that $m<k$.
    Since each cycle $D'_i$ is $b$-goldfish, there exists a vertex $\defnm{z_i} \in V(G)$ such that $V(D'_i)\subseteq B_G(z_i,b)$.
    Let $\defnm{Z} \coloneqq \{z_i : i \in [m]\}$.
    Then $|Z|\leq m<k$, and every cycle $D'_i$ is contained in $B_G(Z,b)$.
    \medskip

    \noindent\textbf{Construction of corsola cycles:}
    We now proceed recursively.
    After step $n$, we will have constructed a collection $\mathcal{D}_n = \{D_1, \dots, D_n\}$ of $n$~$q$-fat $(r_1, r_2, d, \kappa)$-corsola cycles and a set $Y_n \subseteq V(G)$ of size at most~$N(n)$ satisfying the following properties; see \cref{fig:CoralReefWithGoldfish}:
    \begin{enumerate}[label=\rm{(A\arabic*)}]
        \item \label{itm:IH:CoralReef} $\mathcal{D}_n$ is an $(r_1, r_2, d, \kappa, q)$-coral reef of size $n$;
        \item \label{itm:IH:DistanceGoldfish} every cycle in $\mathcal{D}_n$ is at distance at least $q$ from every goldfish cycle in $\mathcal{D}'$;
        \item \label{itm:IH:corsolaPaths} for all distinct $i, j \in [n]$, every $D_i$--$D_j$ path of length less than~$d$ contains a vertex of\\ $B_G(X_{\mathcal{D}_n}, \ell/2) \cup B_G(Y_n, R)$;
        \item \label{itm:IH:corsolaPaths2} for all distinct $i, j \in [n]$, the distance in $G$ between $P_{D_i}$ and $P_{D_j}$ is more than $5q$.
    \end{enumerate}

    We remark that while we have $\mathcal{D}_1 \subseteq \mathcal{D}_2 \subseteq \dots \subseteq \mathcal{D}_n$ as well as $X_{\mathcal{D}_1} \subseteq X_{\mathcal{D}_2} \subseteq \ldots \subseteq X_{\mathcal{D}_n}$, this does not hold for the sets~$Y_n$. In fact, in each step of the construction, we will choose an entirely new set~$Y_n$.

     \begin{figure}[ht]
        \centering
        \includegraphics[width=0.6\linewidth]{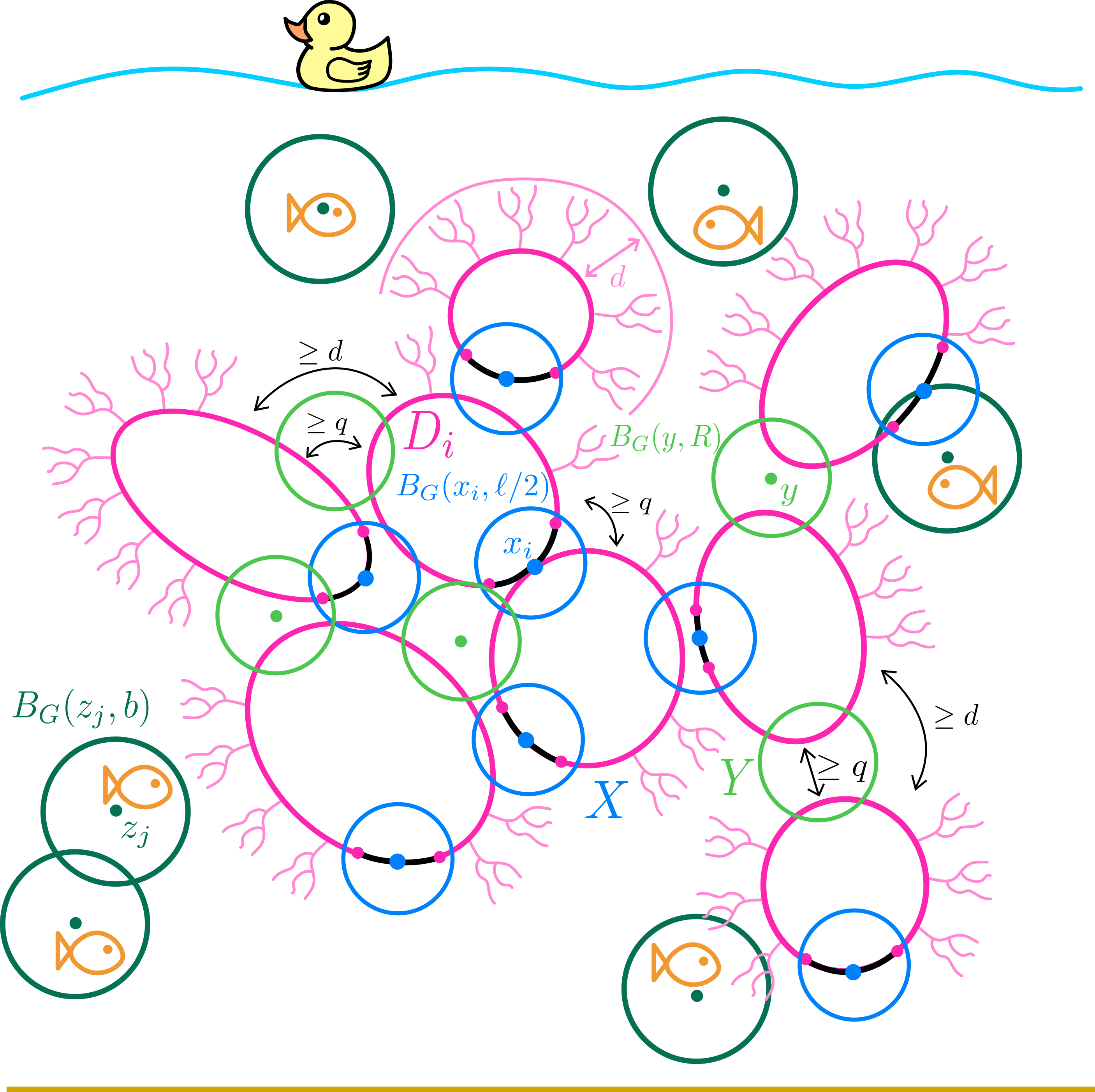}
        \caption{The pink corsola cycles~$D_i$ form an $(r_1,r_2,d,\kappa,q)$-coral reef. Together with the orange goldfish cycles $D'_j$ and the light-green vertex set~$Y$, they satisfy \ref{itm:IH:CoralReef} to \ref{itm:IH:corsolaPaths2}.}

        \label{fig:CoralReefWithGoldfish}
    \end{figure}

    \smallskip
    We first show how to conclude the proof if we find many such cycles $D_i$.

    \begin{claim} \label{claim:GreatBarrierReef}
        Suppose there exist a collection $\mathcal{D}_n$ of $n := 79(k-m)$ $q$-fat cycles $D_i$ and a set $Y_n \subseteq V(G)$ of size at most $N(n)$ such that $\mathcal{D}_n$ and $Y_n$ satisfy \ref{itm:IH:CoralReef} and \ref{itm:IH:DistanceGoldfish}.\footnote{Note that \ref{itm:IH:corsolaPaths} and \ref{itm:IH:corsolaPaths2} are only needed if the construction terminates before $n$ cycles~$D_i$ have been found.} Then $k \cdot K_3$ is a $q$-fat minor of~$G$.
    \end{claim}

    \begin{claimproof}
        Let $F$ be the graph with vertex set $\mathcal{D}_n$, in which two cycles $D_i, D_j \in \mathcal{D}_n$ are adjacent if and only if they are at distance less than $q$ in $G$.
        By \ref{itm:CoralReef:degenerate}, the graph~$F$ has chromatic number at most~$79$ (as can be seen by colouring $F$ greedily in the order $D_1, D_2, \dots, D_n$).
        Hence, $F$ contains an independent set of size at least $n/79 = k-m$.
        Equivalently, there exist $k-m$ cycles in $\mathcal{D}_n$ that are pairwise at distance at least $q$.

        By \ref{itm:IH:DistanceGoldfish}, these corsola cycles, together with the goldfish cycles $D'_j\in\mathcal{D}'$, form a $q$-fat model of $k\cdot K_3$ in $G$.
    \end{claimproof}

    This completes the proof if the recursive construction produces a coral reef of size $79(k-m)$.
    We now start constructing the corsola cycles~$D_i$ and the sets~$Y_i$.
    \medskip

    \noindent\textbf{Choice of the first corsola cycle:}
    Recall that $u = r_{\ref{lem:ApexForestWithoutApex}}(q', a)$, $r_1 \geq s_1 = w_{\ref{lem:ApexForestWithoutApex}}(q')$ and $r_2 \geq s_2 = s_{\ref{lem:ApexForestWithoutApex}}(q')$ (see \eqref{eq:ForestDecomp} and \eqref{eq:Maxr1Andr2}).
    Applying \cref{lem:ApexForestWithoutApex} to $G$ with $U = B_G(Z, 2b)$, $q=q'$ and $d = a$, we obtain that at least one of the following statements holds: \begin{itemize}
        \item $G$ admits an honest $(r_1,r_2)$-radial partial \fd\ $(T, \mathcal{V})$ with support\\ $G-B_G(Z, 2b+u)$; or
        \item $G$ contains a $q'$-fat cycle~$C$ at distance at least~$a$ from $B_G(Z,2b)$.
    \end{itemize}

    In the former case, we define a \gd\ as required in \ref{itm:Smooshing:main:GraphDecomp}.
    For each $z \in Z$, let $V_z := B_G(z, 2b+u+1)$.
    Let $H$ be the graph obtained from the disjoint union of the forest~$T$ and the set $Z$ by adding an edge between distinct $z \in Z$ and $h \in V(T) \cup Z$ whenever $V_z \cap V_h \neq \emptyset$.
    It is straightforward to verify that $(H, \mathcal{V})$ is an honest \gd\ of~$G$.

    By construction, $(H, \mathcal{V})$ and $U := Z$ satisfy \ref{itm:Smooshing:main:GraphDecomp}, since $f_{\ref{thm:ResultAfterSmooshing}}(k) \geq k-1$, $w'_{\ref{thm:ResultAfterSmooshing}}(q) = R_1 \geq r_1$, $s_{\ref{thm:ResultAfterSmooshing}}(q) = R_2 + 1 \geq r_2 + 1$ and $w_{\ref{thm:ResultAfterSmooshing}}(q,g) \geq 2b+u+1$.
    We remark that the parameters $f_{\ref{thm:ResultAfterSmooshing}}(k), w'_{\ref{thm:ResultAfterSmooshing}}(q), s_{\ref{thm:ResultAfterSmooshing}}(q)$ and $w_{\ref{thm:ResultAfterSmooshing}}(q,g)$ were chosen substantially larger than required here for reasons that will become apparent later.
    \smallskip

    In the latter case, recall that $d \geq 5Q+3$, $q' = q'_{\ref{lem:AngryHippo}}(Q)$, $r_1 \geq r'_1 = w_{\ref{lem:AngryHippo}}(Q)$, $r_2 \geq r'_2 = s_{\ref{lem:AngryHippo}}(Q)$, $\kappa = \kappa_{\ref{lem:AngryHippo}}(Q)$, $a = a_{\ref{lem:AngryHippo}}(Q)$ (see \eqref{eq:HappyHippo}) and that $b \geq b_{\ref{lem:AngryHippo}}(Q, d)$ (see \eqref{eq:HappyHippo2}).

    Moreover, for every $c \in V(C)$, the ball $B_G(c, b)$ contains no $Q$-fat model of $K_3$.
    Indeed, such a model would yield a $q$-fat $b$-goldfish cycle at distance at least $(a+b)-b \geq q$ from $B_G(Z,b)$, contradicting the maximality of $\mathcal{D}'$.
    We may therefore apply \cref{lem:AngryHippo} to $U:=B_G(Z, 2b)$ and $C$, with parameters $Q$ and $d$, to obtain a $Q$-fat $(r_1,r_2,d,\kappa)$-corsola cycle $C' = (P_{C'}, W_{C'})$ such that $d_G(P_{C'}, B_G(Z,2b)) \geq Q$.

    Set $\defnm{D_1} := C'$ and $\defnm{Y_1} := \emptyset$.
    Then, $\defnm{\mathcal{D}_1} := \{D_1\}$ is an $(r_1, r_2, d, \kappa, q)$-coral reef of size~$1$, and it satisfies \ref{itm:IH:corsolaPaths}.
    Moreover, by the definition of $b$, we have $b \geq \ell/2$.
    Since $P_{C'}$ is disjoint from $B_G(Z, 2b)$, it follows that, for every $j \in [m]$,
    \begin{equation} \label{eq:Distance:CorsolaGoldfish}
    \begin{aligned}
        d_G(D_1,D'_j) &\geq d_G(D_1, B_G(Z,b)) \geq d_G(P_{C'}, B_G(Z,b)) - ||W_{C'}||/2 \\
        &= d_G(P_{C'}, B_G(Z,2b)) + b - ||W_{C'}||/2 \geq q + b - \ell/2 \geq q.
    \end{aligned}
    \end{equation}
    Thus, $D_1$ satisfies \ref{itm:IH:DistanceGoldfish}.
    This completes the first step of the recursive construction.
    \medskip

    \noindent\textbf{Choice of the (n+1)st corsola cycle:}
    Let $n \in \N$ with $1 \leq n < 79(k-m)$, and suppose that we have already constructed a collection $\defnm{\mathcal{D}_n = \{D_1, \dots, D_n\}}$ of $n$ corsola cycles and a set $\defnm{Y_n} \subseteq V(G)$ of size at most~$N(n)$ such that $\mathcal{D}_n$ and $Y_n$ satisfy \ref{itm:IH:CoralReef} to \ref{itm:IH:corsolaPaths2}.
    Define
    \begin{equation} \label{eq:TheSetO}
        \defnm{O} := \Bigg(\bigcup_{i \in [n]} V(D_i)\Bigg) \cup B_G(X_{\mathcal{D}_n}, \ell) \cup B_G(Y_n, R+\ell/2) \cup B_G(Z, 2b).
    \end{equation}

    Recall once again that $u = r_{\ref{lem:ApexForestWithoutApex}}(q', a)$, $r_1 \geq s_1 = w_{\ref{lem:ApexForestWithoutApex}}(q')$ and $r_2 \geq s_2 = s_{\ref{lem:ApexForestWithoutApex}}(q')$ (see \eqref{eq:ForestDecomp} and~\eqref{eq:Maxr1Andr2}).
    Applying \cref{lem:ApexForestWithoutApex} to $G$ with $U = O$, $q=q'$ and $d = a$, we obtain one of the following:
    \begin{itemize}
        \item $G$ admits an honest $(r_1,r_2)$-radial partial \fd\ $(T, \mathcal{V})$ with support\\ $G-B_G(O, u)$; or
        \item $G$ contains a $q'$-fat cycle $C$ such that $d_G(C,O)\geq a$.
    \end{itemize}

    We first consider the latter case and use $C$ to construct the $(n+1)$st corsola cycle. We then consider the former case, in which we terminate the recursive construction and show that $G$ satisfies \ref{itm:Smooshing:main:GraphDecomp}.
    \smallskip

    Suppose, therefore, that $G$ contains a $q'$-fat model $C$ of $K_3$ with $d_G(C,O)\geq a$.
    Recall once again that $d \geq 5Q+3$, $q' = q'_{\ref{lem:AngryHippo}}(Q)$, $r_1 \geq r'_1 = w_{\ref{lem:AngryHippo}}(Q)$, $r_2 \geq r'_2 = s_{\ref{lem:AngryHippo}}(Q)$, $\kappa = \kappa_{\ref{lem:AngryHippo}}(Q)$, $a = a_{\ref{lem:AngryHippo}}(Q)$ (see \eqref{eq:HappyHippo}) and that $b \geq b_{\ref{lem:AngryHippo}}(Q, d)$ (see \eqref{eq:HappyHippo2}).

    Since $B_G(Z, 2b) \subseteq O$, the same argument as before shows that, for every $c \in V(C)$, the ball $B_G(c, b)$ contains no $Q$-fat model of $K_3$.
    We may therefore apply \cref{lem:AngryHippo} to $U:=O$ and $C$, with parameters $Q$ and $d$, to obtain a $Q$-fat $(r_1,r_2,d,\kappa)$-corsola cycle $C' = (P_{C'}, W_{C'})$ such that $d_G(P_{C'}, O) \geq Q$.

    We add $C'$ to the collection $\mathcal{D}_n$ by setting $\defnm{D_{n+1}} := C'$ and $\defnm{\mathcal{D}_{n+1}} := \mathcal{D}_n \cup \{D_{n+1}\}$.
    Since $\mathcal{D}_n$ satisfies \ref{itm:IH:corsolaPaths2} and $d_G(P_{C'}, O) \geq Q >5q$, we have that $\mathcal{D}_{n+1}$ also satisfies \ref{itm:IH:corsolaPaths2}.
    The same calculation as in~\eqref{eq:Distance:CorsolaGoldfish} shows that $D_{n+1}$ is at distance at least~$q$ from every goldfish cycle $D'_j \in \mathcal{D}'$.
    Hence, $\mathcal{D}_{n+1}$ satisfies~\ref{itm:IH:DistanceGoldfish}.

    It remains to verify that $\mathcal{D}_{n+1}$ satisfies \ref{itm:IH:CoralReef}, and to define a set $Y_{n+1}$ such that $\mathcal{D}_{n+1}$ and $Y_{n+1}$ satisfy \ref{itm:IH:corsolaPaths}.

    \begin{claim}
        The collection $\mathcal{D}_{n+1}$ satisfies \ref{itm:IH:CoralReef}; that is, $\mathcal{D}_{n+1}$ is an $(r_1, r_2,d,\kappa,q)$-coral reef.
    \end{claim}

    \begin{claimproof}
        \ref{itm:CoralReef:Distance}: Since $\mathcal{D}_n$ satisfies \ref{itm:CoralReef:Distance}, it remains only to consider pairs involving the newly added cycle $D_{n+1}$.
        Let $v \in V(D_{n+1})$ be at distance less than $q$ from some $D_i \in \mathcal{D}_n$.
        We claim that $v \in B_G(x_{n+1}, \ell/2)$, where $x_{n+1}$ is the `middle' vertex of $W_{C'}$.
        Indeed, since $\bigcup_{i \in [n]}V(D_i) \subseteq O$ and $d_G(P_{C'}, O) \geq Q \geq q$, no vertex of $P_{C'}$ is at distance less than $q$ from $D_i$.
        Hence, $v \in V(D_{n+1}) \setminus V(P_{C'}) \subseteq B_G(x_{n+1}, \ell/2)$.
        Therefore, $\mathcal{D}_{n+1}$ satisfies \ref{itm:CoralReef:Distance}.
        \smallskip

        \ref{itm:CoralReef:degenerate}: Since $\mathcal{D}_n$ satisfies \ref{itm:CoralReef:degenerate}, it remains to show that $D_{n+1}$ is at distance less than $q$ from at most $78$ cycles in $\mathcal{D}_n$.
        Suppose for a contradiction that there are $79$ such cycles, say $D_{i_1}, \ldots, D_{i_{79}}$.
        For every $j \in [79]$, the preceding argument verifying \ref{itm:CoralReef:Distance} ensures that $d_G(D_{i_j}, W_{C'}) < q$.
        Choose $w_j \in V(W_{C'})$ such that $d_G(D_{i_j}, w_j) < q$.

        Since $W_{C'}$ has length at most $\ell = 77d$, there exist distinct $j, j' \in [79]$ such that $d_{W_{C'}}(w_{j}, w_{j'}) \leq \frac{77d}{78}$.
        Combining the $w_j$--$w_{j'}$ subpath of $W_{C'}$ with a shortest $D_{i_j}$--$w_{j}$ path and a shortest $w_{j'}$--$D_{i_{j'}}$ path yields a $D_{i_j}$--$D_{i_{j'}}$ path~$P$ contained in $B_G(W_{C'}, q-1) \subseteq B_G(P_{C'}, \ell/2+q-1)$.

        By the definition of $O$, we have $B_G(X_{\mathcal{D}_n}, \ell) \cup B_G(Y_n, R + \ell/2) \subseteq O$.
        Hence, $d_G(P_{C'}, O) \geq Q \geq q$ implies $d_G\big(P_{C'},\; B_G(X_{\mathcal{D}_n}, \ell/2) \cup B_G(Y_n, R)\big) \geq \ell/2+q$.
        It follows that $P$ is disjoint from $B_G(X_{\mathcal{D}_n}, \ell/2) \cup B_G(Y_n, R)$.
        On the other hand, the length of~$P$ is at most $\frac{77d}{78} + 2q-2$ and thus less than~$d$ by~\eqref{eq:d}.
        This contradicts \ref{itm:IH:corsolaPaths} for $\mathcal{D}_n$.
        Therefore, $D_{n+1}$ is at distance less than $q$ from at most $78$ cycles in $\mathcal{D}_n$, and so $\mathcal{D}_{n+1}$ satisfies \ref{itm:CoralReef:degenerate}.
    \end{claimproof}

    \begin{claim}\label{cl:def-Yn}
        There is a set $Y_{n+1}$ of at most $N(n+1)$ vertices of~$G$ such that $\mathcal{D}_{n+1}$ and~$Y_{n+1}$ satisfy \ref{itm:IH:corsolaPaths}.
    \end{claim}

    \begin{claimproof}
        Recall that by definition, $X_{\mathcal{D}_{n+1}}$ contains the `middle' vertex of each $W_{D_{j}}$ for every $j \leq n+1$. As $||W_{D_j}|| \leq \ell$ since $D_j$ is $(r_1, r_2,d, \kappa)$-corsola, this implies that any $D_i$--$D_j$ path in $G$ that either starts in $W_{D_i}$ or ends in $W_{D_{j}}$ meets $B_G(X_{\mathcal{D}_{n+1}}, \ell/2)$. Hence, it remains to consider $P_{D_i}$--$P_{D_{j}}$ paths for $1\le i<j \le n+1$.
        Since $\mathcal{D}_{n+1}$ satisfies \ref{itm:IH:corsolaPaths2} and $d\ge 2q$, we now obtain our desired set $Y_{n+1}$ of size at most~$N(n+1)$ by applying \cref{lem:untangling}, where we recall that $R = R_{\ref{lem:untangling}}(q,d,\kappa)$ (by \eqref{eq:Untangling}) and that we defined $N(n+1) = 2(n+1) + \lfloor24k\log k\rfloor$ according to \cref{lem:untangling}.
    \end{claimproof}

    This completes step $n+1$ of the construction.
    Recall that if the construction finds $79(k-m)$ corsola cycles, then \cref{claim:GreatBarrierReef} yields a $q$-fat model of $k\cdot K_3$ in $G$, proving \cref{thm:ResultAfterSmooshing}.\smallskip

    It remains to consider the case in which the construction terminates at some step $n+1\leq 79(k-m)$.
    This occurs precisely when the application of \cref{lem:ApexForestWithoutApex} at that step yields an honest $(r_1,r_2)$-radial partial \fd\ of $G$ with support $V(G-B_G(O,u))$, rather than a $q'$-fat model of $K_3$.
    We use this partial forest-decomposition to show that $G$ satisfies outcome \ref{itm:Smooshing:main:GraphDecomp} of \cref{thm:ResultAfterSmooshing}.
    \bigskip

    \noindent\textbf{\large Construction of the graph-decomposition (outcome~\ref{itm:Smooshing:main:GraphDecomp})}
    \smallskip

    As explained above, we may assume that $G$ admits an honest $(r_1, r_2)$-radial partial \fd\ $(T, \mathcal{V})$ with support $V(G-B_G(O, u))$, where $O$ is defined in~\eqref{eq:TheSetO}.
    \medskip

    \noindent\textbf{Proof outline:}
    Roughly speaking, we aim to combine the partial \fd\ $(T, \mathcal{V})$, whose support is $G-B_G(O,u)$, with the partial \tds\ of~$G$ with support $B_G(P_{D_i}, d)$. Such \tds\ exist because the cycles $D_i\in\mathcal{D}_n$ are corsola cycles.
    Together, these supports cover most of $G$, and our goal is to combine these decompositions into a \gd\ of $G$ satisfying \ref{itm:Smooshing:main:GraphDecomp}.
    \smallskip

    To this end, we would like to apply \cref{lem:Smooshing}.
    There are, however, two main obstacles.
    First, the paths $P_{{D_i}}$ are guaranteed to be far apart only in $G-(B_G(X_{\mathcal{D}_n}, \ell/2) \cup B_G(Y_n,R))$, rather than in~$G$ itself.
    Second, the balls around $X_{\mathcal{D}_n}$, $Y_n$, and $Z$ need not be contained in the support of either $(T,\mathcal{V})$ or the partial \tds\ with support $B_G(P_{D_i},d)$.

    To overcome these difficulties, we first delete from $G$ a ball of bounded radius around $X_{\mathcal{D}_n}$, $Y_n$ and $Z$.
    We then delete some additional vertices from each path $P_{D_i}$, chosen so that the remaining subpaths are pairwise far apart.
    By restricting the original decompositions, these subpaths are themselves corsola paths.

    Then, we define $C$ to be the set of vertices that are sufficiently far from the resulting subpaths and from $X_{\mathcal{D}_n}$, $Y_n$ and $Z$.
    The parameters are chosen so that $B_G(C, y)$ is contained in the support of $(T, \mathcal{V})$.
    We may then apply \cref{lem:Smooshing} to obtain a partial graph-decomposition of $G$ whose support contains $C$ and the relevant neighbourhoods of the subpaths, and whose underlying graph has girth greater than~$g$.
    Its radial width and spread depend only on~$q$.

    The vertices not covered by this partial decomposition lie within bounded distance of $X_{\mathcal{D}_n}$, $Y_n$ and $Z$.
    They can therefore be covered by a bounded number of additional bags, each of bounded radius depending only on~$g$ and~$q$.
    Adding these bags extends the partial decomposition to a graph-decomposition of~$G$, with the newly added vertices of the underlying graph forming a hitting set for all cycles of length at most~$g$.
    \medskip

    \noindent\textbf{Definition of the corsola subpaths:}
    Let $\defnm{P} \coloneqq \bigcup_{i \in [n]}P_{D_i}$ and set
    \begin{equation} \label{eq:TheSetOprime}
    \defnm{O'} := B_G(X_{\mathcal{D}_n}, \ell) \cup B_G(Y_n, R+\ell/2) \cup B_G(Z, 2b) \subseteq O.
    \end{equation}
    For each $i \in [n]$, define $\defnm{O'_i} \coloneqq P_{D_i} \cap B_G(O', d)$ and $\defnm{L_i} \coloneqq B_{P_{D_i}}(O'_i, \lambda)$, where $\defnm{\lambda} \coloneqq 2(d-1) + \kappa/2$.
    Let $P'_1, \ldots, P'_p$ be the components of $P - \bigcup_{i \in [n]}L_i$ and set $\defnm{P'} \coloneqq \bigcup_{j \in [p]} P'_j$.
    The paths $P_{D_i}$ are pairwise disjoint by construction, therefore each $P'_j$ is contained in a unique path $P_{D_i}$.
    \smallskip

    \begin{claim}\label{cl:same-cc-far-apart}
        Let $v, v' $ be two vertices of $P_{D_i} \cap P'$ such that the $v$--$v'$ subpath of $P_{D_i}$ is not contained in $P'$.
        Then, $d_G(v, v') \geq d$.
    \end{claim}

    \begin{claimproof}
        Let $P''$ be the $v$--$v'$ subpath of $P_{D_i}$.
        By assumption, $P''$ contains a vertex $s\in L_i$.
        By the definition of $L_i$, there exists $p \in O'_i$ such that $d_{P_{D_i}}(p, s) \leq \lambda$.
        Moreover, $p$ lies on $P''$: otherwise, the $p$--$s$ subpath of $P_{D_i}$ would contain either $v$ or $v'$, and hence a vertex of $P'$ would belong to $B_{P_{D_i}}(p, \lambda) \subseteq L_i$, a contradiction.
        Since $v, v' \notin L_i$ and $B_{P_{D_i}}(p, \lambda) \subseteq L_i$, we obtain $d_{P_{D_i}}(v, v') > 2\lambda = 4(d-1) + \kappa$.
        Property \ref{itm:Corsola:QuasiGeodesic} of corsolas therefore implies that $d_G(v, v') \geq d$.
    \end{claimproof}

    We now show that each $P'_j$ is a subpath of the path $P_{D_i}$ containing it.
    Suppose not.
    Then the vertices of $P'_j$ can be partitioned into two nonempty sets $R$ and $R'$ that are separated along $P_{D_i}$.
    For every $v \in R$ and $v' \in R'$, the $v$--$v'$ subpath of $P_{D_i}$ is not contained in $P'$. Hence, \cref{cl:same-cc-far-apart} gives $d_G(R, R') \geq d > 1$, contradicting that $P'_j$ is connected. (We remark that alternatively, we could have also ensured earlier that each path $P_{D_i}$ is an induced subgraph of~$G$).
    \medskip

    \noindent\textbf{Application of \cref{lem:Smooshing}:}
    Let $\defnm{C} \coloneqq V(G)\setminus (B_G(P',z) \cup B_G(O,u+y))$.
    We will apply \cref{lem:Smooshing} in $G$ to the paths $P'_1, \ldots, P'_p$ and the set $C$.
    For this, we first verify the assumptions of \cref{lem:Smooshing}.
    Recall that we chose $z,z',y,R_1,R_2$ and~$d$ according to \cref{lem:Smooshing} with respect to $r_1,r_2$ and~$g$, i.e.\ $z \geq z' = z'_{\ref{lem:Smooshing}}(r_1, g+1)$, $y = y_{\ref{lem:Smooshing}}(r_1, g+1)$, $R_1 = w_{\ref{lem:Smooshing}}(r_1)$, $R_2 = s_{\ref{lem:Smooshing}}(r_2)$ and $d \geq d_{\ref{lem:Smooshing}}(z)$ (see \eqref{eq:Smooshing}).
    \medskip

    We first show that every $P'_j$ is an $(r_1, r_2, z)$-corsola path in $G$.
    Fix $j \in [p]$ and let $i \in [n]$ such that $P'_j$ is a subpath of $P_{D_i}$.
    Since $P_{D_i}$ is an $(r_1, r_2, d)$-corsola path, property \ref{itm:Corsola:TreeDecomp} yields an $(r_1, r_2)$-radial partial tree-decomposition of $G$ with support $B_G(P_{D_i}, d)$.
    As $P'_j \subseteq P_{D_i}$ and $z \leq d$, we have $B_G(P'_j, z) \subseteq B_G(P_{D_i}, d)$.
    Restricting the tree-decomposition to $B_G(P'_j, z)$ therefore yields an $(r_1, r_2)$-radial partial tree-decomposition of $G$ with support $B_G(P'_j, z)$.
    Thus, $P'_j$ is an $(r_1, r_2, z)$-corsola path in $G$.
    \smallskip

    Next, we show that the paths $P'_1, \ldots, P'_p$ are pairwise at distance at least $d$ in $G$.
    Let $j, j' \in [p]$ be distinct, and let $i, i' \in [n]$ such that $P'_j$ is a subpath of $P_{D_i}$ and $P'_{j'}$ is a subpath of $P_{D_{i'}}$.
    If $i=i'$ then \cref{cl:same-cc-far-apart} gives $d_G(P'_j, P'_{j'}) \geq d$.
    Suppose that $i \neq i'$.
    By \ref{itm:IH:corsolaPaths}, every path of length less than $d$ between $P_{D_i}$ and $P_{D_{i'}}$ intersects $B_G(X_{\mathcal{D}_n}, \ell/2) \cup B_G(Y_n, R) \subseteq O'$.
    Hence, every vertex of such a path lies at distance less than $d$ from $O'$.
    In particular, its endvertex in $P_{D_i}$ belongs to $O'_i \subseteq L_i$, and therefore does not lie in $P'$.
    It follows that $d_G(P_{D_i} \cap P', P_{D_{i'}} \cap P') \geq d$, and hence $d_G(P'_j, P'_{j'}) \geq d$.
    \smallskip

    By definition, $C \cap B_G(P', z) = \emptyset$.
    Moreover, since $C \cap B_G(O, u+y) = \emptyset$, we have $B_G(C, y) \subseteq V(G) \setminus B_G(O, u)$.
    Thus, restricting $(T, \mathcal{V})$ to $B_G(C, y)$ yields an $(r_1, r_2)$-radial partial forest-decomposition of $G$ with support $B_G(C, y)$.
    \smallskip

    All assumptions of \cref{lem:Smooshing} are therefore satisfied.
    Applying the lemma to $P'_1, \ldots, P'_p$ and $C$, we obtain an $(R_1, R_2)$-radial partial graph-decomposition $(H, \mathcal{W})$ of $G$ with support $B_G(P', z) \cup C$, where $H$ has girth at least $g+1$.
    It remains to extend $(H, \mathcal{W})$ to a \gd\ of~$G$.
    \medskip

    \noindent\textbf{Construction of the \gd:}
    Recall that $\lambda = 2(d-1)+\kappa/2$.
    Set
    \[
    \defnm{\rho} \coloneqq u+y+\lambda+d\quad (= u+y+3d+\kappa/2 - 2).
    \]

    We first show that $V(G) \setminus (B_G(P', z) \cup C) \subseteq B_G(O', \rho)$.
    Let $v \in V(G) \setminus (B_G(P', z) \cup C)$.
    Since $v \notin C$, the definition of $C$ gives $v \in B_G(P',z) \cup B_G(O,u+y)$.
    As $v \notin B_G(P', z)$, it follows that $v \in B_G(O, u+y)$.
    Choose $o \in O$ such that $d_G(o, v) \leq u+y$.
    Since $u+y \leq z$ and $v \notin B_G(P', z)$, we have $o \notin P'$.

    If $o \in O'$, then $d_G(v, O') \leq u+y \leq \rho$, as desired.
    We may therefore assume that $o \notin O'$.
    Since $O \setminus O' \subseteq \bigcup_{i \in [n]} V(D_i)$ (by \eqref{eq:TheSetO} and \eqref{eq:TheSetOprime}), there exists $i \in [n]$ such that $o \in V(D_i)$.
    Moreover, $V(D_i) \setminus V(P_{D_i}) \subseteq B_G(x_i, \ell/2) \subseteq O'$, and hence $o \in V(P_{D_i})$.
    Thus, $o \in V(P) \setminus V(P')$ so $o \in L_i$.
    By definition of $L_i$, there exists a vertex $o' \in O'_i$ such that $d_G(o, o') \leq \lambda$.
    Since $O'_i \subseteq B_G(O', d)$, we obtain $d_G(v, O') \leq d_G(v, o) + d_G(o, o') + d_G(o', O') \leq u + y + \lambda + d = \rho$.
    This proves the claimed inclusion.
    \smallskip

    Define
    \begin{alignat*}{3}
        \defnm{W_x} &:= B_G(x,\; \ell &&+\rho+1) &&\quad\text{ for every } x \in X_{\mathcal{D}_n},\\
        \defnm{W_y} &:= B_G(y,\; R + \ell/2 &&+\rho+1) &&\quad\text{ for every } y \in Y_n, \text{ and}\\
        \defnm{W_z} &:= B_G(z,\; 2b &&+\rho+1) &&\quad\text{ for every } z \in Z.
    \end{alignat*}

    By the definition of $O'$, we have $B_G(O', \rho+1) \subseteq \bigcup_{w \in X_{\mathcal{D}_n} \cup Y_n \cup Z} W_w$.
    In particular, the new bags cover every vertex outside $B_G(P', z) \cup C$, together with all its neighbours.

    For every $w \in X_{\mathcal{D}_n} \cup Y_n \cup Z$, introduce a new vertex \defn{$h_w$} with bag $W_w$, and let $\defnm{U} \coloneqq \{h_w : w \in X_{\mathcal{D}_n} \cup Y_n \cup Z\}$.
    Let \defn{$H'$} be obtained from the disjoint union of~$H$ and $U$ by joining each $h_w \in U$ to every distinct vertex $h \in V(H) \cup U$ such that $W_w \cap W_h \neq \emptyset$.

    Then $(H', \mathcal{W})$ is a graph-decomposition of $G$.
    Indeed, the original bags cover the subgraph $G[B_G(P', z) \cup C]$, while the new bags cover the subgraph of $G$ induced by the remaining vertices and their neighbours.
    Moreover, for every $v \in V(G)$, the new nodes whose bags contain $v$ form a clique and are adjacent to every old node whose bags contain $v$.
    Since the copart $H_v$ is connected whenever it is nonempty, it follows that the copart $H'_v$ is connected.
    \smallskip

    We now bound the radial width and spread of $(H', \mathcal{W})$.
    For every $h \in V(H)$, the bag $W_h$ has radius at most $R_1 = w'_{\ref{thm:ResultAfterSmooshing}}(q)$ in $G$.
    Moreover, for every vertex $v \in V(G)$, the copart $H_v$ has radius at most $R_2$ in $H$.
    The new nodes whose bags contain $v$ form a clique and are adjacent to every node of $H_v$.
    It follows that the copart $H'_v$ has radius at most $R_2+1 = s_{\ref{thm:ResultAfterSmooshing}}(q)$ in $H'$.
    This remains true when $H_v$ is empty, since in that case the nodes of $H'_v$ form a clique.

    The bags indexed by $U$ have radius at most \begin{equation}
        \begin{aligned}
            \defnm{w_{\ref{thm:ResultAfterSmooshing}}(q, g)} &:= \max\{\ell, R+\ell/2,\; 2b\} + \rho + 1 \\
            &\phantom{:}= \max\{\ell,R+\ell/2,\; 2b\} + u+y+3d+\kappa/2 - 1,
        \end{aligned}
        \label{eq:DefinitionOfWidth}
    \end{equation}
    which is independent of $k$.

    Finally, $|U| \leq |X_{\mathcal{D}_n}| + |Y_n| + |Z| \leq n + N(n) + m \leq 79k + N(79k) \leq f_{\ref{thm:ResultAfterSmooshing}}(k)$, and $H'-U=H$ has girth at least $g+1$.
    Therefore, $G$ satisfies \ref{itm:Smooshing:main:GraphDecomp}, completing the proof.
\end{proof}

\section{Proof of coarse Erd\H{o}s-P\'{o}sa (Gotta catch 'em all)}
\label{sec:FinalProof}

In this section, we prove our main results, \cref{th:main,main:CoarseErdosPosa}, and their corollaries \cref{maincor:FarApartEPForLongCycles,maincor:FarApartEPForLongInducedCycles}. We begin by establishing the following more general result.

\begin{restatable}{theorem}{coarseEPgeneral} \label{thm:CoarseEP:General}
    There exist a constant $d_{\ref{thm:CoarseEP:General}}$ and a function $f_{\ref{thm:CoarseEP:General}}: \N \to \N$ with the following property. Let $q \in \N$, let $G$ be a graph, let $\mathcal{C}_q$ denote the collection of all $q$-fat cycles in $G$, and let $\mathcal{C}$ be a collection of cycles of $G$ such that $\mathcal{C}_q \subseteq \mathcal{C}$.
    Then, for every $k \in \N$, one of the following holds: \begin{itemize}
        \item $\mathcal{C}$ contains $k$ cycles that are pairwise at distance at least~$q$ in $G$; or
        \item there exists a set $X \subseteq V(G)$ of size at most $f_{\ref{thm:CoarseEP:General}}(k)$ such that every cycle in $\mathcal{C}$ is at distance at most $d_{\ref{thm:CoarseEP:General}} \cdot q$ from $X$.
    \end{itemize}

    Moreover, $f_{\ref{thm:CoarseEP:General}}$ can be chosen so that $f_{\ref{thm:CoarseEP:General}} \in \mathcal{O}(k \log k)$.
\end{restatable}

We first observe that \cref{main:CoarseErdosPosa} follows immediately from \cref{thm:CoarseEP:General}.

\begin{proof}[Proof of \cref{main:CoarseErdosPosa}]
    Apply \cref{thm:CoarseEP:General} with $\mathcal{C} := \mathcal{C}_q$.
\end{proof}

To prove \cref{maincor:FarApartEPForLongCycles} and \cref{maincor:FarApartEPForLongInducedCycles}, we will rely on the following observation.

\begin{lemma}\label{lem:fat-gives-induced}
    Let $q \geq 2$, let $G$ be a graph and let $C$ be a $q$-fat cycle in $G$.
    Then, $C$ contains an induced cycle of length at least $2q$.
\end{lemma}

\begin{proof}
    Let $P_0, \ldots, P_5$ witness that $C$ is $q$-fat in $G$; that is, for distinct non-consecutive $i, j \in \mathbb{Z}_6$, we have $d_G(P_i, P_j) \geq q$.
    Let $C'$ be a shortest cycle contained in $G[V(C)]$ that visits $P_0, P_1, \ldots, P_5$ in this cyclic order.
    Such a cycle exists, since $C$ itself has this property.\smallskip

    We claim that $C'$ is induced.
    Suppose not, and let $xy$ be a chord of $C'$.
    Let $i, j \in \mathbb{Z}_6$ such that $x \in V(P_i)$ and $y \in V(P_j)$.
    Since $xy \in E(G)$, we have $d_G(P_i, P_j) \leq 1 < q$. Hence, either $i=j$ or $i$ and $j$ are consecutive.
    It follows that one of the two $x$--$y$ subpaths of $C'$, say $Q$, still visits $P_0, P_1, \ldots, P_5, P_0$ in their prescribed cyclic order.
    Therefore, $C'' \coloneqq Q \cup \{xy\}$ is a cycle contained in $G[V(C)]$ with the same visiting property.
    As $xy$ is a chord of $C'$, the cycle $C''$ is shorter than $C'$, a contradiction.
    Thus, $C'$ is induced.

    Finally, the two $P_0$--$P_3$ subpaths of $C'$ each have length at least $d_G(P_0, P_3) \geq q$.
    Hence, $C'$ has length at least $2q$.
\end{proof}

We can now deduce \cref{maincor:FarApartEPForLongCycles} and \cref{maincor:FarApartEPForLongInducedCycles} from \cref{thm:CoarseEP:General} in essentially the same way.

\begin{proof}[Proof of \cref{maincor:FarApartEPForLongCycles}]
    Set $q' \coloneqq \max\{\ell, q\}$, and let $\mathcal{C'}$ be the collection of all cycles in $G$ of length at least $\ell$.
    Let $\mathcal{C} \coloneqq \mathcal{C}_{q'} \cup \mathcal{C}'$.
    By \cref{lem:fat-gives-induced}, every cycle in $\mathcal{C}$ contains a cycle of length at least $\ell$.
    Thus, \cref{maincor:FarApartEPForLongCycles} now follows from \cref{thm:CoarseEP:General}.
\end{proof}

\begin{proof}[Proof of \cref{maincor:FarApartEPForLongInducedCycles}]
    Set $q' \coloneqq \max\{\ell, q\}$, and let $\mathcal{C'}$ be the collection of all induced cycles in $G$ of length at least $\ell$.
    Let $\mathcal{C} \coloneqq \mathcal{C}_{q'} \cup \mathcal{C}'$.
    By \cref{lem:fat-gives-induced}, every cycle in $\mathcal{C}$ contains an induced cycle of length at least $\ell$.
    Thus, \cref{maincor:FarApartEPForLongInducedCycles} now follows from \cref{thm:CoarseEP:General}.
\end{proof}

Most of this section is dedicated to proving \cref{thm:CoarseEP:General}.
In \cref{subsec:CoarseEP:Quasi-Isom}, we derive \cref{th:main,main:CoarseEp:Inf} from \cref{main:CoarseErdosPosa}.
\medskip

\noindent \textbf{Proof outline:}
To prove \cref{thm:CoarseEP:General}, we first apply \cref{thm:ResultAfterSmooshing}.
This yields either a $q$-fat model of $k \cdot K_3$ in~$G$, or a \gd\ $(H, \mathcal{V})$ of $G$ with small radial width, together with a small set $U \subseteq V(H)$ that meets every short cycle of~$H$.
In the former case, we are done since $\mathcal{C}_q \subseteq \mathcal{C}$.

We therefore consider the latter case and analyse the decomposition graph $H$.
In \cref{subsec:CyclesInHAreFatCyclesInG}, we show that every cycle of $H$ sufficiently far from $U$ gives rise to a $q$-fat cycle in~$G$. Moreover, this correspondence preserves distances, in the sense that cycles sufficiently far apart in $H$ yield cycles that are far apart in $G$.

We then apply the result of Dujmović, Joret, Micek and Morin \cite{DJMMDistanceErdosPosa} (see \cref{thm:DistanceErdosPosa}) on finding pairwise distant cycles to $H$.
This either finds many cycles in $H$ that yield $q$-fat cycles in~$G$ pairwise at distance at least~$q$, in which case we are again done, or a small set $Y \subseteq V(H)$ such that $H-B_H(Y,d)$ is a forest for some small $d \in \N$.

For each $y \in Y$, choose a vertex $x_y \in V_y$, and let $X \coloneqq \{x_y : y \in Y\}$.
The decomposition $(H, \mathcal{V})$ then induces a forest-decomposition of $G - B_G(X, D)$ modelled on the forest $F \coloneqq H - B_H(Y, d)$, for a suitable bounded radius $D$.

Finally, in \cref{subsec:FindingCyclesInATreeDecomp}, we analyse this forest-decomposition.
We show that either $G - B_G(X, D)$ contains many cycles from $\mathcal{C}$ that are pairwise far apart in~$G$, or there is a small collection of bags that meets every cycle in $\mathcal{C}$ not already within distance $D$ of $X$. Since all bags have bounded radius, choosing one vertex from each of these bags and adding these vertices to~$X$ yields the desired hitting set.

\subsection{Turning cycles of \texorpdfstring{$H$}{H} into fat cycles of \texorpdfstring{$G$}{G}} \label{subsec:CyclesInHAreFatCyclesInG}

Following the proof outline above, we show that every cycle of $H$ sufficiently far from $U$ gives rise to a $q$-fat cycle in $G$.
We begin with an auxiliary lemma that finds a $q$-fat cycle in a bounded neighbourhood of the original cycle in $H$.

\begin{lemma} \label{lem:dumb-rhino}
    Let $q \in \N$, let $U$ be a set of vertices in a graph $H$ that meets every cycle in $H$ of length at most $30q+5$.
    If $H$ contains a cycle $C$ at distance at least $5q/2+1$ from $U$, then $H$ has a $q$-fat cycle that is contained in $B_H(C, 5q/2)$.
\end{lemma}

\begin{proof}
    Since $C$ avoids $U$, it has length at least $30q+6$.
    Divide $C$ into six subpaths $\defnm{C_0}, \dots, \defnm{C_5}$, each of length at least $5q+1$, so that $\bigcup_{i \in \mathbb{Z}_6} C_i = C$ and so that two consecutive subpaths intersect in precisely one vertex.
    If $C_0, \dots, C_5$ witness that $C$ is $q$-fat, then we are done.
    We may therefore assume that there is a path~$W'$ of length less than~$q$ between two non-consecutive paths~$C_i$ and~$C_j$.
    \smallskip

    Consider all walks $Q$ in $H$ of length at most $5q$ with endvertices $\defn{x},\defn{y} \in V(C)$ such that both $x$--$y$ subpaths of $C$ have length at least $5q+1$.
    Such walks exist because $W'$ is one.
    Choose \defn{$Q$} so that the length of the shorter of these two subpaths is as small as possible, and, subject to this, so that $Q$ is as short as possible.
    Let \defn{$P$} denote the shorter $x$--$y$ subpath of $C$, and let \defn{$P'$} denote the other one.

    We first claim that $Q$ has length exactly~$5q$.
    Suppose that $Q$ has length less than~$5q$, and let $v$ be the neighbour of $x$ on $P$.
    The walk $Q' := Q + xv$ has endvertices $v$ and $y$, and length at most $5q$.
    The $v$--$y$ subpath of $C$ contained in $P$ is shorter than $P$.
    Thus, by the choice of $Q$, this subpath must have length less than $5q+1$.
    Since $P$ has length at least $5q+1$, it follows that $P$ has length exactly $5q+1$.
    The union $Q \cup P$ therefore contains a cycle of length at most $||Q||+||P|| \leq 10q$.
    Hence, this cycle contains a vertex of $U$.
    However, $Q \cup P \subseteq B_H(C, 5q/2)$, contradicting the assumption that $C$ is at distance at least $5q/2+1$ from $U$.
    Consequently, $Q$ has length exactly $5q$.
    The choice of $Q$ also implies that $Q$ is a shortest path between $x$ and $y$, since replacing it by a shorter $x$--$y$ path would leave $P$ and $P'$ unchanged.
    \smallskip

    Divide $Q$ into five subpaths $\defnm{Q_1}, \dots, \defnm{Q_5}$, each of length~$q$, so that $\bigcup_{i \in [5]} Q_i = Q$ and so that two consecutive subpaths intersect in precisely one vertex.
    We claim that $Q_1, \dots, Q_5, P$ witness that $Q \cup P$ is a $q$-fat cycle.
    Since $Q$ is geodesic, any two non-consecutive paths among $Q_1, \ldots, Q_5$ are at distance at least $q$.
    Setting $\defnm{\widetilde{Q}} \coloneqq Q_2 \cup Q_3 \cup Q_4$, it therefore remains to show that $d_H(\widetilde{Q}, P) \geq q$.\looseness=-1
    \smallskip

    Suppose for a contradiction that there is a path~\defn{$W$} of length less than~$q$ between $\widetilde{Q}$ and $P$.
    Let $u \in V(\widetilde{Q})$ and $w \in V(P)$ be the endvertices of $W$.

    Suppose first that one of the subpaths $xPw$ and $wPy$, say $xPw$, has length at least~$5q+1$.
    The walk $xQuWw$ has length less than $(||Q||-q)+q = 5q$ because $u \in V(\widetilde{Q})$. Its endvertices are $x$ and $w$, and the two $x$--$w$ paths of $C$ are $xPw$ and the subpath containing $P'$, both of which have length at least $5q+1$.
    Moreover, $xPw$ is shorter than $P$. This contradicts the choice of $Q$.

    We may therefore assume that both $xPw$ and $wPy$ have length at most $5q$.
    It follows that $P$ has length at most $10q$, and hence $Q \cup P$ contains a cycle of length at most $||Q|| + ||P|| \leq 15q$.
    This cycle must contain a vertex of $U$.
    On the other hand, $Q \cup P \subseteq B_H(C, 5q/2)$, contradicting again that $C$ is at distance at least $5q/2+1$ from~$U$.
    Thus, $P$ is at distance at least~$q$ from~$\widetilde{Q}$, and so $Q_1,\dots,Q_5,P$ witness that $Q \cup P$ is a $q$-fat cycle.
    Since this cycle is contained in $Q\cup C\subseteq B_H(C,5q/2)$, the result follows.
\end{proof}

The previous lemma finds a fat cycle in $H$ near any cycle sufficiently far from $U$.
We now show that, given an honest $H$-decomposition of a graph $G$, this fat cycle can be lifted to a fat cycle in $G$ lying close to the bags indexed by the original cycle.

\begin{lemma} \label{lem:FindFatCycleInG}
    There exist functions $g_{\ref{lem:FindFatCycleInG}}, d_{\ref{lem:FindFatCycleInG}}, R_{\ref{lem:FindFatCycleInG}} : \N^3 \to \N$ such that the following holds.
    Let $q,r_1,r_2 \in \N$, let $U$ be a set of vertices in a graph $H$ that meets every cycle in $H$ of length at most $g_{\ref{lem:FindFatCycleInG}}(q, r_1, r_2)$.
    Let $(H, \mathcal{V})$ be an honest graph-decomposition of a graph~$G$ of radial spread at most~$r_2$ such that $\text{rad}_G(V_h) \leq r_1$ for every $h \in V(H)\setminus U$.
    If $H$ contains a cycle $C$ at distance at least $d_{\ref{lem:FindFatCycleInG}}(q, r_1, r_2)$ from $U$, then $G$ has a $q$-fat cycle that is contained in $B_G(\bigcup_{h \in V(C)} V_h, R_{\ref{lem:FindFatCycleInG}}(q, r_1, r_2))$.
\end{lemma}

\begin{proof}
    Set \begin{align*}
        q' &:= f_{\ref{lem:far-apart-lifts-far-apart}}(q+2r_1-1, r_2) + 1, \\
        g &:= g_{\ref{lem:FindFatCycleInG}}(q, r_1, r_2) := 30q'+5, \\
        d &:= d_{\ref{lem:FindFatCycleInG}}(q, r_1, r_2) := \lceil5q'/2\rceil+1, \text{ and } \\
        R &:= R_{\ref{lem:FindFatCycleInG}}(q, r_1, r_2) := (5q'+1) r_1.
    \end{align*}
    Applying \cref{lem:dumb-rhino} to $C$ in $H$ with parameter~$q'$ yields a $q'$-fat cycle $\defnm{C'} \subseteq B_H(C, 5q'/2)$ in $H$.
    Since $d_H(C, U) \geq d > 5q'/2$, the cycle $C'$ is disjoint from $U$.
    Consequently, $\text{rad}_G(V_h) \leq r_1$ for every $h \in V(C')$.
    Hence, for each $h \in V(C')$, we may choose a connected set $\defnm{V'_h} \subseteq V(G)$ such that $V_h \subseteq V'_h \subseteq B_G(V_h, r_1)$.
    Let $\defnm{C'_0}, \dots, \defnm{C'_5}$ be subpaths of $C'$ witnessing that $C'$ is $q'$-fat in $H$.
    \smallskip

    We first construct a cycle in $G$.
    For each $i \in \mathbb{Z}_6$, let $\defnm{G_i} \coloneqq G[\bigcup_{h \in V(C'_i)} V'_h]$.
    Since $(H, \mathcal{V})$ is honest, the bags corresponding to consecutive vertices of $C'_i$ intersect.
    Since every $V'_h$ is connected, it follows that each $G_i$ is connected.
    Moreover, $G_i$ intersects both $G_{i-1}$ and $G_{i+1}$.

    For each $i \in \{1,3,5\}$, choose a shortest $G_{i-1}$--$G_{i+1}$ path \defn{$P_i$} in $G_i$.
    Thus, the internal vertices of $P_i$ lie outside of $G_{i-1} \cup G_{i+1}$.
    For each $i \in \{0,2,4\}$, choose a path~\defn{$P_i$} in $G_i$ joining the endvertex of $P_{i-1}$ in $G_i$ to the endvertex of $P_{i+1}$ in $G_i$.

    We claim that $\defnm{C''} := \bigcup_{i \in \mathbb{Z}_6} P_i$ is a $q$-fat cycle in $G$ contained in $B_G(\bigcup_{h \in V(C)}V_h, R)$.\smallskip

    We first verify the containment.
    By construction,
    \[
    V(C'') \subseteq \bigcup_{i\in \mathbb{Z}_6} V(G_i) \subseteq \bigcup_{h\in V(C')}V'_h \subseteq \bigcup_{h \in V(C')} B_G\left(V_h, r_1\right).
    \]
    It therefore suffices to show that, for every $o \in V(C')$, $V_o \subseteq B_G(\bigcup_{h \in V(C)}V_h, R-r_1)$.

    Fix $o \in V(C')$ and $v \in V_o$.
    Since $C' \subseteq B_H(C, 5q'/2)$, there is an $o$--$C$ path $W$ in $H$ of length $s \leq 5q'/2$. Write $W \eqqcolon w_0\dots w_s$ where $w_0 = o$ and $w_s \in V(C)$.
    The path $W$ is disjoint from~$U$, since otherwise $d_H(C, U) \leq s \leq 5q'/2 < d$.
    Hence, every bag indexed by a vertex of $W$ has radius at most~$r_1$ in~$G$.

    For each $i \in \{0, \ldots, s-1\}$, choose $a_i \in V_{w_i} \cap V_{w_{i+1}}$, which is possible because the decomposition is honest.
    Both $v$ and $a_0$ lie in $V_{w_0}$, and, for every $i\in\{1,\dots,s-1\}$, both $a_{i-1}$ and $a_i$ lie in $V_{w_i}$.
    Thus, consecutive vertices in the sequence $v, a_0, \ldots, a_{s-1}$ are at distance at most~$2r_1$ in~$G$.
    Since $a_{s-1}\in V_{w_s}$ and $w_s\in V(C)$, it follows that $d_G(v, \bigcup_{h\in V(C)}V_h) \leq 2r_1s \leq 5q'r_1 = R-r_1$.
    This proves the desired containment.
    \smallskip

    It remains to show that $C''$ is a $q$-fat cycle.
    For each $i \in \mathbb{Z}_6$, set $A_i \coloneqq \bigcup_{h\in V(C'_i)}V_h$.
    Since $P_i \subseteq G_i$, we have $V(P_i)\subseteq B_G(A_i,r_1)$.

    Let $i, j \in \mathbb{Z}_6$ be distinct and non-consecutive.
    Since $C'_0,\dots,C'_5$ witness that $C'$ is $q'$-fat, we have $d_H(C'_i,C'_j)\geq q'>f_{\ref{lem:far-apart-lifts-far-apart}}(q+2r_1-1,r_2)$.
    By \cref{lem:far-apart-lifts-far-apart}, $d_G(A_i, A_j) \geq q+2r_1$.
    Consequently, $d_G(P_i, P_j) \geq d_G(A_i,A_j)-2r_1 \geq q$.

    Thus, any two non-consecutive paths among $P_0,\dots,P_5$ are at distance at least $q$, and in particular are disjoint.
    By construction, consecutive paths intersect in precisely their common endvertex.
    Hence, $C''$ is a cycle, and the paths $P_0,\dots,P_5$ witness that it is $q$-fat. Together with the containment established above, this completes the proof.
\end{proof}

\subsection{Finding/Hitting cycles in a tree-decomposition} \label{subsec:FindingCyclesInATreeDecomp}

We now consider the case in which we obtain a partial forest-decomposition of $G$ and analyse the cycles contained in its support.
The next lemma shows that we can either find many such cycles that are pairwise far apart or cover them all by a small number of balls of bounded radius.

\begin{lemma} \label{lem:HittingCyclesInATreeDecomp}
    There is a function $d_{\ref{lem:HittingCyclesInATreeDecomp}}: \N^2 \to \N$ such that the following holds.
    Let $k, q, r_1 \in \N$, and let $G$ be a graph that admits a partial \fd\ of radial width at most $r_1 \in \N$, finite radial spread, and support $Y \subseteq V(G)$.
    Let $\mathcal{C}$ be a collection of finite connected subgraphs of $G[Y \setminus B_G(G - Y,\; \lceil q/2\rceil)]$.
    Then either $\mathcal{C}$ contains $k$ members that are pairwise at distance at least~$q$ in~$G$, or there exists a set $X \subseteq V(G)$ of size less than $k$ such that every member of $\mathcal{C}$ is at distance at most $d_{\ref{lem:HittingCyclesInATreeDecomp}}(q,r_1)$ from~$X$ in $G$.
\end{lemma}

For the proof of \cref{lem:HittingCyclesInATreeDecomp} we need the Erd\H{o}s-P\'{o}sa property for subtrees of a tree.
For finite graphs this was proven by Gy{\'a}rf{\'a}s and Lehel \cite{gyarfas1970helly}.
We require a version for infinite graphs, which we prove in a similar way.
However, an additional condition that the subtrees have finite radius is required.
This is necessary as shown for example by considering a 1-way infinite path and a collection of subtrees consisting of 1-way infinite subpaths.

\begin{lemma}\label{lem:helly}
    Let $k \in \mathbb{N}$, let $F$ be a forest and let $\mathcal{A}$ be a family of finite radius, nonempty connected subgraphs of $F$.
    Then either there exists a set $X \subseteq V(F)$ of size less than $k$ that intersects every member of~$\mathcal{A}$, or $\mathcal{A}$ contains $k$ pairwise disjoint members.
\end{lemma}

\begin{proof}
    Clearly it is enough to prove the lemma in the case that $F$ is connected. The lemma clearly holds for $k=1$.
    So we shall proceed inductively assuming that $k\ge 2$.

    Pick some vertex~$r$ of $F$ arbitrarily.
    For each $T\in \mathcal{A}$, let $c(T):=\min\{ d_F(r,v):v\in V(T)\}$ and let $f(T):= \max\{ d_F(r,v):v\in V(T)\}$.
    Note that $c(T) \leq f(T) < \infty$ since $T$ has finite radius.
    If $\sup_{T\in \mathcal{A} } c(T) = \infty$, then we can  choose some $T_1,\ldots , T_k \in \mathcal{A}$ with $c(T_1)\le f(T_1) < c(T_2) \le f(T_2)< \ldots <c(T_k)\le f(T_k)$. Then $T_1,\ldots , T_k$ are pairwise disjoint.
    So we may now assume that $\sup_{T\in \mathcal{A} } c(T) < \infty$.

    Choose some $T_1\in \mathcal{A}$ with $c(T_1)=\sup_{T\in \mathcal{A} } c(T)$, and let $x$ be a vertex of $T_1$ with $d_F(r,x) = c(T_1) = \sup_{T\in \mathcal{A} } c(T)$.
    Let $F^*$ be the subtree of $F$ consisting of all vertices $u$ so that the path between $u$ and $r$ contains $x$.
    Then, $T_1$ is a subtree of $F^*$ by minimality of $d_F(r, x)$.
    Moreover, $x$ is the unique vertex of $F^*$ that has a neighbour in $F-F^*$ and the unique vertex of $F^*$ that is closest to $r$.
    As $\sup_{T \in \mathcal{A}} c(T) = d_F(r,x)$, it follows that every $T'\in \mathcal{A}$ that intersects $F^*$ contains $x$.

    Let $\mathcal{A}'$ be the set of trees in $\mathcal{A}$ that do not contain $x$ and that are thus subtrees of $F-F^*$.
    By the inductive hypothesis, either $F-F^*$ contains $k-1$ disjoint subtrees $T_2, \ldots , T_k$, or there is a set $X$ of size less than $k-1$ that intersects every member of $\mathcal{A}'$.
    In the first case, $T_1,\ldots , T_k$ are $k$ pairwise disjoint members of~$\mathcal{A}$, and in the second case $X\cup \{x\}$ is a set of size less than $k$ that intersects every member of~$\mathcal{A}$.
\end{proof}

We can now prove \cref{lem:HittingCyclesInATreeDecomp}.

\begin{proof}[Proof of \cref{lem:HittingCyclesInATreeDecomp}]
    We first define the function $d_{\ref{lem:HittingCyclesInATreeDecomp}}$. Let
    \[
        d \coloneqq \defnm{d_{\ref{lem:HittingCyclesInATreeDecomp}}}(q, r_1) \coloneqq \lceil q/2\rceil + r_1.
    \]

    Let $(T, \mathcal{V})$ be a partial \fd\ of $G$ of radial width at most $r_1$ and support~$Y$.
    For each $C \in \mathcal{C}$, let $\defnm{T_{C,q/2}} \coloneqq T[\{t \in V(T) : V_t \cap B_G(C, \lceil q/2\rceil) \neq \emptyset\}]$.
    Since $V(C) \subseteq Y \setminus B_G(G-Y,\; \lceil q/2 \rceil)$, we have $B_G(C, \lceil q/2 \rceil) \subseteq Y$.
    Since $C$ is connected, $B_G(C,\lceil q/2\rceil)$ is connected as well.
    It follows from \ref{itm:H2'} that~$T_{C,q/2}$ is a nonempty connected subgraph of~$T$.
    Moreover, since $C$ is finite and the decomposition has finite radial spread, $T_{C,q/2}$ has finite radius in $T$.

    Set $\defnm{\mathcal{A}} \coloneqq \{T_{C,q/2} : C \in \mathcal{C}\}$.
    By \cref{lem:helly}, either there exists a set~$Z \subseteq V(T)$ of size less than~$k$ that intersects every member of $\mathcal{A}$, or $\mathcal{A}$ contains $k$ pairwise disjoint members.
    \medskip

    Suppose first that there exists such a set $Z$.
    For every $z \in Z$, choose a vertex $x_z \in V(G)$ such that $V_z \subseteq B_G(x_z, r_1)$.
    Set $X \coloneqq \{x_z : z \in Z\}$, and note that $|X| < k$.

    Let $C \in \mathcal{C}$.
    Since $T_{C,q/2} \in \mathcal{A}$, there exists $z \in Z \cap V(T_{C,q/2})$.
    By the definition of $T_{C,q/2}$, we have $d_G(C, V_z) \leq \lceil q/2\rceil$.
    Consequently, $d_G(C, x_z) \leq d_G(C, V_z) + r_1 \leq \lceil q/2\rceil + r_1 = d$.
    Thus, every subgraph in $\mathcal{C}$ is at distance at most $d$ from $X$ in $G$.
    \medskip

    Suppose now that $\mathcal{A}$ contains $k$ pairwise disjoint members.
    Choose $C_1, \ldots, C_k \in \mathcal{C}$ such that the subgraphs $T_{C_1,q/2}, \ldots, T_{C_k,q/2}$ of~$T$ are pairwise disjoint.
    It follows from the definition of the $T_{C_i,q/2}$'s that the sets $B_G(C_1, \lceil q/2\rceil), \ldots, B_G(C_k, \lceil q/2\rceil)$ are pairwise disjoint.
    Therefore, the subgraphs $C_1, \ldots, C_k$ are pairwise at distance at least~$q$ in~$G$.
\end{proof}

\subsection{Proof of \texorpdfstring{\cref{thm:CoarseEP:General}}{Theorem 9.1}}

To complete the proof of \cref{thm:CoarseEP:General}, we use a result of Dujmović, Joret, Micek, and Morin \cite{DJMMDistanceErdosPosa} establishing that cycles that are pairwise far apart have the \EP\ property.
To get better bounds, we use the following strengthening, which follows from \cite{CDGKMMS} by setting $\ell := 3$. \footnote{Actually, they only prove \cref{thm:DistanceErdosPosa} for finite graphs \cite{CDGKMMS}. However, it only takes some minor modifications to extend the proof to infinite graphs. We discuss this in \cref{appendix}.}

\begin{theorem}[{\cite[Theorem~2]{CDGKMMS}}] \label{thm:DistanceErdosPosa}
    There exist a constant $c_{\ref{thm:DistanceErdosPosa}}$ and a function $f_{\ref{thm:DistanceErdosPosa}} : \N \to  \N$ with $f_{\ref{thm:DistanceErdosPosa}} \in \mathcal{O}(k\cdot\log(k))$ such that for all integers $k \geq 1$ and $d \geq 1$, and for every graph~$G$, either $G$ contains $k$ cycles that are pairwise at distance greater than~$d$ in~$G$, or there exists a subset~$X$ of vertices of $G$ with $|X| \leq f_{\ref{thm:DistanceErdosPosa}}(k)$ such that $G - B_G(X, c_{\ref{thm:DistanceErdosPosa}}\cdot d)$ is a forest.
\end{theorem}

To maintain the linear bound for the distance of the cycles from the hitting set (in case where we do not find many cycles), we will also use the following simple result of Davies, Hickingbotham, Illingworth, and McCarty, which describes how taking powers affects fat minors and distances.
For a graph $G$ and a positive integer $q$, the \defn{$q$-th power of $G$} is the graph \defn{$G^q$} on vertex set $V(G)$ where two vertices are adjacent if and only if they are at distance at most $q$ in $G$.

\begin{theorem}[{\cite[Theorem~3]{DHIMFatCounterexample}}] \label{thm:QIto3fatMinorfreegraph}
    Let $q$ be a positive integer, $X$ a graph, and $G$ a graph with no $q$-fat model of $X$. Then the graph $G^q$ contains no 3-fat model of $X$ and for every $u,v\in V(G)$, we have
    \[
    d_{G^q}(u,v)
    \le
    d_G(u,v)
    \le
    q \cdot d_{G^q}(u,v).
    \]
    In particular, $G$ is $q$-quasi-isometric to $G^q$.
\end{theorem}

The proof of \cite[Theorem~3]{DHIMFatCounterexample} in fact establishes the following statement.

\begin{lemma}[\cite{DHIMFatCounterexample}] \label{lem:power-fat-minor-close}
    Let $q$ be a positive integer, let $G$ and $X$ be graphs, and let $M$ be a $3$-fat model of $X$ in $G^q$.
    Then, $G$ has a $q$-fat model of $X$ contained in $B_G(M, \lceil q/2\rceil)$.
\end{lemma}

We now prove \cref{thm:CoarseEP:General}, which we restate for convenience.
\coarseEPgeneral*

\begin{proof}[Proof of \cref{thm:CoarseEP:General}]
    We first define the constant $d_{\ref{thm:CoarseEP:General}}$ and the function $f_{\ref{thm:CoarseEP:General}}$.
    For this, we need to introduce several parameters that will appear in the course of the proof.
    Let
    \begin{align*}
        r_2 &\coloneqq s_{\ref{thm:ResultAfterSmooshing}}(3), \\
        r'_1 &\coloneqq w'_{\ref{thm:ResultAfterSmooshing}}(3), \\
        R &\coloneqq R_{\ref{lem:FindFatCycleInG}}(3, r'_1, r_2), \\
        g &\coloneqq g_{\ref{lem:FindFatCycleInG}}(3, r'_1, r_2), \\
        r_1 &\coloneqq w_{\ref{thm:ResultAfterSmooshing}}(3, g), \\
        D &\coloneqq \max\{2d_{\ref{lem:FindFatCycleInG}}(3, r'_1, r_2), f_{\ref{lem:far-apart-lifts-far-apart}}(2R+3, r_2)\}, \\
        d &\coloneqq \defnm{d_{\ref{thm:CoarseEP:General}}} \coloneqq \max\{r_1 \cdot (2c_{\ref{thm:DistanceErdosPosa}} \cdot D + 3) + 2, d_{\ref{lem:HittingCyclesInATreeDecomp}}(3,r'_1)\}, \\
        \defnm{f_{\ref{thm:CoarseEP:General}}(k)} &\coloneqq f_{\ref{thm:DistanceErdosPosa}}(k+1) + f_{\ref{thm:ResultAfterSmooshing}}(k) + k.
    \end{align*}

    By \cref{thm:ResultAfterSmooshing}, we have $f_{\ref{thm:ResultAfterSmooshing}}(k) \in \mathcal{O}(k \log k)$, and by \cref{thm:DistanceErdosPosa}, we have $f_{\ref{thm:DistanceErdosPosa}}(k) \in \mathcal{O}(k\cdot \log(k))$.
    Consequently, $f_{\ref{thm:CoarseEP:General}}(k) = \mathcal{O}(k \log k)$.
    \smallskip

    We first prove the statement for $q = 3$ and then extend it to all $q \geq 4$.
    The cases $q = 1,2$ follow immediately from the case $q=3$ (because every $3$-fat cycle is also $1$-fat and $2$-fat, and therefore $\mathcal{C}_3 \subseteq \mathcal{C}_1 \subseteq \mathcal{C}$ or $\mathcal{C}_3 \subseteq \mathcal{C}_2 \subseteq \mathcal{C}$, respectively). We only need to increase $d_{\ref{thm:CoarseEP:General}}$ by a factor of~$3$ to accommodate for $d_{\ref{thm:CoarseEP:General}}\cdot 1 \lneq d_{\ref{thm:CoarseEP:General}} \cdot 3$ and  $d_{\ref{thm:CoarseEP:General}}\cdot 2 \lneq d_{\ref{thm:CoarseEP:General}} \cdot 3$.

    We remark that we could use the same proof as for the case of $q=3$ to prove the cases $q=3$ and $q \geq 4$ at once; however, rather than a \emph{constant}~$d_{\ref{thm:CoarseEP:General}}$, we would only obtain a \emph{function}~$d$ (depending on~$q$). Therefore, we rely on this case distinction.
    \medskip

    \textbf{The case $\mathbf{q = 3}$:}
    Since $\mathcal{C}_3 \subseteq \mathcal{C}$, we are done if $G$ contains a $3$-fat model of $k\cdot K_3$.
    We may therefore assume that no such model exists.
    Applying \cref{thm:ResultAfterSmooshing} to $G$ with parameters $k$, $3$ and $g$, we obtain a graph \defn{$H$} and a set \defn{$U$} of at most $f_{\ref{thm:ResultAfterSmooshing}}(k)$ vertices of $H$ such that $U$ hits all cycles in $H$ of length at most $g$, and $G$ admits an honest $(r_1, r_2)$-radial graph-decomposition $(H, \defnm{\mathcal{V}})$ in which $\rad_{G}(V_h) \leq r'_1$ for every $h \in V(H) \setminus U$.

    Let \defn{$H'$} be obtained from $H$ by identifying all vertices of $U$ into a single vertex~\defn{$u$} and deleting all loops and parallel edges.
    Applying \cref{thm:DistanceErdosPosa} to $H'$ with parameters $k+1$ and $D$, we obtain one of the following: \begin{itemize}
        \item $H'$ contains $k+1$ cycles that are pairwise at distance greater than $D$; or
        \item there exists a set $Y' \subseteq V(H')$ of size at most $f_{\ref{thm:DistanceErdosPosa}}(k+1)$ such that $H' - B_{H'}(Y', c_{\ref{thm:DistanceErdosPosa}}\cdot D)$ is a forest.
    \end{itemize}
    \medskip

    Suppose first that $H'$ contains $k+1$ cycles that are pairwise at distance greater than $D$.
    At most one of these cycles can be at distance at most $D/2$ from $u$.
    Hence, at least $k$ of them, say $C_1, \ldots, C_k$, are at distance greater than~$D/2$ from $u$.
    In particular, none of these cycles contains $u$, so they are also cycles in~$H$, and $d_H(C_i, U) \geq D/2 \geq d_{\ref{lem:FindFatCycleInG}}(3,r'_1,r_2)$ for every $i \in [k]$.

    For every $i \in [k]$, applying \cref{lem:FindFatCycleInG} to~$H$, $U$, and $C_i$ with parameters $3, r'_1, r_2$, we obtain a $3$-fat cycle $C'_i$ in $G$ contained in $B_G(\bigcup_{h \in V(C_i)}V_h, R)$.
    Moreover, for distinct $i, j \in [k]$, we have $d_H(C_i, C_j) \geq D \geq f_{\ref{lem:far-apart-lifts-far-apart}}(2R+3, r_2)$.
    Thus, by \cref{lem:far-apart-lifts-far-apart}, $d_G(\bigcup_{h \in V(C_i)}V_h, \bigcup_{h \in V(C_j)}V_h) \geq 2R+3$.
    It follows that $d_G(C'_i, C'_j) \geq 3$.
    Hence, $C'_1, \ldots, C'_k$ form a $3$-fat model of $k \cdot K_3$ in $G$, contradicting our assumption.
    \medskip

    Suppose now that there exists a set $\defnm{Y'} \subseteq V(H')$ of size at most $f_{\ref{thm:DistanceErdosPosa}}(k+1)$ such that $H' - B_{H'}(Y', c_{\ref{thm:DistanceErdosPosa}} \cdot D)$ is a forest.
    Set $\defnm{Y} := (Y' \setminus \{u\}) \cup U$ and $\defnm{F} \coloneqq H - B_H(Y, c_{\ref{thm:DistanceErdosPosa}}\cdot D)$.
    Then, $F$ is a forest.
    Let $\defnm{S} := \{v \in V(G) : V(H_v) \subseteq V(F)\}$ and for every $f \in V(F)$, let $\defnm{W_f} := V_f \cap S$.
    Since $V(F) \cap U = \emptyset$, it follows that $(F, \mathcal{W})$ is an $(r'_1, r_2)$-radial partial forest-decomposition of $G$ with support $S$.

    Let $\mathcal{C}'$ be the collection of all cycles in $\mathcal{C}$ contained in $G[S \setminus B_G(G - S, 2)]$.
    Applying \cref{lem:HittingCyclesInATreeDecomp} to $G$ and $\mathcal{C}'$ with parameters $k$, $3$ and $r'_1$ yields either $k$ cycles in $\mathcal{C}'$ that are pairwise at distance at least~$3$ in~$G$, or a set $Z \subseteq V(G)$ of size less than $k$ such that every cycle in $\mathcal{C}'$ is at distance at most~$d_{\ref{lem:HittingCyclesInATreeDecomp}}(3,r'_1)$ from $Z$ in $G$.
    In the former case, we are done, so we may assume the latter.
    \smallskip

    For every $y \in Y$, choose $x_y \in V(G)$ such that $V_y \subseteq B_G(x_y, r_1)$, and set $X := \{x_y :  y \in Y\} \cup Z$. Then \[|X| \leq |Y| + |Z| \leq |Y'| + |U| + |Z| \leq f_{\ref{thm:DistanceErdosPosa}}(k+1) + f_{\ref{thm:ResultAfterSmooshing}}(k) + k = f_{\ref{thm:CoarseEP:General}}(k).\]

    We claim that every cycle in $\mathcal{C}$ is at distance at most~$d$ from $X$.
    Let $C \in \mathcal{C}$.
    If\\ $C \subseteq G[S \setminus B_G(G-S,2)]$, then $C \in \mathcal{C}'$, and therefore \[d_G(C,X) \leq d_G(C, Z) \leq d_{\ref{lem:HittingCyclesInATreeDecomp}}(3,r'_1) \leq d.\]

    Otherwise, $C$ is at distance at most~$2$ from $G-S$.
    Choose $v \in G-S$ such that $d_G(C, v) \leq 2$.
    Since $v \notin S$, choose $h \in V(H) \setminus V(F)$ with $v \in V_h$.
    Since $h \notin V(F)$, we have $d_H(h, Y) \leq c_{\ref{thm:DistanceErdosPosa}} \cdot D$.
    Choose $y \in Y$ such that $d_H(h, y) \leq c_{\ref{thm:DistanceErdosPosa}} \cdot D$.
    By \cref{lem:dist-H-decomp} applied to $\{v\}$ and $V_y$, we have $d_G(v, V_y) \leq 2r_1 \cdot (c_{\ref{thm:DistanceErdosPosa}} \cdot D + 1)$.
    Since $V_y \subseteq B_G(x_y, r_1)$, it follows that $d_G(v, x_y) \leq r_1 \cdot (2c_{\ref{thm:DistanceErdosPosa}} \cdot D + 3)$.
    Consequently, \[d_G(C, X) \leq d_G(C, x_y) \leq d_G(C, v) + d_G(v, x_y) \leq r_1 \cdot (2c_{\ref{thm:DistanceErdosPosa}} \cdot D + 3) + 2 \leq d.\]
    This completes the proof for $q = 3$.
    \medskip

    \noindent \textbf{The case $\mathbf{q \geq 4}$:}
    Recall that $G^{q}$ is obtained from $G$ by joining every pair of distinct vertices at distance at most~$q$ in $G$.
    Let $\mathcal{C}_3^q$ be the collection of all $3$-fat cycles in $G^{q}$, and set $\mathcal{C'} \coloneqq \mathcal{C}_3^q \cup \mathcal{C}$.
    Applying the case $q=3$ to $G^q$ and $\mathcal{C}'$ with parameter $k$, we obtain either $k$ cycles $C_1, \ldots, C_k \in \mathcal{C}'$ that are pairwise at distance at least~$3$ in $G^{q}$, or a set $X \subseteq V(G)$ of size at most $f_{\ref{thm:CoarseEP:General}}(k)$ such that every cycle in $\mathcal{C}'$ is at distance at most~$d_{\ref{thm:CoarseEP:General}}$ from $X$ in~$G^{q}$.
    \smallskip

    In the latter case, since $\mathcal{C} \subseteq \mathcal{C'}$, by \cref{thm:QIto3fatMinorfreegraph} every $C \in \mathcal{C}$ satisfies $d_G(C, X) \leq q \cdot d_{G^{q}}(C,X) \leq q \cdot d_{\ref{thm:CoarseEP:General}}$, as desired.
    We may therefore assume the former case.
    \smallskip

    Let $I \coloneqq \{i \in [k] : C_i \notin \mathcal{C}\}$ and $k' := |I|$.
    By the definition of $\mathcal{C}'$, the cycles $C_i$ with $i \in I$ are $3$-fat cycles in $G^q$.
    Since they are pairwise at distance at least $3$ in $G^q$, they form a $3$-fat model of $k'\cdot K_3$ in $G^{q}$.
    By \cref{lem:power-fat-minor-close}, $G$ has a $q$-fat model of $k' \cdot K_3$ contained in $\bigcup_{i \in I} B_G(C_i,\lceil q/2\rceil)$.
    Let $D_1,\dots,D_{k'}$ be the $q$-fat cycles of this model.
    Since $\mathcal{C}_q \subseteq \mathcal{C}$, each $D_j$ belongs to $\mathcal{C}$.

    The cycles $C_1, \ldots, C_k$ are pairwise at distance at least~$3$ in $G^q$, and hence pairwise at distance greater than~$2q$ in~$G$.
    Therefore, for every $i \notin I$ and every $j \in [k']$, we have $d_G(C_i, D_{j}) > 2q - \lceil q/2\rceil \geq q$.
    The cycles $C_i$ with $i \notin I$ are pairwise at distance greater than $2q$, while the cycles $D_1, \ldots, D_{k'}$ are pairwise at distance at least $q$ because they form a $q$-fat model of $k'\cdot K_3$ in $G$.
    Consequently, $\{C_i : i \notin I\} \cup \{D_1, \ldots, D_{k'}\}$ is a collection of $k$ cycles in $\mathcal{C}$ that are pairwise at distance at least $q$ in $G$.
    This completes the proof.
\end{proof}

\subsection{Quasi-isometries} \label{subsec:CoarseEP:Quasi-Isom}

We finally prove \cref{th:main}, which we restate here for convenience.

\main*

\begin{proof}
    We first define the constant $\lambda_{\ref{th:main}}$ and the function $f_{\ref{th:main}}$. For this, we need to introduce several parameters that will appear in the course of the proof. Let \begin{align*}
        f &\coloneqq \defnm{f_{\ref{th:main}}} \coloneqq f_{\ref{main:CoarseErdosPosa}}, \\
        d &\coloneqq 3 \cdot d_{\ref{main:CoarseErdosPosa}}, \\
        r &\coloneqq r_{\ref{lem:ApexForestWithoutApex}}(3, d), \\
        r_1 &\coloneqq w_{\ref{lem:ApexForestWithoutApex}}(3), \\
        r_2 &\coloneqq s_{\ref{lem:ApexForestWithoutApex}}(3), \\
        r'_1 &\coloneqq \max\{r+1, r_1\} \text{, and } \\
        \lambda &\coloneqq \defnm{\lambda_{\ref{th:main}}} \coloneqq \lambda_{\ref{prop:q.i.-graph-dec}}(r'_1, r_2+1).
    \end{align*}

    We first prove the statement for $q = 3$ and then extend it to all $q \geq 4$.
    The cases $q = 1,2$ follow immediately from the case $q=3$ (because every $3$-fat model of $k\cdot K_3$ is also $1$-fat and $2$-fat). We only need to increase $\lambda_{\ref{th:main}}$ by a factor of $3$ to accommodate for $\lambda_{\ref{th:main}}\cdot 1\lneq \lambda_{\ref{th:main}}\cdot 3$ and $\lambda_{\ref{th:main}}\cdot 2 \lneq \lambda_{\ref{th:main}}\cdot 3$.

    We remark that we could also use the same proof as for the case $q=3$ to prove the cases $q=3$ and $q\geq 4$ at once; however, rather than a \emph{constant}~$\lambda_{\ref{th:main}}$, we would only obtain a \emph{function}~$\lambda$ (depending on~$q$). Therefore, we rely on this case distinction.
    \medskip

    \noindent \textbf{The case $\mathbf{q=3}$:}
    By \cref{main:CoarseErdosPosa}, there is a set $X \subseteq V(G)$ of at most $f(k)$ vertices such that every $3$-fat model of $K_3$ in $G$ has distance at most $3 \cdot d_{\ref{main:CoarseErdosPosa}} = d$ from $X$.
    Applying \cref{lem:ApexForestWithoutApex} to $G$ and $X$ with parameters $3$ and $d$, we obtain an honest $(r_1, r_2)$-radial partial forest-decomposition $(F, \mathcal{V})$ of $G$ with support $V(G - B_G(X, r))$.

    For every $x \in X$, introduce a new vertex $h_x$ with bag $V_{h_x} \coloneqq B_G(x, r+1)$.
    Let $X' \coloneqq \{h_x : x \in X\}$.
    Let $H$ be obtained from the disjoint union of $F$ and $X'$ by joining each $h_x \in X'$ to every distinct vertex $h \in V(F) \cup X'$ such that $V_{h_x} \cap V_h \neq \emptyset$.
    Then, $(H, \mathcal{V})$ is an honest graph-decomposition of $G$ with radial width at most $r'_1$ and radial spread at most $r_2+1$ and $H-X' = F$ is a forest.
    By \cref{prop:q.i.-graph-dec}, $G$ is $\lambda$-quasi-isometric to $H$, and $H$ contains a set of at most $f(k)$ vertices whose removal turns $H$ into a forest.
    \medskip

    \noindent\textbf{The case $\mathbf{q \geq 4}$:} By \cref{thm:QIto3fatMinorfreegraph}, the graph $G^q$ contains no $3$-fat model of $k \cdot K_3$.
    By the case $q=3$, there is a $\lambda$-quasi-isometry $\sigma$ from $G^q$ to a graph $H$ that contains a set $X\subseteq V(H)$ of size at most $f(k)$ such that $H-X$ is a forest.

    Then, every $h \in V(H)$ is at distance at most $\lambda$ from $\sigma(V(G))$.
    Furthermore, if $u, v \in V(G)$, by \cref{thm:QIto3fatMinorfreegraph} we have
    \[d_H(\sigma(u), \sigma(v)) \geq \lambda^{-1} \cdot d_{G^q}(u, v) - \lambda \geq \lambda^{-1} \cdot (d_{G}(u, v)/q) - \lambda \geq (\lambda \cdot q)^{-1} \cdot d_{G}(u, v) - (\lambda \cdot q).\]
    Similarly, \[d_H(\sigma(u), \sigma(v)) \leq \lambda \cdot d_{G^q}(u, v) + \lambda \leq \lambda \cdot d_{G}(u, v) + \lambda \leq (\lambda \cdot q) \cdot d_{G}(u, v) + (\lambda \cdot q).\]
    Thus, $\sigma$ is a $(\lambda \cdot q)$-quasi-isometry from $G$ to $H$.
\end{proof}

\section{The Scouring of the Shire} \label{sec:BattleOfTheBlackGate}

This section is dedicated to proving the fat minor conjecture, \cref{conj:qi}, for disjoint unions of cycles. More precisely, we prove \cref{conj:qi} for $H = k\cdot K_3$, where $k \in \N \cup \{\infty\}$, which implies it for arbitrary cycles instead of triangles by \cite[Lemma~5.3]{GPCoarseGT}.
The proof builds on the coarse \EP\ theorem, \cref{th:main}.

In \cref{subsec:Mordor}, we first give a proof with an $\mathcal{O}_k(q)$-quasi-isometry. Then, in \cref{subsec:MountDoom}, we refine this result to make the quasi-isometry additive, thereby proving \cref{main:CoarseEp:Infadd,thm:mainFatkK3additive}.

\subsection{The fat minor conjecture for disjoint unions of cycles} \label{subsec:Mordor}

We begin with a proof of \cref{conj:qi} for $\infty\cdot K_3$ since it is much simpler than the proof for $k\cdot K_3$ where $k \in \N$.

\begin{restatable}{theorem}{coarseEPinf} \label{main:CoarseEp:Inf}
    There exists a constant $\lambda_{\ref{main:CoarseEp:Inf}}$ such that the following holds.
    Let $q \in \N$ and let $G$ be a graph that does not contain $\infty \cdot K_3$ as a $q$-fat minor.
    Then, $G$ is $(\lambda_{\ref{main:CoarseEp:Inf}} \cdot q)$-quasi-isometric to a graph with no $\infty \cdot K_3$ minor.
\end{restatable}

\begin{proof}
    Set $\defnm{\lambda_{\ref{main:CoarseEp:Inf}}} := 3\cdot \lambda_{\ref{th:main}}$.
    Apply \cref{th:main} to $G$ for every $k \in \N$ with parameter~$3q$.
    If, for some $k \in \N$, we obtain a $(\lambda_{\ref{th:main}} \cdot 3q)$-quasi-isometry to a graph $H$ that contains a set $X \subseteq V(H)$ of size at most $f_{\ref{th:main}}(k) < \infty$ such that $H-X$ is a forest, then we are done.

    Therefore, we may assume that $G$ contains $k \cdot K_3$ as a $3q$-fat minor for every $k \in \N$.
    We construct a $q$-fat model of $\infty \cdot K_3$ recursively.
    For this, let $C_1$ be a cycle of $G$ witnessing that $K_3$ is a $3q$-fat minor of $G$. Now let $k \geq 1$, and assume that we have chosen cycles $C_1, \ldots, C_k$ that witness that $k \cdot K_3$ is a $q$-fat minor of $G$, i.e.\ each cycle $C_i$ is $q$-fat, and they are pairwise at distance at least~$q$ in~$G$.

    Let $k' := \left|\bigcup_{i \leq k} V(C_i)\right| + 1$, and let $C'_1, \dots, C'_{k'}$ be $k'$ cycles of $G$ witnessing that $k'\cdot K_3$ is a $3q$-fat minor of~$G$.
    In particular, the cycles $C'_i$ are pairwise at distance at least $3q$.
    By the pigeonhole principle, at least one cycle $C'_j$ is at distance at least $q$ from all cycles $C_i$ with $i \leq k$.
    Set $C_{k+1} := C'_j$.
    Then, $C_1, \dots, C_{k+1}$ form a $q$-fat model of $(k+1)\cdot K_3$ in $G$.
\end{proof}

Next, we prove \cref{conj:qi} for $k \cdot K_3$ where $k \in \N$. The proof occupies the rest of this subsection.

\begin{theorem}\label{thm:mainFatkK3}
    There exists a function $f_{\ref{thm:mainFatkK3}} : \N \to \N$ such that the following holds.
    Let $q,k\in \N$ and let $G$ be a graph with no $q$-fat model of $k\cdot K_3$.
    Then $G$ is $(f_{\ref{thm:mainFatkK3}}(k) \cdot q)$-quasi-isometric to a graph with no $k\cdot K_3$ minor.
\end{theorem}

By \cref{thm:QIto3fatMinorfreegraph}, it is enough to prove the theorem for $q=3$. We shall still argue on $q$ more generally as the fatness of forbidden $k\cdot K_3$ minors increases during the proof. But \cref{thm:QIto3fatMinorfreegraph} allows us to get the final linear dependence on $q$ for parameters of the quasi-isometry.
\smallskip

\cref{th:main} provides significant progress towards \cref{thm:mainFatkK3} since forbidding a fat minor is invariant under quasi-isometry (up to changing the constants for the fatness).

\begin{lemma}[{\cite[Observation 2.4]{GPCoarseGT}}]\label{fatfat}
    If two graphs $G,H$ are $\lambda$-quasi-isometric and $G$ contains no $q$-fat model of a graph $J$, then $H$ contains no $Q$-fat model of $J$, where $Q=q\lambda + 14\lambda^3$.
\end{lemma}

By \cref{th:main} and \cref{fatfat}, it is enough to prove \cref{thm:mainFatkK3} for graphs $G$ with a vertex set $X\subseteq V(G)$ such that $G-X$ is a forest and $|X|\le m$.
The quasi-isometry we shall take for this will simply be a contraction.

A graph $H$ is a \defn{depth-$d$ contraction-minor} of a graph $G$ if it can be obtained from $G$ by contracting the edges of a subgraph $F$ such that every connected component of $F$ has diameter at most~$d$ in~$G$.
Observe that $G$ is $(d+1)$-quasi-isometric to such a graph~$H$.
Furthermore, if $G$ contains no $q$-fat $J$~minor, then neither does $H$.
Note that if $H$ is a depth-$d$ contraction-minor of $G$ and $J$ is a depth-$d'$ contraction-minor of $H$, then $J$ is a depth-$\left((d+1)(d'+1)-1\right)$ contraction-minor of $G$.
For a vertex set $Y\subseteq V(G)$ and a contraction-minor $H$ of $G$, we let \defn{$Y_H$} be the set of vertices~$v \in V(H)$ whose corresponding contracted subgraph in $G$ contains a vertex of~$Y$.

As discussed, \cref{thm:mainFatkK3} is now implied by the following lemma.

\begin{lemma}\label{lemma:mainFatkK3}
    There exists a function $f_{\ref{lemma:mainFatkK3}} : \N^3 \to \N$ such that the following holds.
    Let $m,q,k\in \N$ and let $G$ be a graph with no $q$-fat model of $k\cdot K_3$ and with a vertex set $X\subseteq V(G)$ such that $G-X$ is a forest and $|X| \leq m$.
    Then $G$ has a depth-$f_{\ref{lemma:mainFatkK3}}(m,q,k)$ contraction-minor with no $k\cdot K_3$ minor.
\end{lemma}

\begin{proof}[Proof of \cref{thm:mainFatkK3} assuming \cref{lemma:mainFatkK3}]
    Set $g(k) := f_{\ref{lemma:mainFatkK3}}\big(f_{\ref{th:main}}(k),\; 9\lambda_{\ref{th:main}} +378\lambda^3_{\ref{th:main}},\; k\big)+1$ and let $\defnm{f}_{\ref{thm:mainFatkK3}}(k)$ be an upper bound for the parameters of a quasi-isometry that is the composition of a $(3\lambda_{\ref{th:main}})$- and a $g(k)$-quasi-isometry (e.g.\ $\defnm{f}_{\ref{thm:mainFatkK3}}(k) := (3\lambda_{\ref{th:main}}+2) \cdot g(k)$ works).

    As in the proofs in \cref{sec:FinalProof}, we first prove the statement for $q=3$, and then extend it to all $q \in \N$.
    So, let $q=3$. Applying \cref{th:main} for $q=3$ yields that $G$ is $(\lambda_{\ref{th:main}}\cdot 3)$-quasi-isometric to a graph $H$ with a set $X$ of at most $f_{\ref{th:main}}(k)$ vertices such that $H-X$ is a forest. By \cref{fatfat}, $H$ has no $(9\lambda_{\ref{th:main}} + 378\lambda^3_{\ref{th:main}})$-fat $k\cdot K_3$ minor. Applying \cref{lemma:mainFatkK3} to $H$ yields that $H$ is $g(k)$-quasi-isometric to a graph with no $k \cdot K_3$ minor. Composing the two quasi-isometries concludes the proof for $q=3$.

    The cases $q = 1,2$ follow immediately from the case $q=3$ (because every $3$-fat model of $k \cdot K_3$ is also $1$-fat and $2$-fat). We only need to increase $f_{\ref{thm:mainFatkK3}}(k)$ by a factor of~$3$ to accommodate for $f_{\ref{thm:mainFatkK3}}(k) \cdot 1 \lneq f_{\ref{thm:mainFatkK3}}(k) \cdot 3$ and $f_{\ref{thm:mainFatkK3}}(k)\cdot 2 \lneq f_{\ref{thm:mainFatkK3}}(k) \cdot 3$.

    The case $q\geq 4$ follows from the case $q=3$ and \cref{thm:QIto3fatMinorfreegraph}.
\end{proof}

To prove \cref{lemma:mainFatkK3},
we first further simplify the graph $G$ by finding a bounded-depth contraction-minor with no collection of $k$ far apart cycles.
Our main tool for this is \cref{lem:dumb-rhino}.
To make the induction work, we prove the following stronger lemma. In our later application, we will only need the case $Y=\emptyset$.

\begin{lemma}\label{lemma:Fat2Distant}
    There exists a function $f_{\ref{lemma:Fat2Distant}} : \N^3 \to \N$ such that the following holds.
    Let $m,q,k\in \N$ and let $G$ be a graph with vertex sets $Y\subseteq X \subseteq V(G)$ such that every $q$-fat model of $k\cdot K_3$ in~$G$ intersects $Y$ (if there are any), $G-X$ is a forest, and $|X| \leq m$.
    Then $G$ has a depth-$f_{\ref{lemma:Fat2Distant}}(m,q,k)$ contraction-minor $H$ with no collection of $k$~cycles that are pairwise at distance at least $6q+3$ from each other and at distance at least $5q/2+1$ from $Y_H$.
\end{lemma}

\begin{proof}
    For each $q\in \N$, set $\defnm{f}_{\ref{lemma:Fat2Distant}}(1,q,1) \coloneqq 6q$.
    Now for each $m,q,k\in \N$, we inductively define
    \begin{align*}
        \defnm{f}_{\ref{lemma:Fat2Distant}}(m+1,q,k) &\coloneqq (f_{\ref{lemma:Fat2Distant}}(m,q,k)+1)(38q+7), \text{ and} \\
        \defnm{f}_{\ref{lemma:Fat2Distant}}(m,q,k+1) &\coloneqq (f_{\ref{lemma:Fat2Distant}}(m,q,k)+1)(38q+7).
    \end{align*}

    Suppose first that there exist distinct $x,y\in X$ at distance at most $21q+2$. Then let $H'$ be the depth-$(21q+2)$ contraction-minor of $G$ obtained by contracting a shortest path between $x$ and $y$.
    Note that $1\le|X_{H'}|<|X|$ and every $q$-fat model of $k \cdot K_3$ in $H'$ intersects $Y_{H'}$.
    Then, by the inductive hypothesis, there is a depth-$f_{\ref{lemma:Fat2Distant}}(m-1,q,k)$ contraction-minor $H$ of $H'$ (and therefore a depth-$((f_{\ref{lemma:Fat2Distant}}(m-1,q,k)+1)(21q+3))$, and in particular a depth-$f_{\ref{lemma:Fat2Distant}}(m,q,k)$, contraction-minor of $G$) with no $k$ cycles that are pairwise at distance at least $6q+3$ from each other and at distance at least $5q/2+1$ from~$Y_H$, as desired.
    So, we may now assume that the vertices of $X$ are pairwise at distance at least $21q+3$ in~$G$.
    \smallskip

    Next, suppose that for every $x\in X\setminus Y$ there is no geodesic cycle of length between $6q$ and $36q+6$ that visits $x$.
    Then every cycle~$C'$ in~$G$ of length at most $36q+6$ that visits $x$ is contained in $B_G(x,3q)$.
    Indeed, if not, then there is a vertex $c \in V(C')\setminus B_G(x,3q)$, and we may take a shortest cycle~$C$ in $G$ containing $x$ and $c$.
    In particular, $C$ has length at least $6q+2$ and at most $36q+6$.
    If $C$ is not geodesic in $G$, there exists a shortcut $S$ between two vertices $c_1, c_2$ of $C$.
    By considering a shortcut of minimum length, we can assume that $S$ is internally disjoint from $C$.
    By minimality of $C$, we have $\{c_1, c_2\} \cap \{c, x\} = \emptyset$.
    Furthermore, again by minimality of $C$, the vertices $x$ and $c$ lie on different $c_1$--$c_2$ subpaths of $C$.
    Let $C_c$ be the $c_1$--$c_2$ subpath of $C$ that contains $c$, and let $C_x$ be the other subpath.
    Then, $C_c \cup S$ is a cycle that does not contain $x$.
    Since $G-X$ is a forest, $C_c \cup S$ contains some vertex $x' \in X \setminus \{x\}$.
    If $x' \in V(C_c)$ then $x$ and $x'$ both lie on $C$; otherwise $x$ and $x'$ both lie on the cycle $C_x \cup S$.
    In either case, $x$ and $x'$ both lie on a cycle of length at most $36q+6$, so $d_G(x, x') \leq 18q+3$.
    This is a contradiction, so $C$ must be geodesic in $G$.
    This contradicts our assumption on $x$, so every cycle~$C'$ in~$G$ of length at most $36q+6$ that visits $x$ is contained in $B_G(x,3q)$.

    Let $H$ be the graph obtained from $G$ by contracting all balls $B_G(x,3q)$ around the vertices $x\in X \setminus Y$; these balls are pairwise disjoint since any two vertices of $X$ are at distance at least $21q+3$.
    Then, $H$ is a depth-$6q$ contraction-minor of $G$.
    We now show that every cycle $C_H$ in $H$ of length at most $30q+5$ intersects $Y_H$.
    Since $G-X$ is a forest, $H-X_H$ is also a forest, so $C_H$ must intersect $X_H$.
    If $C_H$ contains two distinct vertices in $X_H$, then $C_H$ contains a path $P_H$ of length at most $15q+2$ between two vertices of $X_H$, that is internally disjoint from $X_H$.
    Then, $P_H$ lifts to a path $P_G$ in $G$ of length at most $15q+2$ between $B_G(x, 3q)$ and $B_G(x', 3q)$ for some distinct $x, x' \in X$.
    Therefore, $d_G(x, x') \leq 21q+2$, a contradiction.
    Thus, $C_H$ contains exactly one vertex $x_H \in X_H$.
    If $x_H \in Y_H$ then we are done.
    Otherwise, $x_H$ is obtained by contracting $B_G(x, 3q)$ for some $x \in X \setminus Y$.
    Then, $C_H$ lifts to a cycle $C_G$ of $G$ contained in $B_G(x, 3q) \cup V(C_H) \setminus \{x_H\}$ of length at most $36q+5$.
    Moreover, every vertex of $C_G$ is at distance at most $18q+2$ from $x$ since the part outside $B_G(x, 3q)$ has length at most $30q+5$ and has both endpoints in $B_G(x, 3q)$.
    Hence, $C_G$ does not intersect $X \setminus \{x\}$.
    Since $G-X$ is a forest, $C_G$ must contain $x$.
    Therefore, $C_G$ is a cycle in $G$ of length at most $36q+6$ that visits $x$ and is not contained in $B_G(x, 3q)$, contradicting the property established above for $x \in X \setminus Y$.
    This concludes the proof that $C_H$ intersects $Y_H$.
    Since $G$ has no $q$-fat model of $k \cdot K_3$ avoiding $Y$, neither does $H$ have one avoiding $Y_H$.
    It then follows from \cref{lem:dumb-rhino} that $H$ has no $k$ cycles that are pairwise at distance at least $6q$ in $H$ from each other and at least $5q/2 +1$ from $Y_H$, as desired.
    \smallskip

    So, we may assume now that there exists some $x\in X\setminus Y$ with a geodesic cycle $C$ of length between $6q$ and $36q+6$ that visits $x$.
    In particular, $C$ is a $q$-fat model of $K_3$ in~$G$ and it is contained in $B_G(x,18q+3)$.
    It also does not intersect~$Y$ since the vertices in $X \supseteq Y$ are pairwise at distance at least $21q+3$.
    Since every $q$-fat model of $k \cdot K_3$ in $G$ intersects $Y$, this implies that $k \geq 2$.
    Now consider the depth-$(38q+6)$ contraction-minor $H'$ of $G$ obtained by contracting the ball $B_G(x,19q+3)$ down to the vertex~$x$.
    Let $X'=X$ and $Y'=Y\cup \{x\}$.
    Then $H'$ contains no $q$-fat model of $(k-1)\cdot K_3$ that does not intersect~$Y'$, as otherwise by adding~$C$, this would contradict the fact that $G$ contains no $q$-fat model of $k\cdot K_3$ that does not intersect~$Y$.
    By the induction hypothesis, $H'$ has a depth-$f_{\ref{lemma:Fat2Distant}}(m,q,k-1)$ contraction-minor $H$ with no $(k-1)$ cycles that are pairwise at distance at least $6q+3$ in $H$ from each other and at least $5q/2+1$ from $Y'_H$.
    Since $|Y'_H\setminus Y_H| \leq 1$, it follows that $H$ has no $k$ cycles that are pairwise at distance at least $6q+3$ in $H$ from each other and at least $5q/2+1$ from $Y_H$.
    As $(f_{\ref{lemma:Fat2Distant}}(m,q,k-1)+1) ( 38q+7  ) = f_{\ref{lemma:Fat2Distant}}(m,q,k)$, it now follows that $H$ is the desired depth-$f_{\ref{lemma:Fat2Distant}}(m,q,k)$ contraction-minor of $G$.
\end{proof}

\cref{lemma:mainFatkK3} follows from \cref{lemma:Fat2Distant} and the following lemma.

\begin{lemma}\label{lemma:farkK3}
    There exists a function $f_{\ref{lemma:farkK3}} : \N^2 \to \N$ such that the following holds.
    Let $m,q,k\in \N$ and let $G$ be a graph with no collection of $k$ cycles that are pairwise at distance at least $q$ and with a vertex set $X\subseteq V(G)$ such that $G-X$ is a forest and $
    |X|\leq m$.
    Then $G$ has a depth-$f_{\ref{lemma:farkK3}}(m,q)$ contraction-minor with no $k$ disjoint cycles.
\end{lemma}

\begin{proof}[Proof of \cref{lemma:mainFatkK3} assuming \cref{lemma:farkK3}]
    Apply \cref{lemma:Fat2Distant} with $Y=\emptyset$ and then apply \cref{lemma:farkK3} (with $6q+3$ instead of $q$).
\end{proof}

So, proving \cref{lemma:farkK3} will complete the proof of \cref{thm:mainFatkK3}.
The remainder of this subsection is now dedicated to proving \cref{lemma:farkK3}.
For this, we will examine the different types of collections of $k$ disjoint cycles that a graph as in \cref{lemma:farkK3} can contain and show that each of the (boundedly many) types can be eliminated with a bounded depth contraction-minor.
\smallskip

Consider now a graph $G$ as in \cref{lemma:farkK3}, i.e.\ $G$ contains a vertex set $X\subseteq V(G)$ with $|X|\le m$ such that $F:=G-X$ is a forest, and $G$ contains no collection of $k$ cycles that are pairwise at distance at least~$q$.
Now, consider some collection $\mathcal{C}$ of $k$ disjoint cycles $C_1,\ldots , C_k$ in $G$.
Note that some pair of these cycles is at distance less than~$q$ in~$G$.
We say that $\mathcal{C}$ is \defn{feasible} if each $x\in X$ is within distance at most $q$ of at most one cycle in~$\mathcal{C}$.

Since $F=G-X$ is a forest, every cycle of $G$ consists of some subset of $X$ and paths between these vertices whose interiors are contained in $F$ (or possibly it is a cycle containing a single vertex of~$X$).
So, from $\mathcal{C}$ we can naturally define an auxiliary (multi)graph $\defnm{X(\mathcal{C})}$ (possibly with loops and multiple edges) on vertex set $X$, where we add an edge between two vertices $x, x' \in X$ for each subpath of a cycle in $\mathcal{C}$ between $x$ and $x'$ whose interior is contained in $F$, and for each cycle in $\mathcal{C}$ intersecting only a single vertex $x$ of $X$, we add a loop at~$x$.

If $\mathcal{C} = \{C_1, \dots, C_k\}$ is feasible, then there exists a partition $Z_1\,\dot\cup \ldots \dot\cup\, Z_k$ of $X$ such that for each $i \in [k]$, we have that $B_G(V(C_i),q)\cap X \subseteq Z_i$. We call $(X(\mathcal{C}), Z_1, \ldots , Z_k)$ a \defn{type} of the feasible collection $\mathcal{C}$.
Note that $\mathcal{C}$ may belong to multiple different types as the partition of~$X$ is not unique if there are vertices in~$X$ that are not within distance at most $q$ from any cycle in~$\mathcal{C}$.

When we fix some $X$, we can talk about feasible collections of disjoint cycles in both $(G,X)$ and $(G',X)$ for contraction-minors $G'$ of $G$ that are obtained from $G$ by only contracting some edges of $F$.
Our aim is to show that for any given type $\mathcal{T}$ of feasible collections of $k$ disjoint cycles in $(G,X)$, we can find a bounded-depth contraction-minor $G'$ of $G$ by only contracting edges of~$F$ such that $(G',X)$ contains no feasible collection of $k$ disjoint cycles of type $\mathcal{T}$.

\begin{lemma}\label{Smeagol}
    There exists a function $f_{\ref{Smeagol}} : \N^2 \to \N$ such that the following holds.
    Let $m,q,k\in \N$ and let $G$ be a graph with no collection of $k$ cycles that are pairwise at distance at least $q$ and with a vertex set $X\subseteq V(G)$ such that $G-X$ is a forest and $
    |X| \leq m$.
    Let $\mathcal{T}$ be a type of feasible collections of $k$ disjoint cycles of $(G,X)$.
    Then $(G,X)$ has a depth-$f_{\ref{Smeagol}}(m,q)$ contraction-minor $(G',X)$ (that only contracts edges not incident with $X$) such that $(G',X)$ has no feasible collection of $k$ disjoint cycles of type $\mathcal{T}$.
\end{lemma}

To motivate examining feasible collections of $k$ disjoint cycles of each type like this, let us now show that \cref{Smeagol} implies \cref{lemma:farkK3}.

\begin{proof}[Proof of \cref{lemma:farkK3} assuming \cref{Smeagol}]
    For each $m,q\in \N$, let
    \[
    \defnm{f}_{\ref{lemma:farkK3}}(m,q)
    :=
    (2mq+1)
    \left(
    f_{\ref{Smeagol}}(m,q)+1
    \right)^{
    ((2m)!)\cdot m^{m}
    }
    .
    \]
    Note that the lemma immediately holds if $m<k$ because $G$ can contain at most $m$ disjoint cycles. So we may assume that $k \leq m$.

    For a given feasible collection $\mathcal{C}$ of $k$ disjoint cycles, $X(\mathcal{C})$ is a multigraph with vertex set $X$ and it consists of exactly $k$ disjoint cycles. Clearly, there are only boundedly many such multigraphs, in fact at most $(m+k)!\le (2m)!$.
    The number of partitions $Z_1\,\dot\cup \cdots \dot\cup\, Z_k$ of $X$ is at most $k^m\le m^m$.
    Thus, there are at most $((2m)!)\cdot m^{m}$ types of feasible collections of $k$ disjoint cycles.

    So now, enumerate all these types $\mathcal{T}_1, \ldots , \mathcal{T}_r$ for some $r\le ((2m)!)\cdot m^{m}$.
    By repeatedly applying \cref{Smeagol}, there exists a sequence of graphs $G=G_0,G_1,\ldots,G_r$ such that for each $1\le i \le r$, $(G_i,X)$ is a depth-$f_{\ref{Smeagol}}(m,q)$ contraction-minor of $(G_{i-1},X)$ and $(G_i,X)$ contains no feasible collection of $k$ disjoint cycles of type $\mathcal{T}_i$.
    It follows that $(G_r,X)$ is a contraction-minor of $(G,X)$ of  depth at most
    \[
    \left(
    f_{\ref{Smeagol}}(m,q)+1
    \right)^{
    ((2m)!)\cdot m^{m}
    }
    -1
    .
    \]
    Furthermore, $(G_r,X)$ has no feasible collection of $k$ disjoint cycles of any type, since contracting an edge non-incident to $X$ clearly cannot create a collection of $k$ disjoint cycles of a new type.

    Now every remaining collection of $k$ disjoint cycles in~$G_r$ cannot be feasible, so there is some vertex of $X$ at distance at most $q$ from at least two cycles in this collection.
    Let $G'$ be obtained from $G_r$ by contracting the edges within each $G_r[B_{G_r}(x,q)]$ for $x\in X$.
    Then, $G'$ is a depth-$2mq$ contraction-minor of $G_r$ and contains no collection of $k$ disjoint cycles.
    Hence, $G'$ is the desired contraction-minor of~$G$.
\end{proof}

It therefore remains to prove \cref{Smeagol}. For this, we will examine collections of subpaths of a forest.
We say that two paths $P_1,P_2$ \defn{overlap} if they intersect but neither is contained in the other.
For a forest~$F$ and a collection $\mathcal{P}$ of subpaths of $F$, we say that $\mathcal{P}$ is \defn{plentiful} if for every overlapping pair $P_1,P_2\in \mathcal{P}$ with four distinct endvertices $a_1,b_1,a_2,b_2$, where $P_i$ has endvertices $a_i,b_i$, either both the subpaths of $P_1\cup P_2$ with endvertices $a_1,a_2$ and $b_1,b_2$ are contained in $\mathcal{P}$, or both the subpaths of $P_1\cup P_2$ with endvertices $a_1,b_2$ and $a_2,b_1$ are contained in $\mathcal{P}$.
Note that this in particular implies that for every overlapping pair $P_1,P_2\in \mathcal{P}$ with four distinct endvertices such that $P_1,P_2$ are contained in some common subpath $Q$ of $F$, either $P_1\cap P_2\in \mathcal{P}$, or $P_1 \setminus P_2, P_2\setminus P_1 \in \mathcal{P}$.

\cref{Smeagol} now follows from the following lemma on plentiful collections of subpaths of a forest.

\begin{restatable}{lemma}{Lobelia}\label{Lobelia Sackville-Baggins}
    There exists a function $f_{\ref{Lobelia Sackville-Baggins}} : \N^2 \to \N$ such that the following holds.
    Let $r,q\in \N$, let $F$ be a forest, and
    let $\mathcal{P}_1, \ldots , \mathcal{P}_{r}$ be plentiful collections of subpaths of $F$.
    Then either
    \begin{enumerate}[label=\rm{(\arabic*)}]
        \item \label{itm:Sackville} there exist paths $P_1 \in \mathcal{P}_1, \ldots , P_r \in \mathcal{P}_r$ that are pairwise at distance at least~$q$ in~$F$, or
        \item \label{itm:Baggins} there exists some subforest $J$ of $F$ such that for every path $Q$ of $F$ we have that $|E(Q)\cap E(J)| \le f_{\ref{Lobelia Sackville-Baggins}}(r,q)$ and for every $P_1 \in \mathcal{P}_1, \ldots , P_r \in \mathcal{P}_r$ there exist distinct $i, j \in [r]$ such that $J$ contains a path between $P_i$ and $P_j$.
    \end{enumerate}
\end{restatable}

Let us first show that \cref{Lobelia Sackville-Baggins} implies \cref{Smeagol}.

\begin{proof}[Proof of \cref{Smeagol} assuming \cref{Lobelia Sackville-Baggins}.]
    We set $\defnm{f}_{\ref{Smeagol}} := f_{\ref{Lobelia Sackville-Baggins}}$.
    Let $F:=G-X$, and let $\mathcal{T}= (M, Z_1, \ldots , Z_k)$ be the type of feasible collections of $k$ disjoint cycles of $(G,X)$ given by the premises of \cref{Smeagol}. We may assume that there is at least one feasible collection of type $\mathcal{T}$ in $G$ (as otherwise we are done).
    Enumerate the edges $e_1,\ldots , e_r$ of $M$. Clearly $r\leq m (= |X|)$ since $M$ is a multigraph consisting of $k$ (vertex) disjoint cycles (and possibly loops) on vertex set~$X$.
    For each $i \in [r]$, let $c(i)$ be the unique index such that the component (equivalently cycle or loop) of~$M$ containing $e_i$ is contained in~$M[Z_{c(i)}]$.

    For each $i \in [r]$, if $e_i=xx$ is a loop of $M$, then let $\mathcal{Q}_i$ be the collection of cycles $Q$ of $G$ such that $V(Q)\cap X=\{x\}$ and $B_G(V(Q),q) \cap X \subseteq Z_{c(i)}$, otherwise if $e_i=xy$ is not a loop of $M$, then let $\mathcal{Q}_i$ be the collection of paths $Q$ of $G$ between $x$ and $y$ such that $V(Q)\cap X=\{x,y\}$ and $B_G(V(Q),q) \cap X \subseteq Z_{c(i)}$.
    Observe that if $e_i,e_j$ are edges of different components of~$M$ and $Q_i\in \mathcal{Q}_i$, $Q_j\in \mathcal{Q}_j$ are at distance at most $q$ in $G$, then there is a path between them in $F$ of length at most $q$ since $Z_{c(i)}$ and $Z_{c(j)}$ are disjoint.
    For each $i \in [r]$, let $\defnm{\mathcal{P}_i}:=\{Q- X : Q\in \mathcal{Q}_i\}$, and note that each $\mathcal{P}_i$ is a non-empty collection of subpaths of~$F$.
    Moreover, for every $\ell \in [k]$ and every choice of one path $P_i$ from each $\mathcal{P}_i$ with $c(i) = \ell$, if the paths $P_i$ are pairwise vertex-disjoint then together with $Z_\ell$ they induce a subgraph of $G$ that contains at least one cycle.
    In particular, since $G$ contains no collection of $k$ cycles that are pairwise at distance at least $q$ in $G$, it follows that there are no $P_1 \in \mathcal{P}_1, \ldots , P_r \in \mathcal{P}_r$ that are pairwise at distance at least~$q$ in~$F$.
    \smallskip

    Consider some $\mathcal{P}_i$ with $e_i=xy$ (possibly with $x=y$), and consider some overlapping pair $P_1,P_2\in \mathcal{P}_i$
    with four distinct endvertices $a_1,b_1,a_2,b_2$, where $P_j$ has endvertices $a_j,b_j$.
    Without loss of generality, let $\{a_1,a_2\}\subseteq B_G(x,1)$ and $\{b_1,b_2\}\subseteq B_G(y,1)$.
    Then clearly $\mathcal{P}_i$ also contains the two subpaths of $P_1\cup P_2$ with endvertices $a_1,b_2$ and with endvertices $a_2,b_1$.
    Thus, the collections $\mathcal{P}_1, \ldots , \mathcal{P}_r$ of subpaths are plentiful.

    Then  by \cref{Lobelia Sackville-Baggins} there exists a subgraph $J$ of $F$ as in \ref{itm:Baggins} of \cref{Lobelia Sackville-Baggins}; in particular, every component of $J$ has diameter at most $f_{\ref{Smeagol}}(r,q)$.
    Let $G'=G/E(J)$ be obtained by contracting the edges of~$J$. Then clearly $(G',X)$ is a depth-$f_{\ref{Smeagol}}(r,q)$ contraction-minor of $(G,X)$.
    We claim that $(G',X)$ contains no feasible collection $\mathcal{C}$ of $k$ disjoint cycles of type~$\mathcal{T}$, and is hence as desired. Indeed, any such $\mathcal{C}$ would yield (by uncontracting) a feasible collection~$\mathcal{C}'$ of $k$ disjoint cycles of type~$\mathcal{T}$ in $(G,X)$. But such a collection would add one path $P_i$ to each collection $\mathcal{P}_i$, and these paths $P_i$ would be pairwise disjoint. But then $J$ must contain a path between some distinct paths $P_i, P_j$, so if $P_i, P_j \subseteq C'$ for some $C'\in \mathcal{C}'$, then the corresponding cycle $C \in \mathcal{C}$ is not a cycle, or if $P_i \subseteq C'_1$ and $P_j \subseteq C'_2$ for some cycles $C'_1 \neq C'_2 \in \mathcal{C}'$, then the corresponding cycles in $\mathcal{C}$ are not disjoint, a contradiction.
\end{proof}

So, it remains to prove \cref{Lobelia Sackville-Baggins}.
We need two more preliminary lemmas for this.
The first simple lemma gives us a way to obtain the first outcome of \cref{Lobelia Sackville-Baggins}.
For $q=1$, this is a lemma of Gy{\'a}rf{\'a}s and Lehel \cite{gyarfas1970helly}, and our proof is essentially the same as theirs.

\begin{lemma}\label{lem:distantdudes}
    Let $m,q \in \N$, let $F$ be a forest, and let $\mathcal{A}_1, \ldots , \mathcal{A}_m$ each be a collection of $m$ subtrees of $F$ that are pairwise at distance at least $q$ in $F$.
    Then there exist $T_1 \in \mathcal{A}_1, \ldots , T_m \in \mathcal{A}_m$ that are pairwise at distance at least $q$ in $F$.
\end{lemma}

\begin{proof}
    The lemma trivially holds for $m=1$.
    We proceed inductively.
    Choose a vertex $r$ in a component of $F$ that contains some member of $\mathcal{A}_1 \cup  \cdots \cup \mathcal{A}_m$.
    Choose $T\in \mathcal{A}_1 \cup  \cdots \cup \mathcal{A}_m$ in the component of $F$ containing $r$ so that the distance between $T$ and $r$ is maximised. (Note that each $\mathcal{A}_i$ is finite, and hence $T$ exists even if $F$ is infinite).
    Without loss of generality, we may assume that $T\in \mathcal{A}_m$. Set $T_m := T$.
    Now observe that each of $\mathcal{A}_1, \ldots , \mathcal{A}_{m-1}$ has at most one member at distance less than~$q$ from $T_m$ (because if there were two, then they would be at distance less than $q$ because $F$ is a forest and $T_m$ was chosen furthest away from~$r$).
    So there exist subcollections $\mathcal{A}_1'\subset \mathcal{A}_1, \ldots , \mathcal{A}_{m-1}'\subset \mathcal{A}_{m-1}$ of $m-1$ subtrees of $F$ that are pairwise at distance at least $q$ in $F$ and that are all at distance at least $q$ from $T_m$ in $F$.
    By the inductive hypothesis, we then find subtrees $T_1 \in \mathcal{A}_1, \ldots , T_{m-1} \in \mathcal{A}_{m-1}$ that along with $T_m$ are pairwise at distance at least $q$ in $F$, as desired.
\end{proof}

\cref{lem:distantdudes} gives us a way to take advantage of finding many far apart paths within each collection.
So next we will examine what can be found in collections of subpaths that do not contain a subcollection of many far apart paths.
In order to obtain the desired second outcome of \cref{Merry}, we must restrict ourselves to forests that are the (not necessarily disjoint) union of only a bounded number of paths.

\begin{lemma} \label{Merry}
    Let $\ell,q,r \in \N$,
    let $F$ be a forest that is the (not necessarily disjoint) union of at most $\ell$ paths, and let $\mathcal{P}$ be a plentiful collection of subpaths of $F$.
    Then either
    \begin{itemize}
        \item  there exists some subcollection $\mathcal{P}^*\subseteq \mathcal{P}$ such that $|\mathcal{P}^*| \le 4\ell^2(q+1)r$ and every $P\in \mathcal{P}$ contains some $P^*\in \mathcal{P}^*$ as a subpath, or
        \item there exist $P_1,\ldots , P_r\in \mathcal{P}$ that are pairwise at distance at least $q$ and contained in a common subpath $Q$ of $F$.
    \end{itemize}
\end{lemma}

\begin{proof}
    Since $F$ is the union of at most $\ell$ paths, it contains at most $2\ell$ leaves, and hence at most~$4\ell^2$ maximal subpaths.
    In particular, every subpath $Q$ of $F$ is a subpath of one of these at most~$4\ell^2$ paths.
    So, it suffices to consider each maximal subpath~$Q$ of~$F$ separately, and either find paths as in the second bullet (in which case we are done), or a subcollection $\mathcal{P}^*_Q \subseteq \mathcal{P}$ of (improved) size at most $(q+1)r$ such that every $P \in \mathcal{P}$ that is contained in $Q$ contains some path $P^* \in \mathcal{P}^*_Q$ (in which case the union $\mathcal{P}^*$ over the collections $\mathcal{P}^*_Q$ is as desired).
    \smallskip

    So let $Q$ be a maximal subpath of~$F$, and let $\mathcal{P}_Q$ be the subcollection of all paths in $\mathcal{P}$ that are contained in $Q$. Then $\mathcal{P}_Q$ is plentiful because $\mathcal{P}$ is plentiful and  if overlapping paths $P_1,P_2$ are contained in $Q$, then all six paths between their endvertices are also contained in~$Q$.

    Let $\mathcal{P}^*_Q \subseteq \mathcal{P}_Q$ be the set of inclusionwise minimal elements of $\mathcal{P}_Q$.
    Since $\mathcal{P}_Q$ is plentiful, distinct $P_1^*,P_2^*\in \mathcal{P}^*_Q$ can only intersect in at most one vertex, which then must be a common endvertex.
    So, ordering the paths $P_1^*, \ldots, P_w^*$ in $\mathcal{P}^*_Q$ according to their order along the path $Q$, we observe that for each $i$, $P_i^*$ is at distance at least $q$ from $P_{i+q+1}^*$ in $Q$ and hence in $F$.
    Thus, if $|\mathcal{P}^*_Q| > (q+1)r$, then we can get the second outcome of the lemma by taking $P_1^*,\ldots , P_{1+(r-1)(q+1)}^*$.
    In the other case we have that $|\mathcal{P}^*_Q| \le (q+1)r$, as desired.
\end{proof}

We are now ready to prove \cref{Lobelia Sackville-Baggins} and thus complete the proof of \cref{thm:mainFatkK3}.
We restate \cref{Lobelia Sackville-Baggins} for convenience.

\Lobelia*

\begin{proof}
    For each $q\in \N$, set $\defnm{f}_{\ref{Lobelia Sackville-Baggins}}(1,q) := 1$ and then for each $r\in \N$ with $r>1$, we inductively define
    \[
    \defnm{f}_{\ref{Lobelia Sackville-Baggins}}(r,q)
    :=
    16r^5(q+1) (2q + f_{\ref{Lobelia Sackville-Baggins}}(r-1,q)).
    \]
    Clearly the lemma holds for $r=1$, so we shall proceed inductively.
    \smallskip

    Suppose first that there exists some $i \in [r]$ such that for any $r-1$ subpaths $P_1,\ldots , P_{r-1}$
    of $F$, there exists some $P\in \mathcal{P}_i$ that is at distance at least $q$ from $P_1,\ldots , P_{r-1}$ in $F$.
    Without loss of generality, we may assume that $\mathcal{P}_r$ has this property.
    If there exist paths $P_1 \in \mathcal{P}_1, \ldots , P_{r-1} \in \mathcal{P}_{r-1}$ such that $P_1,\ldots , P_{r-1}$ are pairwise at distance at least $q$ in $F$, then we may choose some $P_r\in \mathcal{P}_r$ at distance at least $q$ in $F$ from $P_1,\ldots , P_{r-1}$, giving outcome~\ref{itm:Sackville} of the lemma.
    Otherwise, by the induction hypothesis,
    there exists some subgraph $J$ of $F$ such that for every path $Q$ of $F$ we have that $|E(Q)\cap E(J)| \le f_{\ref{Lobelia Sackville-Baggins}}(r-1,q)\le f_{\ref{Lobelia Sackville-Baggins}}(r,q)$ and for every $P_1 \in \mathcal{P}_1, \ldots , P_{r-1} \in \mathcal{P}_{r-1}$ there exist distinct $i,j \in [r-1]$ such that $J$ contains a path between $P_i$ and $P_j$.
    This gives outcome~\ref{itm:Baggins} of the lemma.
    \smallskip

    So, we may assume now that for every $i \in [r]$ there exist subpaths $P_1^i,\ldots , P_{r-1}^i$ of $F$ such that every $P\in \mathcal{P}_i$ is within distance $q$ of at least one of $P_1^i,\ldots , P_{r-1}^i$ in $F$.
    Let \defn{$F'$} be the union of all such paths for all $i \in [r]$.
    Then $F'$ is a forest that is the union of at most $r(r-1) \leq r^2$ paths, and every $P\in \bigcup_{i=1}^r \mathcal{P}_i$ is within distance $q$ of $F'$ in $F$.
    By adding at most $r^2-1$ subpaths of $F$ to $F'$, we may assume that each component of $F$ contains at most one component of $F'$. Moreover, $F'$ is the union of at most $2r^2$ paths.

    For each $P\in \bigcup_{i=1}^r \mathcal{P}_i$ that intersects $F'$, let $\defnm{P'}:=P\cap F'$, and note that $P'$ is a path.
    For each $P\in \bigcup_{i=1}^r \mathcal{P}_i$ that does not intersect $F'$, let $P'$ be the subpath of $F'$ consisting simply of the vertex of $F'$ closest to $P$ (and so within distance $q$ of $P$).
    So, each $P'$ is a subpath of $F'$.
    Moreover, for any two paths $P_1, P_2$ in the same component of $F$, each edge of the $P'_1$--$P'_2$ path in $F'$ is contained in the $P_1$--$P_2$ path in $F$. In particular, this implies that any two paths $P_1, P_2$ satisfy $d_{F'}(P'_1, P'_2) \leq d_F(P_1, P_2)$.
    For each $i \in [r]$, let $\defnm{\mathcal{P}_i'}:=\{P': P\in \mathcal{P}_i \}$.
    Observe that each $\mathcal{P}_i'$ is plentiful since each $\mathcal{P}_i$ is plentiful.
    Indeed, given two overlapping paths in $\mathcal{P}_i'$ with four distinct endvertices, choose preimages in $\mathcal{P}_i$. These preimages are overlapping and have four distinct endvertices, so plentifulness of $\mathcal{P}_i$, followed by projecting to $\mathcal{P}_i'$, gives one of the two required pairs of paths in $\mathcal{P}_i'$
    \smallskip

    Suppose now that for each $i \in [r]$ the collection $\mathcal{P}_i'$ contains $r$ paths $P_{1,i}', \ldots , P_{r,i}'$ that are pairwise at distance at least $q$ in $F'$. In particular, they are at distance at least~$q$ in~$F$ (because $F$ is a forest and each of its components contains at most one component of~$F'$, so if $P_{j,i}$, $P_{\ell,i}$ lie in the same component of~$F$, the (unique) path between them is contained in~$F'$).
    Then, their respective corresponding paths $P_{1,i}, \ldots , P_{r,i}$ in $\mathcal{P}_i$ are pairwise at distance at least~$q$ in~$F$ (because we only potentially moved them closer together when passing to the paths $P'_{j,i}$).
    Then, \cref{lem:distantdudes} yields outcome~\ref{itm:Sackville} of the lemma.
    \smallskip

    So, without loss of generality, we may now assume that $\mathcal{P}_r'$ contains no $r$ subpaths that are pairwise at distance at least $q$ in $F'$.
    Then, by \cref{Merry}, there exists some subcollection $\defnm{\mathcal{P}^*_r} \subseteq \mathcal{P}'_r$ such that $|\mathcal{P}^*_r| \le 16r^5(q+1)$ and every $P'\in \mathcal{P}_r'$ contains some $P^*\in \mathcal{P}_r^*$ as a subpath.
    For each $P'\in \mathcal{P}_r^*$,
    let \defn{$J_{P'}$} be the subgraph $F[B_F(P',q)]-E(P')$ of $F$.
    In particular, every subpath~$Q$ of~$F$ contains at most $2q$ edges of $J_{P'}$, i.e.\ $|E(Q)\cap E(J_{P'})| \le 2q$.

    Consider some $P' \in \mathcal{P}_r^*$, and for each $i \in [r-1]$, let \defn{$\mathcal{P}_{i,P'}$} be the subcollection of $\mathcal{P}_i$ consisting of all the paths in $\mathcal{P}_i$ that are vertex-disjoint from $J_{P'}$.
    In particular, each $P_i\in \mathcal{P}_{i,P'}$ must be at distance at least $q$ from $P'$ in $F$.
    Choose $P_r\in\mathcal{P}_r$ whose projection onto $F'$ is $P'$. For every $P_i\in\mathcal{P}_{i,P'}$, its projection $P_i'$ is disjoint from $P'$: indeed, either $P_i'\subseteq P_i$, or $P_i'$ is a single vertex at distance at most $q$ from $P_i$, whereas $d_F(P_i,P')>q$. Since $F$ is a forest, every path between $P_i$ and $P_r$ therefore meets $P'$, so $d_F(P_i,P_r)\ge d_F(P_i,P')>q$.
    Then, if there exist $P_1\in \mathcal{P}_{1,P'}, \ldots , P_{r-1} \in \mathcal{P}_{r-1,P'}$ that are pairwise at distance at least $q$ in $F$, we get outcome~\ref{itm:Sackville} of the lemma.
    So, we may assume otherwise.
    Therefore, by the induction hypothesis, there is some
    subgraph \defn{$W_{P'}$} of $F$ such that for every path~$Q$ of $F$ we have that $|E(Q)\cap E(W_{P'})| \le f_{\ref{Lobelia Sackville-Baggins}}(r-1,q)$ and for every $P_1 \in \mathcal{P}_{1,P'}, \ldots , P_{r-1} \in \mathcal{P}_{r-1,P'}$ there exist distinct $i, j \in [r-1]$ such that $W_{P'}$ contains a path between $P_i$ and $P_j$.

    Now, let $\defnm{J}:=\bigcup_{P'\in \mathcal{P}_r^*} (J_{P'}\cup W_{P'})$.
    Then we have for every subpath $Q$ of $F$ that
    \[
    |E(Q)\cap E(J)|
    \le
    |\mathcal{P}_r^*|
    (2q+ f_{\ref{Lobelia Sackville-Baggins}}(r-1,q) )
    \le
    16r^5(q+1)
    (2q +
    f_{\ref{Lobelia Sackville-Baggins}}(r-1,q) )
    =
    f_{\ref{Lobelia Sackville-Baggins}}(r,q).
    \]
    Consider some $P_1 \in \mathcal{P}_1, \ldots , P_r \in \mathcal{P}_r$.
    Then there exists some $P'\in \mathcal{P}_r^*$ such that $P'$ is a subpath of $P_r'$.
    So, $P'_r$ intersects $P'$, and thus $P_r$ intersects $J_{P'}$.
    If for some $i \in [r-1]$, we have that $P_i\not\in \mathcal{P}_{i,P'}$, then also $P_i$ intersects $J_{P'}$.
    If $P'_r$ is just a single vertex, then so is $P'$, and $J_{P'} \subseteq J$ is connected and therefore contains a path between $P_i$ and $P_r$.
    Otherwise, $P'_r$ is contained in $P_r$ and therefore $J_{P'} \subseteq J$ again contains a path between $P_i$ and $P_r$ because $J_{P'} \cup P'$ is connected.

    So, we may assume that $P_1 \in \mathcal{P}_{1,P'}, \ldots , P_{r-1} \in \mathcal{P}_{r-1,P'}$.
    But then there exist some distinct $i, j \in [r-1]$ such that $W_{P'}$, and therefore $J$, contains a path between $P_i$ and $P_j$.
    Thus, $J$ gives us outcome~\ref{itm:Baggins} of the lemma, as desired.
\end{proof}

This completes the proof of \cref{thm:mainFatkK3}.

\subsection{Additive quasi-isometries} \label{subsec:MountDoom}

In this section we refine \cref{thm:mainFatkK3,main:CoarseEp:Inf} to obtain \cref{thm:mainFatkK3additive,main:CoarseEp:Infadd}, whose quasi-isometries have only an additive error.
We require the following theorem of Berger and Seymour \cite{berger2024bounded} (see \cite{nguyen2025asymptoticpath}), which essentially says that quasi-isometries to forests can be improved to only have an additive error.

\begin{theorem}\label{additivetree}
    For every $\lambda\in \N$ there exists a constant $C\in \N$ such that if there is a $\lambda$-quasi-isometry~$\phi$ from a graph $G$ to a forest $F$, then there is a forest $F'$ obtainable from $F$ by contracting and subdividing edges such that $\phi$ is a $(1,C)$-quasi-isometry from $G$ to $F'$.
\end{theorem}

To deduce \cref{thm:mainFatkK3additive} from \cref{thm:mainFatkK3}, we would like to extend \cref{additivetree} to allow for a bounded number of additional vertices.
Conveniently, a (special, much simpler case of a) lemma of Nguyen, Scott, and Seymour \cite[3.1]{nguyen2025asymptoticpath} provides this inductive step.

\begin{lemma}[\cite{nguyen2025asymptoticpath}] \label{additivevertex}
    Let $\mathcal{H}$ be a minor-closed class of graphs such that for every $\lambda \in \N$ there exists a constant $C \in \N$ such that if there is a $\lambda$-quasi-isometry $\phi$ from a graph $G$ to a graph $H\in \mathcal{H}$, then there is a graph $H'$ obtainable from $H$ by contracting and subdividing edges such that $\phi$ is a $(1,C)$-quasi-isometry from~$G$ to~$H'$.

    Then for every $\lambda^* \in \N$ there exists some $C^* \in \N$ such that if there is a $\lambda^*$-quasi-isometry $\phi$ from a graph $G$ to a graph $H$ that has a vertex $v\in V(H)$ such that $H-v\in \mathcal{H}$, then there is a graph $H'$ obtainable from $H$ by contracting and subdividing edges such that $\phi$ is a $(1,C^*)$-quasi-isometry from $G$ to $H'$.
\end{lemma}

By \cref{additivetree} and repeatedly applying \cref{additivevertex}, we obtain the following.

\begin{lemma}\label{lemma:additive}
    There exists a function $f_{\ref{lemma:additive}} : \N^2 \to \N$ such that the following holds.
    Let $m,\lambda\in \N$, let $\phi$ be a $\lambda$-quasi-isometry from a graph~$G$ to a graph~$H$, and assume there exists a set $X\subseteq V(H)$ such that $H-X$ is a forest and $|X|\le m$.
    Then there is a graph $H'$ obtainable from $H$ by contracting and subdividing edges such that $\phi$ is a $(1,f_{\ref{lemma:additive}}(m,\lambda ))$-quasi-isometry from~$G$ to~$H'$. \qed
\end{lemma}

Now, \cref{thm:mainFatkK3additive} follows from \cref{thm:mainFatkK3} and \cref{lemma:additive}.

\begin{proof}[Proof of \cref{thm:mainFatkK3additive}]
    Let $\defnm{f}_{\ref{thm:mainFatkK3additive}}(k,q) := f_{\ref{lemma:additive}}(g(k),f_{\ref{thm:mainFatkK3}}(k)\cdot q)$ where $g$ is the function from the \EP~theorem \cite{EPTheorem}. By \cref{thm:mainFatkK3}, $G$ is $(f_{\ref{thm:mainFatkK3}}(k)\cdot q)$-quasi-isometric to a graph $H$ with no $k \cdot K_3$ minor. Hence, by applying \cref{lemma:additive} to $G$, $H$ and any set $X \subseteq V(H)$ of size at most $g(k)$ that hits all cycles in $H$, it follows that $G$ is $(1, f_{\ref{thm:mainFatkK3additive}}(k,q))$-quasi-isometric to a graph $H'$ obtained from $H$ by contracting and subdividing edges. As these two operations do not create any new cycles, $k \cdot K_3$ is not a minor of $H'$. This concludes the proof.
\end{proof}

Unfortunately, we cannot simply apply \cref{lemma:additive} to deduce \cref{main:CoarseEp:Infadd} from \cref{main:CoarseEp:Inf} since there is no upper bound on the number of vertices that need to be deleted from~$H$ to obtain a forest.
However, since we just want $H$ to remain $\infty \cdot K_3$ minor-free, we have the advantage that we can add a finite number of additional vertices to $H$ to manually improve the quasi-isometry (instead of applying \cref{additivevertex}).

Roughly speaking, we first apply \cref{main:CoarseEp:Inf} to obtain an $\mathcal{O}(q)$-quasi-isometry $\phi'$ to some graph $H$ with a finite set $X \subseteq V(H)$ such that $F:=H-X$ is a forest. We then use \cref{additivetree} to improve the part of $\phi'$ that maps to the forest $F$ to a quasi-isometry that only has an additive error. We then define a new graph~$H'$, which will essentially be the union of $F$ and $X$ together with some `shortcut set'~$S$, which will be finite because $X$ is finite, and which will ensure that we can extend $\phi'$ to an overall additive quasi-isometry. We now describe the first step.
\smallskip

Consider a graph $G$ that is $\lambda$-quasi-isometric to some graph $H$ with a finite set $X\subseteq V(H)$ such that $H-X$ is a forest.
By \cref{prop:q.i.-graph-dec}~\ref{itm:q.i.-graph-dec:QItoDec}, $G$ admits an honest $(r_1,r_2)$-radial $H$-decomposition, with $r_1,r_2$ depending only on $\lambda$.
For each $x\in X$, choose some $u_x\in V_x$ and let $U=\{u_x:x\in X\}$.
Then, $G$ admits an $(r_1,r_2)$-radial partial $(H-X)$-decomposition with support $V(G-B_G(U, 2r_1))$.
By deleting vertices and edges of $H-X$ appropriately, we can make this partial graph-decomposition honest.
Let $G'$ be the graph obtained from $G - B_G(U, 2r_1)$ by adding, for every bag $V_h$, a vertex $v_h$ and pairwise internally-vertex-disjoint paths of length $r_1$ from $v_h$ to all vertices in $V_h$.
Then, $G'$ has an honest $(r_1, r_2)$-radial \fd\ (with distances now measured in~$G'$) indexed by $H-X$ such that
    \begin{enumerate}[label=\rm{(\arabic*)}]
        \item \label{itm:Mordor:1} $d_G(u, v) \leq d_{G'}(u,v) \leq  d_{G-B_G(U, 2r_1)}(u,v)$ for every $u,v \in V(G-B_G(U, 2r_1))$, and
        \item \label{itm:Mordor:2} for every $z\in V(G')$, there exists some $v\in V(G-B_G(U, 2r_1))$ with $d_{G'}(z,v) \le r_1$.
    \end{enumerate}
Moreover, by \cref{prop:q.i.-graph-dec}~\ref{itm:q.i.-graph-dec:DecToQI}, $G'$ is $\lambda_{\ref{prop:q.i.-graph-dec}}(r_1,r_2)$-quasi-isometric to a forest~$F$.
So, by applying \cref{additivetree} to $G'$ and $F$, and combining it with \ref{itm:Mordor:1} and \ref{itm:Mordor:2}, we have the following lemma.

\begin{lemma}\label{bored}
    There exists a function $f_{\ref{bored}} : \N \to \N$ such that the following holds.
    Let $\lambda \in \N$ and let $G$ be a graph.
    If $G$ is $\lambda$-quasi-isometric to a graph $H$ with a finite vertex set $X\subseteq V(H)$ such that $H-X$ is a forest, then there exists some vertex set $U\subseteq V(G)$ with $|U| \le |X|$, a forest~$F$ and a function $\phi: V(G-B_G(U,f_{\ref{bored}}(\lambda))) \to V(F)$ such that:
    \begin{enumerate}[label=\rm{(\arabic*)}]
    \item $d_G(u, v) - f_{\ref{bored}}(\lambda) \leq d_{F}(\phi(u),\phi(v)) \leq  d_{G-B_G(U,f_{\ref{bored}}(\lambda))}(u,v) + f_{\ref{bored}}(\lambda)$\\ for every~$u,v \in V(G-B_G(U,f_{\ref{bored}}(\lambda)))$, and
    \item for every vertex $h$ of $F$, there exists some $v \in V(G-B_G(U,f_{\ref{bored}}(\lambda)))$ such that $d_{F}(h,\phi(v)) \leq f_{\ref{bored}}(\lambda)$. \qed
\end{enumerate}
\end{lemma}

By \cref{main:CoarseEp:Inf}, the (infinite) \EP~theorem \cite{EPTheorem}, and \cref{bored}, we obtain the following lemma.

\begin{lemma}\label{addkK3inf}
    There exists a function $f_{\ref{addkK3inf}} : \N \to \N$ such that the following holds.
    Let $q\in \N$ and let $G$ be a graph with no $q$-fat model of $\infty \cdot K_3$.
    Then there exists a finite vertex set $U\subseteq V(G)$, a forest $F$, and a function $\phi: V\big(G-B_G(U,f_{\ref{addkK3inf}}(q))\big) \to V(F)$ such that:
    \begin{enumerate}[label=\rm{(\arabic*)}]
    \item $d_G(u, v) - f_{\ref{addkK3inf}}(q) \leq d_{F}(\phi(u),\phi(v)) \leq  d_{G-B_G(U,f_{\ref{addkK3inf}}(q))}(u,v) + f_{\ref{addkK3inf}}(q)$\\ for every~$u,v \in V(G-B_G(U,f_{\ref{addkK3inf}}(q)))$, and
    \item\label{itm:addkK3inf-2} for every vertex $h$ of $F$, there exists some $v \in V(G-B_G(U,f_{\ref{addkK3inf}}(q)))$ such that $d_{F}(h,\phi(v)) \leq f_{\ref{addkK3inf}}(q)$. \qed
\end{enumerate}
\end{lemma}

We are now ready to prove \cref{main:CoarseEp:Infadd}.

\begin{proof}[Proof of \cref{main:CoarseEp:Infadd}.]
For each $q \in \N$, set
\[
\defnm{f}_{\ref{main:CoarseEp:Infadd}}(q)
\coloneqq
8f_{\ref{addkK3inf}}(q) +4
.
\]

Let $\defnm{U},\defnm{F},\defnm{\phi}$ be given by \cref{addkK3inf} applied to~$G$. Set $\defnm{B} := B_G(U, f_{\ref{addkK3inf}}(q))$.
For every distinct $u,v\in U$ within the same component of $G$, let $P_{u,v}$ be a shortest path in $G$ between $u$ and $v$.
Let \defn{$S$ }be the subgraph of $G$ consisting of the vertices of $U$ and the union of all such paths~$P_{u,v}$.
Note in particular that $S$ is a finite subgraph.
Now let \defn{$H$} be obtained from the disjoint union of $S$ and~$F$ by
\begin{enumerate}[label=\rm{(\roman*)}]
    \item \label{itm:Mordor:i} adding a path of length $2f_{\ref{addkK3inf}}(q)+1$ between $u$ and $\phi(v)$\\
    for each $u\in U$ and $v\in V( G-B)$ with $d_G(u,v) = f_{\ref{addkK3inf}}(q) +1$,
    \item \label{itm:Mordor:ii} adding a path of length $2f_{\ref{addkK3inf}}(q)+1$ between $s$ and $\phi(s)$\\
    for each $s\in V(S)\setminus B$, and
    \item \label{itm:Mordor:iii} adding a path of length $d_G(U,s)$ between $s$ and each $u\in U$ with $d_G(u,s) = d_G(U,s)$\\
    for each $s\in V(S) \cap B$.
\end{enumerate}
Note that $H$ is $\infty \cdot K_3$ minor-free since $V(S)$ is finite and $H-V(S)$ is a forest.
\smallskip

Let $\defnm{\psi}: V(G) \to V(H)$ be defined by $\psi(u):=u$ for every $u\in U$, $\psi(v):=\phi(v)$ for every $v\in V( G-B)$, and for every $v\in B \setminus U$ we let $\psi(v):=u$ for some $u\in U$ with $d_G(u,v) \le f_{\ref{addkK3inf}}(q)$.

Note that for every $u,v \in U$, we have that $d_H(\psi(u), \psi(v)) \le d_{G}(u,v)$.
\smallskip

We now show that $\psi$ is a $(1, f_{\ref{main:CoarseEp:Infadd}}(q))$-quasi-isometry from $G$ to $H$.

\noindent \ref{quasiisom:2}: Clearly for every vertex $h$ of $H$, there exists a vertex $v$ of $G$ such that $d_{H}(h,\psi(v)) \leq 2 f_{\ref{addkK3inf}}(q) +1 \le f_{\ref{main:CoarseEp:Infadd}}(q)$. (This is because $\phi$ satisfies \ref{itm:addkK3inf-2} and because of \ref{itm:Mordor:ii}.)
\smallskip

\noindent \ref{quasiisom:1}: Consider some $u,v\in V(G)$. We first show that $d_H(\psi(u), \psi(v)) \leq d_G(u,v) + f_{\ref{main:CoarseEp:Infadd}}(q)$. For this, let $P$ be a shortest path between $u$ and $v$ in $G$. If $u,v \in V(G-B)$, then $\psi(u), \psi(v) \in V(F)$, and hence
\[
d_{H}(\psi(u),\psi(v))
\overset{H \supseteq F}{\le}
d_{F}(\psi(u),\psi(v)) \overset{\ref{addkK3inf}}{\leq}  d_{G-B}(u,v) + f_{\ref{addkK3inf}}(q).
\]
In particular, if $P$ is a subpath of $G-B$, then $d_{G-B}(u,v) =
d_{G}(u,v)$ and thus
\begin{equation} \label{eq:Mordor:1}
d_{H}(\psi(u),\psi(v)) \leq d_{G}(u,v) + f_{\ref{addkK3inf}}(q),
\end{equation}
so we may assume otherwise.

Starting from $u$, let $u^*$ be the first vertex of $P$ contained in $B_G(B,1)$, and let $v^*$ be the last.
If $u \notin B$, then the subpath of $P$ between $u$ and $u^*$ is a subpath of $G-B$.
In particular, either $u \in B$ and then $u=u^*$ and thus $d_H(\psi(u),\psi(u^*)) = d_G(u,u^*)$, or $u\notin B$ and then
$d_H(\psi(u),\psi(u^*))\leq d_{G}(u,u^*) + f_{\ref{addkK3inf}}(q)$ by \eqref{eq:Mordor:1}.
Similarly, $d_H(\psi(v),\psi(v^*))\leq d_{G}(v,v^*) + f_{\ref{addkK3inf}}(q)$.

Let $u'$ be a vertex of $U$ closest to $u^*$ in $G$, and similarly, let $v'$ be a vertex of $U$ closest to $v^*$ in~$G$.
If $u^* \notin B$, then $d_H(\psi(u^*), \psi(u')), d_H(\psi(v'), \psi(v^*)) \le  2f_{\ref{addkK3inf}}(q) + 1$ by \ref{itm:Mordor:i}.
If $u^*\in B$, then both $\psi(u^*)$ and the closest $u'$ lie in $U$, and $d_H(\psi(u^*),u')\le d_G(\psi(u^*),u') \le 2f_{\ref{addkK3inf}}(q)$.
Moreover, $d_H(\psi(u'), \psi(v')) \le d_G(u',v')$ by the choice of~$S$.
Hence,
\begin{align*}
    d_H(\psi(u^*),\psi(v^*))
&\le
d_H(\psi(u^*),\psi(u'))
+
d_H(\psi(u'),\psi(v'))
+
d_H(\psi(v'),\psi(v^*)) \\
&\le
d_G(u',v') +
4f_{\ref{addkK3inf}}(q) +2 \le d_G(u^*,v^*) + 6f_{\ref{addkK3inf}}(q)+4,
\end{align*}
where we used that $d_G(u',v') \leq d_G(u^*,v^*) + 2f_{\ref{addkK3inf}}(q) + 2$ by the triangle inequality. Therefore,
\begin{align*}
    d_H(\psi(u),\psi(v))
&\le
d_H(\psi(u),\psi(u^*))
+
d_H(\psi(u^*),\psi(v^*))
+
d_H(\psi(v^*),\psi(v)) \\
&\le
d_G(u,u^*) + d_G(u^*,v^*) + d_G(v^*,v) +
8f_{\ref{addkK3inf}}(q) +4 \\
&=
d_G(u,v) +
8f_{\ref{addkK3inf}}(q) +4 = d_G(u, v) + f_{\ref{main:CoarseEp:Infadd}}(q),
\end{align*}
which completes the proof of the upper bound in \ref{quasiisom:1}.
\medskip

We now show that $d_H(\psi(u), \psi(v)) \geq d_G(u,v) - f_{\ref{main:CoarseEp:Infadd}}(q)$. For this, let $Q$ be a shortest path between $\psi(u)$ and $\psi(v)$ in $H$.
We will replace $Q$ by a walk in $G$.

Call the paths added in \ref{itm:Mordor:i} and \ref{itm:Mordor:ii} \emph{passes}.
Each pass $P$ joins a vertex $\defnm{s_P}\in V(S)$ to $\phi(a_P)$ for some $\defnm{a_P}\in V(G-B)$ with $d_G(s_P,a_P)\leq f_{\ref{addkK3inf}}(q)+1$.
Since $P$ has length $2f_{\ref{addkK3inf}}(q)+1$, replacing it by a shortest $s_P$--$a_P$ path in $G$ saves at least $f_{\ref{addkK3inf}}(q)$ in length.
On the other hand, a path $R$ in $F$ between $\phi(x)$ and $\phi(y)$ can be replaced by an $x$--$y$ path in $G$ at an additional cost of at most $f_{\ref{addkK3inf}}(q)$, since \cref{addkK3inf} gives
\[
d_G(x,y)
\leq d_F(\phi(x),\phi(y))+f_{\ref{addkK3inf}}(q)
\leq |\!|R|\!|+f_{\ref{addkK3inf}}(q).
\]

If $Q$ is contained in $F$, this already gives $d_G(u,v)\leq |\!|Q|\!|+f_{\ref{addkK3inf}}(q)$.
Otherwise, split $Q$ at its vertices in $S$.
Each resulting segment that is vertex-disjoint from $F$ is either an edge of $S$ or a path added in \ref{itm:Mordor:iii}.
In both cases, it can be replaced by a path in $G$ of the same length and with the same endvertices.

Each remaining segment consists of a subpath $R$ of $F$, possibly a single vertex, together with one or two passes.
To apply the replacements above, choose the preimage at each end of $R$ as follows: use $a_P$ if that end is incident with a pass $P$, and use $u$ or $v$ if it is the corresponding end of $Q$.
Then, the additional cost of at most $f_{\ref{addkK3inf}}(q)$ for $R$ is offset by the saving of at least $f_{\ref{addkK3inf}}(q)$ from a pass.
Thus, the whole segment can be replaced by a path in $G$ without increasing its length.

The replacements agree at the vertices of $S$, so they concatenate to a walk in $G$ of length at most $|\!|Q|\!|$.
At either end, if the original vertex $x\in\{u,v\}$ belongs to $B$, this walk ends at $\psi(x)$ rather than $x$.
We therefore attach an $x$--$\psi(x)$ path in $G$ of length at most $f_{\ref{addkK3inf}}(q)$.
The resulting walk joins $u$ to $v$, and hence, in all cases,
\[
d_G(u,v)
\leq |\!|Q|\!|+2f_{\ref{addkK3inf}}(q)
= d_H(\psi(u),\psi(v))+2f_{\ref{addkK3inf}}(q)
\leq d_H(\psi(u),\psi(v))
     +f_{\ref{main:CoarseEp:Infadd}}(q),
\]
as required.
\end{proof}

\section{Length and geodesic metric spaces} \label{sec:lengthspaces}

In this short penultimate section, we discuss versions of our results for geodesic spaces and length spaces.

A \defn{geodesic metric space} is a metric space $(X,d)$ such that for every $x,y\in X$ there exists a geodesic between $x$ and $y$.
A geodesic metric space is \defn{proper} if every closed ball is compact.
A \defn{length space} is a metric space $(X,d)$ such that for every $x, y \in  X$ and $\varepsilon > 0$
there is an $x$--$y$ arc of length at most $d(x, y) + \varepsilon$. Thus every geodesic metric space is a length space.
In particular, the geometric realization of every connected graph, with each edge assigned length $1$, is a length space.
We sketch how to extend our results to length spaces.

It is straightforward and  well known (see for example \cite{GPCoarseGT}) that every length space is quasi-isometric to a connected graph (with uniform constants).
In fact, every length space $(X, d)$ is $1$-quasi-isometric to the graph with vertex set $X$ where $x,y\in X$ are adjacent if $d(x,y)\le 1$.
\cref{fatfat} also holds for length spaces~\cite{GPCoarseGT}.
Thus, for every $q$, there exists $q'$, depending only on $q$ and the fixed quasi-isometry constants, such that whenever length spaces $X$ and $Y$ are quasi-isometric with these constants, if $X$ contains a graph $H$ as a $q'$-fat minor, then $Y$ also contains $H$ as a $q$-fat minor.
From these two observations, the length space extensions of \cref{th:main,main:CoarseEp:Inf,thm:mainFatkK3additive,thm:mainFatkK3}, and therefore of \cref{cor:asdim1,cor:asdim2} follow.
The length space version of \cref{main:CoarseErdosPosa} can be proven using (the length space version of) \cref{th:main} and \cref{lem:helly}.

We now prove \cref{main:manningextension} using \cref{main:CoarseEp:Inf} (or rather its extension to length spaces).

\manningextension*

\begin{proof}[Proof of \cref{main:manningextension}.]
    First we show that the first three bullets are equivalent.
    By \cref{fatfat}, the first bullet implies the second.
    Clearly the second bullet implies the third.
    Suppose now that for some $q>0$, $G$ does not contain $\infty \cdot K_3$ as a $q$-fat minor.
    Now, by \cref{main:CoarseEp:Inf} and the (infinite) \EP~theorem~\cite{EPTheorem}, $G$ is quasi-isometric to a connected graph $H$ containing a finite vertex set $X$ such that $H-X$ is a forest.
    If $|X|=0$, then $H$ is simply a tree as we desired, so we may assume that $|X|\ge 1$.
    Let $R$ be a finite connected subgraph of $H$ that contains $X$.
    Let $H'=H/E(R)$, where the contracted vertex is $v$.
    Then $H'-v$ is a forest and $H$ is quasi-isometric to $H'$.
    By adding some edges to $H'$ between neighbours of $v$, we can obtain a graph $H''$ quasi-isometric to $H'$ and such that $H''-v$ is a tree.
    Then, $G$ is quasi-isometric to $H''$.
    Hence, the first three bullets are equivalent.

    So, we now assume that $G$ is a proper geodesic metric space.
    The equivalence of the fourth and fifth bullets is (essentially) Manning's theorem \cite{Manning05} (see \cite{GPCoarseGT}).
    Clearly the fifth bullet implies the first three.
    So, it remains to show that the first bullet implies the fourth.

    First we must observe that proper geodesic metric spaces are quasi-isometric to locally finite graphs.
    Let $U$ be a subset of vertices of $G$ that are pairwise at distance at least $1/3$, such that $U$ is maximal with this property.
    Observe that every point $z$ of $G$ is at distance less than $1/3$ from $U$.
    Then it is easy to show that $G$ is quasi-isometric to the graph $J$ on vertex set $U$ where $x,y\in U$ are adjacent if they are at distance at most $1$ in $G$ (see e.g. \cite{GPCoarseGT}).
    Consider some $x\in U$ and the (compact) closed ball $B=B_G(x,1)$.
    We claim that $B\cap U$ is finite.
    Indeed, otherwise we could choose an infinite sequence of distinct points of $B\cap U$.
    By compactness, this sequence would have a convergent subsequence, which is impossible since any two distinct points of $U$ are at distance at least $1/3$.
    Therefore, $J$ is locally finite.

    Suppose now that $G$ is quasi-isometric to a graph $H$ containing a vertex $v$ such that $H-v$ is a tree.
    By applying the construction from \cite[Lemma~13]{EG24} to a quasi-isometry from $J$ to $H$, we may further assume that $H$ is locally finite.
    Let $S$ be the subgraph of $H$ consisting of all of its cycles, and let $H/E(S)$ be the graph obtained from $H$ by contracting all edges of~$S$. Clearly $H/E(S)$ is a tree, so it remains to show that $H$ and $H/E(S)$ are quasi-isometric.
    For this it is enough to show that $S$ is a finite subgraph.
    This follows from the fact that all cycles of $H$ contain $v$, there is a unique path in $H-v$ between any two vertices of $H-v$, and $v$ has finite degree.
\end{proof}

\section{Discussion}\label{sec:discuss}

We conclude the paper with a discussion of some open problems.

\paragraph{The fat minor conjecture for forests.}

In \cref{sec:BattleOfTheBlackGate}, we bootstrapped \cref{th:main} to prove \cref{thm:mainFatkK3}, thereby establishing \cref{conj:qi} for $H=k\cdot K_3$.
A similar strategy may apply when $H$ is a forest.
The path-width theorem~\cite{robertson1983graph} provides a natural starting point: graphs excluding a fixed forest as a minor have bounded path-width, and conversely, every class of bounded path-width excludes some forest as a minor.

Nguyen, Scott, and Seymour~\cite{NSSFatTree} confirmed that the path-width theorem holds in a coarse way: graphs excluding a fixed forest as a fat minor are quasi-isometric to graphs of bounded path-width.
Thus, as in our setting, one can first pass to a quasi-isometric graph with a restricted structure.
The question is whether this approximation can be refined to exclude the prescribed forest itself.
Like the graphs with a feedback vertex set of bounded size that arise in our setting, graphs of bounded path-width have a relatively simple structure, giving some reason to hope that such a refinement is possible.
As an illustration of this simplicity, the coarse Menger conjecture~\cite{AHJKWCoarseMenger,GPCoarseGT} holds for graphs of bounded path-width~\cite{divoux2025asymptotic} but fails in general~\cite{nguyen2025asymptotic,nguyen2025counterexample}.
We therefore conjecture that \Cref{conj:qi} holds when $H$ is a forest.

\begin{conjecture}\label{conj:pathwidth}
    For every forest $F$, there exists a constant $\lambda_F$ such that the following holds.
    Let $q\in\N$ and let $G$ be a graph with no $q$-fat model of $F$.
    Then $G$ is $(\lambda_F\cdot q)$-quasi-isometric to a graph with no $F$ minor.
\end{conjecture}

There is also reason to ask for an additive version of this conjecture, in analogy with \cref{thm:mainFatkK3additive}.
Nguyen, Scott, and Seymour~\cite{nguyen2025asymptoticpath} showed that if a graph $G$ is $(M,A)$-quasi-isometric to a graph $H$ of path-width at most $k$, then $G$ is $(1,A')$-quasi-isometric to a graph $H'$ that is a minor of a subdivision of $H$, for some constant $A'$ depending only on $M,A,k$.
However, subdivisions can introduce new forest minors, so this result does not directly show that \cref{conj:pathwidth} implies its additive analogue.
Nevertheless, it suggests the following conjecture.

\begin{conjecture}\label{conj:pathwidth2}
    For every forest $F$, there exists a function $f_F:\N\to\N$ such that the following holds.
    Let $q\in\N$ and let $G$ be a graph with no $q$-fat model of $F$.
    Then $G$ is $(1,f_F(q))$-quasi-isometric to a graph with no $F$ minor.
\end{conjecture}

\medskip

\paragraph{Coarse \EP\ properties.}

We say that a graph $H$ has the \defn{weak coarse \EP\ property} if there exist functions $f:\N\to\N$ and $g:\N^2\to\N$ such that, for every graph $G$ and all $q,k\in\N$, either $G$ contains a $q$-fat model of $k\cdot H$, or there is a set $X\subseteq V(G)$ with $|X|\le f(k)$ such that $B_G(X,g(q,k))$ meets every $q$-fat model of $H$ in $G$.
If the radius can be bounded by a function of $q$ alone, then $H$ has the \defn{coarse \EP\ property}.
Thus, \cref{main:CoarseErdosPosa} establishes that every fixed cycle $C_{\ell}$ has the coarse \EP\ property.

As with the ordinary \EP\ property for minor models~\cite{robertson1986graph}, planarity is necessary for the (weak) coarse \EP\ property.
Unlike in the ordinary setting, however, planarity is not sufficient: Albrechtsen and Davies~\cite{ADWeakCounterex} exhibited a planar graph that does not have the weak coarse \EP\ property.
It is therefore natural to ask which (planar) graphs have the weak coarse \EP\ property, and which have the coarse \EP\ property.

The coarse path-width theorem of Nguyen, Scott, and Seymour~\cite{NSSFatTree} implies that every forest has the weak coarse \EP\ property\footnote{Their theorem states that, for every forest $F$, every graph with no $q$-fat model of $k\cdot F$ admits a path-decomposition in which each bag can be covered by at most $b_F(k)$ balls of radius $r_F(q,k)$.
A packing--hitting argument analogous to those in \cref{sec:FinalProof} then yields at most $f_F(k)$ balls of radius $g_F(q,k)$ meeting every $q$-fat model of $F$.}.
We believe that the dependence on $k$ in the radius of the hitting sets can be removed.

\begin{conjecture}
    Every forest has the coarse \EP\ property.
\end{conjecture}

Beyond forests and cycles, a natural next case to consider is $K_{2,3}$.
Here, \cref{conj:qi} is already known to hold: graphs with no $q$-fat model of $K_{2,3}$ are quasi-isometric to cacti~\cite{fujiwara2023coarse,AJKWFatK4}.
This provides a possible starting point for the following conjecture.

\begin{conjecture}
    $K_{2,3}$ has the coarse \EP\ property.
\end{conjecture}

\medskip

\paragraph{Far apart minor models.}

A different variant of the coarse \EP\ property asks for minor models that are pairwise far apart, without requiring them to be fat.
We say that a graph $H$ has the \defn{far apart \EP\ property} if there exist functions $f,g:\N\to\N$ such that, for every graph $G$ and all $q, k\in\N$, either $G$ contains $k$ minor models of $H$ that are pairwise at distance at least $q$, or there is a set $X\subseteq V(G)$ with $|X|\le f(k)$ such that $G-B_G(X,g(q))$ has no $H$ minor.

Dujmović, Joret, Micek, and Morin~\cite{DJMMDistanceErdosPosa} proved that $K_3$ has the far apart \EP\ property, with a function $g$ linear in $q$.
Chudnovsky, Dujmović, Joret, R.~Kaul, Micek, Morin, and Scott~\cite{CDGKMMS} recently extended this result to every fixed cycle $C_{\ell}$; see also \cref{maincor:FarApartEPForLongCycles}.

Albrechtsen and Davies~\cite[Conjecture~7.6]{ADWeakCounterex} recently conjectured the following extension to all planar graphs, again with a function $g$ linear in $q$.

\begin{conjecture}[\cite{ADWeakCounterex}]
    For every planar graph $H$, there exist a function $f_H:\N\to\N$ and a constant~$C_H$ such that the following holds for every graph $G$ and all $q,k\in\N$.
    If $G$ does not contain $k$ minor models of $H$ that are pairwise at distance at least $q$, then there is a set $X\subseteq V(G)$ with $|X|\le f_H(k)$ such that $G-B_G(X,C_H \cdot q)$ has no $H$ minor.
\end{conjecture}

The case $q=2$, in which the models are required to be pairwise anticomplete, was recently proved by Chudnovsky, Reinald, and Thomassé~\cite{CRTAnticompletePlanarEP}, with balls of radius~$1$.

Motivated by \cref{maincor:FarApartEPForLongInducedCycles}, we also propose the corresponding statement for induced minors.

\begin{conjecture} \label{conj:InducedEP}
    For every planar graph $H$, there exist a function $f_H:\N\to\N$ and a constant~$C_H$ such that the following holds for every graph $G$ and all $q,k\in\N$.
    If $G$ does not contain $k$ induced minor models of $H$ that are pairwise at distance at least $q$, then there is a set $X\subseteq V(G)$ with $|X|\le f_H(k)$ such that $G-B_G(X,C_H \cdot q)$ contains no $H$ induced minor.
\end{conjecture}

The weaker versions of these two conjectures, replacing the radius $C_H \cdot q$ by $g(q,k)$ for some function $g:\N^2\to\N$, would already be of interest.
A possible approach is to adapt Robertson and Seymour’s proof of the \EP\ property for planar minor models~\cite{robertson1986graph}: show that graphs without the desired packing are quasi-isometric to graphs of bounded tree-width, and then use a \td\ of small width to obtain the required hitting set.
The counterexample of Albrechtsen and Davies~\cite{ADWeakCounterex} rules out the analogous strategy for fat-minor exclusion.

\subsection*{Statement of AI Use}

We used ChatGPT 5.6 Sol and ChatGPT 6 Ultra only to help with making minor edits to the final paper.

\subsection*{Acknowledgements}

We thank Vida Dujmović, Gwenaël Joret, Piotr Micek, and Pat Morin for helpful comments.

\bibliographystyle{alpha}
\bibliography{coolnew}

@article{DHIMFatCounterexample,
    author={Davies, James and Hickingbotham, Robert and Illingworth, Freddie and McCarty, Rose},
    title={Fat minors cannot be thinned (by quasi-isometries)},
    year={2026},
    journal={Analysis and Geometry in Metric Spaces},
    volume = {14},
    number = {1},
    pages = {20250036}
}

@article{AJKWFatK4,
    author = {Albrechtsen, Sandra and Jacobs, Raphael W. and Knappe, Paul and Wollan, Paul},
    title = {A characterisation of graphs quasi-isometric to ${K}_4$-minor-free graphs},
    journal = {Combinatorica},
    volume = {45},
    pages = {61},
    year = {2025}
}

@article{gyarfas1970helly,
  title={A {H}elly-type problem in trees},
  author={Gy{\'a}rf{\'a}s, Andr{\'a}s and Lehel, Jen\H{o}},
  journal={Combinatorial Theory and its Applications},
  volume={4},
  pages={571--584},
  year={1970},
  publisher={North-Holland, Amsterdam}
}

@article{GPCoarseGT,
  author = {Agelos Georgakopoulos and Panos Papasoglu},
  title  = {Graph minors and metric spaces},
  journal = {Combinatorica},
  volume = {45},
  pages = {33},
  year = {2025}}

@article{Manning05,
author = {Jason Fox Manning},
title = {{Geometry of pseudocharacters}},
volume = {9},
journal = {Geometry \& Topology},
number = {2},
publisher = {MSP},
pages = {1147--1185},
year = {2005},
doi = {10.2140/gt.2005.9.1147},
URL = {https://doi.org/10.2140/gt.2005.9.1147}
}

@article{ADEFJKW23,
      title={A structural duality for path-decompositions into parts of small radius}, 
      author={Sandra Albrechtsen and Reinhard Diestel and Ann-Kathrin Elm and Eva Fluck and Raphael W. Jacobs and Paul Knappe and Paul Wollan},
      journal = {Innovations in Graph Theory},
      volume = {3},
      pages = {207--246},
      year={2026}
}

@unpublished{ADWeakCounterex,
    author = {Albrechtsen, Sandra and Davies, James},
    title = {Counterexample to the conjectured coarse grid theorem},
    note={\arxiv{2508.15342}},
    year = {2025}
}

@inproceedings{ahn2025coarse,
  title={A coarse {E}rd{\H{o}}s-{P}{\'o}sa theorem},
  author={Ahn, Jungho and Gollin, J. Pascal and Huynh, Tony and Kwon, O-joung},
  booktitle={Proceedings of the 2025 Annual ACM-SIAM Symposium on Discrete Algorithms (SODA)},
  pages={3363--3381},
  year={2025},
  organization={SIAM}
}

@unpublished{ADGFatK2t,
    author = {Albrechtsen, Sandra and Distel, Marc and Georgakopoulos, Agelos},
    title = {Excluding ${K}_{2,t}$ as a fat minor},
    note = {\arxiv{2510.14644}},
    year = {2025}
}

@unpublished{NSSFatTree,
    author = {Nguyen, Tung and Scott, Alex and Seymour, Paul},
    title = {Asymptotic structure. {III}. {E}xcluding a fat tree},
    note = {\arxiv{2509.09035}},
    year = {2025}
}

@article{EPTheorem,
    author = {Erd{\H{o}}s, Paul and P{\'o}sa, Lajos},
    title = {On independent circuits contained in a graph},
    journal = {Canadian Journal of Mathematics},
    volume = {17},
    pages = {347-352},
    year = {1965}
}

@unpublished{ADGSmallCounterexFatMinorConj,
    author = {Albrechtsen, Sandra and Distel, Marc and Georgakopoulos, Agelos},
    title = {Small counterexamples to the fat minor conjecture},
    note = {\arxiv{2601.05761}},
    year = {2026}
}

@unpublished{CMPRRAntiCompleteEPforLongHoles,
    author = {Czyżewska, Jadwiga and Masařík, Tomáš and Pilipczuk, Marcin and Reinald, Amadeus and Rzążewski, Paweł},
    title = {Induced {E}rd{\H{o}}s--{P}{\'o}sa property for long holes, long thetas, and beyond},
    note ={\arxiv{2607.07697}},
    year = {2026}
}

@unpublished{CDGKMMS,
    author = {Chudnovsky, Maria and Dujmović, Vida and Joret, Gwenaël and Kaul, Raj and Micek, Piotr and Morin, Pat and Scott, Alex},
    title = {Far-apart {E}rd{\H{o}}s--{P}{\'o}sa property of long cycles},
    note = {\arxiv{2607.12136}},
    year = {2026}
}

@article{BBEGLPSAsymptoticDimMinorClosed,
  title={Asymptotic dimension of minor-closed families and {Assouad--Nagata} dimension of surfaces},
  author={Bonamy, Marthe and Bousquet, Nicolas and Esperet, Louis and Groenland, Carla and Liu, Chun-Hung and Pirot, Fran{\c{c}}ois and Scott, Alexander},
  journal={Journal of the European Mathematical Society},
  volume={26},
  number={10},
  pages={3739--3791},
  year={2024}
}

@misc{HIMWAsDimPlanarFat,
    author = {Walczak, Bartosz},
    title = {Asymptotic dimension of graphs excluding an induced or fat minor},
    note = {Seminar talk at Jagiellonian University given on June 10, 2026},
    howpublished = {\url{https://www.youtube.com/watch?v=x7DFkXGBdJk}}
}

@article{simonovits1967new,
  title={A new proof and generalizations of a theorem of {E}rd{\H{o}}s and {P}{\'o}sa on graphs without $k+ 1$ independent circuits},
  author={Simonovits, Mikl{\'o}s},
  journal={Acta Mathematica Hungarica},
  volume={18},
  number={1-2},
  pages={191--206},
  year={1967},
  publisher={Akad{\'e}miai Kiad{\'o}, co-published with Springer Science+ Business Media BV~…}
}

@article{robertson1986graph,
  title={Graph minors. {V}. {E}xcluding a planar graph},
  author={Robertson, Neil and Seymour, Paul},
  journal={Journal of Combinatorial Theory, Series B},
  volume={41},
  number={1},
  pages={92--114},
  year={1986},
  publisher={Elsevier}
}

@article{liu2022packing,
  title={Packing topological minors half-integrally},
  author={Liu, Chun-Hung},
  journal={Journal of the London Mathematical Society},
  volume={106},
  number={3},
  pages={2193--2267},
  year={2022},
  publisher={Wiley Online Library}
}

@article{thomassen1988presence,
  title={On the presence of disjoint subgraphs of a specified type},
  author={Thomassen, Carsten},
  journal={Journal of Graph Theory},
  volume={12},
  number={1},
  pages={101--111},
  year={1988},
  publisher={Wiley Online Library}
}

@unpublished{fujiwara2023coarse,
  title={A coarse-geometry characterization of cacti},
  author={Fujiwara, Koji and Papasoglu, Panos},
  note={\arxiv{2305.08512}},
  year={2023}
}

@article{fujiwara2007note,
  title={A note on spaces of asymptotic dimension one},
  author={Fujiwara, Koji and Whyte, Kevin},
  journal={Algebraic \& Geometric Topology},
  volume={7},
  number={2},
  pages={1063--1070},
  year={2007},
  publisher={Mathematical Sciences Publishers}
}

@article{hagen2014weak,
  title={Weak hyperbolicity of cube complexes and quasi-arboreal groups},
  author={Hagen, Mark F.},
  journal={Journal of Topology},
  volume={7},
  number={2},
  pages={385--418},
  year={2014},
  publisher={Oxford University Press}
}

@article{bestvina2015constructing,
  title={Constructing group actions on quasi-trees and applications to mapping class groups},
  author={Bestvina, Mladen and Bromberg, Ken and Fujiwara, Koji},
  journal={Publications math{\'e}matiques de l'IH{\'E}S},
  volume={122},
  number={1},
  pages={1--64},
  year={2015},
  publisher={Springer}
}

@article{behrstock2017hierarchically,
  title={Hierarchically hyperbolic spaces, {I}: Curve complexes for cubical groups},
  author={Behrstock, Jason and Hagen, Mark and Sisto, Alessandro},
  journal={Geometry \& Topology},
  volume={21},
  number={3},
  pages={1731--1804},
  year={2017},
  publisher={Mathematical Sciences Publishers}
}

@article{benjamini2022triangulations,
  title={Triangulations of uniform subquadratic growth are quasi-trees},
  author={Benjamini, Itai and Georgakopoulos, Agelos},
  journal={Annales Henri Lebesgue},
  volume={5},
  pages={905--919},
  year={2022}
}

@unpublished{margolis2024coarse,
  title={Coarse homological invariants of metric spaces},
  author={Margolis, Alexander},
  note={\arxiv{2411.04745}},
  year={2024}
}

@article{nguyen2025counterexample,
  title={A counterexample to the coarse {M}enger conjecture},
  author={Nguyen, Tung and Scott, Alex and Seymour, Paul},
  journal={Journal of Combinatorial Theory, Series B},
  volume={173},
  pages={68--82},
  year={2025},
  publisher={Elsevier}
}

@unpublished{nguyen2025asymptotic,
  title={Asymptotic structure. {I}{V}. {A} counterexample to the weak coarse {M}enger conjecture},
  author={Nguyen, Tung and Scott, Alex and Seymour, Paul},
  note={\arxiv{2508.14332}},
  year={2025}
}

@article{robertson1983graph,
  title={Graph minors. {I}. {E}xcluding a forest},
  author={Robertson, Neil and Seymour, Paul},
  journal={Journal of Combinatorial Theory, Series B},
  volume={35},
  number={1},
  pages={39--61},
  year={1983},
  publisher={Elsevier}
}

@article{bell2008asymptotic,
  title={Asymptotic dimension},
  author={Bell, Greg and Dranishnikov, Alexander},
  journal={Topology and its Applications},
  volume={155},
  number={12},
  pages={1265--1296},
  year={2008},
  publisher={Elsevier}
}

@misc{RaymondDynamic,
  author = {Raymond, Jean-Florent},
  title = {Dynamic {E}rd{\H{o}}s--{P}{\'o}sa Listing},
  howpublished = {\url{https://perso.ens-lyon.fr/jean-florent.raymond/Erdős-Pósa/}},
  note = {Accessed: 2026-08-12}
}

@inproceedings{bonamy2025local,
  title={Local constant approximation for dominating set on graphs excluding large minors},
  author={Bonamy, Marthe and Gavoille, Cyril and Picavet, Timoth{\'e} and Wesolek, Alexandra},
  booktitle={Proceedings of the ACM Symposium on Principles of Distributed Computing},
  pages={77--87},
  year={2025}
}

@article{brandstadt1999distance,
  title={Distance approximating trees for chordal and dually chordal graphs},
  author={Brandst{\"a}dt, Andreas and Chepoi, Victor and Dragan, Feodor},
  journal={Journal of Algorithms},
  volume={30},
  number={1},
  pages={166--184},
  year={1999},
  publisher={Elsevier}
}

@inproceedings{bonamy2026meta,
  title={Meta-Theorems for Cuttable Distributed Problems},
  author={Bonamy, Marthe and Das, Avinandan and Gavoille, Cyril and Picavet, Timoth{\'e} and Suomela, Jukka and Wesolek, Alexandra},
  booktitle={ACM Symposium on Principles of Distributed Computing},
  pages={382--389},
  year={2026}
}

@article{alon2002covering,
  title={Covering a hypergraph of subgraphs},
  author={Alon, Noga},
  journal={Discrete Mathematics},
  volume={257},
  number={2-3},
  pages={249--254},
  year={2002},
  publisher={Elsevier}
}

@article{EPLongCircuits,
    author = {\'Etienne Birmelé and J. Adrian Bondy and Bruce A. Reed},
    title = {The {E}rd{\H{o}}s–{P}{\'o}sa property for long circuits},
    journal = {Combinatorica},
    volume = {27(2)},
    pages ={135–145}, 
    year = {2007} 
}

@article{EPLongCyclesPrescribedSet,
    author = {Henning Bruhn and Felix Joos and Oliver Schaudt},
    title = {Long cycles through prescribed vertices have the {E}rd{\H{o}}s–{P}{\'o}sa property},
    journal = {Journal of Graph Theory},
    volume = {87(3)},
    pages = {275–284}, 
    year = {2018}
}

@article{EPLongCyclesTighterFct,
    author = {Samuel Fiorini and Audrey Herinckx},
    title = {A tighter {E}rd{\H{o}}s–{P}{\'o}sa function for long cycles},
    journal = {Journal of Graph Theory}, 
    volume = {77(2)},
    pages = {111–116}, 
    year ={2014}
}

@article{EPLongCyclesTightFct,
    shorthand = {MNSW17},
    author = {Frank Mousset and Andreas Noever and Nemanja {\v{S}}korić and Felix Weissenberger},
    title = {A tight {E}rd{\H{o}}s–{P}{\'o}sa function for long cycles},
    journal = {Journal of Combinatorial Theory, Series B}, 
    volume = {125},
    pages = {21–32}, 
    year = {2017}
}

@article{EPCycleGroup,
    author = {Gollin, J. Pascal and Hendrey, Kevin and Kwon, O-joung and Oum, Sang-il and Yoo, Youngho},
    title = {A unified {E}rd{\H{o}}s–{P}{\'o}sa theorem for cycles in graphs labelled by multiple abelian groups},
    journal = {Mathematische Annalen}, 
    volume = {393},
    pages = {2507–2559}, 
    year = {2025},
}

@article{EPCycleModPrescribedSet,
    author = {Kakimura, Naonori and Kawarabayashi, Ken-ichi},
    title = {Packing cycles through prescribed vertices under modularity constraints},
    journal = {Advances in Applied Mathematics}, volume = {49(2)},
    pages = {97–110}, 
    year = {2012}
}

@article{EPCyclesConnGroup,
    author = {Kawarabayashi, Ken-ichi and Wollan, Paul},
    title = {Non-zero disjoint cycles in highly connected group labelled graphs},
    journal = {Journal of Combinatorial Theory, Series B}, 
    volume = {96(2)},
    pages = {296–301},
    year = {2006},
}

@article{EPCycleOddConn,
    author = {Dieter Rautenbach and Bruce Reed},
    title = {The {E}rd{\H{o}}s-{P}{\'o}sa property for odd cycles in highly connected graphs},
    journal = {Combinatorica}, 
    volume = {21(2)},
    pages = {267–278}, 
    year = {2001}
}

@article{EPCycleOddConn2,
    author = {Carsten Thomassen},
    title = {The {E}rd{\H{o}}s-{P}{\'o}sa property for odd cycles in graphs of large connectivity},
    journal = {Combinatorica}, 
    volume = {21(2)},
    pages = {321–333}, 
    year = {2001}
}

@article{EPCycleMod,
    author = {Paul Wollan},
    title = {Packing cycles with modularity constraints},
    journal = {Combinatorica}, 
    volume = {31(1)},
    pages = {95–126},
    year = {2011}
}

@article{EPCycleOddPrescribed,
    author = {Felix Joos}, 
    title = {Parity linkage and the {E}rd{\H{o}}s–{P}{\'o}sa property of odd cycles through prescribed vertices in highly connected graphs},
    journal = {Journal of Graph Theory}, 
    volume = {85(4)},
    pages = {747–758}, 
    year = {2017}
}

@article{EPCyclePrescribed,
    author = {Kakimura, Naonori  and Kawarabayashi, Ken-ichi and Marx, Dániel},
    title = {Packing cycles through prescribed vertices},
    journal = {Journal of Combinatorial Theory, Series B},
    volume = {101(5)},
    pages = {378–381}, 
    year = {2011}
}

@article{EPCyclePrescribed2,
    author = {Matteo Pontecorvi and Paul Wollan}, 
    title = {Disjoint cycles intersecting a set of vertices},
    journal = {Journal of Combinatorial Theory, Series B}, 
    volume = {102(5)},
    pages = {1134–1141}, 
    year = {2012}
}

@article{EPPlanar,
    author = {Batenburg, Wouter Cames van and Huynh, Tony and Joret, Gwena{\" e}l and Raymond, Jean-Florent},
	journal = {Advances in Combinatorics},
	doi = {10.19086/aic.10807},
	year = {2019},
	title = {A tight {Erd}{\H o}s-{P}{\' o}sa function for planar minors},    
}

@article{birmele2003tree,
  title={Tree-width and circumference of graphs},
  author={Birmelé, Étienne},
  journal={Journal of Graph Theory},
  volume={43},
  number={1},
  pages={24--25},
  year={2003},
  publisher={Wiley Online Library}
}

@article{tychonoff1935funktionenraum,
  title={{\"U}ber einen {F}unktionenraum},
  author={Tychonoff, Andrey},
  journal={Mathematische Annalen},
  volume={111},
  number={1},
  pages={762--766},
  year={1935},
  publisher={Springer}
}

@unpublished{divoux2025asymptotic,
  title={Asymptotic structure. {V}. {T}he coarse {M}enger conjecture in bounded path-width},
  author={Divoux, Alex and Nguyen, Tung and Scott, Alex and Seymour, Paul},
  note={\arxiv{2509.08762}},
  year={2025}
}

@unpublished{AKAnticompleteLongS,
    author = {Ahn, Jungho and Kwon, O-joung},
    title = {Erd{\H{o}}s-{P}{\'o}sa property for induced packings of long {S}-cycles},
    note = {\arxiv{2608.22349}},
    year = {2026}
}

@unpublished{nguyen2025asymptoticpath,
  title={Asymptotic structure. {I}{I}. {P}ath-width and additive quasi-isometry},
  author={Nguyen, Tung and Scott, Alex and Seymour, Paul},
  note={\arxiv{2509.09031}},
  year={2025}
}

@article{AHJKWCoarseMenger,
    author = {Albrechtsen, Sandra and Huynh, Tony and Jacobs, Raphael W. and Knappe, Paul and Wollan, Paul},
    title = {A {M}enger-type theorem for two induced paths},
    journal = {SIAM Journal on Discrete Mathematics},
    volume = {38},
    number = {2},
    pages = {1438-1450},
    year = {2024}
}

@incollection{GroAsyInv,
	Author = {Gromov, Misha},
	Title = {{Asymptotic invariants of infinite groups}},
	Booktitle = {{Geometric group theory, Vol.~2 (Sussex, 1991)}},
	Number = {182},
	Pages = {1--295},
	Publisher = {Camb.\ Univ.~Press},
	Series = {London Math.~Soc.~Lecture Note Ser.},
	Year = {1993},
}

@book{DruKapBook,
	Author = {Druţu, Cornelia and Kapovich, Michael},
	Note = {Electronic version available at http://people.maths.ox.ac.uk/drutu/tcc2/ChaptersBook.pdf},
	Publisher = {AMS},
	Title = {{Geometric Group Theory}},
	Series = {Colloquium Publications},
	Volume = {63},
	Year = 2018
	}

@article{McMFat,
	title = {Fat minors in finitely presented groups},
	author = {MacManus, Joseph},
	journal = {Combinatorica},
    volume = {45},
    pages = {40},
    year = {2025}
}

@unpublished{BLPPSep,
      title={Coarse Balanced Separators in Fat-Minor-Free Graphs}, 
      author={{\'E}douard Bonnet and Hung Le and Marcin Pilipczuk and Michał Pilipczuk},
      note = {\arxiv{2604.11318}}, 
      year = {2026},
}

@article{papasoglu2023polynomial,
  title={Polynomial growth and asymptotic dimension},
  author={Papasoglu, Panos},
  journal={Israel Journal of Mathematics},
  volume={255},
  number={2},
  pages={985--1000},
  year={2023},
  publisher={Springer}
}

@inproceedings{abrishami2026burling,
  title={Burling graphs in graphs with large chromatic number},
  author={Abrishami, Tara and Bria{\'n}ski, Marcin and Davies, James and Du, Xiying and Masa{\v{r}}{\'\i}kov{\'a}, Jana and Rz{\k{a}}{\.z}ewski, Pawe{\l} and Walczak, Bartosz},
  booktitle={Proceedings of the 2026 Annual ACM-SIAM Symposium on Discrete Algorithms (SODA)},
  pages={3978--3998},
  year={2026},
  organization={SIAM}
}

@article{dvovrak2025asymptotic,
  title={Asymptotic dimension of intersection graphs},
  author={Dvo{\v{r}}{\'a}k, Zden{\v{e}}k and Norin, Sergey},
  journal={European Journal of Combinatorics},
  volume={123},
  pages={103631},
  year={2025},
  publisher={Elsevier}
}

@unpublished{davies2025string,
  title={String graphs are quasi-isometric to planar graphs},
  author={Davies, James},
  note ={\arxiv{2510.19602}},
  year={2025}
}

@article{liu2025assouad,
  title={Assouad--{N}agata dimension of minor-closed metrics},
  author={Liu, Chun-Hung},
  journal={Proceedings of the London Mathematical Society},
  volume={130},
  number={3},
  pages={e70032},
  year={2025},
  publisher={Wiley Online Library}
}

@article{distel2026proper,
  title={Proper minor-closed classes of graphs have {A}ssouad--{N}agata dimension 2},
  author={Distel, Marc},
  journal={Combinatorics, Probability and Computing},
  volume={35},
  number={1},
  pages={1--25},
  year={2026},
  publisher={Cambridge University Press}
}

@article{chepoi2012constant,
  title={Constant approximation algorithms for embedding graph metrics into trees and outerplanar graphs},
  author={Chepoi, Victor and Dragan, Feodor F. and Newman, Ilan and Rabinovich, Yuri and Vaxes, Yann},
  journal={Discrete \& Computational Geometry},
  volume={47},
  number={1},
  pages={187--214},
  year={2012},
  publisher={Springer}
}

@article{kerr2023tree,
  title={Tree approximation in quasi-trees},
  author={Kerr, Alice},
  journal={Groups, Geometry, and Dynamics},
  volume={17},
  number={4},
  pages={1193--1233},
  year={2023}
}

@article{berger2024bounded,
  title={Bounded-diameter tree-decompositions},
  author={Berger, Eli and Seymour, Paul},
  journal={Combinatorica},
  volume={44},
  number={3},
  pages={659--674},
  year={2024},
  publisher={Springer}
}

@article{DJMMDistanceErdosPosa,
    author = {Dujmović, Vida and Joret, Gwenaël and Micek, Piotr and Morin, Pat},
    title = {Erd{\H{o}}s-{P}\'{o}sa property of cycles that are far apart},
    journal = {Journal of the London Mathematical Society},
    volume = {114(2)},
    pages = {e70607},
    year = {2026}
}

@unpublished{CRTAnticompletePlanarEP,
    author = {Maria Chudnovsky and Amadeus Reinald and Stéphan Thomassé},
    title = {Forbidding anticomplete planar minors: {I}nduced {E}rd{\H{o}}s-{P}\'{o}sa property and Maximum Independent Set in {QP}},
    eprinttype={arXiv},
    eprint = {2607.09646},
    year = {2026},
    note = {\arxiv{2607.09646}}, 
}

@unpublished{papasoglu2026additive,
    author={Papasoglu, Panos and Swenson, Eric},
    title={Additive quasi-isometries and cacti},
    eprinttype={arXiv},
    eprint = {2609.15917},
    year = {2026},
    note = {\arxiv{2609.15917}}, 
}

@article{EG24, 
    title={Coarse Geometry of Quasi-Transitive Graphs Beyond Planarity}, 
    volume={31}, 
    url={https://www.combinatorics.org/ojs/index.php/eljc/article/view/v31i2p41}, 
    DOI={10.37236/12661}, 
    number={2}, 
    journal={The Electronic Journal of Combinatorics}, 
    author={Esperet, Louis and Giocanti, Ugo}, 
    year={2024}, 
    month={May}, 
    pages={\#P2.41}
}

@inproceedings{CCTZ26,
author = {Chang, Hsien-Chih and Conroy, Jonathan and Tan, Zihan and Zheng, Da Wei},
title = {Cutting Planarians: Planar Emulators for String Graphs},
year = {2026},
isbn = {9798400725364},
publisher = {Association for Computing Machinery},
address = {New York, NY, USA},
url = {https://doi.org/10.1145/3798129.3800917},
doi = {10.1145/3798129.3800917},
booktitle = {Proceedings of the 58th Annual ACM Symposium on Theory of Computing},
pages = {2140–2151},
numpages = {12},
location = {Salt Lake City, UT, USA},
series = {STOC '26}
}

\appendix
\section{\texorpdfstring{\cref{thm:DistanceErdosPosa}}{Theorem 9.7} for infinite graphs}\label{appendix}

In this section, we describe the minor modifications needed to adapt the proof of \cref{thm:DistanceErdosPosa} from \cite[Theorem~2]{CDGKMMS} to infinite graphs.
Unfortunately, we do not see a way to directly extend \cref{thm:DistanceErdosPosa} to the infinite setting by using its finite version.
It seems necessary to look more closely at their proof and slightly tweak it to get the infinite version. The following assumes understanding or concurrent reading of their proof.

There are essentially only three steps in the proof where one needs to be at least slightly more careful in the infinite setting. It is only the second step that requires an actual modification.

The first is that, as we do, they apply Simonovits' theorem \cite{simonovits1967new} (see \cref{lem:simon}).
One should be cautious, because the theorem is false for infinite graphs as shown, for example, by the infinite cubic tree.
However, as in our proof of \cref{lem:untangling}, they only apply \cref{lem:simon} to a finite graph.

The second is in their application of the following theorem of Alon \cite{alon2002covering}.

\begin{theorem}[\cite{alon2002covering}]\label{alon}
    For all positive integers $k$ and $c$, for every finite forest $F$ and every collection $\mathcal{A}$ of non-null subgraphs of $F$ each having at most $c$ connected components, either
    \begin{itemize}
        \item $\mathcal{A}$ has $k$ pairwise disjoint members, or
        \item there exists a subset $X \subseteq V(F)$ with $|X| \le 2c^2(k - 1)$ such that $X \cap V(A) \not= \emptyset$ for all $A \in \mathcal{A}$.
    \end{itemize}
\end{theorem}

The issue here is similar to our discussion around \cref{lem:helly}. Indeed, \cref{alon} is false for infinite graphs as can be shown by considering a 1-way infinite path and a collection of subtrees consisting of 1-way infinite subpaths.
Fortunately, similarly to our application of \cref{lem:helly}, the proof of \cref{thm:DistanceErdosPosa} from \cite[Theorem~2]{CDGKMMS} only applies \cref{alon} to a family $\mathcal{A}$ such that the connected components of subgraphs of $\mathcal{A}$ all have finite radius (at least in the case of their theorem that we use, being $\ell = 3$, as we will discuss last).

Therefore, this issue is resolved by instead applying the following infinite version of \cref{alon} with the extra finite radius condition on the subgraphs in $\mathcal{A}$.
We provide a proof of this infinite extension using its finite version, \cref{alon}, and \cref{lem:helly} (or more precisely its $k=2$ case which is basically the Helly property).

\begin{theorem}\label{alon2}
    For all positive integers $k$ and $c$, for every forest $F$ and every collection $\mathcal{A}$ of non-null subgraphs of $F$ each having at most $c$ connected components and whose connected components all have finite radius, either
    \begin{itemize}
        \item $\mathcal{A}$ has $k$ pairwise disjoint members, or
        \item there exists a subset $X \subseteq V(F)$ with $|X| \le 2c^2(k - 1)$ such that $X \cap V(A) \not= \emptyset$ for all $A \in \mathcal{A}$.
    \end{itemize}
\end{theorem}

\begin{proof}
Assume that $\mathcal{A}$ has no $k$ pairwise disjoint members.
For each $A\in \mathcal{A}$, let $C(A)$ be the set of its connected components (of which there are at most $c$).
Let
\[
Z = \prod_{A\in \mathcal{A}}
\left(
[2c^2(k - 1)] \times C(A)
\right).
\]
Note that $Z$ is compact by Tychonoff's theorem \cite{tychonoff1935funktionenraum}.
For $z\in Z$ and $A\in \mathcal{A}$, we let $(j^z_A, C^z_A)$ denote the $A$-coordinate of~$z$.

Consider some finite nonempty $\mathcal{B} \subseteq \mathcal{A}$.
Let $Z_\mathcal{B}$ be the set of elements~$z$ of $Z$ such that for every $A,B\in \mathcal{B}$ with $j_A^z=j_{B}^z$, we have that $C_A^z$ and $C_B^z$ intersect.
We first show that $Z_{\mathcal{B}}$ is nonempty.
For this, take some minimal subgraph $F_\mathcal{B}$ of $F$ such that for every connected component $C$ of every $B\in \mathcal{B}$, the subgraph $F_\mathcal{B} \cap C$ of~$F$ is non-null and connected, and for every intersecting pair $B,B'\in \mathcal{B}$, the subgraphs $F_\mathcal{B} \cap B$ and $F_\mathcal{B} \cap B'$ intersect.
Then $F_\mathcal{B}$ is finite because $\mathcal{B}$ is finite.
So, applying \cref{alon} to $F_{\mathcal{B}}$ and $\{B \cap F_{\mathcal{B}} : B \in \mathcal{B}\}$, there exists some $X_\mathcal{B} \subseteq V(F_{\mathcal{B}}) \subseteq V(F)$ with $n:=|X_\mathcal{B}|\le 2c^2(k - 1)$ such that $X_\mathcal{B} \cap V(B) \not= \emptyset$ for all $B \in \mathcal{B}$.
Indeed, if there were $B_1, \ldots, B_k \in \mathcal{B}$ such that the subgraphs $B_i \cap F_{\mathcal{B}}$ are pairwise disjoint, then the subgraphs $B_i$ would also be pairwise disjoint by construction of $F_{\mathcal{B}}$, so we would have found $k$ pairwise disjoint members of $\mathcal{A}$, which we assumed does not exist.
Thus, enumerating $X_\mathcal{B}$ by $\{x_1, \dots, x_n\}$ and letting $z$ be defined by setting, for every $B\in \mathcal{B}$, its $B$-coordinate to be $(i,D_{B,i})$ where $i \in [n]$ and $D_{B,i} \in C(B)$ are such that $x_i \in D_{B,i}$ (and choosing all other coordinates arbitrarily), yields $z \in Z_{\mathcal{B}}$.
This shows that $Z_\mathcal{B}$ is nonempty.

Now $\{Z_\mathcal{B} : \mathcal{B}\subseteq \mathcal{A}, |\mathcal{B}|<\infty\}$ is a family of closed sets with the finite intersection property.
Indeed, each set $Z_{\mathcal{B}}$ is closed because there are only finitely many coordinates ($|\mathcal{B}|$ many) where a $z \in Z_{\mathcal{B}}$ cannot take on any value, and for the remaining coordinates, each has at most finitely many choices.
Moreover, given any finite collection of finite subsets $\mathcal{B}_1, \dots, \mathcal{B}_n$ of~$\mathcal{A}$, the set $Z_{\mathcal{B}'}$, where $\mathcal{B}' := \bigcup_{i \leq n} \mathcal{B}_i$, is nonempty and contained in each $Z_{\mathcal{B}_i}$.
So, by compactness, there is some $z\in Z$ that is contained in every $Z_{\mathcal{B}}$.
Now for each $1\le i \le 2c^2(k - 1)$, we apply \cref{lem:helly} for $k=2$ to the set $\mathcal{C}_i := \{C_A^z : A \in \mathcal{A} \text{ and } j_A^z = i\}$. By the choice of~$z$, any two $C \in \mathcal{C}_i$ intersect, which by \cref{lem:helly} implies that there exists some vertex $x_i\in V(F)$ such that $x_i \in V(C_A^z)$ for all $C_A^z \in \mathcal{C}_i$. In this case, $x_i \in A$.
Then, taking $X:=\{x_1, \ldots , x_{2c^2(k - 1)}\}$, we have that $X \cap V(A) \neq \emptyset$ for all $A \in \mathcal{A}$ (because $x_{j^z_A} \in A$), as desired.
\end{proof}

The last minor technicality relates to the second.
They apply a theorem of Birmelé \cite{birmele2003tree} that graphs without a cycle of length at least $\ell$ have tree-width at most $\ell-1$.
Although this theorem holds for infinite graphs, the potential issue is that \cref{alon} is applied using the tree-decomposition $(T, \mathcal{V})$ given by \cite{birmele2003tree}. If for some vertex $v$, the subtree $T_v$ had infinite radius in~$T$, then in their process the condition of \cref{alon2} that each connected component of each member of $\mathcal{A}$ has finite radius could be violated.
However, we are only interested in the case of their theorem for cycles of length at least $\ell=3$, not longer cycles as they study. In this case they are again only applying \cite{birmele2003tree} in the $\ell=3$ case, giving a forest.
So then this technicality of considering a tree-decomposition rather than just the actual forest becomes rather redundant.
Even if one still follows their tree-decomposition route, then by choosing $(T, \mathcal{V})$ so that each $T_v$ has radius at most~$1$ in~$T$ (as can be done in the case of forests), the resulting collection $\mathcal{A}$ of subgraphs will still have the desired property of their connected components having finite radius.

\end{document}